\documentclass[hidelinks]{memo-l}

\usepackage{amssymb,
enumitem,
mathrsfs,
hyperref,
tikz,
upgreek
}
\usepackage[T1]{fontenc}
\usepackage[shadow,loadshadowlibrary]{todonotes}

\usepackage{etoolbox}
\usepackage{ifthen}
\newcommand{\addaxioms}[1]{\ifthenelse{\equal{#1}{0}}{\toggletrue{ax}}{\togglefalse{ax}}}
\newtoggle{ax}
\addaxioms{1} 
\newcommand{\axioms}[1]{\iftoggle{ax}{\todo[color=blue!50]{#1}}{}}

\newcommand{\RR}{\mathbb{R}}

\newcommand{\pre}[2]{{}^{#1}#2}

\newcommand{\Nbhd}{\boldsymbol{N}}

\newcommand{\markdef}[1]{\textbf{#1}}

\newcommand{\lS}{\lambda\text{-}\boldsymbol{\Sigma}}
\newcommand{\lP}{\lambda\text{-}\boldsymbol{\Pi}}
\newcommand{\lD}{\lambda\text{-}\boldsymbol{\Delta}}
\newcommand{\lB}{\lambda\text{-}\mathbf{Bor}}
\newcommand{\gBor}[1]{#1\text{-}\mathbf{Bor}}
\newcommand{\lPSP}{\ensuremath{\lambda \text{-}\mathrm{PSP}}}
\newcommand{\lwPSP}{\ensuremath{\lambda \text{-}\mathrm{wPSP}}}
\newcommand{\lbPSP}{\ensuremath{\lambda \text{-}\mathrm{BorPSP}}}
\newcommand{\U}{\mathcal{U}}
\newcommand{\V}{\mathcal{V}}
\newcommand{\UU}{\mathbb{U}}
\newcommand{\pI}{\mathbf{I}}
\newcommand{\pII}{\mathbf{II}}
\newcommand{\On}{\mathrm{Ord}}
\newcommand{\Cn}{\mathrm{Card}}

\newcommand{\id}{\operatorname{id}}
\newcommand{\pow}{\mathscr{P}}

\newcommand{\lh}{\operatorname{lh}}

\newcommand{\dom}{\operatorname{dom}}

\newcommand{\cf}{{\operatorname{cf}}}

\newcommand{\lev}{\operatorname{lev}}
\newcommand{\p}{p} 
\newcommand{\Def}{\operatorname{def}}
\newcommand{\weight}{\operatorname{wt}}
\newcommand{\dens}{\operatorname{dens}}

\newcommand{\ran}{\operatorname{rng}}

\newcommand{\crt}{\operatorname{crt}}
\newcommand{\Ult}{\operatorname{Ult}}

\newcommand{\ZF}{{\sf ZF}}
\newcommand{\AC}{{\sf AC}}
\newcommand{\DC}{{\sf DC}}
\newcommand{\AD}{{\sf AD}}
\newcommand{\GCH}{{\sf GCH}} 
\newcommand{\ZFC}{{\sf ZFC}}

\newenvironment{enumerate-(a)}{\begin{enumerate}[label={\upshape (\alph*)}, leftmargin=2pc]}{\end{enumerate}}
\newenvironment{enumerate-(a)-r}{\begin{enumerate}[label={\upshape (\alph*)}, leftmargin=2pc,resume]}{\end{enumerate}}
\newenvironment{enumerate-(a)-5}{\begin{enumerate}[label={\upshape (\alph*)}, leftmargin=2pc,start=5]}{\end{enumerate}}
\newenvironment{enumerate-(A)}{\begin{enumerate}[label={\upshape (\Alph*)}, leftmargin=2pc]}{\end{enumerate}}
\newenvironment{enumerate-(A)-r}{\begin{enumerate}[label={\upshape (\Alph*)}, leftmargin=2pc,resume]}{\end{enumerate}}
\newenvironment{enumerate-(i)}{\begin{enumerate}[label={\upshape (\roman*)}, leftmargin=2pc]}{\end{enumerate}}
\newenvironment{enumerate-(i)-r}{\begin{enumerate}[label={\upshape (\roman*)}, leftmargin=2pc,resume]}{\end{enumerate}}
\newenvironment{enumerate-(I)}{\begin{enumerate}[label={\upshape (\Roman*)}, leftmargin=2pc]}{\end{enumerate}}
\newenvironment{enumerate-(I)-r}{\begin{enumerate}[label={\upshape (\Roman*)}, leftmargin=2pc,resume]}{\end{enumerate}}
\newenvironment{enumerate-(1)}{\begin{enumerate}[label={\upshape (\arabic*)}, leftmargin=2pc]}{\end{enumerate}}
\newenvironment{enumerate-(1)-r}{\begin{enumerate}[label={\upshape (\arabic*)}, leftmargin=2pc,resume]}{\end{enumerate}}
\newenvironment{itemizenew}{\begin{itemize}[leftmargin=2pc]}{\end{itemize}}

\newtheorem{theorem}{Theorem}[section]
\newtheorem{lemma}[theorem]{Lemma}
\newtheorem{corollary}[theorem]{Corollary}
\newtheorem{proposition}[theorem]{Proposition}

\newtheorem{fact}[theorem]{Fact}
\newtheorem{claim}{Claim}[theorem]
\theoremstyle{definition}
\newtheorem{defin}[theorem]{Definition}
\newtheorem{example}[theorem]{Example}

\theoremstyle{remark}
\newtheorem{remark}[theorem]{Remark}

\newenvironment{proofsketch}{\begin{proof}[Sketch of the proof]}{\end{proof}}
\newenvironment{proofsketchcite}[2]{\begin{proof}[Sketch of the proof] (See~\cite[#2]{#1} for more details.)}{\end{proof}}
\newenvironment{proofcompare}[2]{\begin{proof}[Proof] (Compare with~\cite[#2]{#1}.)}{\end{proof}}

\numberwithin{section}{chapter}
\numberwithin{equation}{chapter}

\makeindex

\begin{document}

\frontmatter

\title{Generalized Descriptive Set Theory \\ at Singular Cardinals of Countable Cofinality}


\author[V.~Dimonte]{Vincenzo Dimonte}
\address{Dipartimento di Scienze Matematiche, Informatiche e Fisiche, Universit\`{a} degli Studi di Udine, via delle Scienze 206, 33100 Udine --- Italy}
\email{vincenzo.dimonte@uniud.it}
\thanks{}

\author[L.~Motto Ros]{Luca Motto Ros}
\address{Dipartimento di matematica \guillemotleft{Giuseppe Peano}\guillemotright, Universit\`a degli Studi di Torino, Via Carlo Alberto 10, 10121 Torino --- Italy}
\email{luca.mottoros@unito.it}
\thanks{Research supported by the project PRIN~2017 ``Mathematical Logic: models, sets, computability'', prot.~2017NWTM8R and the project PRIN 2022 ``Models, Sets and Classifications'' Code no. 2022TECZJA CUP G53D23001890006, funded by the European Union – NextGenerationEU – PNRR M4 C2 I1.1. The authors are members of GNSAGA (INdAM)}

\date{\today}





\begin{abstract}
We provide a comprehensive development of the basics of descriptive set theory for non-separable complete metric spaces whose weight is a singular cardinal \( \lambda \) of countable confinality. Somewhat unexpectedly, the resulting theory is remarkably similar to the classical one, although the methods used are necessarily fairly different and combine ideas and results from general topology, infinite combinatorics, and set theory. 

More in detail, we study \(\lambda\)-Polish spaces and standard \(\lambda\)-Borel spaces (characterization of the generalized Cantor and Baire spaces, analogues of the Cantor-Bendixson theorem, classification up to \(\lambda\)-Borel isomorphism, etc.), their \(\lambda\)-Borel hierarchy (structural properties, changes of topologies, and so on), \(\lambda\)-analytic sets (including generalizations of the Lusin separation theorem and of the Souslin theorem), \( \lambda \)-coanalytic sets (including \( \lP^1_1 \)-ranks and alike), and \(\lambda\)-projective sets. We also consider more advanced topics, and provide e.g.\ various uniformization results for \(\lambda\)-Borel set; these in turn lead to fundamental applications to the study of \(\lambda\)-Borel equivalence relations, such as a generalization of the celebrated Feldman-Moore theorem. Finally, we study a natural generalization of the classical Perfect Set Property, and develop tools to show that all definable sets enjoy such property under suitable large cardinal assumptions, most notably including Woodin's \( \mathsf{I0}(\lambda) \).
\end{abstract}

\maketitle

\tableofcontents


\mainmatter


%
%
%

\chapter{Introduction}

\section{Generalized descriptive set theory} \label{ref:GDST}

According to the standard reference~\cite{Kechris1995} for the subject,
\begin{quotation}
Descriptive set theory is the study of ``\textbf{definable sets}'' in \textbf{Polish} (i.e.,\ 
separable completely metrizable) spaces. In this theory, sets are classified in
hierarchies, according to the complexity of their definitions, and the 
structure of the sets in each level of these hierarchies is systematically analyzed.
\end{quotation}
A central role in classical descriptive set theory is played by the Cantor space \( \pre{\omega}{2} \) of all 
infinite binary sequences and the Baire space \( \pre{\omega}{\omega} \) of all infinite 
sequences of natural numbers, both equipped with the product topology. The standard basis for \( \pre{\omega}{\omega} \) is constituted by the sets 
\( \Nbhd_s = \{ x \in \pre{\omega}{\omega} \mid s \subseteq x \} \), where \( s \) ranges 
over finite sequences of natural numbers, and similarly for the Cantor space. For several results (more precisely, for those results relying only on the Borel 
structure of a Polish space), one can safely restrict the attention to such spaces, which 
are well-understood and have a rich combinatorial structure. Nevertheless, a great part 
of the success of descriptive set theory in finding applications in various branches of 
mathematics is due to the fact that Polish spaces are ubiquitous: it is arguably the 
interplay between this general setup embracing the most diverse situations and the 
manageability of spaces like the Cantor and the Baire space which allowed for the great 
impact this area had in modern mathematics.

As for structural properties of definable subsets, one of the first and more important ones is certainly the Perfect Set Property, i.e.\ whether the Cantor space \( \pre{\omega}{2} \) can be topologically embedded in those subsets of  Polish spaces which have cardinality greater than \( \omega \). This is a topological version of the Continuum Hypothesis \( \mathsf{CH} \), but it is stronger than that. Cantor's insight was to show that such property holds for closed sets, and this was later extended, working in \( \mathsf{ZFC} \) alone, to all Borel and even analytic sets. Using deep set theoretical techniques and exploiting the special role of the spaces \( \pre{\omega}{2} \) and \( \pre{\omega}{\omega} \), this regularity property was pushed beyond the level of analytic sets. 
One way to achieve this is to work with the so-called \( \kappa \)-weakly homogeneously Souslin sets and split the problem in two subtasks, that is:
\begin{theorem}[{\cite[Theorem 32.7]{Kanamori}}]\label{thm:intro1}
If \( A \) is \( \kappa \)-weakly homogeneously Souslin, then \( A \) has the Perfect Set Property.
\end{theorem}

\begin{theorem}[{\cite{MaStee1988}}]\label{thm:intro2}
Assuming the existence of sufficiently large cardinals, all projective sets (and in fact all sets in $L(\RR)$) are \( \kappa \)-weakly homogeneously Souslin for a suitable measurable cardinal \( \kappa \).
\end{theorem}

Given the great achievements of (classical) descriptive set theory, it is natural to ask whether one can extend it to more general situations. Indeed, in the literature it is easy to find research directions which might be (and, at least in some cases, have been) called \emph{generalized} descriptive set theory. We describe here three important instances of this phenomenon, in which very important and deep results have been obtained through the work of uncountably many great mathematicians, often with very different backgrounds.

\smallskip

(A) In the realm of general topology, one can find various attemps of adapting the techniques from descriptive set theory to non-separable complete metric spaces.%
\footnote{Sometimes these attemps undergo the name of \emph{non-separable} descriptive set theory, see e.g.~\cite{Hansell1983}.} 
For example, A.\ H.\ Stone~\cite{Stone1962} provides a deep analysis of definable subsets of the so-called \emph{\( k\)-Baire space} \( B(k) \), where \( k \) is an arbitrary uncountable cardinal, and of the \( k \)-Cantor space \( C(k) \) for cardinals \( k \) with countable cofinality.
Both spaces are completely metrizable spaces of weight \( k \), and will play an important role in our work as well, see Chapter~\ref{sec:lambda-Polish}. 
Although the notion of a \( k \)-analytic set introduced in~\cite[Section 8]{Stone1962} is equivalent to ours when \( k = \lambda \) is an uncountable cardinal that fits the setup of this paper (see Chapter~\ref{sec:lambda-analytic}), Stone chose to consider only \emph{classical} Borel sets rather than their generalization to \( k^+ \)-Borel sets (see Chapter~\ref{sec:Borel}). As discussed at the end of~\cite{Stone1962}, this causes a serious mismatch between the two notions, and in particular the failure of the natural generalization of Souslin's theorem relating bi-\(k\)-analytic sets to the Borel ones. 
Hansell~\cite{Han72,Han73} exploited instead \( \sigma \)-discrete unions and related concepts to introduce ``hyper-Borel sets'' and ``\( k \)-Souslin%
\footnote{Hansell's definition is different from the notion of \( \kappa \)-Souslin set considered in set theory and discussed in Section~\ref{sec:Souslinsets}.} 
sets'' as natural generalizations to the non-separable context of the classical Borel and analytic (or \( \aleph_0 \)-Souslin) sets, respectively. Although in~\cite{Han72} it is shown that a Souslin-like theorem can be obtained in this setup, in~\cite{Han73} the author realized that his notion of ``\( k \)-Souslin sets'' fails to properly generalize the classical notion of analytic sets, and in particular it is much weaker than Stone's notion of a \( k \)-analytic set.

\smallskip

(B) \label{approachB} In set theory, another natural generalization of descriptive set theory is obtained by replacing \( \omega \) with an uncountable cardinal \( \kappa \) in all basic definitions and concepts. In this way, one is led to study spaces of \( \kappa \)-sequences like the generalized Cantor space \( \pre{\kappa}{2} \) or the generalized Baire space \( \pre{\kappa}{\kappa} \), both endowed with a natural topology which here we call bounded topology (see Section~\ref{sec:boundedtopology}). This idea reappeared several times in the literature (see e.g.~\cite{Vaught1974,MekVaa1993}),
but a systematic development of the theory has been initiated only about 15 years ago.
Before this paper, only uncountable cardinals \( \kappa \) satisfying \( \kappa^{<\kappa} = \kappa \) were considered.%
\footnote{A notable exception is~\cite{AM}, which however contains some red herrings. For example, in that work the authors always consider \( \pre{\kappa}{\kappa} \) as the generalization of the Baire space, a choice that is arguably wrong when \( \kappa \) is singular.} 
This is mostly due to technical reasons, but it has the unpleasant consequence of excluding all singular cardinals at once,
and thus it makes this approach incompatible with that of part (A) because \( \pre{\kappa}{2} \) is metrizable if and only if \( \kappa \) has countable cofinality.
This line of research was quite successful and led to deep connections with e.g.\ Shelah's stability theory in model theory (\cite{Friedman:2011nx,Mangraviti2018,Moreno2017,MorenoUnp}). However, it has also several weaknesses. Besides the limitation on the cofinality of \( \kappa \), the theory is often confined to the study of the two spaces \( \pre{\kappa}{2} \) and \( \pre{\kappa}{\kappa} \), and most of the natural generalizations of the classical results turn out to be independent of \( \mathsf{ZFC} \), if not plainly false (like the generalization of the Souslin's theorem).
Nevertheless, a general trend has emerged: the stronger are the large cardinal conditions that \( \kappa \) satisfies, the more stable is the resulting generalized descriptive set theory. But interestingly enough, the largest cardinal assumptions conceived so far, like Woodin's axiom \( \mathsf{I0} \), imply that the cardinal at hand has countable cofinality, and thus it escapes the framework under discussion. 

\smallskip

(C) \label{approachC} In his seminal paper~\cite{Woodin2011}, Woodin realized that the initial segment \( V_{\lambda+1} \)  of the von Neumann's hierarchy of sets can be equipped with a somewhat natural topology, and observed that many structural results about its definable subsets under \( \mathsf{I0}(\lambda) \) exhibit a striking analogy with those concerning the definable subsets of \( \pre{\omega}{2} \)  under the Axiom of Determinacy \( \mathsf{AD} \).%
\footnote{Recall that \( \AD \) is the assertion ``Every \( A \subseteq \pre{\omega}{\omega} \) is determined''. Although it contradicts the full Axiom of Choice \( \AC \), it is compatible with weaker forms of it. For example, under sufficiently large cardinals assumptions we have that \( L(\RR) \models \ZF + \AD + \DC(\RR) \).}
Although it is natural to conceive the gathering of those results as a generalization of classical descriptive set theory, it must be noted that the analogy is not developed further, and that the motivations and the techniques used in the \( \mathsf{I0} \) context are quite different from the classical ones. 
Also, the space \( V_{\lambda+1} \) is a bit exotic for the average mathematician, and it is unclear whether such results can be applied to other reasonable spaces. Therefore, here the interest appears to be more on the metamathematical side, and indeed few applications outside set theory were to be expected.

\smallskip

Although the above approaches are apparently unrelated and sometimes even incompatible with each other,
our purpose is precisely to reconcile all of them, taking advantage of the strengths of each one to overcome the limitations of the other ones.
 Our work results in a very elegant and smooth theory, significantly mimicking the classical one in many respects. 
 A side effect of this unified approach is that of separating those results from approach (C) that are really technical and specific to set theory (such as 
 the proof that a wide class of sets in \( L(V_{\lambda+1}) \) are \( \UU(j) \)-representable under \( \mathsf{I0}(\lambda) \), which might be viewed as a higher analogue of Theorem~\ref{thm:intro2}), from the practical consequences that, in analogy with Theorem~\ref{thm:intro1}, can be derived from them without large cardinal assumptions and in a more accessible way, such as the \(\lambda\)-Perfect Set Property.

\section{Setup and main results} \label{sec:setup}

The crucial move in realizing our program is to concentrate on singular uncountable cardinals \(\lambda\) of countable cofinality. It turns out that under mild assumptions on \(\lambda\), all prominent spaces from the three approaches (A), (B), and (C) described in Section~\ref{ref:GDST} (that is: \( B(\lambda) \), \( C(\lambda) \), \( \pre{\lambda}{2} \), and \( V_{\lambda+1} \)) are homeomorphic to each other (Theorem~\ref{thm:homeomorphictoCantor}), allowing us to take advantage at once of the strengths of each of those setups. There are basically four conditions on \(\lambda\) that will be used, listed here in increasing strength order:
\begin{enumerate-(1)}
\item \label{requirementsonlambda-1}
\underline{\(\lambda\) is \(\omega\)-inaccessible}, that is, \( \kappa^\omega < \lambda \) for all \( \kappa < \lambda \). This makes the theory of \(\lambda\)-Polish spaces smoother and allows for very useful cardinality computations in metric spaces (see e.g.\ Lemma~\ref{lem:densityvscardinality} and Corollary~\ref{cor:cardinalityunderomegainaccessible}). It also makes a bit of choice available, as it implies e.g.\ that \( \RR \) is well-orderable and that \( \lambda \) is larger than the continuum \( 2^{\aleph_0} \).
\item \label{requirementsonlambda-2}
\underline{\( 2^{< \lambda} =  \lambda\)} or, equivalently, \(\lambda\) is strong limit. This is equivalent to the fact that the generalized Cantor space \( \pre{\lambda}{2} \) from approach (B) is a \(\lambda\)-Polish space. It is also a crucial condition for e.g.\ the \(\lambda\)-Borel sets being well-behaved.
\item \label{requirementsonlambda-3}
\underline{\( \beth_\lambda = \lambda \)}, that is, \(\lambda\) is a fixed point of the beth function (see the end of Section~\ref{subsec:ordandcard}). It is a condition equivalent to \( |V_\lambda| = \lambda \), which in turn is equivalent to the fact that the space \( V_{\lambda+1} \) from approach (C) is \(\lambda\)-Polish when equipped with Woodin's topology. For this reason it becomes obviously relevant each time that we want to employ Woodin's machinery.
\item \label{requirementsonlambda-4}
\underline{Woodin's \(\mathsf{I0}(\lambda) \)}. This is one of the strongest large cardinal assumtpions ever conceived. As we will show, it implies that all definable subsets of \(\lambda\)-Polish spaces are quite well-behaved, and in particular that they all have the \(\lambda\)-Perfect Set Property. 
\end{enumerate-(1)}

As for the axiomatic setup, unless otherwise specified our basic theory will be 
\[ 
\ZF + \AC_\lambda(\pre{\lambda}{2}) .
\]
(See Section~\ref{subsec:weakchoice} for the definitions and more information on weak choice axioms.) 
This is in line with what is done in classical descriptive set theory, whose basics are developed in \( \ZF + \AC_\omega(\RR) \), and allows for the use of our results in inner models in which the full Axiom of Choice \( \AC \) may fail, e.g.\ the model \( L(V_{\lambda+1}) \) under \( \mathsf{I0}(\lambda) \). Some of the initial results from Chapter~\ref{sec:lambda-Polish} work in even weaker theories, like \( \ZF \) alone, \( \ZF + \AC_\omega(\pre{\lambda}{2}) \), or \( \ZF + \AC_\lambda (\pre{\omega}{\lambda}) \) (but notice that \( \AC_\lambda (\pre{\omega}{\lambda}) \) is equivalent to \( \AC_\lambda(\pre{\lambda}{2}) \) if \(\lambda\) is strong limit). In the opposite direction, the reader uninterested in these subtleties may safely assume that we are working in the standard theory \( \ZFC \) (in which case some of the proofs could even be simplified), and skip all comments clarifying to what extent certain results can be obtained in choiceless settings.

Our presentation follows closely the structure of~\cite{Kechris1995}, to allow for a timely comparison between the classical results and those in the generalized context. When no new argument is needed, we just refer the reader to the corresponding result in~\cite{Kechris1995}, unless certain subtleties about the construction or the amount of choice needed to carry it out force us to discuss more details.
Of course, even if the statements are often natural generalizations of the classical one, in most cases a new proof is in order. 
It might be worth noticing that in some cases the new arguments are specific to the uncountable setting, but on a few occasions they also work for \( \lambda = \omega \), and thus they provide alternate proofs of the corresponding classical theorems. 
Sometimes this also leads to results in classical descriptive set theory which apparently were overlooked in the literature (see e.g.\ Proposition~\ref{prop:surjectionpreservingcompact} and Theorem~\ref{thm:changeoftopologyforuniformizations}).

We now discuss the content of the various chapters and the main results obtained in this paper.

\medskip

\noindent
\textbf{Chapter~\ref{chapter:preliminaries}.}
We collect all notions and basic results from general topology and set theory that are needed in the sequel, so the chapter can be safely skipped if the reader has a high level of familiarity with both areas. However, since we are going to work in a setting where only a small amount of choice is available, some results in topology need to be reworked in detail, so it might be worth having a look at Section~\ref{subsec:weakchoice} anyway.

\medskip

\noindent
\textbf{Chapter~\ref{sec:lambda-Polish}.} 
We consider \(\lambda\)-Polish spaces, namely, completely metrizable spaces with weight at most \( \lambda \). Our analysis reveals that while \( \pre{\lambda}{2} \) is a suitable analog of the Cantor space in the generalized context, the ``correct'' generalization of the Baire space in the countable-cofinality case is Stone's \( B(\lambda) = \pre{\omega}{\lambda} \) from~\cite{Stone1962}, and not \( \pre{\lambda}{\lambda} \), as previously believed (see e.g.~\cite{AM}).
 This is an instance of a more general phenomenon: in order to obtain significant generalizations, when \( \lambda \) is singular one cannot simply replace \(\omega\) with \(\lambda\) in all occurrences, but rather carefully choose between \(\lambda\) and its cofinality, depending on the situation at hand. 
 Also, zero-dimensionality must be intended in the sense of the Lebesgue covering dimension, and not in the sense of inductive dimension. (The two notions coincide for separable metrizable spaces, but the former is stronger in the context of non-separable spaces.)
Under mild assumptions on \(\lambda\), the theory of \( \lambda \)-Polish spaces in the countable-cofinality case turns out to be very close to that of ``classical'' Polish spaces, which correspond to the case \( \lambda = \omega \) (see Sections~\ref{sec:basicpropertiesforpolish}--\ref{sec:lambda-perfect}). There are notable exceptions, though. For example, when \(\lambda\) is uncountable there is only one type of perfect kernel for \(\lambda\)-Polish spaces \( X \) with \( \dim(X) = 0 \) (Corollary~\ref{cor:characterizationCantor2}) because, indeed, \( B(\lambda) \) is the only \(\lambda\)-perfect \(\lambda\)-Polish space with Lebesgue covering dimension \( 0 \) (Theorem~\ref{thm:characterizationCantor2}). And all \(\lambda\)-Lindel\"of (a natural analogue of \(\sigma\)-compacteness) \(\lambda\)-Polish spaces are necessarily small (Corollary~\ref{cor:lambdaLindelofissmall}). 
In many cases the arguments we use are adaptations of the ones used e.g.\ in~\cite{Kechris1995}, although some care is needed to avoid unnecessary uses of strong forms of choice axioms. 
We conclude the chapter with a brief discussion on why this setup is meaningful only when \( \cf(\lambda) = \omega \) (Section~\ref{subsec:Polishforothercofinalities}).

\medskip

\noindent
\textbf{Chapter~\ref{sec:Borel}.}
We introduce the collection \( \lB(X) \) of \(\lambda\)-Borel sets, namely, of the sets in the smallest \( \lambda^+ \)-algebra (equivalently, since \(\lambda\) is singular, \( \lambda \)-algebra) on the \(\lambda\)-Polish space \( X \) containing all its open sets (Section~\ref{sec:Boreldefandbasicfacts}). As shown in Proposition~\ref{prop:BorelvslambdaBorel}, \( \lB(X) \) is much larger than the collection of all ``classical'' Borel subsets of \( X \) considered e.g.\ by Stone~\cite{Stone1962,Stone1972}. It turns out that the general behavior of the associated \(\lambda\)-Borel hierarchy follows closely that of the classical Borel hierarchy on Polish spaces (Section~\ref{sec:Borelhierarchy}). Moving to more advanced results, there are two good news and a bad one. 
On the positive side, we still have the change-of-topology technique at disposal (Section~\ref{sec:changeoftopology}), and the levels of the \(\lambda\)-Borel hierarchy satisfy the expected structural properties (Section~\ref{sec:structuralproperties}), even though the methods used to obtain these results significantly differ from the ones employed in the classical setting. The bad news is that no form of Borel determinacy is available --- this is basically folklore, but we include a thorough discussion in Section~\ref{sec:noBoreldeterminacy} for the records. 

\medskip

\noindent
\textbf{Chapter~\ref{sec:lambda-analytic}.}
We define the collection \( \lS^1_1(X) \) of \(\lambda\)-analytic subsets of a \(\lambda\)-Polish space \( X \), and study their basic properties (Sections~\ref{sec:defanalytic} and~\ref{sec:basifactsanalytic}).
By~\cite[Theorem 19]{Stone1962}, our definition is equivalent to the one introduced in~\cite[Section 8]{Stone1962} (for a general cardinal \(\lambda\)). In that paper, however, the author just relates this notion with ``classical'' Borel subsets of \( X \), as in general it is not immediate to see that \( \lS^1_1(X) \) is closed under intersections of size \(\lambda\) --- see Proposition~\ref{prop:closurepropertiesofanalytic}. Although natural in that setup, this is a serious limitation because, as hinted in~\cite[Section 8.6]{Stone1962}, it prevents one to find e.g.\ a suitable generalization of the Lusin's separation theorem. 
This is where ideas from approach (B) on p.\ \pageref{approachB} help in overcoming the difficulties encountered in approach (A) from the same page.
Indeed, our approach 
reveals that, when restricting to strong limit cardinals \(\lambda\) of countable cofinality, the class \( \lS^1_1(X) \) has stronger closure properties and contains all \(\lambda\)-Borel sets, an improvement which ultimately relies on Theorem~\ref{thm:homeomorphictoCantor}\ref{thm:homeomorphictoCantor-2} and is an illuminating example of why it is useful to mix the approaches discussed in Section~\ref{ref:GDST}. On the one hand, this allows us to extend some of Stone's results from ``classical'' Borel sets to \(\lambda\)-Borel sets (compare e.g.\ the last part of Theorem~\ref{thm:analyticPSP} with~\cite[Theorem 6]{Stone1962}). On the other hand, one gets the natural generalization of the mentioned Lusin's separation result, 
and obtains as a consequence the generalized Soulslin's theorem stating that \( \lD^1_1(X) = \lB(X) \) 
(Section~\ref{subsec:Lusinseparation}). Other separation theorems are obtained in Section~\ref{sec:otherseparationtheorems}.
Along the way (Section~\ref{sec:standardBorel}), we also consider standard \(\lambda\)-Borel spaces, notably including the Effros \(\lambda\)-Borel space \( F(X) \) of closed subsets of \( X \) --- to deal with it, we had to employ new ideas because in the generalized context we cannot use compactification arguments for cardinality reasons. Additionally, we record a number of facts concerning definable equivalence relations and \(\lambda\)-Polish groups which readily follows from the theory already developed. We conclude the chapter with a brief discussion on \(\lambda\)-projective and \( \kappa \)-Souslin sets (Sections~\ref{sec:lambda-projective} and~\ref{sec:Souslinsets}, respectively), and on various forms of definable absoluteness (Section~\ref{sec:absolutenessnew}).

\medskip

\noindent
\textbf{Chapter~\ref{chapter:uniformization}.}
This chapter is devoted to two more advanced topics: uniformizations and ranks. In Section~\ref{sec:uniformization}, we prove various results concerning the existence of \(\lambda\)-Borel uniformizations of a given \(\lambda\)-Borel set. The crucial observation is that one can characterize this in terms of isolated points, or suitable changes of topologies (Theorem~\ref{thm:changeoftopologyforuniformizations}). Notably, the characterization works also when \( \lambda = \omega \), and appears to be new in that context as well. Moving from this, we then prove uniformization results for \(\lambda\)-Borel sets with countable vertical sections (Theorem~\ref{thm:ctblsections}), or with compact vertical sections (Theorem~\ref{thm:compactsections2}). Notice that the proofs of the corresponding results in classical descriptive set theory heavily use Baire category arguments. Since this tool is missing in the generalized context, we had to employ completely different ideas that are specific to uncountable cardinals larger than the continuum, and thus cannot be used in the classical setting. (We do not know if there is a proof covering simultaneously the classical and the generalized case.) Among the remarkable consequences of the above uniformization results, we provide analogues of the celebrated Feldman-Moore theorem for countable \( \lambda \)-Borel equivalence relations (Theorem~\ref{thm:feldman-moore}), and of Slaman-Steel's ``marker lemma'' (Lemma~\ref{lem:markerlemma}). We also briefly discuss a generalization of hyperfiniteness, called hyper-\( \kappa \)-smallness, which naturally arises in the generalized context (Lemma~\ref{lem:hypersmall} and Corollary~\ref{cor:hypersmall}). 
In Section~\ref{sec:lambdacoanalyticranks}, we instead review the theory of \( \lP^1_1 \)-ranks, leading to important results on the structural properties of \(\lambda\)-analytic and \(\lambda\)-coanalytic sets (Theorem~\ref{thm:structuralpropertiesforcoanalytic}), and on the existence of \( \lP^1_1 \)-codes for \(\lambda\)-Borel sets (Theorem~\ref{thm:borelcodes}) as a consequence of the fact that the pointclass \( \lP^1_1 \) is ranked (Theorem~\ref{thm:Pi11isranked}). We also provide a natural boundedness theorem (Theorem~\ref{thm: boundednessforanalytic}) and two reflection theorems (Theorems~\ref{thm:firstreflection} and~\ref{thm:secondreflection}). 
All these results are to be intended as tools made available for the further development of the theory in future works. Besides a few crucial observations, their proofs are straighforward generalizations of the classical arguments, and are thus omitted here. 

\medskip

\noindent
\textbf{Chapter~\ref{chapter:lambda-PSP}.}
We consider various possible generalizations of the classical Perfect Set Property, and show that they are indeed all equivalent to each other under mild assumptions on \(\lambda\). This is a consequence of the fact that under \( 2^{< \lambda} = \lambda \), all \(\lambda\)-analytic sets satisfy the strongest one, which is called \(\lambda\)-Perfect Set Property and is denoted by \lPSP{} (see Section~\ref{sec:PSPforanalytic}). Then we show that this result is best possible if we work in \( \ZF(\mathsf{C}) \) alone: as in the classical case, it is consistent that there are \(\lambda\)-coanalytic sets without the \lPSP{} (Section~\ref{sec:noPSPcoanalytic}). Section~\ref{sec:PSPgames} is devoted to the study of certain Gale-Stewart games whose determinacy implies the \lPSP{} (and in some cases it is even equivalent to it). Finally, in Section~\ref{sec:representalePSP} we introduce the concepts of \( \UU \)-representable and TC-representable sets. This is a generalization on the notion of \( \kappa \)-weakly homogenously Souslin sets, and is inspired by Woodin's definition of $\mathbb{U}(j)$-representable sets under \( \mathsf{I0} \) (see~\cite{Woodin2011}).
The main result is that if \( 2^{< \lambda} = \lambda \), then all TC-representable sets have the \lPSP{} (Theorem~\ref{thm:mainPSP}). This fact is crucial to unlock the applications considered in next chapter.

\medskip

\noindent
\textbf{Chapter~\ref{sec:applications}.}
The goal of this chapter is twofold. First, we connect \( \UU \)-representability and TC-representability to similar notions already present in the literature. In particular,
we show that setting back \( \lambda = \omega \) in our definitions of representability we recover the well-studied notion of weakly homogeneously Souslin sets (Section~\ref{sec:representabilityvsweaklyhomogeneouslysouslin}), while working under the large cardinal assumptions \( \mathsf{I0}(\lambda) \) our notions align with Woodin's \( \UU(j) \)-representability and his version of the tower condition (Section~\ref{sec:representabilityvsU(j)-representability}). This is no surprise, as our definitions are clearly inspired to those older concepts, but it reinforces the claim that the theory developed in this paper is really a higher analogue of classical descriptive set theory, giving rise to a very similar picture (perhaps surprisingly, given the outcome of the development of generalized descriptive set theory for regular cardinals undertaken in approach (B) from p.\ \pageref{approachB}).
The second goal is to study how large the collection of sets having the \( \lPSP \) can be under large cardinal assumptions, and in particular what is the influence of Woodin's \( \mathsf{I0}(\lambda) \) in this respect (Section~\ref{sec:whichsetsarerepresentable}). This is where approach (C) on p.\ \pageref{approachC} comes into play. Exceptionally, this part is developed within \( \ZFC \), as customary in the area.
Among other things, we reprove a beautiful and deep theorem of Cramer stating that under \( \mathsf{I0}(\lambda) \), all the \(\lambda\)-projective sets (and indeed also all sets in the inner model \( L(V_{\lambda+1}) \)) have the \( \lPSP \). Finally, in Section~\ref{sec:furtherresults} we make some further observations on how tight is the connection between the axiom \( \mathsf{I0}(\lambda) \) and its consequences on the \( \lambda \)-Perfect Set Property.

%

\chapter{Notation and terminology} \label{chapter:preliminaries}

As customary in set theory, \( \ZFC \) denotes the Axioms of Zermelo-Fraenkel, and \( \ZF \) denotes the axioms of \( \ZFC \) minus the Axiom of Choice \( \AC \) (see e.g.\ \cite[Chapter 1]{Jech2003}). Since we want to discuss matters related to weak forms of choice principles, in this chapter we exceptionally work in \( \ZF \) alone and always specify how much choice is needed for the various results.

\section{Ordinals and cardinals} \label{subsec:ordandcard}

We denote  the class of all ordinals by \( \On \), and the class of all cardinals by \( \Cn \).
Greek letters \( \alpha , \beta, \gamma, \dotsc \), possibly variously decorated, usually denote ordinals, while \( \kappa, \lambda, \mu, \dotsc \) are reserved for cardinals. The smallest infinite cardinal is denoted by \(\omega\), and will often be confused with the set of natural numbers \( \mathbb{N} \). If \( \preceq \) is a well-order on a set \( X \), we let \( \mathrm{ot}(\preceq) \) be its order type, i.e.\ the unique \( \alpha \in \On \) that is order isomorphic to \( (X,\preceq) \).

The von Neumann hierarchy \( (V_\alpha)_{\alpha \in \On} \) is a stratification of the universe of sets \( V \), and is defined by transfinite recursion over the ordinals by setting \( V_0 = \emptyset \), \( V_{\alpha+1} = \pow(V_\alpha) \), and \( V_\gamma = \bigcup_{\alpha < \gamma} V_\alpha \) for \( \gamma \) limit (see e.g.\ \cite[Chapter 6]{Jech2003}). The rank of a set \( X \in V \) is the smallest \( \alpha \in \On \) such that \( X \in V_\alpha \).

In choiceless settings like ours, the notion of cardinality is usually defined as follows. We write \( X \asymp Y \) to say that \( X \) and \( Y \) are in bijection, and \( [X]_{\asymp} \) is the collection of all sets \( Y \) of minimal rank that are in bijection with \( X \), i.e.~it is the equivalence class of \( X \) under \( \asymp \), cut-down using Scott's trick.
Working in \( \ZF \), the cardinality of a set \( X \) is defined to be 
\[
|X| = 
\begin{cases}
\text{the unique \( \kappa \in \Cn \) such that \( \kappa \asymp X \)} &\text{if \( X \) is well-orderable,}
\\
[X]_{\asymp} & \text{otherwise.}
\end{cases}
\]
Thus \( \AC \) implies that every cardinality is a cardinal.
Set \( |X| \leq |Y| \) if and only if there is an injection from \( X \) into \( Y \), and \( |X| < |Y| \) if \( |X| \leq |Y| \) but \( |X| \neq |Y| \).
By the Shr\"oder-Bernstein theorem, \( |X| = |Y| \) if and only if \( |X| \leq |Y| \leq |X| \).

We remark once again that since we are not assuming \( \AC \), expressions like ``the set \( X \) has size smaller than \( \nu \)'' for some cardinal \( \nu \) are somewhat stronger than usual, as \( |X| < \nu \) means that \emph{\( X \) is well-orderable} and its cardinal(ity) \( \kappa \) satisfies \( \kappa < \nu \). Also, statements like ``either \( |X| \leq \nu \) or \( \nu < |X| \)'' might be nontrivial, because \( |X| \leq \nu \) might fail simply because the set \( X \) is not well-orderable, and therefore we cannot automatically infer that \( \nu \) injects into \( X \) from the fact that \( |X| \nleq \nu \).

A sequence \( (\alpha_i)_{i < \beta} \) is cofinal in an ordinal \( \delta \) if \( \alpha_i < \delta \) for all \( i < \beta \), and for all \( \gamma < \delta \) there is \( i < \beta \) such that \( \gamma \leq \alpha_i \). The cofinality of \( \delta \), denoted by \( \cf(\delta) \), is the smallest \(\beta\) for which there is a sequence \( (\alpha_i)_{i < \beta} \) cofinal in \( \delta \). Notice that \( \cf(\delta) \leq \delta \) and that \( \cf(\delta) \) is always a cardinal. A cardinal \( \lambda \) is called regular if \( \cf(\lambda) = \lambda \), and singular otherwise. In this paper we will mostly be interested in uncountable cardinals \(\lambda\) with the smallest possible cofinality, that is \( \cf(\lambda) = \omega \). An example of such a cardinal is \( \aleph_\omega \).

The successor of a cardinal \( \kappa \) is denoted by \( \kappa^+ \). The cardinal \(\lambda\) is called successor if \( \lambda = \kappa^+ \) for some \( \kappa \in \Cn \), and limit otherwise. Cardinal exponentiation is defined as usual by setting, for \( \lambda,\kappa \in \Cn \)
\[ 
\lambda^\kappa = |\pre{\kappa}{\lambda}|,
 \] 
where \( \pre{\kappa}{\lambda} = \{ f \mid f \colon \kappa \to \lambda \}  \). 
We also set \( \lambda^{< \kappa} = \left|  \bigcup_{\mu < \kappa} \pre{\mu}{\lambda} \right| \).
Notice however that in choiceless settings \( \lambda^\kappa \) and \( \lambda^{< \kappa} \) might not be cardinals, as the set \( \pre{\kappa}{\lambda} \) might not be well-orderable if we do not assume \( \AC \). The usual inequalities \( 2^\kappa > \kappa \) and \( \kappa^{\cf(\kappa)} > \kappa \) for \( \kappa \) an infinite cardinal are still true in our context, as there are obvious injections from \( \kappa \) into both \( \pre{\kappa}{2} \) and \( \pre{\cf(\kappa)}{\kappa} \), while in \( \ZF \) alone one can prove that there are no surjections from \( \kappa \) onto either \( \pre{\kappa}{2} \) or \( \pre{\cf(\kappa)}{\kappa} \) (for the latter see e.g.\ the proof of~\cite[Theorem 3.11]{Jech2003}). 

We will use various largeness conditions on a cardinal \(\lambda\): although their formulation is standard, their actual meaning is a bit peculiar in choiceless models of set theory, as they imply that indeed some choice is available below the given \(\lambda\). For example, an infinite cardinal \(\lambda\) is called \(\omega\)-inaccessible if \( \kappa^\omega < \lambda \) for all \( \kappa < \lambda \) (in particular, \(\lambda\) must be uncountable). This means that the set \( \pre{\omega}{\kappa} \) is well-orderable (even if we are not assuming \( \AC \) in the background) and its cardinality, which is then a cardinal, is smaller than \(\lambda\). In particular, when \(\lambda\) is \(\omega\)-inaccessible the reals are well-orderable and \(\lambda\) is larger than the continuum \( 2^{\aleph_0} \). 

In the same fashion, an infinite cardinal \(\lambda\) is strong limit if \( 2^\kappa < \lambda \) for all \( \kappa < \lambda \). This again means that each \( \pre{\kappa}{2} \) is well-orderable and has size smaller than \( \lambda \). In particular, such a \(\lambda\) must be an \(\omega\)-inaccessible limit cardinal. The condition of being strong limit can be weakened by only requiring that \( 2^\kappa \leq \lambda \) for all \( \kappa < \lambda \). However, if \(\lambda\) is singular the two notions coincide. One direction is obvious. Conversely, assume that \( 2^\kappa \leq \lambda \) for all \( \kappa < \lambda \) and that, towards a contradiction, there is \( \kappa < \lambda \) such that \( 2^{\kappa} = \lambda \). By monotonicity of cardinal exponentiation, we can assume that \( \cf(\lambda) \leq \kappa < \lambda \). Then
\[ 
 2^\kappa = 2^{\kappa \cdot \cf(\lambda)} = (2^\kappa)^{\cf(\lambda)} = \lambda^{\cf(\lambda)} > \lambda ,
 \] 
a contradiction. Under a small fragment of choice, namely the \(\lambda\)-axiom of choice \( \AC_\lambda(\pre{\lambda}{2}) \) discussed in Section~\ref{subsec:weakchoice}, being strong limit is also equivalent to requiring that \( 2^{<\lambda} = \lambda \).

In \( \ZF \), the beth function \( \beth \colon \On \to V, \, \alpha \mapsto \beth_\alpha \) is defined as follows. Recursively define \( B_0 = \omega \), \( B_{\alpha+1}  = \{ f \mid f \colon B_\alpha \to 2 \} = \pre{B_\alpha}{2} \), and \( B_\gamma = \bigcup_{\alpha < \gamma} B_\alpha \) for \( \gamma \in \On \) limit. Then set \( \beth_\alpha = |B_\alpha| \).  Once again, \( \beth_\alpha \) need not to be a cardinal if we do not assume \( \AC \). The condition \( \beth_\lambda = \lambda \) considered in this paper is then a quite strong requirement, which implies that \( \beth_\alpha \) is a cardinal for all \( \alpha \leq \lambda \). It also implies that \(\lambda\) is strong limit, but it is strictly stronger than that because, in fact, \emph{every} limit \( \alpha \leq \lambda \) is such that \( \beth_\alpha \) is a strong limit cardinal. Finally, the requirement \( \beth_\lambda = \lambda \) implies that we have full choice below \(\lambda\), in the sense that it is equivalent to asserting that \( V_\lambda \) is well-orderable and \( |V_\lambda| = \lambda \).

While the existence of e.g.\ regular strong limit cardinals (aka inacessible cardinals) is independent of \( \mathsf{ZFC} \), the existence of singular strong limit cardinals can be proved without additional requirements if we have enough choice at disposal. For example, as observed above, if \( \beth_\omega \) is a cardinal, then it is strong limit and has countable cofinality. Additional set theoretical assumptions may yield more examples of such cardinals: for example, the Generalized Continuum Hypothesis \( \GCH \) implies that, under choice, all limit cardinals are strong limit.

\section{Topological spaces} \label{subsec:topologicalspaces}

If \( X \) and \( Y \) are topological spaces we write \( X \approx Y \) to say that they are homeomorphic. The weight \( \weight(X) \) of a topological space \( X \) is the smallest cardinal \( \mu \) such that \( X \) has a well-orderable basis of size \( \mu \), while the density character \( \dens(X) \) of \( X \) is the smallest cardinal \( \nu \) such that \( X \) has a well-orderable dense subset of size \( \nu \). Without choice, however, it might happen%
\footnote{It is enough to consider a non-well-orderable set \( X \) with the discrete topology.} 
 that a space \( X \) has neither a well-orderable basis nor a well-orderable dense subset, in which case the weight and density character are undefined. The following fact easily follows from the trivial observation that if \( \weight( X ) =  \nu \), then \( X \)
 has at most \( 2^\nu \)-many open sets.

\begin{fact} \label{fct:weightstronglimit}
If \( \kappa \) is strong limit and \( X \) has at least \( \kappa \)-many%
\footnote{That is, there is an injective function from \( \kappa \) into the topology of \( X \).}
 different open sets, then \( \weight(X) \not< \kappa \).
\end{fact}

A subset \( A \) of \( X \) is discrete if each \( x \in A \) has an open neighborhood \( U_x \) such that \( U_x \cap A = \{ x \} \). If \( X \) is metrizable, without loss of generality we can assume that \( U_x \cap U_y = \emptyset \) for distinct \( x,y \in A \) (but without metrizability this might become a stronger requirement). This can be further strengthened as follows.

\begin{defin}
Let \( r \in \RR^+ \) be a strictly positive real and \( (X,d) \) be a metric space. A set \( A \subseteq X \) is
\( r \)-spaced%
\footnote{In the literature, a set with this property is also called \( r \)-packing.}
 if \( d(x,y) \geq r \) for all distinct \( x,y \in A \). 
\end{defin}

The above condition is obviously equivalent to requiring that for all \( x \in A \), the open ball 
\[ 
B_d(x,r)  = \{ z \in X \mid d(x,z)<r \}
\]
does not contain any point of \( A \) besides \( x \) itself.
Notice that any subset of a  discrete (respectively, \( r \)-spaced) set is still discrete (respectively, \( r \)-spaced). Using the Axiom of Choice, it is easy to see that the density character of a metric space (equivalently, its weight) is strictly related to the size of its discrete subspaces, as well as to the size of its \( r \)-spaced sets. 

\begin{lemma}[\( \AC \)] \label{lem:r-spaces-AC}
Let \( \nu \) be an infinite cardinal and \( (X,d) \) be a metric space. The following are  equivalent:
\begin{enumerate-(1)}
\item \label{lem:r-spaces-AC-1}
\( \weight(X) > \nu \) (equivalently, \( \dens(X) > \nu \)).
\item \label{lem:r-spaces-AC-2}
There is a discrete subset of \( X \) of size larger than \( \nu \).
\item \label{lem:r-spaces-AC-3}
There is a set \( A \subseteq Y \) of cardinality greater than \( \nu \) which is%
\footnote{Sets which are \( r \)-spaced for some \( r > 0 \) are sometimes called uniformly discrete.}
 \( r \)-spaced for some \( r > 0 \). 
\end{enumerate-(1)}  
\end{lemma}

\begin{proof}
The implications~\ref{lem:r-spaces-AC-3} \( \Rightarrow \)~\ref{lem:r-spaces-AC-2} and \ref{lem:r-spaces-AC-2}~\( \Rightarrow \)~\ref{lem:r-spaces-AC-1} are trivial, so we only discuss~\ref{lem:r-spaces-AC-1} \( \Rightarrow \)~\ref{lem:r-spaces-AC-3}.
Using Zorn's lemma, fix a maximal \( 2^{-n} \)-spaced set \( B_n  \subseteq X \). By maximality, the set \( B = \bigcup_{n \in \omega} B_n \) is dense in \( X \). Thus  \( |B_{\bar{n}}| > \nu \) for some \( \bar{n} \in \omega \), as otherwise \( |B| \leq \nu \times \omega  = \nu \), contradicting~\ref{lem:r-spaces-AC-1}. It is thus enough to set \( A = B_{\bar{n}}\).
\end{proof}

We will prove the same result without appealing to Zorn's lemma (i.e.\ the full Axiom of Choice \( \AC \)) in Lemma~\ref{lem:r-spaces}.

An open cover of a topological space \( X \) is a family \( \mathcal{U} \) of open sets such that \( \bigcup \mathcal{U} = X \).
Given an infinite cardinal \( \nu \), we say that
 \( X \) is \( \nu \)-Lindel\"of if every open cover \( \mathcal{U} \) of \( X \) admits a well-orderable subcover \( \mathcal{U}' \subseteq \mathcal{U} \) such that \( |\mathcal{U}'| < \nu \). Thus \( \omega \)-Lindel\"ofness is just the standard compactness property. A refinement of an open cover \( \mathcal{U} \) of \( X \) is an open cover \( \mathcal{V} \) of \( X \) such that for all \( V \in \mathcal{V} \) there is \( U \in \mathcal{U} \) such that \( V \subseteq U \). The open cover \( \mathcal{V} \) is locally finite if every \( x \in X \) has an open neighborhood \( W \) such that the set \( \{ V \in \mathcal{V} \mid W \cap V \neq \emptyset \} \) is finite. The space \( X \) is paracompact if each of its open covers has a locally finite open refinement. Under \( \AC \), every metrizable space \( X \) is paracompact. However, using Rudin's argument from~\cite{Rudin1969} one sees that if \( X \) admits a well-orderable basis, then \( \AC \) is not needed to prove that \( X \)
is paracompact. Indeed, once a well-ordered basis for \( X \) is fixed, then every open cover can be refined to a well-ordered open cover, from which one can in turn recover \emph{in a canonical way} (that is, without using any form of the axiom of choice and with an explicit and constructive procedure) a refinement witnessing paracompactness.

Suppose that \( (X_i, \tau_i) \) is a family of topological spaces, where \( i \in I \) for some set \( I \). The product space \( \prod_{i \in I} X_i \) can be endowed with several topologies, usually generated by some collection of sets of the form \( \prod_{i \in I} U_i \) with \( U_i \in \tau_i \) for all \( i \in I \); the support of \( \prod_{i \in I} U_i \) is the set \( \{ i \in I \mid U_i \neq X_i \} \). The usual product topology on \( \prod_{i \in I} X_i \) is thus the topology generated by all sets \( \prod_{i \in I} U_i \) with finite support. We generalize this to arbitrary cardinals \( \mu \) as follows: The \( < \mu \)-supported topology on \( \prod_{i \in I} X_i \) is the topology generated by all sets \(\prod_{i \in I} U_i \) having well-orderable support of size smaller than \( \mu \). In particular, the \( < \omega \)-supported product topology coincides with the classical product topology, while if \( |I| < \mu \) then the \( < \mu \)-supported product topology coincides with the box topology.

We denote by \( \mathrm{cl}(A) \) the closure of \( A \subseteq X \).
A set \( A  \subseteq X \) is \( G_\delta \) if it can be written as a countable intersection of open sets and, dually, it is \( F_\sigma \) if it is a countable union of closed sets. It is well-known that, working in \( \ZF \) alone, if \( (X,d) \) is a metric space, then every closed set \( F \subseteq X \) can be written \emph{in a canonical way} as a countable intersection of open sets: indeed, \( F = \bigcap_{n \in \omega} U_n \) where
\[ 
U_n = \bigcup \{ B_d(x, 2^{-n}) \mid x \in F \}.
 \] 
The following very useful fact easily follows.

\begin{fact} \label{fct:booleancombinationsareFsigma}
Let \( X \) be a metrizable space. If \( A \subseteq X \) is a Boolean combination of open sets, than \( A \) can be \emph{canonically} written as \( A = \bigcap_{n \in \omega} U_n = \bigcup_{n \in \omega} C_n \), where each \( U_n \) is open and each \( C_n \) is closed.
\end{fact}

\section{Sequences and trees} \label{subsec:trees}
When \(\alpha\) is an ordinal and \( X \) a nonempty set, we denote by
\( \pre{\alpha}{X} \) the set of all sequences of order type \( \alpha \) and values in \(X\), i.e.\
\[ 
\pre{\alpha}{X} = \{ f \mid f \colon \alpha \to X  \}.
 \] 
When considering cardinals \( \mu, \lambda \), the set \( \pre{\mu}{\lambda} \) is thus the Cartesian product of \( \mu \)-many copies of \( \lambda \), that is, \( \pre{\mu}{\lambda} = \prod_{i < \mu} X_i \) where \( X_i = \lambda \) for all \( i < \mu \).
We also set \( \pre{<\mu}{X} = \bigcup_{\beta < \mu} \pre{\beta}{X} \) and \( \pre{ \leq \mu}{X} = \pre{<\mu}{X} \cup \pre{\mu}{X} \). 

Given sequences \( s \) and \( t \), we denote by \( \lh(s) \) the length of \( s \) (which is \( \mu \) if \( s \in \pre{\mu}{\lambda} \) and an ordinal smaller than \( \mu \) if \( s \in \pre{<\mu}{\lambda} \)), by \( s {}^\smallfrown{} t \) the concatenation of \( s \) and \( t \), and by \( s \restriction \alpha \) the restriction of \( s \) to length \( \alpha \leq \lh(s) \). When \( t = \langle \beta \rangle \) is a sequence of length \( 1 \) we sometimes write \( s {}^\smallfrown{} \beta \) instead of the more formal \( s {}^\smallfrown{} \langle \beta \rangle \). Two sequences are incomparable if none is an extension of the other one, and comparable otherwise.

Sequences of cardinals \( (\lambda_i)_{i < \mu} \) will sometimes be denoted by \( \vec{\lambda}^{(\mu)} \), dropping the reference to \( \mu \) when it is clear from the context. To uniformize our notation for sets of sequences, given such a \( \vec{\lambda} =  \vec{\lambda}^{(\mu)}  = (\lambda_i)_{i < \mu}\) and \( \lambda = \sup_{i < \mu} \lambda_i \) we 
set 
\[ 
\pre{\mu}{(\vec{\lambda})}  = \{  x \in \pre{\mu}{\lambda} \mid \forall i < \mu \, (x(i) < \lambda_i) \} = \prod_{i < \mu} \lambda_i,
\]
and 
\[ 
\pre{<\mu}{(\vec{\lambda})} = \{ s \in \pre{< \mu}{\lambda} \mid \forall i < \lh(s) \, (s(i) < \lambda_i) \} = \bigcup_{j < \mu} \left( \prod_{i < j } \lambda_i \right).
 \] 

A subset \( T \subseteq \pre{< \mu}{\lambda} \) is called (\( \mu \)-)tree on \( \lambda \) if it is closed under initial segments, i.e.\ if \( s \subseteq t \in T \) implies \( s \in T \).
For \( \alpha < \mu \), the \( \alpha \)-th level of \( T \) is the set 
\[
 \mathrm{Lev}_\alpha(T) = \{ t \in T \mid \lh(t) = \alpha \} .
 \]
 A tree \( T \) is called pruned if there is no terminal node in \( T \), i.e.\ all sequences of \( T \) can be properly extended in \( T \). The body of a tree \( T \subseteq \pre{<\mu}{\lambda} \) is the set
\[ 
[T] = \{ x \in \pre{\mu}{\lambda} \mid x \restriction \beta \in T \text{ for all } \beta < \mu \}.
 \] 
Notice that the notion of body strictly depends
 on the choice of \( \mu \) and \( \lambda \), and not just on the set \( T \): when it is necessary, we thus write \( [T]^\mu_\lambda \) instead of \( [T] \) to make explicit the ambient space we are referring to.%
\footnote{For example, if \( T = \pre{<\mu}{\lambda} \) and \( \mu < \mu' \), then \( [T]^\mu_\lambda = \pre{\mu}{\lambda} \neq \emptyset \), while \( [T]^{\mu'}_\lambda = \emptyset \). Similarly, if \( \lambda ' > \lambda \), then \( [T]^\mu_{\lambda'} \) is a proper subset of its ambient space \( \pre{\mu}{\lambda'} \), while this does not happen for \( [T]^\mu_\lambda \).} 
Notice also that, trivially, \(  \pre{\mu}{\lambda}  = [\pre{<\mu}{\lambda}]\) and \( \pre{\mu}{(\vec{\lambda})} = [\pre{< \mu}{(\vec{\lambda})}] \) if \( \vec{\lambda} = (\lambda_i)_{i < \mu} \).

An \(\omega\)-tree \( T \subseteq \pre{<\omega}{\lambda} \) is called well-founded if \( [T] = \emptyset \), and ill-founded otherwise. If \( T \) is ill-founded, we can define its leftmost branch \( x \in [T] \) be recursion on \( n \in \omega \) letting \( x(n) \) be the least \( \alpha < \lambda \) such that 
\( (x \restriction n) {}^\smallfrown{} \alpha \subseteq z \) for some \( z \in [T] \). 
This gives us a definable, choiceless way to select an element from \( [T] \). If instead \( T \) is well-founded, we can define a rank function \( \rho_T \)  on it as follows. Recall that to each well-founded (strict) relation  \( \prec \) on a set \( X \) we can associate its rank function \( \rho_{\prec} \colon X \to \On \) by recursively setting \( \rho_\prec(x) = \sup \{ \rho_\prec(y)+1  \mid y \prec x \} \). The rank of \( \prec \) is then \( \rho(\prec) = \sup \{ \rho_\prec(x)+1 \mid x \in X \} \); if \( X \) is well-orderable, then clearly \( \rho(\prec) < |X|^+ \). Applying this to \( X = T \) with \( \prec \) the relation of strict containment \( \supsetneq \) we get our \( \rho_T \). In particular, \( \rho_T(t) = 0 \) if and only if \( t \) is terminal in \( T \), and \( \rho_T(t) = \sup \{ \rho_T(s)+1 \mid s \in T \wedge t \subsetneq s \} \) otherwise. The rank of the tree \( T \) is then the ordinal \( \rho(T) = \sup \{ \rho_T(t)+1 \mid t \in T \} = \rho_T(\emptyset) +1 \). 

When the \(\omega\)-tree \( T \subseteq \pre{<\omega}{\lambda} \) is ill-founded, we can still define a rank function \( \rho_T \) as above on its well-founded part
\( \mathrm{WF}(T) = \{ s \in T \mid \neg \exists x \in [T] \, (s \subseteq x ) \} \).
Indeed, if \( t \subseteq t' \in T \) for some \( t \in \mathrm{WF}(T) \), then \( t' \in \mathrm{WF}(T) \), and the relation \( \supsetneq \) is well-founded on \( \mathrm{WF}(T) \).
Accordingly, we then let \( \rho(T) = \sup \{ \rho_T(t) + 1 \mid t \in \mathrm{WF}(T) \} \). 
Notice that \( \rho_T(t) < \lambda^+ \) for every \( t \in \mathrm{WF}(T) \). Sometimes it is convenient to extend \( \rho_T \) to the whole \( \pre{< \omega}{\lambda} \) by setting \( \rho_T(t) = \lambda^+ \) if \( t \in T \setminus \mathrm{WF}(T) \), and \( \rho_T(t) = 0 \) if \( t \notin T \).

The Kleene-Brouwer ordering \( \leq_{\mathrm{KB}} \) on \( \pre{< \omega}{\lambda} \) is the linear order defined by setting, for distinct \( s ,t \in \pre{<\omega}{\lambda} \), 
\( s \leq_{\mathrm{KB}} t \) if and only if either \( t \subsetneq s \), or else \( s \) and \( t \) are incomparable and \( s_i < t_i \) for the smallest \( i \) such that \( s_i \neq t_i \). The order \( \leq_{\mathrm{KB}} \) clearly contains  infinite decreasing chains, but it is not hard to see that an \(\omega\)-tree \( T \subseteq \pre{<\omega}{\lambda} \) is well-founded if and only if the restriction of \( \leq_{\mathrm{KB}} \) to \( T \) is a well-ordering.

The concept of tree and the related notions can be naturally adapted when \( \lambda \) is replaced by any
 given set \(A \), that is when \( T \subseteq \pre{< \mu}{A} \). We will in particular consider the case where \( A = \lambda \times \lambda \): in this case, we will naturally identify sequences in \( \pre{<\mu}{(\lambda \times \lambda)} \) with pair of sequences (of the same length) in \( \pre{<\mu}{\lambda} \times \pre{< \mu}{\lambda} \), and similarly for the sequences of length \(\mu\) appearing in the body of such  a tree. Finally, if \( T \subseteq {}^{< \mu}{(\lambda \times \lambda)} \) is a \(\mu\)-tree, the section tree \( T(x) \) of \( T \) at \( x \in \pre{\mu}{\lambda} \) is defined by
\[ 
T(x) = \{ s \in \pre{< \mu}{\lambda} \mid (x \restriction \lh(s), s) \in T \},
 \] 

and the section tree \( T(s) \) of \( T \) at \( s \in \pre{<\mu}{\lambda} \) is defined by

\[ 
T(s) = \{ u \in \pre{\lh(s)}{\lambda} \mid (s \restriction \lh(u), u) \in T \}.
 \]

\section{Weak forms of choice} \label{subsec:weakchoice}

As in the classical case, it is worth developing our (generalized) descriptive set theory under the weakest possible choice assumption. 
We first provide a purely \( \ZF \)-reformulation of Lemma~\ref{lem:r-spaces-AC}.

\begin{lemma} \label{lem:r-spaces}
Let \( (X,d) \) be a metric space with a well-orderable dense subset.
For every infinite cardinal \( \nu \), the following are equivalent:
\begin{enumerate-(1)}
\item \label{lem:r-spaces-2a}
\( \dens(X) > \nu \);
\item \label{lem:r-spaces2b}
there is a discrete \( A \subseteq X \) with \( |A| > \nu \);
\item \label{lem:r-spaces-2c}
then there is a set \( A \subseteq X \) such that \( A \) is \( r \)-spaced for some \( r > 0 \) and \( |A| > \nu \).
\end{enumerate-(1)}
In particular, if \( \dens(X) = \lambda \) for some infinite cardinal \(\lambda\), then for all infinite \( \nu < \lambda \) there is a set \( A \subseteq X \) which is \( r \)-spaced for some \( r > 0  \) and  \(|A| =  \nu \).
\end{lemma}

\begin{proof}
We just need to prove~\ref{lem:r-spaces-2a} \( \Rightarrow \) \ref{lem:r-spaces-2c}, the rest is trivial. 
Fix an enumeration \( (x_\alpha)_{\alpha < \mu} \) of a dense subset \( D \) of \( X \), for some cardinal \( \mu \). 
For each \( n \in \omega \), recursively construct a well-orderable \( 2^{-n} \)-spaced set 
\( B_n \subseteq D \) as follows. At stage \( \alpha < \mu \), consider the element 
\( x_\alpha \in D \): if it is at distance at least \( 2^{-n} \) from all the points that have been 
added to \( B_n \) so far, add \( x_\alpha \) to \( B_n \) as well; if not, move to the next 
stage. We claim that the well-orderable set \( B = \bigcup_{n \in \omega} B_n \) is dense in \( X \). If not, there would be 
\( y \in X \) and \( m \in \omega \) such that the ball \( P \) centered in \( y \) with radius \( 2^{-m} \) is disjoint from \( B \). Let \( \beta < \mu \) be such that 
\( d(y,x_\beta) < 2^{-(m+1)} \). Then \( x_\beta \in P \) and so \( x_\beta \notin B_{m+1} \). 
On the other hand, for each \( z \in B_{m+1} \subseteq B \) we have \( d(z,x_\beta) \geq 2^{-(m+1)} \), 
as otherwise \( d(y,z) < 2^{-m} \), contradicting \( P \cap B = \emptyset \). This means that when constructing \( B_{m+1} \), at stage \( \beta \) the element \( x_\beta \) would have been added to \( B_{m+1} \), a contradiction. It then follows that there is \( \bar{n} \in \omega \) such that \( |B_{\bar{n}}| > \nu \), as otherwise \( |B| \leq \nu \times \omega = \nu \) and hence \( \dens(X) \leq \nu \). Setting \( A =B_{\bar{n}} \) we are done.
\end{proof}

\begin{remark} \label{rmk:r-spaces}
The above proof shows that if the well-orderable dense set \(  D \subseteq X \) is fixed in advance, then a set \( A \) as in~\ref{lem:r-spaces-2c} can be constructed in a \emph{canonical} way (just apply the above procedure and take \( \bar{n} \in \omega \) to be minimal such that \( |B_{\bar{n}}| > \nu \)).
\end{remark}

\begin{fact}[Folklore] \label{fct:fromweighttosize}
Let \( X \) be a metrizable space and \( \nu \) be a cardinal. If \( \weight(X) \leq \nu \), then \( |X| \leq \nu^\omega \).
\end{fact}

\begin{proof}
Let \( \mathcal{B} = \{ B_\alpha \mid \alpha < \nu \} \) be a basis for \( X \). Then the map \( f \colon X \to  \pre{\omega}{\nu} \) defined by letting \( f(x)(n) \) be the smallest \( \alpha < \nu \) such that \( x \in B_\alpha \subseteq B_d(x, 2^{-n}) \) is injective.
\end{proof}

As discussed in Section~\ref{subsec:ordandcard},  certain natural conditions  on \(\lambda\) entail some form of choice below \( \lambda \), and this in turn allows us to get more information on the cardinality of metrizable spaces with density character at most \( \lambda\).


\begin{lemma} \label{lem:densityvscardinality}
Assume that \( \lambda \) is \(\omega\)-inaccessible, and let \( X \) be metrizable. 
Then \(  \weight(X) < \lambda \) if and only if \( \dens(X) < \lambda \) if and only if \( X \) is well-orderable and \( | X | < \lambda \).
\end{lemma}

\begin{proof}
Clearly if \( X \) is well-orderable, then \( \dens(X) \) is defined and \( \dens(X) \leq |X| \); similarly, if \( \dens(X) \) is defined, then so is \( \weight(X) \), and \( \weight(X) \leq \dens(X) \). This gives all backward implications. The remaining implication follows from Fact~\ref{fct:fromweighttosize}: indeed, by \( \omega \)-inaccessibility of \(\lambda\) it follows that if \( \weight(X) = \nu < \lambda \), then \( \pre{\omega}{\nu} \), and hence also \( X \), is well-orderable, and \( |X| \leq \nu^\omega < \lambda \).
\end{proof}

\begin{corollary} \label{cor:cardinalityunderomegainaccessible}
Assume  that \( \lambda \) is \(\omega\)-inaccessible. Let \( X \) be a metrizable space with \( \dens(X) \leq \lambda \). Then either \( X \) is well-orderable and \( |X| \leq \lambda \), or else \( \lambda < |X| \leq  \lambda^\omega  \).
\end{corollary}

Notice that if the second alternative holds, then \( X \) might not be well-orderable, and thus \( |X| \) needs not to be a cardinal in general (all inequalities must be intended in the \( \ZF \)-sense).

\begin{proof}
We distinguish two cases. If \( \dens(X) < \lambda \), then \( X \) is well-orderable and \( |X| < \lambda \) by Lemma~\ref{lem:densityvscardinality}. If \( \dens(X) = \lambda \), then any well-ordered dense subset of \( X \) witnessing such equality is also a witness of \( |X| \geq \lambda \). Thus if there is also an injection from \( X \) to \( \lambda \), then \( X \) is well-orderable and \( |X| = \lambda \); otherwise \( |X| > \lambda \). Finally, \( |X| \leq \lambda^\omega \) follows from \( \weight(X) \leq \dens(X) = \lambda \) and Fact~\ref{fct:fromweighttosize}.
\end{proof}

Given a cardinal \(\lambda\) and a set \( X \), the axiom \( \AC_\lambda(X) \) of \(\lambda\)-choice over \( X \) is the statement: 
\begin{quotation}
For every sequence \( (A_\alpha)_{\alpha < \lambda} \) of nonempty subsets of \( X \) there is a function \( f \colon \lambda \to X \) such that \( f(\alpha) \in A_\alpha \) for all \( \alpha < \lambda \). 
\end{quotation}
Any \( f \) as above is called choice function for the given sequence of sets \( (A_\alpha)_{\alpha < \lambda} \). Obviously, if there is a surjection from a set \( X \) onto \( Y \) and \( \mu \leq \lambda \), then \( \AC_\lambda(X) \) implies \( \AC_\mu(Y) \). The (full) \(\lambda\)-choice axiom \( \AC_\lambda \) asserts that \( \AC_\lambda(X) \) holds for all sets \( X \). 
\begin{fact} \label{fct:successorsareregular}
Assume \( \AC_\lambda \). Then
\begin{enumerate-(i)}
\item
\( \lambda^+ \) is a regular cardinal.
\item
Every well-ordered union of size \( \lambda\) of sets of cardinality at most \( \lambda \) is still well-orderable and of size at most \( \lambda \).
\end{enumerate-(i)}
\end{fact}

The amount of \( \AC_\lambda \) needed for the above fact depends on the properties of \(\lambda\) and on the nature of the sets under consideration, see e.g.\ Fact~\ref{fct:lambdaunionewithoutchoice}.

In this paper we will mostly be dealing with metrizable spaces of weight at most \( \lambda\), that is, topological spaces \( X \) admitting a compatible metric and having a well-ordered basis \( \mathcal{B} = \{ B_\alpha \mid \alpha < \mu \} \) of size \( \mu \leq \lambda\). Such an \( X \), if nonempty, is easily seen to be a surjective image of \( \pre{\omega}{\mu} \) via the functions sending \( z \in \pre{\omega}{\mu} \) to the point \( x \) if \( \bigcap_{n \in \omega} B_{z(n)}  = \{ x \} \), and to some fixed \( x_0 \in X \) in all remaining cases (that is, if  \( \bigcap_{n \in \omega} B_{z(n)} \) either is empty or contains more than one point). Since \( \pre{\omega}{\mu} \) is in turn a surjective image of \( \pre{\omega}{\lambda} \), it follows that under \( \AC_\lambda(\pre{\omega}{\lambda}) \) we also get \( \AC_\lambda(X) \) for all metrizable spaces of weight at most \( \lambda \). This entails, still assuming \( \AC_\lambda(\pre{\omega}{\lambda}) \), that \( \weight(X) \leq \lambda \) if and only if \( \dens(X) \leq \lambda \) for all metrizable spaces \( X \), and in such case \( \weight(X) = \dens(X) \).

Similarly, if \( X \) and \( \mathcal{B}  \) are as in the previous paragraph, then there is an obvious surjection of \( \pre{\mu}{2} \) onto the open sets of \( X \), namely, the function sending \( z \in \pre{\mu}{2} \) to \( \bigcup \{ U_\alpha \mid {\alpha < \mu} \wedge  {z(\alpha) =1} \} \). This in turn induces canonical surjections from e.g.\ \( \pre{\lambda}{2} \) onto families of size \(\lambda\) of open or closed sets (and thus onto the collection of clopen partitions of \( X \)), onto the collection of subsets of \( X \) which are intersections of \(\lambda\)-many open sets, onto the continuous functions between two given \(\lambda\)-Polish spaces, and so on. The process can even be iterated through all ordinals smaller than \( \lambda^+ \) to show that there is a canonical surjection from \( \pre{\lambda}{2} \) onto each of the classes \( \lS^0_\alpha (X)\), \( \lP^0_\alpha(X) \), and \( \lB(X) \) considered in Section~\ref{sec:Borelhierarchy}. Assuming \( \AC_\lambda(\pre{\lambda}{2}) \) in the background theory, this will allow us to make \(\lambda\)-choices over all such classes.

It is easy to see that \( \AC_\lambda(\pre{\lambda}{2}) \) implies \( \AC_\lambda(\pre{\omega}{\lambda}) \), as \(  \pre{\lambda}{2} \) surjects onto \( \pre{\omega}{\lambda} \) (notice that \( \pre{\lambda}{2} \) is in one-to-one correspondence with \( \pre{\lambda}{\lambda} \), which in turns surjects onto \( \pre{\omega}{\lambda} \) via the obvious restriction function). 
Conversely, if we further assume that \( 2^{< \lambda} =  \lambda \) and \( \cf(\lambda) = \omega \), then there is also a surjection from \( \pre{\omega}{\lambda} \) onto \( \pre{\lambda}{2} \). Indeed, let \( (\lambda_i)_{i < \omega} \) be an increasing sequence cofinal in \( \lambda \), and for each \( i  \in \omega \), fix a surjection \( \rho_i \colon \lambda \to \pre{\lambda_i}{2} \). Then the function sending \( z \in \pre{\omega}{\lambda} \) to \( \bigcup_{i < \omega} \rho_i(z(i)) \) if \( \rho_i(z(i)) \subseteq \rho_j(z(j)) \) for all \( i < j  \), and to some fixed \( x_0 \in \pre{\lambda}{2} \) otherwise, is clearly surjective. Thus for such cardinals \(\lambda\) the axiom \( \AC_\lambda(\pre{\omega}{\lambda}) \) entails \( \AC_\lambda(\pre{\lambda}{2}) \). The following results are probably known and should be considered folklore, but we could not find a reference for them and so we include their proofs as well.

\begin{fact} \label{fct:lambdaunionewithoutchoice}
Assume \( \AC_\lambda(\pre{\lambda}{2}) \). Then 
\begin{enumerate-(i)}
\item \label{fct:lambdaunionewithoutchoice-1}
\( \lambda^+ \) is a regular cardinal.
\item \label{fct:lambdaunionewithoutchoice-2}
Let \( X \) be a metrizable space of weight (equivalently, density character) at most \(  \lambda \). If \( ( A_\alpha )_{\alpha < \lambda} \) is a family of subsets of \( X \) such that \( |A_\alpha| \leq \lambda \) for all \( \alpha < \lambda \), then \( |\bigcup_{\alpha < \lambda} A_\alpha| \leq \lambda \).
\end{enumerate-(i)}
\end{fact}

\begin{proof}
\ref{fct:lambdaunionewithoutchoice-1}
It is enough to show that \( \sup_{i < \lambda} \gamma_i < \lambda^+ \) for any sequence
 \( (\gamma_i)_{i < \lambda} \) of ordinals such that \( \lambda \leq \gamma_i < \lambda^+ \).  Choose corresponding well-orderings \( \preceq_i \) on \(\lambda\) such that \( (\lambda, {\preceq_i}) \) has order type \( \gamma_i \) for all \( i < \lambda \). This requires only \( \AC_\lambda(\pre{\lambda}{2}) \), as each \( \preceq_i \) can be coded via characteristic function as an element of \( \pre{\lambda \times \lambda}{2} \), and the latter space can be canonically identified with \( \pre{\lambda}{2} \) (see Section~\ref{sec:defanalytic} for more details). Then the order type of the well-ordering \( \preceq \) on \( \lambda \times \lambda \) defined by 
\[ 
(i,\alpha) \preceq (j, \beta) \iff i < j \, \vee \, ({i = j} \wedge {\alpha \preceq_i \beta}) 
\]
is larger than each \( \gamma_i \), and thus \( |\sup_{i < \lambda} \gamma_i| \leq |\mathrm{ot}(\preceq)| = \lambda \). 

\ref{fct:lambdaunionewithoutchoice-2}
It is enough to show that any function \( f \colon \lambda \to X \) can be coded as an element of 
\( \pre{\lambda}{2} \). Granting this, we can then use \( \AC_\lambda(\pre{\lambda}{2}) \) to choose a sequence of enumerations \( f_\alpha \colon \lambda \to X \) of the sets \( A_\alpha \), and then construct an enumeration of \( \bigcup_{\alpha < \lambda} A_\alpha \) in the obvious way.

By Fact~\ref{fct:fromweighttosize},  there is a canonical injection from \( X \) to \( \pre{\omega}{\lambda} \). The latter can then be canonically injected into \( \pre{\lambda}{\lambda} \), which can in turn be identified with \( \pre{\lambda}{2} \). Composing all these injections together with the initial function \( f \), one sees that such an \( f \) can be coded as an element of \( \pre{\lambda}{(\pre{\lambda}{2})} \). Since the latter is clearly in bijection with \( \pre{\lambda}{2} \) we are done. 
\end{proof}

In most of the results of this paper, we will have some amount of unrestricted choice below \(\lambda\) (granted by mild large cardinal conditions like \(\omega\)-inaccessibility, being strong limit, or \( \beth_\lambda = \lambda \)), a restricted form of \(\lambda\)-choice at level \(\lambda\) (namely, \( \AC_\lambda(\pre{\omega}{\lambda}) \)), but no choice will be required beyond \(\lambda\). This is in line with the usual setup of classical descriptive set theory, where the theory is developed with full choice below \(\omega\) (since we are in the finite realm), a restricted form of countable choice at level \(\omega\) (namely, \( \AC_\omega(\RR) \) or, equivalently, \( \AC_\omega(\pre{\omega}{\omega}) \)), but no choice beyond \(\omega\).

The following is another restricted form of choice that will be useful later. Given a cardinal \(\lambda\) and a set \( X \), the axiom \( \DC_\lambda(X) \) of \(\lambda\)-choice over \( X \) is the statement:%
\footnote{The more familiar Principle of Dependent Choice \( \DC \) from e.g.\ \cite[p.\ 50]{Jech2003} is just a reformulation of our $\DC_\omega$.} 
\begin{quotation}
If $R \subseteq \pre{<\lambda}{X} \times X$ is a relation such that for every \( s \in \pre{<\lambda}{X} \) there is \( x \in X \) such that \( s \mathrel{R} x \), then there exists a function $f \colon \lambda \to X$ such that $(f \restriction \alpha) \mathrel{R} f(\alpha)$ for every $\alpha<\lambda$.
\end{quotation}

Obviously, if there is a surjection from a set \( X \) onto \( Y \) and \( \mu \leq \lambda \), then \( \DC_\lambda(X) \) implies \( \DC_\mu(Y) \). The (full) \(\lambda\)-choice axiom \( \DC_\lambda \) asserts that \( \DC_\lambda(X) \) holds for all sets \( X \).
One can check that if $\lambda$ is singular, then $\DC_\lambda$ is equivalent to the fact that $\DC_\kappa$ holds for all \( \kappa<\lambda \). 
Moreover, it is easy to see that $\DC_\lambda$ implies $\AC_\lambda$, for every $\lambda \in \Cn$. Finally, recall that the assertion that \( \DC_\lambda \) holds for all \( \lambda \in \Cn \) is equivalent to the full $\AC$ (see e.g.\ \cite[Chapter 8]{Jech1973}).

\section{Spaces of sequences and bounded topology} \label{sec:boundedtopology}

When \( \mu \) is infinite and \( X \) is nonempty, the set \( \pre{\mu}{X}\) will be often endowed with the so-called bounded topology \( \tau_b \), i.e.\ the topology generated by the sets
\[ 
\Nbhd_s(\pre{\mu}{X}) = \{ x \in \pre{\mu}{X} \mid s \subseteq x \}
 \] 
for \( s \in \pre{< \mu}{X} \). When there is no danger of confusion, we will write \( \Nbhd_s^\mu \) or even just \( \Nbhd_s \) instead of \( \Nbhd_s(\pre{\mu}{X}) \): notice however that the choice of \( \mu \) and \( X \) is relevant also for the sequences \( s \) that can be used in the indexing of the basis. 

In most cases, we will consider the case where \( X \) is an infinite cardinal \(\lambda\).
It is clear that when \( \mu = \omega \), the bounded topology coincide with the product \( \tau_p \) of the discrete topology on \( \lambda \), while if \( \mu > \omega \) it becomes strictly finer than \( \tau_p \). Moreover, if \( \mu \) is regular than the bounded topology coincides with the \( < \mu \)-supported product of the discrete topology on \( \lambda \). The bounded topology can also be characterized as the unique topology on \( \pre{\mu}{\lambda} \) whose closed sets coincide with the bodies of trees: \( C \subseteq \pre{\mu}{\lambda} \) is (\( \tau_b \)-)closed if and only if \( C = [T] \) for some (pruned) tree  \( T \subseteq \pre{< \mu}{\lambda} \).

When \( \pre{\mu}{\lambda} \) is equipped with the bounded topology, it has a basis (and also a dense set) of cardinality \( \lambda^{< \mu} = | \pre{< \mu}{\lambda}| \),  and this is optimal: if \( \mathcal{B} \) is a basis for \( \pre{\mu}{\lambda} \) (respectively: if \( D \subseteq \pre{\mu}{\lambda} \) is dense), then \( |\mathcal{B}| \geq \lambda^{<\mu} \) (respectively, \( |D| \geq \lambda^{<\mu} \)).
In particular, \( \pre{\omega}{\lambda} \) has weight and density \( \lambda \), and the same is true of \( \pre{\lambda}{2} \) if \( 2^{< \lambda} = \lambda \). 
Moreover, \( \pre{\mu}{\lambda} \) is metrizable if and only if it is first-countable, if and only if \( \cf(\mu) = \omega \): 
in such case, a compatible metric can be obtained by fixing a strictly increasing sequence \( (\mu_n)_{n \in \omega} \) cofinal in  \( \mu \) and setting,  for all distinct \( x,y \in \pre{\mu}{\lambda} \),  \( d(x,y) = 2^{-n} \) if and only if \( n \in \omega \) is smallest such that \( x \restriction \mu_n \neq y \restriction \mu_n \). Notice that such metric is actually an ultrametric, and that it is complete. The interested reader may consult~\cite[Sections 3 and 6]{AM} for proofs and more details: although literally speaking only the cases \( \lambda = 2 \) and \( \lambda = \mu \) are explicitly considered in~\cite{AM}, the corresponding proofs can easily be adapted to our more general setup.

Unless otherwise stated, subspaces \( X \subseteq \pre{\mu}{\lambda} \) are always 
endowed with the relative topology. In particular, the sets of the form 
\( \Nbhd_s \cap X \), sometimes denoted by \( \Nbhd_s(X) \), constitute a basis for 
\( X \). Notice that, contrarily to the case \( X = \pre{\mu}{\lambda} \),  some of the 
\( \Nbhd_s(X) \) might be empty: for this reason, the basis for \( X \) is sometimes indexed 
only over those \( s \) for which \( \Nbhd_s(X) \neq \emptyset \). For example, the basis 
for \( X = \pre{\mu}{2}  \subseteq \pre{\mu}{\lambda} \) (where \( 2 = \{ 0,1 \} \)) is often 
indexed only over binary sequences of length smaller than \( \mu \).
Notice also that if \( \cf(\mu) = \omega \), 
then \( X \subseteq \pre{\mu}{\lambda} \) is completely metrizable if and only if \( X \) is 
\( G_\delta \) in \( \pre{\mu}{\lambda} \), i.e.\ it can be written as a countable intersection 
of \( \tau_b \)-open subsets of \( \pre{\mu}{\lambda} \) 
(see e.g.~\cite[Theorem 3.11]{Kechris1995}).

For many results we will need the following notions. Let \( X \) be an arbitrary metric space. A \( \lambda \)-scheme on \( X \) is a family \( \{ B_s \mid s \in \pre{<\omega}{\lambda} \} \) of subsets of \( X \) such that 
\begin{itemizenew}
\item
\( B_{s {}^\smallfrown{} \beta} \subseteq B_s \) for all \( s \in \pre{< \omega}{\lambda} \) and \( \beta < \lambda \);
\item
\( \lim_{i \to \infty} \mathrm{diam}(B_{x \restriction i}) = 0 \) for all \( x \in \pre{\omega}{\lambda} \).
\end{itemizenew}
Such a scheme canonically induces a well-defined continuous map \( f \colon D \to X \) where \( D = \{ x \in \pre{\omega}{\lambda} \mid \bigcap_{i \in \omega} B_{x \restriction i} \neq \emptyset \} \) and, for \( x \in D \),
\[ 
f(x) = \text{ the unique point in \( \bigcap\nolimits_{i \in \omega} B_{x \restriction i} \)}. 
\]
It is not hard to see that
\begin{fact} \label{fct:scheme}
\begin{enumerate-(i)}
\item
If \( B_{s {}^\smallfrown{} \beta} \cap B_{s {}^\smallfrown{}  \beta'} = \emptyset \) for all \( s \in \pre{< \omega}{\lambda} \) and distinct \( \beta, \beta' < \lambda \), then \( f \) is injective and \( f( \Nbhd_s \cap D) = B_s \cap f(D) \). If moreover each \( B_s \) is open, then \( f \) is a (topological) embedding.
\item
If \( B_\emptyset = X \) and \( B_s = \bigcup_{\beta < \lambda} B_{s {}^\smallfrown{} \beta} \) for all \( s \in \pre{< \omega}{\lambda} \), then \( f \) is surjective. Moreover \( f(\Nbhd_s \cap D) = B_s \), so if each \( B_s \) is open then \( f \) is an open map.
\item
If the metric of \( X \) is complete and \( \mathrm{cl}(B_{s {}^\smallfrown{} \beta}) \subseteq B_s \) for all \( s \in \pre{<\omega}{\lambda} \) and \( \beta < \lambda \), 
then \( D = [T] \) where \( T \) is the tree
\[
T = \{ s \in \pre{< \omega}{\lambda} \mid B_s \neq \emptyset \},
\]
and thus \( D \) is closed in \( \pre{\omega}{\lambda} \);
if moreover each \( B_s \) is nonempty, then \( D = \pre{\omega}{\lambda} \).
\end{enumerate-(i)}
\end{fact}

These facts will be tacitly used whenever we will deal with \(\lambda\)-schemes.

\section{Transitivity} \label{subsec:transitivity}

We say that a set $X$ is transitive if and only if $x\subseteq X$ whenever \( x \in X \).
The transitive closure of an arbitrary set $X$, denoted by $\mathrm{trcl}(X)$, is the smallest transitive superset of $X$. It can be defined by recursion letting $\mathrm{trcl}_0(X)=X$ and, for each $n\in\omega$, $\mathrm{trcl}_{n+1}(X)=\bigcup\mathrm{trcl}_n(X)$; then, $\mathrm{trcl}(X)=\bigcup_{n\in\omega}\mathrm{trcl}_n(X)$.

Given an infinite regular cardinal $\kappa$, we let $H_\kappa$ be the collection of all sets $X$ that are hereditarily of cardinality less than $\kappa$, i.e., such that $|\mathrm{trcl}(X)|<\kappa$. Assuming $\ZFC$, we have that if $\kappa$ is a regular uncountable cardinal, then $H_\kappa$ is a model of $\ZFC$ minus the Power Set axiom, and that $|H_\kappa|=2^{<\kappa}$ for any $\kappa>\omega$.

We say that a binary relation $R$ on $X$ is extensional if for every$x,y\in X$ we have that \( x = y \) whenever $\forall z\in X\, (z \mathrel{R} x \iff z \mathrel{R} y)$. The relation \( R \) is well-founded if there is no \( (x_n)_{n \in \omega} \in \pre{\omega}{X} \) such that \( x_{n+1} \mathrel{R} x_n \) for every \( n \in \omega \). Suppose that $R$ is a well-founded extensional relation on $X$. Then the Mostowski collapsing function \( \pi \colon X \to V \) is defined by recursion on $R$ by letting $\pi(x)=\{\pi(y) \mid y \mathrel{R} x\}$. The Mostowski collapse $M$ of \( X \) is simply the range of $\pi$: it is clearly a transitive set, and $\pi$ is an isomorphism between $(X,R)$ and $(M,{\in})$ (see e.g.\ \cite[Theorem 6.5]{Jech2003}). 

\section{Constructibility and relative constructibility} \label{sec:constructibilityetc}


A set $X \subseteq M$ is definable (with parameters) in a model of set theory $(M,{\in})$ if there are a formula $\upvarphi(v_0, v_1, \dotsc, v_n)$ in the language of set theory\footnote{This is just the language consisting of a single binary relational symbol, usually interpreted as the membership relations in models of set theory; for this reason, the symbol itself is often denoted by \( \in \).} and parameters $a_1,\dots,a_n\in M$ such that $X=\{a\in M \mid M \models \upvarphi[a/v_0,a_1/v_1,\dots,a_n/v_n]\}$. We set 
\[ 
\Def(M)=\{ X\subseteq M \mid X \text{ is definable in }M \}.
\] 
Of course, if $M$ is a set, then so is $\Def(M)$. 

\begin{defin}
The constructible universe is the class \( L=\bigcup_{\alpha\in \On}L_\alpha \), where the sets \( L_\alpha \) are defined by 
 transfinite recursion setting
 \begin{itemizenew}
  \item $L_0=\emptyset$
  \item
  $L_{\alpha+1}=\Def(L_\alpha)$;
	\item $L_\gamma = \bigcup_{\alpha<\gamma}L_\alpha$ if $\gamma \in \On$ is limit.
 \end{itemizenew}
\end{defin}

The constructible universe $L$ is a model of \( \ZFC \).


By the definition it is clear that the class sequence $(L_\alpha)_{\alpha \in \On}$ is definable in the language of set theory without parameters, and therefore the property ``$x$ is constructible'' (i.e.\ $x\in L$, which can be rephrased as $\exists \alpha\, (x\in L_\alpha)$) is expressible with a first-order formula in the language of set theory. The sentence ``$\forall x\, (x\text{ is constructible})$'', which is again expressible by a first-order sentence in the language of set theory, is called Axiom of Constructibility, and is abbreviated by $\mathsf{V=L}$.
A finer analysis (see e.g.\ \cite[Chapter 13]{Jech2003}) reveals that the first-order formula expressing ``$x\in L_\alpha$'' is absolute among models of \( \ZFC \). In particular, 
\[
L \models \mathsf{V=L}, 
\]
and if $M$ is an inner model of \( \ZF \), then $L\subseteq M$. 
The universe $L$ is also a model of ``Global Choice'', namely, in \( L \) one can define a class function \( \tau \) such that \( \tau(z) \in z \) for every nonempty \( z \in L \). Indeed, \( L \) is well-ordered by a definable class $<_L$ with the property that $x<_L y$ whenever $x\in L_\alpha$ and $y\in L_\beta\setminus L_\alpha$ for some ordinals $\alpha<\beta$. 

We collect these two definability properties for future reference (see~\cite[Theorem 3.3]{Kanamori}). 

\begin{proposition} \label{prop:constructibility}
\begin{enumerate-(i)}
\item \label{prop:constructibility-1}
There is a sentence $\upsigma_0$ in the language of set theory such that, for any $M$ set or class, if $(M, \in) \models \upsigma_0$ then either $M=L$, or there is a limit ordinal $\delta > \omega$ such that $M=L_\delta$.

\item \label{prop:constructibility-2}
There is a formula $\upvarphi_0(u,v)$ in the language of set theory such that for every limit ordinal $\delta> \omega$, every $b\in L_\delta$, and every $a \in V$, we have that $a<_L b$ if and only if $a\in L_\delta$ and $(L_\delta,\in) \models \upvarphi_0[a/u,b/v]$.
\end{enumerate-(i)}
\end{proposition}

Part~\ref{prop:constructibility-2} of Proposition~\ref{prop:constructibility} says that $<_L$ is absolute among models of $\mathsf{V=L}$, and it is the key property to shift from abstract definability in the universe to the descriptive-set-theoretical hierarchies. It derives from the fact that $<_L$ is defined locally, i.e., $<_L$ is actually the union of the orders $<_{L_\alpha}$ on each \( L_\alpha \), where the definition \( \upvarphi_0(u,v) \) of $<_{L_\alpha}$ is the same for every $L_\alpha$. Therefore if $(L_\alpha,{\in}) \models \upvarphi_0[a/u,b/v]$ for two given elements $a ,b \in L $, then $a<_{L_\alpha}b$, and hence  $a<_L b$; conversely, if $a<_L b$, then for every $\alpha \in \On$ such that $a,b\in L_\alpha$ we have that $a <_{L_\alpha} b$, and hence $(L_\alpha,{\in}) \models \upvarphi_0[a/u,b/v]$. 

Finally, recall that $L \models  \GCH$. The key fact for proving this is 
that $\pow(\kappa)\cap L\subseteq L_{\kappa^+}$ for every \( \kappa \in \Cn \).

%

There is a relativized version \( L(X) \) of the constructible universe \( L \).

\begin{defin}
For any set $X$, we let \( L(X) = \bigcup_{\alpha \in \On} L_\alpha(X) \),  where the sets \( L_\alpha(X) \) are defined by transfinite recursion setting
 \begin{itemizenew}
  \item $L_0(X)=\mathrm{trcl}(\{X\})$; 
  \item $L_{\alpha+1}(X)=\Def(L_\alpha(X))$;
	\item $L_\gamma(X)=\bigcup_{\alpha<\gamma}L_\alpha(X)$ if $\gamma \in \On$ is limit.
 \end{itemizenew}
\end{defin}

The most popular instance of the above construction is $L(\RR)$. Recall that \( L(\RR) \) is not necessarily a model of \( \ZFC \), but it is a model of \( \ZF \). In fact, if $M$ is an inner model of \( \ZF \) such that $\mathsf{tr}(X)\in M$, then $L(X)\subseteq M$.

As for $L$, there is some absoluteness also for $L(X)$: the first-order formula that defines ``$x\in L_\alpha(X)$'' uses $X$ as a parameter, and it is absolute between models of $\ZF$ that contains $X$.

While it is not necessarily true that there is a global well-order on $L(X)$, with techniques similar to those employed in the case of \( L \) it is possible to build a definable surjection $\Phi \colon \On \times X^{<\omega}\to L(X)$ so that $\Phi \restriction (\alpha\times X^{<\omega})\in L(X)$  for each $\alpha \in \On$. Moreover, if $\kappa$ is a cardinal in $L(X)$, then $\Phi\restriction(\kappa\times X^{<\omega}) \colon \kappa\times X\to L_\kappa(X)$. This proves that there is a well-ordering of $\mathsf{trcl}(X)$ in $L(X)$ if and only if $L(X) \models \AC$.

In many cases the class $L(X)$ satisfies weak forms of choice. For example, $L(\RR) \models \DC_\omega$ if we assume $\ZF + \DC_\omega$ in \( V \) (see e.g.\ \cite[Proposition 11.13]{Kanamori}), and, similarly, $L(V_{\lambda+1}) \models \DC_\lambda$  if we assume $\ZF +\DC_\lambda$ in \( V \) (see e.g.\ \cite[Lemma 4.10]{Dimonte2018}). 

\section{Boolean algebras, filters, and ultrafilters} \label{subsec:ultrafilters}

Let \(\nu\) be a cardinal and \( X \) be any set. A \(\nu\)-algebra on \( X \) is a collection \( \mathcal{B} \subseteq \pow(X) \) which contains \( \emptyset \) and is closed under complements and well-ordered unions shorther than \(  \nu \). The \( \nu \)-algebra \( \mathcal{B} \) is generated by \( \mathcal{A} \subseteq \pow(X) \) if it is the smallest \(\nu\)-algebra containing \( \mathcal{A} \), in which case the elements of \( \mathcal{A} \) are often called generators of \( \mathcal{B} \) --- it is easy to see that \( \mathcal{B} \) can easily be obtained by closing the family \( \mathcal{A} \cup \{ \emptyset \} \) under complements and unions of length smaller than \( \nu \). A \(\nu\)-algebra is called \( \mu \)-generated if there is a collection \( \mathcal{A} \) of size at most \(\mu\) generating \( \mathcal{B} \).

 A filter \( \mathcal{F} \) on a set $X$ is a collection of subsets of $X$ closed under supersets and finite intersections. We tacitly assume that every filter we consider is proper, namely, that $\mathcal{F} \neq \pow(X)$ or, equivalently, \( \emptyset \notin \mathcal{F} \). A filter \( \mathcal{F} \) is principal if there is a nonempty set $X_0\subseteq X$ such that \( \mathcal{F}=\{A\subseteq X \mid A\supseteq X_0 \} \).
A filter is $\nu$-complete, for some $\nu \in \Cn$, if and only if it is closed under well-ordered intersections shorter than $\nu$.

A filter \( \mathcal{F} \) is an ultrafilter if and only if for each $A \subseteq X$, either $A \in \mathcal{F}$ or \( X\setminus A \in \mathcal{F} \); equivalently, \( \mathcal{F} \) is an ultrafilter if it is maximal for inclusion among all proper filters on \( X \).
In this paper, ultrafilters are usually denoted by \( \mathcal{U} \) or \( \V \).
Notice that an ultrafilter \( \U \) is principal if and only if it is generated by a single element \( \bar x \in X \), that is, \( \U = \{ A \subseteq X \mid  \bar x \in A \} \) for some \( \bar x \in X \).

\section{Large cardinals}

An uncountable cardinal $\kappa$ is called (strongly) inaccessible if it is regular and strong limit, that is, \( 2^{\nu} < \kappa \) for every \( \nu < \kappa \).%
\footnote{Note that, according to our notation, this implies a modicum of choice. There are also definitions of inaccessible cardinals in $\ZF$, but we will not need them in this paper.} 
Under $\ZFC$, we have that $H_\kappa$ is a model of the full $\ZFC$ if and only if $\kappa$ is an inaccessible cardinal. It is not possible to prove in $\ZFC$ that there are inaccessible cardinals, as the existence of an inaccessible cardinal implies the consistency of $\ZFC$. Cardinals with such a property are called large cardinals. Our reference for large cardinals is \cite{Kanamori}.

A cardinal $\kappa$ is measurable if and only if $\kappa>\omega$ and there is a $\kappa$-complete nonprincipal ultrafilter on $\kappa$.
Under $\ZFC$, if $\kappa$ is a measurable cardinal, then there exists a $\kappa$-complete normal ultrafilter on $\kappa$ (see e.g.\ \cite[Theorem 10.20]{Jech2003}), where ``normal'' means that the ultrafilter is closed under diagonal intersections of sequences of sets of length \( \kappa \); moreover, \( \kappa \) is inaccessible (and in fact it is the \( \kappa \)-th inaccessible cardinal). Therefore measurable cardinals are large cardinals. 

We will use the fact that, even in choiceless settings, if $\U$ is a nonprincipal ultrafilter on a set \( X \) that is at least \( \omega \)-complete, then there is a  maximal cardinal $\kappa$, sometimes called completeness number of \( \U \), such that $\U$ is $\kappa$-complete, and such $\kappa$ is indeed a measurable cardinal (see e.g.\ \cite[Lemma 1.1.4]{Larson2004}). 



Let $M$ and $N$ be two sets or two classes. A function $j \colon M\to N$ is called an elementary embedding if for every formula $\upvarphi(v_0, \dotsc, v_n)$ in the language of set theory and every $a_0, \dotsc, a_n \in M$, 
\[
M\models \upvarphi[a_0/v_0, \dotsc, a_n/v_n] \quad \iff \quad N\models \upvarphi[j(a_0)/v_0, \dotsc. j(a_n)/v_n]. 
\]
The critical point \( \crt(j) \) of an elementary embedding $j$ is the smallest $\alpha \in \On$ such that $j(\alpha)\neq\alpha$, if it exists. The existence of a measurable cardinal is equivalent to the existence of a nontrivial elementary embedding $j \colon V \to M$, where $M$ is an inner model of $V$. 

The following large cardinal assumption, due to Woodin, will be crucially used in Chapter~\ref{sec:applications}.

\begin{defin} \label{def:I0}
Let $\lambda \in \Cn$ be uncountable. We say that $\mathsf{I0}(\lambda)$ holds if there exists a nontrivial elementary embedding $j \colon L(V_{\lambda+1})\to L(V_{\lambda+1})$ such that $\crt(j) < \lambda$.
The assertion that \( \mathsf{I0}(\lambda) \) holds for some \( \lambda \in \Cn \) is abbreviated by \( \mathsf{I0} \).
\end{defin}

The axiom \( \mathsf{I0}(\lambda) \) is at the top of the large cardinal hierarchy. It implies that $\lambda$ has countable cofinality: indeed, if $\lambda_0=\crt(j)$ and $\lambda_{n+1}=j(\lambda_n)$, then $\lambda = \sup_{n \in \omega} \lambda_n$. Moreover, each $\lambda_n$ is $m$-huge for every $m\in\omega$, and in particular it is a measurable cardinal. Notably, if \( \mathsf{I0}(\lambda) \) holds then $V_{\lambda+1}$ cannot be well-ordered in $L(V_{\lambda+1})$, therefore $L(V_{\lambda+1}) \models \neg\AC$; however, \( L(V_{\lambda+1}) \models \DC_\lambda \). For more informations on \( \mathsf{I0} \), see~\cite{Cramer2017} and~\cite{Dimonte2018}.

%
%
%

\chapter{\( \lambda \)-Polish spaces} \label{sec:lambda-Polish}

From this point onward, as declared  we will tacitly work in our base theory 
\[ 
\ZF + \AC_\lambda(\pre{\lambda}{2}) . 
\]
However, some of the results in this chapter can be proved in even weaker theories, so we sometimes briefly discuss the exact amount of choice needed for the specific statement at hand.

\section{Definition and examples} \label{sec:def-example-Polish}

The following definition naturally generalizes that of a Polish space, which corresponds to the case \( \lambda = \omega \).

\begin{defin}
Let \( \lambda \) be an infinite cardinal. A topological space \( X \) is \markdef{\(\lambda\)-Polish} if it is completely metrizable and \( \weight(X) \leq \lambda \).
\end{defin}

 When dealing with a \(\lambda\)-Polish space, we will often implicitly assume that it comes with a corresponding well-ordered basis of size at most \( \lambda \); the terminology ``basic open set'', if not specified otherwise, tacitly refers to sets in such a basis. Recall also that the condition \( \weight(X) \leq \lambda \) is equivalent to \( \dens(X) \leq \lambda \) because \( X \) is metrizable. 

\begin{example} \label{xmp:lambda-Polish}
The following are examples of \(\lambda\)-Polish spaces.
\begin{enumerate-(a)}
\item
Any (classical) Polish space---although in this context they somewhat trivialize if \( 2^{\aleph_0} \leq  \lambda \) (this will often be the case).
\item
Any discrete space \( X \) of weight at most \( \lambda \). Notice that in this case \( X \) is well-orderable and \( |X| \leq \lambda \).
\item \label{xmp:lambda-Polish-2}
The space 
\[ 
B(\lambda) =  \pre{\omega}{\lambda} 
\] 
and, if \( \cf(\lambda) = \omega \) and \( \vec{\lambda} = (\lambda_i)_{i \in \omega} \) is cofinal in \(\lambda\), the space 
\[ 
C(\lambda) = \pre{\omega}{(\vec{\lambda})} =  \prod_{i \in \omega} \lambda_i , 
\]
where both \( B(\lambda) \) and \( C(\lambda) \) are endowed with the product of the discrete topologies on \( \lambda \) and \(  \lambda_i \) or, equivalently, with the (relativization of the) bounded topology \( \tau_b \) on \( \pre{\omega}{\lambda} \). Such spaces were studied in~\cite{Stone1962}, where in particular it is shown that \( B(\lambda) \approx C(\lambda) \). The latter result is obtained in~\cite{Stone1962} through a topological characterization of \( B(\lambda) \). We will give a direct proof of (a generalization of) this fact in Theorem~\ref{thm:homeomorphictoCantor}, and also provide several alternative topological characterizations of \( B(\lambda) \) when \( \cf(\lambda) = \omega \) in Theorems~\ref{thm:characterizationCantor} and~\ref{thm:characterizationCantor2}.
\item \label{xmp:lambda-Polish-3}
The space 
\[ 
\pre{\lambda}{2} 
\]
of all binary \( \lambda \)-sequences equipped with the bounded topology is called generalized Cantor space. It is a \(\lambda\)-Polish space if and only if \( \cf(\lambda) = \omega \) and \( 2^{< \lambda} = \lambda \) or, equivalently, if \(  \cf(\lambda) = \omega \) and \(\lambda\) is strong limit (see~\cite[Proposition 3.12]{AM}); this is precisely condition~\ref{requirementsonlambda-2} from Section~\ref{sec:setup}. Notice also that the space \( \pre{\lambda}{\lambda} \), which at first glance might be proposed as a generalization of the Baire space, has weight \( \lambda^{< \lambda} \), and thus it becomes \(\lambda\)-Polish if and only if \( \cf(\lambda) = \omega \) and \( \lambda^{< \lambda} = \lambda \). Since the latter condition is equivalent to ``\( 2^{< \lambda} = \lambda \) and \( \lambda \) is regular'', this happens only when \( \lambda = \omega \), in which case we are back to the usual Baire space \( \pre{\omega}{\omega} \).
\item \label{xmp:lambda-Polish-4}
Let \( \lambda \) be a limit cardinal.
Then the Woodin topology on $V_{\lambda+1}$, which was introduced in~\cite[p.\ 298]{Woodin2011}, is the topology generated by the basic open sets 
\begin{equation*}
O_{(a,\alpha)}=\{A\subseteq V_\lambda \mid A\cap V_\alpha=a\}
\end{equation*}
with $\alpha<\lambda$ and $a\subseteq V_\alpha$.  Such basis has cardinality 
\( |V_\lambda| \). Assume that \( V_\lambda \) is well-orderable, so that \( |V_\lambda | \) is a cardinal. Then $|V_\lambda|$ is strong 
limit because \(\lambda\) is limit, and therefore this is the least possible cardinality for a basis of the Woodin topology by Fact~\ref{fct:weightstronglimit}. Therefore 
the weight of $V_{\lambda+1}$  is $|V_\lambda|$. If 
$\cf(\lambda)=\omega$, then we can define a complete metric on $V_{\lambda+1}$ by 
fixing a strictly increasing sequence $(\lambda_i)_{i\in\omega}$ cofinal in $\lambda$ 
and setting, for all distinct $x,y\in V_{\lambda+1}$, $d(x,y)=2^{-i}$ if and only if $i \in \omega$ is
smallest such that $x\cap V_{\lambda_i}\neq y\cap V_{\lambda_i}$. Conversely, if 
$V_{\lambda+1}$ is first-countable, then $\cf(\lambda)=\omega$. Therefore 
$V_{\lambda+1}$ is (completely) metrizable if and only if $\cf(\lambda)=\omega$, and in such case 
$V_{\lambda+1}$ endowed with the Woodin topology is a $|V_\lambda|$-Polish space.
Thus \( V_{\lambda+1} \) is a \(\lambda\)-Polish space whenever \( \cf(\lambda) = \omega\) and \( |V_\lambda| = \lambda \). Notice that the latter condition is equivalent to \( \beth_\lambda= \lambda \), that is, condition~\ref{requirementsonlambda-3} from Section~\ref{sec:setup}.
\item
If \( X \) is \(\lambda\)-Polish, then the hyperspace \( K(X) \) of compact subsets of \( X \) endowed with the Vietoris topology is still \(\lambda\)-Polish (see~\cite[Section 4.F]{Kechris1995}).
\item
All Banach spaces with density character \(\lambda\), viewed as topological spaces, are \(\lambda\)-Polish. This includes e.g.\ the Banach space \( c_0^\lambda \) consisting of those sequences \( (x_\alpha)_{\alpha < \lambda} \in \RR^\lambda \) such that for all \( \varepsilon \in \RR^+ \) one has \( |x_\alpha| < \varepsilon \) for all but finitely many \( \alpha < \lambda \), equipped with the pointwise operations and the sup norm.
\end{enumerate-(a)}
\end{example}

The class of \(\lambda\)-Polish spaces is clearly closed under \( G_\delta \) subspaces,%
\footnote{Indeed, by~\cite[Theorem 3.11]{Kechris1995} a subspace of a \(\lambda\)-Polish space \( X \) is \(\lambda\)-Polish if and only if it is \( G_\delta \) in \( X \).}
sums 
(i.e.\ disjoint unions) of size at most \(  \lambda \), and countable products. However, if \( \cf (\lambda) = \omega \) such class is not closed under 
 products of size \( \kappa > \omega \), independently of the chosen support for the product topology. Indeed, let \( (Y_\alpha)_{\alpha < \kappa }\) be a family of discrete spaces of size \( \lambda \). Consider the product 
\( Y = \prod_{\alpha < \kappa} Y_\alpha \) equipped with the \( < \mu \)-supported product topology, 
where \( \mu \) is an infinite cardinal such that \( \mu \leq \kappa \). If \( \mu > \omega \), then the weight of \( Y \), if defined, 
would be at least \( \lambda^\omega = \lambda^{\cf(\lambda)} > \lambda \); if instead 
\( \mu = \omega \), then \( Y \) would be non-first-countable, and hence not metrizable. In both cases, 
the product \( Y \) is not \(\lambda\)-Polish. (In contrast, we anticipate that a strictly related kind of spaces, namely the standard \(\lambda\)-Borel spaces, are closed under products of size at most \( \lambda \) if \( 2^{< \lambda} = \lambda \) --- see Section~\ref{sec:standardBorel}, and in particular Proposition~\ref{prop:productofstandardBorelspaces}.)

Extension theorems like the ones by Tietze~\cite[Theorem 1.3]{Kechris1995}, Kuratowski~\cite[Theorem 3.8]{Kechris1995}, and Lavrentiev~\cite[Theorem 3.9]{Kechris1995} are in particular true for \(\lambda\)-Polish spaces,  as well as Urysohn's lemma~\cite[Theorem 1.2]{Kechris1995}. Moreover, every \(\lambda\)-Polish space is Baire by Baire's category theorem~\cite[Theorem 8.4]{Kechris1995}. There are a number of consequence of this that remains true in our generalized setup, e.g.\ the topological \( 0 \)-\(1 \) law~\cite[Theorem 8.46]{Kechris1995}. However, the results that can be proved using category arguments are sparse in the realm of \(\lambda\)-Polish spaces with \( \lambda > \omega \): for example, the natural notion of ``Borelness'' in this context is, arguably, \(\lambda\)-Borelness (see Section~\ref{sec:Borel}), but there are very simple \(\lambda\)-Borel subsets of \( B(\lambda) \) that do not have the Baire property. 

\begin{example} 
Let \(\lambda\) be such that \( 2^{\aleph_0} \leq \lambda \). Let \( A  \subseteq \pre{\omega}{2} \) be a flip set, that is, a set such that for every \( x,y \in \pre{\omega}{2} \), if  \( \exists! n \in \omega \, ( x(n) \neq y(n)) \) then \( x \in A \iff y \notin A \) --- such a set can be constructed recursively using any well-order of \( \pre{\omega}{2} \), which exists by our assumption on \( \lambda \). Let \( f \colon \lambda \to 2 \) be the map such that \( f(\alpha) = 0 \) if and only if \( \alpha \) is even, i.e.\ it is of the form \( \gamma+k \) with \( \gamma = 0 \) or \( \gamma \) limit, and \( k \in \omega \) even, and let \( g \colon B(\lambda) \to \pre{\omega}{2} \) be defined using \( f \) coordinatewise. Finally, let \( B = g^{-1}(A) \). The set \( A \) is the union of \( 2^{\aleph_0} \)-many singletons, and hence \( B \) is a union of at most \(\lambda\)-many closed sets because \( g \) is continuous and \( 2^{\aleph_0} \leq \lambda \): this shows that \( B \) is \(\lambda\)-Borel, and in fact it belongs to the second level of the \(\lambda\)-Borel hierarchy (see Section~\ref{sec:Borelhierarchy}). 
Notice that the set \( B \) is neither meager nor comeager in \( \Nbhd_s \), for any \( s \in \pre{<\omega}{\lambda} \). 
Indeed, let \( \phi \colon \lambda \to \lambda \) be the bijection swapping every even ordinal with its successor. Then the map \( h \colon \Nbhd_s \to \Nbhd_s \) defined by letting \( h(x)(\lh(s)) = \phi(x(\lh(s))) \) and \( h(x)(n) = x(n) \) for every \( n \neq \lh(s) \)
is a homeomorphism of \( \Nbhd_s \) into itself such that \( h(B \cap \Nbhd_s) = \Nbhd_s \setminus B \). If \( B \) where either meager or comeager in \( \Nbhd_s \), then we would obtain that \( \Nbhd_s \) is meager in itself, against the fact that it is a clopen subspace of \( B(\lambda) \), and hence a \(\lambda\)-Polish space. Therefore \( B \) cannot have the Baire property because of~\cite[Proposition 8.26]{Kechris1995}.
\end{example}

In view of this,
one may be tempted to generalize the notion of meagerness by 
considering \( \nu \)-meager sets with \( \nu > \omega \), where
a subset of a \(\lambda\)-Polish space is \( \nu \)-meager if it can be written as a \( \nu \)-sized union of nowhere dense sets. Setting \( \nu = \lambda \), one would then get that all \( \lambda \)-Borel sets have the \( \lambda \)-Baire property, i.e.\ they differ from an open set by a \( \lambda \)-meager set. However such a notion would be useless because e.g.\ \( B(\lambda) \) is already \( \omega_1 \)-meager, hence all of its subsets trivially have the \( \nu \)-Baire property for any \( \nu > \omega \). To see this, for \( \alpha < \omega_1 \) set \( U_\alpha = \{ x \in B(\lambda) \mid x(n) = \alpha \text{ for some } n \in \omega \} \) and observe that all the sets \( U_\alpha \) are open dense, yet \( \bigcap_{\alpha < \omega_1 } U_\alpha = \emptyset \).

Notice also that in the context of metrizable spaces with a well-orderable basis, complete metrizability is equivalent 
to the space being strong Choquet. This follows from~\cite[Section 8.E]{Kechris1995} once we observe that  separability can be dropped because~\cite[Lemma 8.20]{Kechris1995} directly follows from paracompactness.
Moreover, if \( X \) has a basis \( \mathcal{B} \) of size at most \( \lambda \) and \( f \colon X \to Y \) is an open continuous surjection, then \( \{ f(U) \mid U \in \mathcal{B} \} \) is a basis for \( Y \) witnessing \( \weight(Y) \leq \lambda \).
 It readily follows that Sierpi\'nski's theorem on open continuous images of Polish spaces can be generalized as follows.

\begin{proposition} \label{prop:sierpinski}
Let \( X \) be a \(\lambda\)-Polish space and \( Y \) be a metrizable space. 
If there is a continuous open surjection of \( X \) onto \( Y \), then \( Y \) is \(\lambda\)-Polish as well.
\end{proposition}

We finally mention that by another result of Sierpi\'nski~\cite{Sierpinski1945},  if 
\( 2^{<\lambda} = \lambda \) then there is a (non necessarily unique) universal 
\(\lambda\)-Polish space, that is, a \(\lambda\)-Polish space \( \mathfrak{U}_\lambda \) in which every 
\(\lambda\)-Polish space can be isometrically embedded (as a closed set). However, \( \mathfrak{U}_\lambda \) is not necessarily a higher analogue of the Urysohn space, in the sense that it cannot be 
\( \lambda \)-homogeneous%
\footnote{A space \( X \) is called \(\lambda\)-homogenous if every \emph{partial} isometry of \( X \) of size less than \( \lambda \) can be extended to an auto-isometry of the whole space.}
 unless \( \lambda \) is regular (see~\cite{Katetov1986}).

\section{Basic properties}\label{sec:basicpropertiesforpolish}

\emph{Unless otherwise stated, from now on we let \( \lambda > \omega \) be a limit cardinal such that \( \cf(\lambda) = \omega \). We also fix once and for all a  strictly increasing sequence  \( \vec{\lambda} = (\lambda_i)_{i \in \omega} \) of cardinals cofinal in \(\lambda\). When \( 2^{< \lambda} = \lambda \), so that we have enough choice below \(\lambda\) to ensure that \( \kappa^+ \) is regular for every \( \kappa < \lambda \) by Fact~\ref{fct:lambdaunionewithoutchoice}\ref{fct:lambdaunionewithoutchoice-1}, we further convene that each \(\lambda_i\) is regular.}

\medskip

 In particular, for the sake of definiteness we let \( C(\lambda) \) be the space defined as in Example~\ref{xmp:lambda-Polish}\ref{xmp:lambda-Polish-2} using such cardinals \( \lambda_i \) (the choice of the sequence \( \vec{\lambda} \) is irrelevant by Remark~\ref{rmk:homeomorphictoCantor} below). 

The interested reader might appreciate that all results in this section can be proved in \( \ZF \) alone. For example,
in the next lemma no choice is actually required because \( \pre{<\omega}{\lambda} \) and all its subsets can be canonically well-ordered using G\"odel's enumeration.

\begin{lemma}	 \label{lem:homeomorphictoCantor} \axioms{\( \ZF \)}
Let \( T \subseteq \pre{< \omega}{\lambda} \) be a pruned tree. Then \( [T] \approx B(\lambda) \) if and only if 
\begin{equation} 
\label{eq:largepartition}
\forall s \in T \, \big( |\{ t \in T \mid t \supseteq s \}| \geq \lambda \big).
 \end{equation} 
\end{lemma}

\begin{proof}
In \( B(\lambda) \), every nonempty open set has weight \( \lambda \). If~\eqref{eq:largepartition} fails as witnessed by some \( s \in T \), then the nonempty basic open set \( \Nbhd_s([T]) \) of \( [T] \) would have weight smaller than \( \lambda \), hence \( [T] \not\approx B(\lambda) \).

For the other direction, first notice that
condition~\eqref{eq:largepartition} is equivalent to
\begin{equation} 
\label{eq:largepartition2}
\forall s \in T  \, \forall \nu < \lambda \, \exists n > \lh(s) \, \big(  |\{ t \in \mathrm{Lev}_n( T ) \mid t \supseteq s \}| \geq \nu \big).
 \end{equation} 

Indeed, if~\eqref{eq:largepartition2} fails then there are \( s \in T \) and \( \nu < \lambda \) such that \(  |\{ t \in \mathrm{Lev}_n( T ) \mid t \supseteq s  \}| < \nu \) for all \( n > \lh(s) \): this implies that \( |\{ t \in T \mid t \supseteq s \}| \leq \nu \times \omega < \lambda \), contradicting~\eqref{eq:largepartition}.  The other implication is obvious, as~\eqref{eq:largepartition2} implies that \( |\{ t \in T \mid t \supseteq s \}| \geq \nu \) for all \( \nu < \lambda \).

So let us assume that~\eqref{eq:largepartition2} holds and prove that \( [T] \approx \pre{\omega}{\lambda} = B(\lambda) \).
Say that  \( T' \subseteq T\) is a partition of \( t \in T \) if the elements of \( T' \) properly extend \( t \), are pairwise incomparable, and for all \( t \subseteq x \in [T] \) there is \( t' \in T' \) such that \( t' \subseteq x \). (This is equivalent to requiring that \( \{ \Nbhd_{t'}([T]) \mid t' \in T' \} \) is a clopen partition of \(  \Nbhd_t([T])  \), whence the name).

\begin{claim}
For every \( t \in T \) there is a \emph{canonical} partition of \( t \) of size \(\lambda\).
\end{claim}

\begin{proof}[Proof of the claim]
Fix \( t \in T \). Apply~\eqref{eq:largepartition2} with \( s = t \) and \( \nu = \omega \), and let \( n > \lh(t) \) be smallest 
such that the set \( \{ t' \in \mathrm{Lev}_n(T) \mid {t' \supseteq t} \} \) is infinite. Let \( (t_i)_{i < \alpha} \) be the enumeration of such set according to G\"odel's ordering, for the appropriate \( \alpha \geq \omega \). For each \( i < \omega \), apply 
again~\eqref{eq:largepartition2} with \( s =  t_i \) and \( \nu = \lambda_i \) and let 
\( m_i > \lh(t_i)  = n \) be smallest such that the set \( T_i = \{ t' \in \mathrm{Lev}_{m_i}(T) \mid t' \supseteq t_i \} \) has size 
greater than \( \lambda_i \). Then
\[ 
T' = \{ t_i \mid \omega \leq i < \alpha \} \cup  \bigcup\nolimits_{i < \omega} T_i 
 \] 
is a partition of \( t \) of size \( \lambda \).
\end{proof}
 
Define \( \varphi \colon \pre{<\omega}{\lambda} \to T \) by induction on \( \lh(s) \) as follows. Set \( \varphi(\emptyset) = \emptyset \). Given any \( s \in \pre{< \omega}{\lambda} \), by the claim we have a canonical partition \( T' \subseteq T \) of \( \varphi(s) \) of size \( \lambda \). Enumerate without repetitions (following G\"odel's ordering again) \( T' \) as \( (t_\beta)_{\beta < \lambda} \), and then set \( \varphi(s {}^\smallfrown{} \beta) = t_\beta \) for all \( \beta < \lambda \). It is clear \( \varphi \) is monotone and such that \( \lh(\varphi(s)) \geq \lh(s) \) for all \( s \in \pre{<\omega}{\lambda} \), hence the map
\[ 
f_\varphi \colon \pre{\omega}{\lambda}  \to [T] \colon x \mapsto \bigcup\nolimits_{i < \omega} \varphi(x \restriction i)
 \] 
is well-defined. By construction, \( f_\varphi \) is a bijection. Moreover, it is a homeomorphism because  \( f(\Nbhd_s) = \Nbhd_{\varphi(s)}([T]) \) for all \( s \in \pre{< \omega}{\lambda} \) and 
\( \{ \Nbhd_{\varphi(s)}([T]) \mid s \in \pre{< \omega}{\lambda} \} \) is, by construction, a basis for the topology of \( [T] \).
\end{proof}

The next theorem directly shows that, under the appropriate hypothesis on \(\lambda\), all spaces in items \ref{xmp:lambda-Polish-2}--\ref{xmp:lambda-Polish-4} of  Example~\ref{xmp:lambda-Polish} are homeomorphic. Part~\ref{thm:homeomorphictoCantor-2} of Theorem~\ref{thm:homeomorphictoCantor} has been implicitly considered in~\cite[Section 7]{DzaVaa}, although in a slightly different setup. 

\begin{theorem} \label{thm:homeomorphictoCantor}
\begin{enumerate-(i)}
\item \label{thm:homeomorphictoCantor-1}
For any sequence \( \vec{\lambda}' =  (\lambda'_i)_{i \in \omega} \)  of cardinals  cofinal in \( \lambda \),
\[ 
B(\lambda) \approx \pre{\omega}{(\vec{\lambda}')} =  \prod_{i \in \omega} \lambda'_i.
 \] 
In particular, \( B(\lambda) \approx C(\lambda) \).
\item \label{thm:homeomorphictoCantor-2}
If  \( 2^{<\lambda} = \lambda \), we further have 
\[ 
\pre{\lambda}{2} \approx B(\lambda) \approx C(\lambda) .
\]
\item \label{thm:homeomorphictoCantor-3}
If moreover $|V_\lambda|=\lambda$, then
\[
V_{\lambda+1} \approx \pre{\lambda}{2} \approx B(\lambda) \approx C(\lambda) .
\]
\end{enumerate-(i)}
\end{theorem}

Notice that the conditions on \(\lambda\) in the various items are exactly the ones needed to guarantee that the mentioned spaces are \(\lambda\)-Polish: so one can briefly say that the spaces \( B(\lambda) \), \( C(\lambda) \), \( \pre{\lambda}{2} \) and \( V_{\lambda+1} \) are all homeomorphic whenever they are \(\lambda\)-Polish.

\begin{proof}
\ref{thm:homeomorphictoCantor-1}
Recall from Section~\ref{subsec:trees} that \( \pre{\omega}{(\vec{\lambda}')} = [\pre{<\omega}{(\vec{\lambda}')}] \).
Since \( T = \pre{<\omega}{(\vec{\lambda}')} \) clearly satisfies~\eqref{eq:largepartition} (because the \( \lambda'_i \) are cofinal in \( \lambda \)), the result follows from Lemma~\ref{lem:homeomorphictoCantor}.

\ref{thm:homeomorphictoCantor-2}
The space \( \pre{\lambda}{2} \) is obviously homeomorphic to \( \prod_{i < \omega} \pre{\lambda_i}{2} \) where each \( \pre{\lambda_i}{2} \) is endowed with the discrete topology. Thus \( \pre{\lambda}{2} \approx \prod_{i \in \omega} \lambda'_i \) where \( \lambda'_i = | \pre{\lambda_i}{2}| < \lambda \) (the last inequality follows from  \( 2^{< \lambda} = \lambda\)). Since \( \lambda_i < 2^{\lambda_i} = \lambda'_i \) the \( \lambda'_i \) are cofinal in \(\lambda\), thus the result follows from part~\ref{thm:homeomorphictoCantor-1}.

\ref{thm:homeomorphictoCantor-3}
 For each $i\in\omega$, let $V^-_{\lambda_i+1} = \{ A  \setminus V_{\lambda_{i-1}} \mid A \subseteq V_{\lambda_i}\}=\pow(V_{\lambda_i}\setminus V_{\lambda_{i-1}})$ and $V_{\lambda_{-1}}=\emptyset$. We are going to prove that $V_{\lambda+1}\approx  \prod_{i \in \omega} \lambda'_i$ where \( \lambda'_i = |V^-_{\lambda_i+1}| \); the \( \lambda'_i \)'s are clearly cofinal in \( \lambda \) because so are the cardinals \( \lambda_i \). By part~\ref{thm:homeomorphictoCantor-1}, it will then follow that  $V_{\lambda+1}\approx C(\lambda) \approx B(\lambda)$, hence also \( V_{\lambda+1} \approx \pre{\lambda}{2} \) by part~\ref{thm:homeomorphictoCantor-2} (which can be applied here because \( |V_\lambda| = \lambda \) implies \( 2^{< \lambda} = \lambda \) when \(\lambda\) is limit). 
 
For each $i\in\omega$, fix a bijection  $s_i \colon \lambda_i'\to V^-_{\lambda_i+1}$. Then $s \colon \prod_{i \in \omega} \lambda'_i\to V_{\lambda+1}$, defined by $s(x)=\bigcup_{i\in\omega}s_i(x(i))$ for all \( x \in\prod_{i \in \omega} \lambda'_i \), is a homeomorphism, since $\{O_{(a,\lambda_i)} \mid i\in\omega \wedge a\subseteq V_{\lambda_i}\}$ is a basis for the Woodin topology on \( V_{\lambda+1} \). 
\end{proof}

\begin{remark} \label{rmk:homeomorphictoCantor}
By Theorem~\ref{thm:homeomorphictoCantor}\ref{thm:homeomorphictoCantor-1}, the choice of the sequence \( (\lambda_i)_{i \in \omega} \) in the definition of \( C(\lambda) \) is irrelevant and can be changed when needed (provided that we always work with a sequence cofinal in \(\lambda\)). In particular, the chosen sequence needs not to be increasing.
\end{remark}

Recall that a subset \( C \subseteq X \) of a topological space \( X \) is a retract of \( X \) if there is a continuous function \( f \colon X \to C \) such that \( f(x) = x \) for all \( x \in C \) (in particular, \( f \) is surjective). 
By~\cite[Proposition 2.8]{Kechris1995} we get%
\footnote{As formulated in~\cite{Kechris1995}, the proof requires a bit of dependent choices when dealing with trees on arbitrary sets \( A \). In contrast, we only need to work with trees on the well-ordered set \(\lambda\), thus in our case the argument clearly works in \( \ZF \).}

\begin{proposition} \label{prop:retract} \axioms{\( \ZF \)}
Every nonempty closed subset of \( B(\lambda) \) is a retract of the entire space.
\end{proposition}

The next proposition shows in particular that \( \pre{\omega}{\lambda} =  B(\lambda) \approx C(\lambda) \) is surjectively universal for \(\lambda\)-Polish spaces. It is already stated and proved (in a slightly more precise form, which is not relevant here) in~\cite[Section 3.3]{Stone1962}. 

\begin{proposition} \label{prop:surjection} \axioms{\( \ZF \)}
Let \( X \) be \(\lambda\)-Polish. Then there is a closed set \( F \subseteq B(\lambda) \) and a continuous bijection \( f \colon F \to X \). If moreover \( X \neq \emptyset \), then there is a continuous surjection \( g \colon B(\lambda) \to X \).
\end{proposition}

One way to prove Proposition~\ref{prop:surjection} is to adapt
the proof of~\cite[Theorem 7.9]{Kechris1995} and build an appropriate \(\lambda\)-scheme on \( X \), still requiring the sets in the scheme to be \( F_\sigma \). This is an illuminating example of the fact that not all classical proofs must be blindly generalized by replacing \( \omega \) with \(\lambda\) everywhere: indeed, in this case this would lead us to build a scheme consisting of sets which are union of \(\lambda\)-many closed sets, a move that unfortunately would not make the argument go through. Because of this subtlety (and for the sake of completeness) we feel that it is fair towards the reader to provide the details of the proof. Moreover, since the argument in~\cite[Theorem 7.9]{Kechris1995} requires a form of dependent choice, which is not necessarily available in our setup, we slightly refine the argument and show that we do not need any choice principle at all.

\begin{proofcompare}{Kechris1995}{Theorem 7.9} 
Fix a compatible complete metric \( d \) on \( X \) with \( d \leq 1 \).
By Fact~\ref{fct:scheme} it is enough to build a \( \lambda \)-scheme \( \{ B_s \mid s \in \pre{< \omega}{\lambda} \} \) on \( X \) such that for all \( s \in \pre{< \omega}{\lambda} \) and \( \beta, \beta' < \lambda \)
\begin{enumerate-(1)}
\item
\( B_\emptyset = X \) and \( B_s =\bigcup_{\beta< \lambda} B_{s {}^\smallfrown{} \beta} \);
\item
\( B_{s {}^\smallfrown{}  \beta } \cap B_{s {}^\smallfrown{} \beta'} = \emptyset \) if \( \beta \neq \beta' \);
\item
\( \mathrm{cl}(B_{s {}^\smallfrown{} \beta}) \subseteq B_s \);
\item
\( \mathrm{diam}(B_s) \leq 2^{- \lh(s)} \).
\end{enumerate-(1)}
We will additionally construct
for each \( B_s \) a family \( (C^s_n)_{n \in \omega} \) of closed sets such that \( B_s = \bigcup_{n \in \omega} C^s_n\), so that \( B_s \) is in particular an \( F_\sigma \) set.

The construction is by recursion on \( \mathrm{lh}(s) \). Set \( B_\emptyset = X \) and \( C^\emptyset_n = X \) for all \( n \in \omega \). Now assume that \( B_s \) and \( (C^s_n)_{n \in \omega} \) have been constructed as above. 
Cover \( B_s \) with basic open sets \( D_{\alpha} \), \( \alpha < \lambda \), with diameter at most \( 2^{-(\lh(s)+1)} \). Consider the sets
\[ 
F_{\alpha,n} = \left( D_\alpha \setminus \bigcup\nolimits_{\beta< \alpha} D_\beta \right) \cap \left( C^s_n \setminus \bigcup\nolimits_{m < n} C^s_m \right).
 \] 
 for \( \alpha < \lambda \) and \( n \in \omega \).
Clearly, \( \mathrm{cl}(F_{\alpha,n}) \subseteq \mathrm{cl}(C^s_n) = C^s_n \subseteq B_s \) and \( \mathrm{diam}(F_{\alpha,n}) \leq \mathrm{diam}(D_\alpha) \leq 2^{-(\lh(s)+1)} \). Finally, \( \bigcup_{\alpha< \beta, n \in \omega} F_{\alpha,n} = B_s \) because both families \( \{ C^s_n \mid n \in \omega \} \) and \( \{ D_\alpha \mid \alpha < \lambda \} \) cover \( B_s \). Thus it is enough to let \( (B_{s {}^\smallfrown{} \beta})_{\beta < \nu} \) be any enumeration %
 without repetitions of those sets \( F_{\alpha,n} \) which are nonempty (where \( 1 \leq \nu \leq \lambda \) is the cardinality of such family), and 
\( B_{s {}^\smallfrown{} \gamma} = \emptyset \) for \( \nu \leq \gamma < \lambda \) if any such \(\gamma\) exists. 
Since by construction each \( B_{s {}^\smallfrown{} \beta} \) is a Boolean combination of open sets, by Fact~\ref{fct:booleancombinationsareFsigma} we can canonically write it as a countable union \( \bigcup_{n \in \omega} C^{s {}^\smallfrown{} \beta}_n \) of closed sets \( C^{s {}^\smallfrown{} \beta}_n \); for \( \nu \leq \gamma < \lambda \), we just set \( C^{s {}^\smallfrown{} \gamma}_n = \emptyset \) for all \( n \in \omega \).

The additional part follows from Proposition~\ref{prop:retract}.
\end{proofcompare}

\begin{remark}\label{rmk:inverseofinjection}
Let \( f \) and \( F \) be as in the proof of Proposition~\ref{prop:surjection}. Then  \( f(\Nbhd_s(F)) = B_s \), hence the image of any basic open set through \( f \) is an \( F_\sigma \) subset of \( X \).
\end{remark}

If we drop injectivity, we can require that the continuous surjection \( g \colon B(\lambda) \to X \) from the above proposition is also open (Corollary~\ref{cor:surjection2}). The easy and short proof of this fact is the natural adaptation of the classical one: our interest in it is that it yields to the technical fact Proposition~\ref{prop:tree-basis} that will frequently be used later. Let us first isolate the following notion.

\begin{defin} \label{def:tree-basis}
A \markdef{tree-basis} for a \(\lambda\)-Polish space \( X \) is a family \( \mathcal{B} = \{ U_s \mid s \in T \} \) of subsets of \( X \) indexed by some \( \omega \)-tree \( T \subseteq \pre{<\omega}{\lambda} \) (called \markdef{index tree} of \( \mathcal{B} \)) such that  for all \( s \in T \) and \( \alpha < \lambda \) with \( s {}^\smallfrown{} \alpha \in T \)
\begin{enumerate-(a)}
\item \label{def:tree-basis-1}
\( U_s  \) is a nonempty open set; 
\item \label{def:tree-basis-2}
\( \mathrm{cl}(U_{s {}^\smallfrown{} \alpha}) \subseteq U_s \);
\item \label{def:tree-basis-3}
\( U_\emptyset = X \) and \( U_s = \bigcup \{ U_{s {}^\smallfrown{} \beta} \mid s {}^\smallfrown{}  \beta \in T \} \);
\item \label{def:tree-basis-4}
\(  \mathrm{diam}(U_s) \leq 2^{-\lh(s)} \) (where the diameter is computed with respect to any complete metric \( d \leq 1 \) on \( X \) fixed in advance).
\end{enumerate-(a)}
\end{defin}

The name tree-basis comes from the fact that every family \( \mathcal{B} = \{ U_s \mid s \in T \} \) satisfying conditions~\ref{def:tree-basis-3}--\ref{def:tree-basis-4} above is automatically a basis%
\footnote{Given any open ball \( B \) centered in \( x \in X \) with radius \( 2^{-k} \) for some \( k \in \omega \), by condition~\ref{def:tree-basis-4} we have \( x \in U_s \subseteq B \) for any \( s \in \mathrm{Lev}_{k+1}(T) \) such that \( x \in U_s \); such an \( s \) exists because by condition~\ref{def:tree-basis-3}, we have \( X = \bigcup \{ U_s \mid s \in \mathrm{Lev}_n(T) \} \) for every \( n \in \omega \).}
 for the the topology of \( X \), whose inclusion relation is ``controlled'' by its index tree \( T \). A tree-basis \( \mathcal{B} \) clearly constitutes a \(\lambda\)-scheme on \( X \) whose induced function is denoted by \( f_\mathcal{B} \). By Fact~\ref{fct:scheme}, conditions~\ref{def:tree-basis-1}--\ref{def:tree-basis-4} ensure that this is always a continuous open surjection from the closed set \( [T] \subseteq B(\lambda) \) onto \( X \).

\begin{proposition} \label{prop:tree-basis} \axioms{\( \ZF \)}
Every nonempty \(\lambda\)-Polish space \( X \) admits a tree-basis with index tree the full \( \pre{< \omega}{\lambda} \). 
\end{proposition}

\begin{proof} 
Fix a compatible complete metric \( d \) on \( X \) with \( d \leq 1 \), and a basis for \( X \) of size at most \( \lambda \).
We construct the tree-basis
\( \mathcal{B} = \{ U_s \mid s \in \pre{<\omega}{\lambda} \} \) for \( X \)
by recursion on \( \lh(s) \), starting with \( U_\emptyset = X \). Assuming that \(  U_s \) has been defined, let \( (U_{s {}^\smallfrown{}  \beta})_{\beta < \lambda} \) be an enumeration (possibly with repetitions) of all nonempty basic open sets having diameter at most \( 2^{-(\lh(s)+1)} \) and closure contained in \( U_s \). It is then clear that conditions~\ref{def:tree-basis-1}--\ref{def:tree-basis-4} from Definition~\ref{def:tree-basis} are satisfied.
\end{proof}


\begin{corollary} \label{cor:surjection2} \axioms{\( \ZF \)}
For every nonempty \(\lambda\)-Polish space \( X \) there is a continuous open surjection \( g \colon B(\lambda) \to X \).
\end{corollary}

A partial converse to Corollary~\ref{cor:surjection2} is given by Proposition~\ref{prop:sierpinski}. Combining these two results  we get a characterization of \(\lambda\)-Polish spaces depending just on their prototype \( B(\lambda) \).

\begin{corollary} \label{cor:charwithsurj} \axioms{\( \ZF \) }
Let \( X \) be a nonempty metrizable space.
Then \( X \) is \(\lambda\)-Polish
if and only if there is a continuous open surjection \( g \colon B(\lambda) \to X \).
\end{corollary}

Along the same lines, we prove the following result, which will be used in the sequel.
It might be a known fact, at least in the case \( \lambda =\omega \), but we could not trace it back in the literature. Also, it does not seem to directly follow from either Proposition~\ref{prop:surjection} and Corollary~\ref{cor:surjection2}, or their proof.

\begin{proposition} \label{prop:surjectionpreservingcompact} \axioms{\( \ZF \) }
Let \( \lambda \) be a limit cardinal with countable cofinality, and let  \( X \) be a \(\lambda\)-Polish space. Then there is a closed \( F \subseteq B(\lambda) \) and a continuous surjection \( f \colon F \to X \) such that \( f^{-1}(K) \) is compact for every compact \( K \subseteq X \).
\end{proposition}

\begin{proof}
For \( n \geq 1 \) we recursively build canonical open covers \( \mathcal{U}^{(n)} \) of \( X \)  such that:
\begin{enumerate-(1)}
\item
\( |\mathcal{U}^{(n)}| \leq \lambda \);
\item
\( \mathcal{U}^{(n)} \) is locally finite; 
\item
\( U \neq \emptyset \) and \( \mathrm{diam}(U) \leq 2^{-n} \) for every \( U \in \mathcal{U}^{(n)} \) (where the diameter is computed with respect to any compatible complete metric on \( X \) fixed in advance);
\item  
 for each \( U \in \mathcal{U}^{(n+1)} \) there is \( U' \in \mathcal{U}^{(n)} \) such that \( \mathrm{cl}(U) \subseteq U' \). 
 \end{enumerate-(1)}
Indeed, at step \( n \geq 1 \) it is enough to consider the family \( \mathcal{V} \) 
consisting of all basic open sets \( V \subseteq X \) such that \( \mathrm{diam}(V) \leq 2^{-n} \) 
and \( \mathrm{cl}(V) \subseteq U \) for some \( U \in \mathcal{U}^{(n-1)} \) (where 
\( \mathcal{U}^{(0)} = \{ X \} \)), which clearly is an open cover of \( X \), and then refine it in the canonical way 
to the desired \( \mathcal{U}^{(n)} \) using paracompactness and the fact that \( \weight(X) \leq \lambda \) (see the discussion at the end of Section~\ref{subsec:topologicalspaces}).

Now we label the elements of such open covers with sequences from \( \pre{<\omega}{\lambda} \) in order to get a family \( \{ B_s \mid s \in T \} \) with \( T \subseteq \pre{< \omega}{\lambda} \) a tree satisfying the following conditions:
\begin{enumerate-(a)}
\item
\( B_\emptyset = X \);
\item \label{cond:bijection}
if \( n \geq 1 \), the map \( s \mapsto B_s \) is a bijection between \( \mathrm{Lev}_n(T) \) and \( \mathcal{U}^{(n)} \), so that in particular \( \mathrm{diam}(B_s) \leq 2^{-\lh(s)} \); 
\item
\( \mathrm{cl}(B_{s {}^\smallfrown{} \alpha}) \subseteq B_s \) whenever \( s {}^\smallfrown{} \alpha \in T \).
\end{enumerate-(a)}
 This can be done recursively on \( n = \lh(s) \) for \( s \in T \) by letting e.g.\  \( \{ B_{s {}^\smallfrown{} \alpha } \mid \alpha < I \} \) (for the appropriate \( I \leq \lambda \)) be an enumeration without repetitions of those \( U \in \mathcal{U}^{(n+1)} \) such that \( \mathrm{cl}(U) \subseteq B_s \) but \( \mathrm{cl}(U) \not\subseteq B_t \) for all \( t \in \mathrm{Lev}_n( T )\) with \( t <_{\mathrm{lex}} s \), where \( <_{\mathrm{lex}} \) is the strict part of the lexicographical order. (This procedures also determines which \( s {}^\smallfrown{} \alpha \) will be added to \( T \); it does not require choice because each \( \mathcal{U}^{(n)} \) is well-orderable by construction.)
 
 Extend the above family to a \(\lambda\)-scheme on \( X \) by setting \( B_s = \emptyset \) for all \( s \in \pre{<\omega}{\lambda} \setminus T \), and let \( f \) be the induced continuous function. Since for all relevant \( s \) and \( \beta \) we have \( \mathrm{cl}(B_{s {}^\smallfrown{} \beta}) \subseteq B_s \) and \( B_s \neq \emptyset \iff s \in T \), the domain of \( f \) will be the closed set \( F = [T] \). We claim that \( f \) is as required.
 
To check surjectivity, fix any \( x \in X \) and let \( T_x = \{ t \in T \mid x \in B_t \} \), which is clearly a subtree of \( T \). The levels \( \mathrm{Lev}_n(T_x)  \) of \( T_x \) are finite. Indeed, given \( n \geq 1 \) fix an open neighborhood \( W_x \) of \( x \) such that \( W_x \) intersects only finitely many elements of \( \mathcal{U}^{(n)} \): since \( t \in \mathrm{Lev}_n(T_x) \) implies \( B_t \cap W_x \neq \emptyset \), the finiteness of \( \mathrm{Lev}_n(T_x) \) follows from condition~\ref{cond:bijection}. Moreover, each \( \mathrm{Lev}_n(T_x) \) is nonempty, as \( \{ B_t \mid t \in \mathrm{Lev}_n(T) \} = \mathcal{U}^{(n)} \) covers \( X \). Thus by K\"onig's lemma%
\footnote{Again, no choice is needed for this because \( T_x \subseteq \pre{<\omega}{\lambda} \) can be canonically well-ordered using G\"odel's enumeration of finite sequences over \(\lambda\).}
 there is \( y \in [T_x] \subseteq [T] \). Since \( x \in B_{y \restriction n} \) for all \( n \in \omega \), it follows that \( f(y) = x \), as desired.

Finally, let \( K \subseteq X \) be compact. The set \( f^{-1}(K) \) is closed in \( F \) and thus in \( B(\lambda) \).
To show that \( f^{-1}(K) \) is compact, it is enough to show that for each \( n \geq 1 \) there is a finite set \( A^{(n)}_K \subseteq \pre{n}{\lambda} \) such that \( f^{-1}(K) \subseteq \bigcup_{t \in A^{(n)}_K} \Nbhd_t \). Indeed, after replacing each \( A^{(n)}_K \) with its subset 
\[ 
\left\{ t \in A^{(n)}_K \mid t \restriction m \in A^{(m)}_K \text{ for every } 1 \leq m \leq n \right\} , 
\]
we can assume without loss of generality that \( T_K = \{ \emptyset \} \cup \bigcup_{n \geq 1} A^{(n)}_K \) is an \( \omega \)-tree, which is finitely branching by the finiteness of \( A^{(n)}_K \). Therefore \( f^{-1}(K) \), being a closed subset of the compact set \(  [T_K] \), is compact itself.

Let us now construct the required set \( A^{(n)}_K \), for any fixed \( n \geq 1 \).
For each \( x \in K \) let \( W_x \) be the basic open set with smallest index such that \( x \in W_x \) and \( W_x \) intersects only finitely many elements of \( \mathcal{U}^{(n)} \). By condition~\ref{cond:bijection}, this means that there are just finitely many \( t \in \pre{n}{\lambda} \) such that \( B_t \cap W_x \neq \emptyset \). Since the sets \( W_x \) cover \( K \) and the latter is compact, there are finitely many \( x_0, \dotsc, x_\ell \in K \) such that \( K \subseteq \bigcup_{m \leq \ell} W_{x_m} \): this implies that the set \( A^{(n)}_K = \{ t \in \pre{n}{\lambda} \mid B_t \cap K \neq \emptyset \} \subseteq \{ t \in \pre{n}{\lambda} \mid B_t \cap  \bigcup_{m \leq \ell} W_{x_m} \neq \emptyset \} \) is finite. To show that \( f^{-1}(K) \subseteq \bigcup_{t \in A^{(n)}_K} \Nbhd_t \), just notice that if \( y \in F \) is such that \( f(y) \in K \), then \( f(y) \) witnesses \( B_{y \restriction n} \cap K \neq \emptyset \), hence \( y \restriction n \in A^{(n)}_K \).
\end{proof}

\section{Lebesgue covering dimension \( 0 \)}

In the literature on general topology, there are a few alternative definitions of 
``(Lebesgue) covering dimension \( 0 \)'', but all of them are equivalent for metrizable 
spaces. We will indeed restrict to such spaces and adopt the following one, which is sometimes referred to as ``ultraparacompactness''. 

\begin{defin}
A metrizable space \( X \) has (\markdef{Lebesgue}) \markdef{covering dimension \( 0 \)}, in symbols \( \mathrm{dim}(X) = 0 \), if every open covering of \( X \) can be refined to a clopen partition%
\footnote{A family \( \mathcal{P} \subseteq \pow(X) \setminus \{ \emptyset \} \) is a partition of \( X \) if \( \bigcup \mathcal{P} = X \) and \( A \cap B = \emptyset \) for all distinct \( A,B \in \mathcal{P} \); it is a clopen partition if consists of clopen (equivalently, open) subsets of \( X \).}
 of \( X \).
\end{defin}

Recall that the Lebesgue covering dimension \( \mathrm{dim}(X) \), the large inductive dimension \( \mathrm{Ind}(X) \), and the small inductive dimension \( \mathrm{ind}(X) \) all coincide for \emph{separable} metrizable spaces (actually, being normal and second-countable suffices). For arbitrary metrizable spaces \( X \) this is no longer true. The condition
 \( \mathrm{dim}(X) =  0 \) is always equivalent to \( \mathrm{Ind}(X) =  0 \), and  it also implies \( \mathrm{ind}(X) = 0 \) (i.e.\ \( X \) admits a clopen basis).%
 \footnote{More generally, for a metrizable space \( X \) it holds \( \mathrm{ind}(X) \leq \mathrm{Ind}(X) = \mathrm{dim}(X) \).}
  Nevertheless, in the non-separable context there are complete metric spaces with a clopen basis which have Lebesgue covering dimension greater than \( 0 \) (see~\cite{Roy1968}): thus we are taking the strongest among the possible topological definitions of  zero-dimensionality, a clarification that by the above discussion is not necessary in the classical case \( \lambda = \omega \).
  
Suppose that \( X \) is a metrizable space such that \( \mathrm{dim}(X) = 0 \). 
If \( X \) has a well-orderable basis, any clopen partition \( \mathcal{U} \) of \( X \) must be of size at most \( \weight(X) \). Indeed, if \( \mathcal{B} = \{ U_\alpha \mid \alpha < \weight(X) \} \) is a basis for \( X \), then the map sending each \( U \in \mathcal{U} \) into the smallest \( \alpha < \weight(X) \) such that \( \emptyset \neq U_\alpha \subseteq U \) is a well-defined injection from \( \mathcal{U} \) to \( \weight(X) \).
We will also frequently use the obvious fact that if \( \mathrm{dim}(X) = 0 \) and \( C \subseteq X \) is closed, then \( \mathrm{dim}(C) = 0 \). (Recall however that the closedness condition on \( C \) can actually be dropped: this follows from~\cite[Corollary 3.1.20]{Engelking1978} and the fact that metric spaces are perfectly normal.)

Spaces of sequences of the form \( \pre{\mu}{\lambda} \) (with \( \mu, \lambda \) arbitrary infinite cardinals) always have Lebesgue covering dimension \( 0 \) when equipped with the bounded topology. Indeed, if \( \mathcal{U} \) is an open covering of \( \pre{\mu}{\lambda} \) let 
\[ 
A = \{ s \in \pre{< \mu}{\lambda} \mid \Nbhd_s \subseteq U \text{ for some } U \in \mathcal{U} \}
 \] 
and 
\[ 
A' = \{ s \in A \mid \forall \beta < \lh(s) \, (s \restriction \beta \notin A  \}.
 \] 
Then \( \mathcal{U}' = \{ \Nbhd_s \mid s \in A' \} \) is a clopen partition of \( \pre{\mu}{\lambda} \) refining \( \mathcal{U} \). In particular,
\[ 
\mathrm{dim}(B(\lambda)) = \mathrm{dim}(C(\lambda)) = \mathrm{dim}(\pre{\lambda}{2}) = 0.
 \] 

\begin{proposition} \label{prop:dim(X)=0} \axioms{\( \AC_\omega(\pre{\lambda}{2}) \)}
If \( X \) is a \(\lambda\)-Polish space, then \(\mathrm{dim}(X) = 0 \) if and only if \( X \approx C \) for some closed \( C \subseteq B(\lambda) \).
\end{proposition}

We follow the natural generalization of the proof sketched in~\cite[Theorem 7.8]{Kechris1995}. However, it must be noticed that the construction of the desired \( \lambda \)-scheme from \( \mathrm{ind}(X) = 0 \) alone would require the class of basic clopen sets to be closed under intersections of arbitrarily large size below \( \lambda \). This is plainly false for nontrivial metrizable 
spaces, as closure under infinite intersections (that is, \( \omega_1 \)-additivity) already implies that the space is discrete. Our detailed argument below makes evident where the stronger hypothesis \( \mathrm{dim}(X) = 0 \) is needed for non-separable spaces, and shows that again one does not need an axiom of dependent choices to carry out the recursive constructions.

\begin{proofcompare}{Kechris1995}{Theorem 7.8} 
For the nontrivial direction, it is enough to fix a compatible complete metric \( d \leq 1 \) on \( X \) and build a \( \lambda \)-scheme \( \{ B_s \mid s \in \pre{< \omega}{\lambda} \} \) on \( X \) such that for all \( s \in \pre{< \omega}{\lambda} \) and distinct \( \beta, \beta' < \lambda \)
\begin{enumerate-(1)}
\item
\( B_{s {}^\smallfrown{} \beta} \cap B_{s {}^\smallfrown{} \beta'} = \emptyset \);
\item
\( B_\emptyset= X \) and \( B_s = \bigcup_{\beta < \lambda} B_{s {}^\smallfrown{} \beta} \);
\item
\( B_s \) is clopen and \( \mathrm{diam}(B_s) \leq 2^{-\lh(s)} \).
\end{enumerate-(1)}
(Notice that the last condition, together with the definition of \( \lambda \)-scheme, already gives that \( \mathrm{cl}(B_{s {}^\smallfrown{}  \beta}) \subseteq B_s \).) 

Given \( n > 0 \), let \( \mathcal{U}_n \) be the open cover of \( X \) consisting of all basic open sets with diameter at most \( 2^{-n} \). Since \( \weight(X) \leq \lambda \) and \( \dim(X) = 0 \), for 
each \( n > 0 \) we can fix%
\footnote{Here we are using \( \AC_\lambda(\pre{\lambda}{2}) \), but actually \( \AC_\omega(\pre{\lambda}{2}) \) would suffice. The same holds for the proof of Theorem~\ref{thm:characterizationCantor}, where we use the clopen partitions \( \mathcal{U}'_n \) again.}
 a well-orderable clopen partition \( \mathcal{U}'_n = \{ V^{(n)}_\alpha \mid \alpha < \nu_n \} \) (for some \( \nu_n \leq \lambda \)) which refines \( \mathcal{U}_n \). 

The construction of the sets \( B_s \)  is then by recursion on \( \mathrm{lh}(s) \), starting with \( B_\emptyset = X \). Suppose that \( B_s \) satisfies the above conditions. 
%
%
Then it is enough to let \( (B_{s {}^\smallfrown{} \beta})_{\beta < \nu} \) be an enumeration without repetitions, for the appropriate \( \nu \), of the nonempty clopen sets of the form \( B_s \cap V^{(\lh(s)+1)}_\alpha \) with \( \alpha < \nu_{\lh(s)+1} \), and \( B_{s {}^\smallfrown{} \gamma} = \emptyset \) for \( \nu \leq \gamma < \lambda \) if there is any such \(\gamma\).
\end{proofcompare}

This in particular entails~\cite[Theorem 2]{Stone1962}, i.e.\ that every \( G_\delta \) of \( B(\lambda) \) is homeomorphic to a closed subspace of \( B(\lambda) \). Moreover, combining the previous results with Proposition~\ref{prop:retract} we also get the following result, stated (in a more precise form) as Theorem 3 and Theorem \( 3' \) in~\cite{Stone1962}.

\begin{corollary} \label{cor:dim(X)=0} \axioms{\( \AC_\omega (\pre{\lambda}{2}) \)}
If \( X \) is \(\lambda\)-Polish and \( \mathrm{dim}(X)=0 \), then every nonempty closed subset of \( X \) is a retract of the entire space. 
\end{corollary}

The following theorem gives a characterization of \( B(\lambda) = \pre{\omega}{\lambda} \). By Theorem~\ref{thm:homeomorphictoCantor}, it also characterizes \( C(\lambda) \), the generalized Cantor space \( \pre{\lambda}{2} \) (if \( 2^{< \lambda} =  \lambda \)), and Woodin's \( V_{\lambda+1} \) (if \( |V_\lambda| = \lambda \)). 
Notice also that the characterization using condition~\ref{thm:characterizationCantor-2b} naturally generalizes the Alexandrov-Urysohn characterization of the classical Baire space \( \pre{\omega}{\omega} \), see e.g.\ \cite[Theorem 7.7]{Kechris1995}. (Recall that by definition \( \omega \)-Lindel\"of is just compactness.) The characterization using~\ref{thm:characterizationCantor-2''} (which actually works for arbitrary infinite cardinals \(\lambda\) and not just for the case \( \cf(\lambda)=\omega \)) is essentially due to Stone, see~\cite[Theorem 1]{Stone1962}.

\begin{theorem} \label{thm:characterizationCantor} \axioms{\( \AC_\omega(\pre{\lambda}{2}) \)}
The space \( B(\lambda) \) is the unique  \( \lambda \)-Polish space \( X \) such that:
\begin{enumerate-(1)}
\item \label{thm:characterizationCantor-1}
\( \mathrm{dim}(X) = 0 \);
\item \label{thm:characterizationCantor-2}
\( X \) is everywhere of weight \(\lambda\), i.e.\ every (cl)open subspace of \( X \) has weight  \( \lambda \).
\end{enumerate-(1)}
In~\ref{thm:characterizationCantor-2} we can of course replace ``weight'' with ``density character''. \axioms{Per questo serve \( \AC_\lambda(\pre{\omega}{\lambda}) \)}
Moreover, such condition can be replaced by any of the following:
\begin{enumerate}[leftmargin=2pc, label = \upshape (2\alph*)]
\item \label{thm:characterizationCantor-2'}
Every (cl)open subset of \( X \) has discrete subsets of size \( \nu \), for all \( \nu < \lambda \). 
\item \label{thm:characterizationCantor-2''}
Each (cl)open subset of \( X \) has a discrete subset of size \( \lambda \).
\item \label{thm:characterizationCantor-2a}
For every \( \nu < \lambda \), all \( \nu \)-Lindel\"of subsets of \( X \) have empty interior.
\item \label{thm:characterizationCantor-2b}
Every \( \lambda \)-Lindel\"of subset of \( X \) has empty interior.
\item \label{thm:characterizationCantor-2c}
\( \weight(X) = \lambda \) and
\( X \) is strongly homogeneous (or \( h \)-homogeneous), that is, every nonempty clopen subset of \( X \) is homeomorphic to \( X \).
\end{enumerate} 
If \( \lambda \) is \(\omega\)-inaccessible,%
 then one can further replace~\ref{thm:characterizationCantor-2} with any of the following conditions:
\begin{enumerate}[leftmargin=2pc, label = \upshape (2\alph*), resume]
\item \label{thm:characterizationCantor-2d}
Every (cl)open subsets of \( X \) has cardinality at least \( \lambda \).
\item \label{thm:characterizationCantor-2e}
Every (cl)open subsets of \( X \) has cardinality greater than \(  \lambda \).
\item \label{thm:characterizationCantor-2f}
Every (cl)open subsets of \( X \) has cardinality \(  \lambda^\omega \).
\item \label{thm:characterizationCantor-2g}
\( B(\lambda) \) can be embedded as a clopen set into any (cl)open subset of \( X \).
\end{enumerate}
Finally, if \( 2^{< \lambda} =  \lambda \) (and thus, in particular, \(\lambda\) is \(\omega\)-inaccessible), then in the above statements we can  replace all occurrences of \( B(\lambda) \) and \( \lambda^\omega \) with \( \pre{\lambda}{2} \) and \( 2^\lambda \), respectively.
\end{theorem}

\begin{proof}
It is clear that \( B(\lambda) \) has properties~\ref{thm:characterizationCantor-1}--\ref{thm:characterizationCantor-2}, as well as all the variants ~\ref{thm:characterizationCantor-2'}--\ref{thm:characterizationCantor-2g} of~\ref{thm:characterizationCantor-2}
(for~\ref{thm:characterizationCantor-2e} we are using the fact that \( \lambda^\omega = \lambda^{\cf(\lambda)} > \lambda \)).

To prove the first characterization, by Lemma~\ref{lem:homeomorphictoCantor}
 it is enough to show that if \( X \) is a \(\lambda\)-Polish space satisfying conditions~\ref{thm:characterizationCantor-1}--\ref{thm:characterizationCantor-2}, then there is a pruned tree \( T \subseteq \pre{<\omega}{\lambda} \) satisfying~\eqref{eq:largepartition} and such that \( X \approx [T] \). To this aim, we fix a compatible complete metric \( d \leq 1 \) on \( X \) and build a \( \lambda \)-scheme \( \{ B_s \mid s \in \pre{<\omega}{\lambda} \} \) on \( X \) such that for all \( s \in \pre{< \omega}{\lambda} \) and distinct \( \beta, \beta' < \lambda \)
\begin{enumerate-(1)}
\item
\( B_{s {}^\smallfrown{} \beta} \cap B_{s {}^\smallfrown{} \beta'} = \emptyset \);
\item
\( B_\emptyset= X \) and \( B_s = \bigcup_{\beta < \lambda} B_{s {}^\smallfrown{} \beta} \);
\item
\( B_s \) is clopen and \( \mathrm{diam}(B_s) \leq 2^{-\lh(s)} \);
\item
if \( B_s \neq \emptyset \), then \( | \{ \beta < \lambda \mid B_{s {}^\smallfrown{} \beta} \neq \emptyset \} | \geq \lambda_{\lh(s)} \).
\end{enumerate-(1)}
(The latter condition is to ensure that the tree \( T = \{ s \in \pre{< \omega}{\lambda} \mid B_s \neq \emptyset \} \) satisfies condition~\eqref{eq:largepartition}, as required.)

As usual, the construction is by recursion on \( \mathrm{lh}(s) \). 
Set \( B_\emptyset = X \) and
let \( \mathcal{U}'_n = \{ V^{(n)}_\alpha \mid \alpha < \nu_n \} \) be as in the proof of Proposition~\ref{prop:dim(X)=0}. Next assume that \( B_s \) satisfies all the above conditions. We can assume that \( B_s \neq \emptyset \) (otherwise we just set \( B_{s {}^\smallfrown{}  \beta} = \emptyset \) for every \( \beta < \lambda \)). For each \( n > \lh(s) \) let \( \mathcal{P}_n \) be the (well-orderable) clopen partition of \( B_s \) given by the nonempty sets of the form \( B_s \cap V^{(n)}_\alpha \) with \( \alpha < \nu_n \). Then we must have \( | \mathcal{P}_n| \geq \lambda_{\mathrm{lh}(s)} \) for some \( n > \mathrm{lh}(s) \): otherwise the family \( \mathcal{B}= \bigcup_{n > \mathrm{lh}(s)} \mathcal{P}_n \) would have size ta most \(  \lambda_{\mathrm{lh}(s)} \times \omega = \lambda_{\mathrm{lh}(s)} \), and since \( \mathcal{B} \) is a basis for \( B_s \) this would imply \( \weight(B_s) < \lambda \), a contradiction. Let \( n > \mathrm{lh}(s) \) be smallest such that \( |\mathcal{P}_n| \geq \lambda_{\mathrm{lh}(s)} \) and let \( (B_{s {}^\smallfrown{} \beta})_{\beta < |\mathcal{P}_n|} \) be  an enumeration without repetitions of \( \mathcal{P}_n \) and \( B_{s {}^\smallfrown{} \beta} = \emptyset \) for \( |\mathcal{P}_n| \leq \beta < \lambda \): the \( B_{s {}^\smallfrown{} \beta} \) are clearly as required.

%
%

\medskip

We now come to the alternative characterizations. It is enough to show that for any \(\lambda\)-Polish space \( X \), condition~\ref{thm:characterizationCantor-2} is implied by any of its variants. In particular, since~\ref{thm:characterizationCantor-2''} \( \Rightarrow \) \ref{thm:characterizationCantor-2'},
\ref{thm:characterizationCantor-2b} \( \Rightarrow \) \ref{thm:characterizationCantor-2a}, 
\ref{thm:characterizationCantor-2g}  \( \Rightarrow \) \ref{thm:characterizationCantor-2f} (when \(\lambda\) is \(\omega\)-inaccessible),
\ref{thm:characterizationCantor-2f} \( \Rightarrow \) \ref{thm:characterizationCantor-2e}, and \ref{thm:characterizationCantor-2e} \( \Rightarrow \) \ref{thm:characterizationCantor-2d}, it is enough to consider the case of~\ref{thm:characterizationCantor-2'}, \ref{thm:characterizationCantor-2a}, \ref{thm:characterizationCantor-2c}, and~\ref{thm:characterizationCantor-2d}.

For the first one, observe that any discrete set in a topological space with a well-orderable basis is well-orderable itself and has size at most the weight of the space (this gives the contrapositive of~\ref{thm:characterizationCantor-2'} \( \Rightarrow \) \ref{thm:characterizationCantor-2} because \(\lambda\) is limit). 
For the second one, notice that if a topological space has weight \( \nu < \lambda  \), then it is trivially \( \nu^+  \)-Lindel\"of: applied to (cl)open subsets of \( X \) and using the fact that \( \nu^+ < \lambda \) because \(\lambda\) is limit, this gives the contrapositive of~\ref{thm:characterizationCantor-2a} \( \Rightarrow \) \ref{thm:characterizationCantor-2}. 
The implication \ref{thm:characterizationCantor-2c} \( \Rightarrow \) \ref{thm:characterizationCantor-2} is straightforward.
Finally, the contrapositive of~\ref{thm:characterizationCantor-2d} \( \Rightarrow \) \ref{thm:characterizationCantor-2} follows from Lemma~\ref{lem:densityvscardinality}.
\end{proof}

\section{\( \lambda \)-Perfect spaces} \label{sec:lambda-perfect}

The following definition generalizes the notion of isolated point and perfect space to arbitrary cardinals \(\lambda\). As usual, setting \( \lambda = \omega \) one recovers (up to equivalence) the classical definitions.

\begin{defin} \label{def:lambda-perfect}
Let \( X \) be a topological space and \( \lambda \) be an infinite cardinal. A point \( x \in X \) is \markdef{\( \lambda \)-isolated} in \( X \) if there is an open neighborhood \( U \) of \( x \) such that \( |U| < \lambda \). The space \( X \) is \markdef{\(\lambda\)-perfect} if it has no \( \lambda \)-isolated points. A subspace \( P \subseteq X \) is \markdef{\(\lambda\)-perfect}
 (\markdef{in \( X \)}) if it is closed and \(\lambda\)-perfect as a subspace.
\end{defin}

Since the above definition crucially refers to the cardinality of subsets of \( X \), in view of Lemma~\ref{lem:densityvscardinality} it becomes more meaningful when \(\lambda\) is \(\omega\)-inaccessible and \( X \) is metrizable. Moreover, in such a situation 
a point in \( X \) is \(\lambda\)-isolated if and only if it has an open neighborhood of weight smaller than \( \lambda \) (by Lemma~\ref{lem:densityvscardinality} again). Thus for such a \(\lambda\),  a metrizable space \( X \) with a well-orderable basis is \(\lambda\)-perfect if and only if it is everywhere of weight at least \( \lambda \).

This immediately provides a further variant of the characterizations in Theorem~\ref{thm:characterizationCantor}. It is reminiscent of Brouwer's characterization of the Cantor space \( \pre{\omega}{2} \) (see e.g.\ \cite[Theorem 7.4]{Kechris1995}), but notably dropping any compactness condition.

\begin{theorem} \label{thm:characterizationCantor2}
Assume that \(\lambda\) is \(\omega\)-inaccessible.
The space \( B(\lambda) \)  is the unique nonempty \(\lambda\)-perfect \( \lambda \)-Polish space \( X \) with \( \dim(X) =  0 \). If \( 2^{< \lambda} =  \lambda \), the same applies to the generalized Cantor space \( \pre{\lambda}{2} \).
\end{theorem}

We now consider \(\lambda\)-Polish spaces of arbitrary (Lebesgue covering) dimension, and study when \( B(\lambda) \) can be (topologically) embedded into them.
Two of the well-known embedding theorems for classical Polish spaces \( X \) are the following: 
\begin{enumerate-(1)}
\item \label{cantor}
 \( C(\omega) \approx \pre{\omega}{2} \) embeds into \( X \) (necessarily as a closed set, by compactness of \( \pre{\omega}{2} \)) if and only if \( |X| > \omega \) (see~\cite[Corollary 6.5]{Kechris1995});
\item \label{hurewicz}
\( B(\omega) = \pre{\omega}{\omega} \) embeds into \( X \) as a closed set if and only if \( X \) is not \(\sigma\)-compact, i.e.\ it cannot be written as a countable union of compact sets (see~\cite[Theorem 7.10]{Kechris1995}).
\end{enumerate-(1)}

Notice that both statements are of the form: a certain canonical space can be embedded 
into \( X \) (as a closed set) if and only if \( X \) is not too small. In our generalized setting 
\( \lambda > \omega \), the two statements tend to collapse into a unique one. Indeed, 
on the one hand \( B(\lambda) \approx C(\lambda) \) by 
Theorem~\ref{thm:homeomorphictoCantor}; on the other hand, any set of size at most 
\( \lambda\) is trivially \( \lambda \)-compact (that is, a union of 
\( \lambda \)-many compact sets), and the converse holds as well \axioms{\( \AC_\lambda(\pre{\lambda}{2}) \)}
if  \(  2^{\aleph_0} \leq \lambda \) because compact metrizable spaces have size at most \( 2^{\aleph_0} \). Things would not change if instead of considering compactness we consider \( \lambda \)-Lindel\"ofness. It will follow from our results (Corollary~\ref{cor:lambdaLindelofissmall}) that if \( \lambda \) is \(\omega\)-inaccessible, then \axioms{\( \AC_\lambda(\pre{\lambda}{2}) \)} any \(\lambda \)-Lindel\"of \(\lambda\)-Polish space is well-orderable and has size at most \( \lambda \). (The same applies of course to spaces which are \( \lambda \)-sized unions of closed, or even just  \( G_\delta \), \( \lambda \)-Lindel\"of subsets of a \(\lambda\)-Polish space.)\axioms{\( \AC_\lambda(\pre{\lambda}{2}) \)}

Together with Corollary~\ref{cor:cardinalityunderomegainaccessible},
the above discussion
entails that \axioms{\( \AC_\lambda(\pre{\lambda}{2} \)}
the following would be the common analogue of both  statements~\ref{cantor} and~\ref{hurewicz} when 
\( \lambda \) is \( \omega \)-inaccessible and \( X \) is \(\lambda\)-Polish: 
\begin{quotation}
 \( B(\lambda) \) embeds into  \( X \) (as a closed set) if and only if \( |X| > \lambda \). 
 \end{quotation}
 This will be proved in Corollary~\ref{cor:perfect1} (recall that in our setup \( \lambda^\omega = \lambda^{\cf(\lambda)} > \lambda \)). The nontrivial direction of that result follows from the more technical Theorem~\ref{thm:perf} below, which holds without any further assumption on \(\lambda\): its essence is that \( B(\lambda) \) embeds into any \(\lambda\)-Polish space which is everywhere large enough.

\begin{theorem} \label{thm:perf}
Let \( X \) be a completely metrizable space which is everywhere of density character \( \lambda \). Then \( B(\lambda) \) embeds into \( X \) as a closed  set.
\end{theorem}

In particular, \axioms{\( \AC_\lambda(\pre{\lambda}{2}) \)} the theorem applies to all \(\lambda\)-Polish spaces satisfying condition~\ref{thm:characterizationCantor-2} of Theorem~\ref{thm:characterizationCantor}.

\begin{proof}
Let \( D \) be a well-orderable dense subset of \( X \).
By Lemma~\ref{lem:homeomorphictoCantor} it is enough to show that
 there is a pruned tree \( T \subseteq \pre{<\omega}{\lambda} \) satisfying~\eqref{eq:largepartition} and such that \( [T] \) embeds into \( X \) as a closed set. To this aim, we fix a compatible complete metric \( d \leq 1 \) on \( X \) and build a \( \lambda \)-scheme \( \{ B_s \mid s \in \pre{<\omega}{\lambda} \} \) on \( X \) together with a family \( \{ r_s \mid s \in \pre{<\omega}{\lambda} \setminus \{ \emptyset \} \} \) of positive reals such that for all \( s \in \pre{< \omega}{\lambda} \) and distinct \( \beta, \beta' < \lambda \):
\begin{enumerate-(1)}
\item
\( B_{s {}^\smallfrown{} \beta} \cap B_{s {}^\smallfrown{} \beta'} = \emptyset \);
\item
\( B_\emptyset = X \), \( B_s \) is open, and \( \mathrm{diam}(B_s) \leq 2^{-\lh(s)} \);
\item \label{thm:generalizedcantorembedding-3}
\( \mathrm{cl}(B_{s {}^\smallfrown{} \beta}) \subseteq B_s \);
\item \label{thm:generalizedcantorembedding-4}
the distance between \( B_{s {}^\smallfrown{} \beta} \) and \( B_{s {}^\smallfrown{} \beta'} \) is at least \(  r_s \) (if they are both nonempty);
\item
if \( B_s \neq \emptyset \), then \( B_{s {}^\smallfrown{} \beta} \neq \emptyset  \) if and only if \( \beta <  \lambda_{\lh(s)} \).
\end{enumerate-(1)}
The latter condition is to ensure that the tree \( T = \{ s \in \pre{< \omega}{\lambda} \mid B_s \neq \emptyset \} \) is as required (indeed, we get \( [T] = C(\lambda)  \)), while condition~\ref{thm:generalizedcantorembedding-4}, together with~\ref{thm:generalizedcantorembedding-3}, ensures that the range of the embedding induced by the \(\lambda\)-scheme is closed. 

As usual, the construction of the \(\lambda\)-scheme (and of the reals \( r_s \)) is by recursion on \( \lh(s) \), starting with \( B_\emptyset = X \). 
Assume that \( B_s  \) satisfies all the above conditions. If \( B_s = \emptyset \) it is enough to set \( B_{s {}^\smallfrown{} \beta} = \emptyset \) for all \( \beta < \lambda \)  (and \( r_s > 0 \) be arbitrary), hence we can assume \( B_s \neq \emptyset \). For \(  \lambda_{\lh(s)} \leq \beta < \lambda \), set  \( B_{s {}^\smallfrown{}  \beta} = \emptyset \). Then apply Lemma~\ref{lem:r-spaces} with \( Y = B_s \) and \( \nu = \lambda_{\lh(s)} \) and enumerate the resulting \( r \)-spaced set as \( (x_\beta)_{\beta < \lambda_{\lh(s)}} \): by Remark~\ref{rmk:r-spaces} such a set and the real \( r \) can be canonically extracted from the well-orderable set \( D \cap B_s \), which is clearly dense in \( B_s \). Set \( r_s = r/3 \) and, for each \( \beta < \lambda_{\lh(s)} \), set \( B_{s {}^\smallfrown{} \beta} = B_d(x_\beta, 2^{-n} ) \) where \( 0 \neq n \in \omega \) is smallest such \( 2^{-n} \leq   \min \{ 2^{-(\lh(s)+2)},r_s\} \) and \( B_d(x_\beta, 2^{-(n-1)}) \subseteq B_s \).
\end{proof}

In view of the fact that, when \(\lambda\) is \(\omega\)-inaccessible,  \(\lambda\)-perfectness amounts to require that the \(\lambda\)-Polish space be everywhere of weight \(\lambda\), we get:

\begin{corollary} \label{cor:lambda-perfect} \axioms{Da qui in poi \( \AC_\lambda(\pre{\lambda}{2}) \) per tutti i risultati della sezione}
Assume that \(\lambda\) is \(\omega\)-inaccessible. Then \( B(\lambda) \) can be embedded as a closed set into any nonempty \(\lambda\)-perfect \(\lambda\)-Polish space.
\end{corollary}

\begin{theorem} \label{thm:decomp}
Let \( X \) be \(\lambda\)-Polish. Then \( X \) can  be written as a disjoint union \( P \cup C \), where \( P \) is \(\lambda\)-perfect and \( C \) is a well-orderable open set of size at most \( \lambda \). If \(\lambda\) is \(\omega\)-inaccessible, then such a decomposition is unique and the set \( P \) is called \markdef{\(\lambda\)-perfect kernel} of \( X \).
\end{theorem}

Although the proof follows the classical argument (see e.g.~\cite[Theorem 6.4]{Kechris1995}), we report it here for the sake of completeness and to make it evident where \(\omega\)-inaccessibility is needed in the second part.

\begin{proof}
Let \( C \) be the union of all basic open sets \( U \) such that \( |U| \leq \lambda \) and set \( P = X \setminus C \). Then \( C \) is open and \( |C| \leq \lambda \) by Fact~\ref{fct:lambdaunionewithoutchoice}\ref{fct:lambdaunionewithoutchoice-2}. It follows that \( P \) is closed and also \(\lambda\)-perfect: if \( U \cap P \) is a nonempty (relative) basic open subset of \( P \) with \( | U \cap P | < \lambda \), then \( |U| \leq \lambda \) because \( U \setminus P \subseteq C \), hence we would have \( U \subseteq C \), in contradiction with \( P \cap U \neq \emptyset \).

Assume now that \( \lambda \) is \(\omega\)-inaccessible and let \( X = P' \cup C' \) with%
\footnote{Interestingly, we do not need to assume \( P' \cap C' = \emptyset \). Indeed, it is a corollary of the proof that if \( P' \) and \( C' \) are as required, then they must necessarily be disjoint.}  
\( P' \) a \(\lambda\)-perfect set and \( C' \) open such that \( |C'| \leq \lambda \).
Clearly \( C' \subseteq C \) by definition of \( C \). Let \( x \in P' \) be arbitrary, and 
assume towards a contradiction that \( x \in C \). Then by definition of \( C \) there would be some basic open neighborhood \( U \) of \( x \) such that \( |U| \leq \lambda \). On the other hand,
the space \( Y = U \cap P' \) is \(\lambda\)-Polish because it is a \( G_\delta \)-subset of \( X \) (Fact~\ref{fct:booleancombinationsareFsigma}). Moreover it is \(\lambda\)-perfect because it is relatively open in the \(\lambda\)-perfect space \( P' \), and it is nonempty because \( x \in Y \). Hence the space \( B(\lambda) \) embeds into \( Y \) by Corollary~\ref{cor:lambda-perfect}. Since \( \cf(\lambda) = \omega \)  we then have \( |U| \geq |U \cap P'| \geq |B(\lambda)| = \lambda^\omega > \lambda \), contradicting the choice of \( U \). Thus \( x \in P = X \setminus C \), and since \( x \) was an arbitrary element of \( P' \), this shows \( P' \subseteq P \). Since \( X = P' \cup C' \), it then follows \( P' = P \) and \( C' = C \), as desired.
\end{proof}

Combining this with Theorem~\ref{thm:characterizationCantor2}, we get the following curious phenomenon. 

\begin{corollary} \label{cor:characterizationCantor2}
Assume that \(\lambda\) is \(\omega\)-inaccessible and that \( X \) is a \(\lambda\)-Polish space with \( \mathrm{dim}(X) = 0 \).  Then the \(\lambda\)-perfect kernel of \( X \), if nonempty, is homeomorphic to \( B(\lambda) \).
\end{corollary}

This is in strong contrast with the situation in the classical setup \( \lambda = \omega \), where we can find many different kinds of perfect kernels of zero-dimensional Polish spaces. 

\begin{remark}
As for the classical case \( \lambda = \omega \), one can obtain  the decomposition \( P \cup C \) of \( X \) as in Theorem~\ref{thm:decomp} via (generalized) Cantor-Bendixson derivatives. More precisely, given a \( \lambda \)-Polish space \( Y \) let
\[ 
Y' = \{ y \in Y \mid y \text{ is not \(\lambda\)-isolated in } Y \}
 \] 
be the \(\lambda\)-Cantor-Bendixson derivative of \( Y \), and notice that \( Y' \) is closed in \( Y \) and \( | Y \setminus Y'| \leq \lambda \). This operation can be iteratively applied to our \(\lambda\)-Polish space \( X \) as follows (the definition is by recursion on all ordinals):
\begin{align*}
X^{(0)} & = X \\
X^{(\alpha+1)} & = (X^{(\alpha)})' \\
X^{(\gamma)} & = \bigcap_{\beta < \gamma} X^{(\beta)} \qquad \text{if \(\gamma\) is limit}.
\end{align*}
Since \( X \) has weight at most \( \lambda \) and all the sets \( X^{(\beta)} \) are closed, there is \( \alpha < \lambda^+ \) such that \( X^{(\alpha)} = X^{(\alpha+1)} \), that is \( X^{(\alpha)} \) is \(\lambda\)-perfect. It then follows that \( P = X^{(\alpha)} \) is the \(\lambda\)-perfect kernel of \( X \) (whenever this terminology makes sense, i.e.\ if \(\lambda\) is \(\omega\)-inaccessible). 
\end{remark}

Arguing as in the classical case, Corollary~\ref{cor:lambda-perfect} and Theorem~\ref{thm:decomp} yields to the following dichotomy.

\begin{corollary} \label{cor:perfect1}
Assume that \(\lambda\) is \(\omega\)-inaccessible
and let \( X \) be a \(\lambda\)-Polish space. Then either \( X \) is well-orderable and \( |X| \leq \lambda \), or else \( B(\lambda) \) embeds into \( X \) as a closed set (and hence \( |X| =  \lambda^\omega \)).
\end{corollary}

By Theorem~\ref{thm:homeomorphictoCantor} we also get the following analogue of~\cite[Corollary 6.5]{Kechris1995}.

\begin{corollary} \label{cor:perfect2}
Assume that \( 2^{< \lambda} =  \lambda\) 
and let \( X \) be a \(\lambda\)-Polish space. Then either \( X \) is well-orderable and  \( |X| \leq \lambda \), or else \( \pre{\lambda}{2} \) embeds into \( X \) as a closed set (and thus \( |X| = 2^\lambda \)).
\end{corollary}

\begin{remark} \label{rmk:not-omega-inaccessible}
It is not possible to extend Corollaries~\ref{cor:lambda-perfect} and~\ref{cor:perfect1} to cardinals \(\lambda\) of countable cofinality which are not \(\omega\)-inaccessible, even when restricting to \(\lambda\)-Polish spaces of maximal weight \(\lambda\). Indeed, if \(\lambda\) is not \(\omega\)-inaccessible then by definition there is \( \kappa < \lambda \) such that \(  \kappa^\omega \not <  \lambda \). Letting \( X \) be the sum of \(\lambda\)-many copies of \( B(\kappa) \) we get a \(\lambda\)-Polish space of weight \(\lambda\) and cardinality greater than \( \lambda \) which is \(\lambda\)-perfect, yet \( B(\lambda) \) cannot  embed into it. Indeed, on the one hand every nonempty open subset of \( X \) has size at least \( \kappa^\omega \), so no point of \( X \) is \(\lambda\)-isolated by the choice of \( \kappa \). On the other hand, since \( \weight(B(\kappa)) = \kappa < \lambda \)
the preimage of each copy of \( B(\kappa) \) under a topological embedding should have weight at most \( \kappa \), but this is impossible because every  nonempty open subset of \( B(\lambda) \) has weight \( \lambda \). A similar observation applies to Corollary~\ref{cor:perfect2}: if \( \lambda \) is singular but not strong limit, then \( \pre{\lambda}{2} \) has weight \( 2^{< \lambda} > \lambda \) and thus cannot be embedded into any \(\lambda\)-Polish space.
\end{remark}

We conclude this section by proving the announced results on \(\lambda\)-Lindel\"of \(\lambda\)-Polish spaces.

\begin{corollary} \label{cor:lambdaLindelofissmall}
Let \(\lambda\) be \(\omega\)-inaccessible. If a \(\lambda\)-Polish space \( X \) is \(\lambda\)-Lindel\"of, then \( X \) is well-orderable and \( |X| \leq \lambda \).
\end{corollary} 

\begin{proof}
If not, by Corollary~\ref{cor:perfect1} the space \( B(\lambda) \) would embed as a closed set in \( X \) and would thus be \(\lambda\)-Lindel\"of as well, which is obviously false.
\end{proof}

\section{\( \lambda \)-Polish spaces for other cofinalities} \label{subsec:Polishforothercofinalities}

So far, in this chapter on \(\lambda\)-Polish spaces we just considered the case when \( \cf(\lambda) = \omega \) because that is the main focus of the present paper. However, most of the results can in fact be proved for arbitrary uncountable cardinals \(\lambda\), and actually some of the proofs even simplify when \( \lambda \) has uncountable cofinality. In what follows we  thus assume that \( \cf(\lambda) > \omega \) (which in particular implies \( \lambda > \omega \)) and discuss which results from Sections~\ref{sec:def-example-Polish}--\ref{sec:lambda-perfect} 
still hold in this other setup. For the sake of brevity we just sketch the part of the proofs that need to be modified, leaving to the interested reader to fill in all the details.

First of all, notice that in the new setup \( C(\lambda) \) is undefined while \( \pre{\lambda}{2} \) and \( \mathrm{V}_{\lambda+1} \) are not \(\lambda\)-Polish because they are not even first-countable, hence there cannot be any analogue of Theorem~\ref{thm:homeomorphictoCantor}. However, the more technical Lemma~\ref{lem:homeomorphictoCantor} continue to be true: this is just for the records, as we do not have any use for this fact at the moment. Propositions~\ref{prop:retract},  \ref{prop:surjection}, and~\ref{prop:tree-basis}, Corollaries~\ref{cor:surjection2} and~\ref{cor:charwithsurj}, and Proposition~\ref{prop:surjectionpreservingcompact}  holds as well (the proofs are identical).

We next move to \(\lambda\)-Polish spaces \(  X \) with \( \mathrm{dim}(X) = 0 \). Both  Proposition~\ref{prop:dim(X)=0} and Corollary~\ref{cor:dim(X)=0} clearly remain true (their proofs go unchanged). Also the characterizations of \( B(\lambda) \) in Theorem~\ref{thm:characterizationCantor} work in the more general setting, with the following exceptions and restrictions:
\begin{itemizenew}
\item
conditions~\ref{thm:characterizationCantor-2'} and \ref{thm:characterizationCantor-2a} can be included only when \( \lambda \) is limit;%
\footnote{Notice however that when \( \lambda \) is not limit, we can still directly prove \ref{thm:characterizationCantor-2''}~\( \Rightarrow \)~\ref{thm:characterizationCantor-2} and \ref{thm:characterizationCantor-2b}~\( \Rightarrow \)~\ref{thm:characterizationCantor-2}, by-passing the missing conditions~\ref{thm:characterizationCantor-2'} and \ref{thm:characterizationCantor-2a}.}
\item
condition~\ref{thm:characterizationCantor-2e} must be removed: if \( \lambda \) is \(\omega\)-inaccessible of uncountable cofinality, then \( \lambda^\omega = \lambda \) and thus \( |B(\lambda)| = \lambda \) (in a sense this indicates that such setting is less natural when studying spaces of the form \( B(\lambda) \), see below for a more thorough discussion on this issue);
\item
the additional part concerning \( \pre{\lambda}{2} \) must be removed, as such space is not even \(\lambda\)-Polish.
\end{itemizenew}
Some parts of the proof of Theorem~\ref{thm:characterizationCantor} need to be adapted as well. In particular, in the proof of the main characterization (i.e.\ the one involving condition~\ref{thm:characterizationCantor-2}), we now have to%
\footnote{Any space of the form \( \prod_{i \in \omega} \lambda_i \) has weight \( \sup_{i \in \omega} \lambda_i \). If \( \lambda_i < \lambda \) for all but finitely many \( i \in \omega \) and \( \cf(\lambda) > \omega \), then such a space cannot be homeomorphic to \( B(\lambda) \) because it would contain clopen sets of weight smaller than \( \lambda \). If instead \( \lambda_i = \lambda \) for infinitely many \( i \in \omega \), then \( \prod_{i \in \omega} \lambda_i \) would be homeomorphic to \( B(\lambda) \) but then in the proof it would no longer be useful to consider such a space instead of directly working with \( B(\lambda) \).}
directly require that \( B_s \neq \emptyset \) \emph{for all \( s \in \pre{< \omega}{\lambda} \)} (this forces the resulting tree \( T \) to be the full \( \pre{< \omega}{\lambda} \), whence the induced map is a homeomorphism between \( B(\lambda) \) itself and \( X \)). This is possible because, in the notation of the proof of Theorem~\ref{thm:characterizationCantor}, we can now infer that \( |\mathcal{P}_n| = \lambda \) for some \( n \geq \lh(s) \), as otherwise the basis \( \mathcal{B} = \bigcup_{n \geq \lh(s)} \mathcal{P}_n \) would be such that  \( | \mathcal{B}| =  \sup_{n \geq \lh(s)} | \mathcal{P}_n| < \lambda \) (here is where we use the fact that we are now assuming \( \cf(\lambda) > \omega \)).

Finally, let us consider \(\lambda\)-perfect \(\lambda\)-Polish spaces. 
Theorem~\ref{thm:characterizationCantor2} obviously remains true.
Since now \( \cf(\lambda) > \omega \), one can also improve the final part of Lemma~\ref{lem:r-spaces} by further requiring that \( |A| = \lambda \). 
Indeed, assume \( \dens(X) = \lambda \) and define the sets \( B_n \) as in the proof of Lemma~\ref{lem:r-spaces}: if \( |B_n| < \lambda \) for all \( n \in \omega \), then the dense set \( B = \bigcup_{n \in \omega} B_n \) would have size \( \sup_{n \in \omega} |B_n| < \lambda = \dens(X) \), a contradiction.
This allows us to reprove Theorem~\ref{thm:perf} in our new setup, and thus also Corollary~\ref{cor:lambda-perfect}. Formally, the decomposition Theorem~\ref{thm:decomp} holds as well, except for the fact that if \( \lambda \) is \(\omega\)-inaccessible the decomposition might not be unique:%
\footnote{Notice however that the uniqueness part is not needed for the other results.} 
for a counterexample, let \( X \) be any \(\lambda\)-Polish space and observe that we can equivalently take \( P = \emptyset \) (this works because \( |X| = \lambda^\omega =  \lambda \) by Fact~\ref{fct:fromweighttosize} and the fact that \( \lambda \) is an \( \omega \)-inaccessible cardinal with uncountable cofinality) or \( P = X \). These observations trivialize the subsequent Corollaries~\ref{cor:perfect1} and~\ref{cor:perfect2} as stated. (Notice also that the second alternative in the latter is out of question because \( 2^\lambda > \lambda \) while, as argued, all \(\lambda\)-Polish spaces have size at most \( \lambda \) if \(\lambda\) is \(\omega\)-inaccessible and has uncountable cofinality.) However, some of their technical variants can be nontrivial also when \( \cf(\lambda) > \omega \). For example,
using the Cantor-Bendixson derivation argument one can restate
 Corollary~\ref{cor:perfect1}  as: For each \(\lambda\)-Polish space \( X \), either \( X = \bigcup_{\alpha < \rho} V_\alpha \) where \( \rho < \lambda^+ \) and \( (V_\alpha)_{\alpha < \rho} \) is an increasing sequence of open sets such that \( |V_{\alpha+1} \setminus V_\alpha| < \lambda \) for all \( \alpha < \rho \), or else \( B(\lambda) \) embeds into \( X \) as a closed set.

The above discussion hints to the fact that, although some of our results trivialize in the 
uncountable cofinality case, it makes sense to consider and study the class of 
\(\lambda\)-Polish spaces also when \( \cf(\lambda) > \omega \). However, it is then 
unreasonable to proceed with the development of a \emph{generalized} descriptive set 
theory for such spaces, as we plan to do in the next chapters. This is due to the fact that, as already mentioned,  all \(\lambda\)-Polish spaces have size at most 
\( \lambda^\omega \) by Fact~\ref{fct:fromweighttosize}. Thus if \( \lambda^\omega = \lambda \), which is always the case if 
e.g.\ \( \cf(\lambda) > \omega \) and we work in a model of \( \GCH \), or if \( \cf(\lambda)> \omega \) and \( \lambda \) is \(\omega\)-inaccessible, then all 
\(\lambda\)-Polish spaces have size at most \( \lambda \). 
This immediately trivializes the notion of \(\lambda^+\)-Borel 
set that will be introduced at the very beginning of Chapter~\ref{sec:Borel}, as well as all subsequent definability 
notions. In contrast, notice that it makes perfectly sense to develop \emph{classical} 
descriptive set theory (that is, the theory of classical (\( \omega_1 \)-)Borel sets, (\(\omega\)-)analytic 
sets, and so on) for \(\lambda\)-Polish spaces, independently of \(\lambda\) (see 
e.g.~\cite{Stone1962,Stone1972}).

Despite the above observation, it is interesting to notice that \( \kappa \)-Polish spaces for \( \kappa \geq \lambda \) will indeed play a role in the definition of \( \kappa \)-Souslin subsets of \( \lambda \)-Polish spaces. This is a notion which is relevant to our study of definable subsets of \( \lambda \)-Polish spaces when \( \cf(\lambda) = \omega \) (Section~\ref{sec:Souslinsets}).


%
%
%

\chapter{\( \lambda^{(+)} \)-Borel sets} \label{sec:Borel}

\section{Definitions and basic facts} \label{sec:Boreldefandbasicfacts}

The following definition naturally generalizes to an arbitrary \(\nu\) the usual notion of a Borel set (which corresponds to the case \( \nu=\omega \)).

\begin{defin}
Let \( X, Y \) be topological spaces and \( \nu \) be an infinite cardinal. A set \( B \subseteq X \) is \markdef{\( \nu^+ \)-Borel} if it belongs to the  \( \nu^+ \)-algebra generated by the open sets of \( X \). The collection of \( \nu^+ \)-Borel subsets of \( X \) is denoted by \( \gBor{\nu^+}(X) \).

A function \( f \colon X \to Y \) is \markdef{\( \nu^+ \)-Borel} (\markdef{measurable}) if the preimage of any open (equivalently, \( \nu^+ \)-Borel) subset of \( Y \) is \(\nu^+ \)-Borel in \( X \).  
A function \( f \colon X \to Y \) is a \markdef{\( \nu^+ \)-Borel isomorphism} if it is bijective and both \( f \) and \( f^{-1} \) are \( \nu^+ \)-Borel; when such an \( f \) exists, we say that \( X \) and \( Y \) are  \markdef{\( \nu^+ \)-Borel isomorphic}.
Finally, a \markdef{\( \nu^+ \)-Borel embedding} is a function \( f \colon X \to Y \) which is a \( \nu^+ \)-Borel isomorphism on its image.
\end{defin}

Thus \( \gBor{\omega_1}(X) \) coincides with the collection of all (classical) Borel sets, i.e.\ it is the \(\sigma\)-algebra generated by the topology of \( X \). We will of course be mostly interested in the case \( \nu = \lambda \). Since in our setup \( \lambda \) is singular, the collection of \( \lambda^+ \)-Borel subsets of \( X \) may equivalently be described as the smallest \( \lambda \)-algebra containing all open sets. For this reason, we drop the ``\(+ \)'' from the above terminology and notation and just speak e.g.\ of \( \lambda \)-Borel sets and functions. Accordingly, we denote by \( \lB(X) \) the collection of all \( \lambda \)-Borel subsets of \( X \), dropping the reference to \( X \) when there is no danger of confusion, and writing instead \( \lB(X,\tau) \) when it is necessary to specify the topology \(\tau\) on \( X \) we are referring to. Notice also that if \( Y \subseteq X \) then \( \lB(Y) = \{ B \cap Y \mid B \in \lB(X) \} \).

\begin{remark}
When \( X = \pre{\lambda}{2} \) and \(2^{< \lambda} = \lambda\), the bounded topology and the product topology on \( \pre{\lambda}{2} \) generate the same 
\( \lambda^+ \)-Borel sets. In fact, the bounded topology \( \tau_b \) is finer than the product topology. Conversely, every \( \tau_b \)-basic open set \( \Nbhd_s \) is closed in the product topology, thus the result easily follows once we impose that there are only \(  \lambda \)-many \( \tau_b \)-basic open sets, namely, \( 2^{< \lambda} = \lambda \).
\end{remark}

Arguing as in the classical case (see~\cite[Proposition 12.4]{Kechris1995}), one easily gets:

\begin{proposition} \label{prop:borelvsgraph}
Let \( X \) be \(\lambda\)-Polish and \( Y \) be a metrizable%
\footnote{Actually, it is enough that \( Y \) be a \( T_0 \)-space.}
 space of weight at most \( \lambda \). If \( f \colon X \to Y \) is \(\lambda\)-Borel, then its graph
\[ 
\mathrm{graph}(f) = \{ (x,y) \in X \times Y \mid f(x) = y \}
 \] 
is \(\lambda\)-Borel in \( X \times Y \).
\end{proposition}

\begin{proof}
Let \( \mathcal{B} = \{ U_\alpha \mid \alpha < \lambda \} \) be a basis for \( Y \), and notice that for all \( x \in X \) and \( y \in Y \)
\[ 
f(x) = y \quad \text{if and only if} \quad \forall \alpha < \lambda \, (y \in U_\alpha \iff f(x) \in U_\alpha). \qedhere
 \] 
\end{proof}

If \( Y \) is \(\lambda\)-Polish
the converse of the above theorem holds as well, as we will show in Theorem~\ref{thm:borelvsgraph}. 

\section{The \(\lambda\)-Borel hierarchy} \label{sec:Borelhierarchy}

 In the rest of this chapter, we will frequently use the fact that we are assuming \( \AC_\lambda (\pre{\lambda}{2}) \) in the background. For example, when arguing that the length of the \(\lambda\)-Borel hierarchy is (at most) \( \lambda^+ \) we need that \( \lambda^+ \) be regular, 
and under \( \AC_\lambda(\pre{\lambda}{2}) \) this is true by Fact~\ref{fct:lambdaunionewithoutchoice}.
Assume from now on that \( X \) is metrizable.


The \( \lambda^+ \)-algebra \( \lB(X) \) can be naturally stratified as follows. Define by recursion on the ordinal \(\xi \geq 1 \)
\[
\begin{cases}
\lS^0_1(X)  = \text{open sets} & \\
\lS^0_\xi(X)  = \left\{ \bigcup_{\alpha< \lambda} A_\alpha \,\, \big| \,\, X \setminus A_\alpha \in  \lS^0_{\xi'}(X) \text{ for some } \xi' < \xi \right\}, & \text{if } \xi > 1 
\end{cases}
\]
and then set
\begin{align*}
\lP^0_\xi(X)& = \{ X \setminus A \mid A \in \lS^0_\xi(X) \} \\
\lD^0_\xi(X) & = \lS^0_\xi(X) \cap \lP^0_\xi(X).
\end{align*}
Arguing as in the classical case, it is not hard to see that for all \( 1 \leq \xi < \xi' \)
\begin{equation} \label{eq:Borelhierarchy}
\lS^0_\xi(X), \lP^0_\xi(X) \subseteq \lD^0_{\xi'}(X) \subseteq \lS^0_{\xi'}(X), \lP^0_{\xi'}(X),
 \end{equation}
 and that since \( \lambda^+ \) is regular
\[ 
\lB(X) = \bigcup_{1 \leq \xi < \lambda^+} \lS^0_\xi(X) = \bigcup_{1 \leq \xi< \lambda^+} \lP^0_\xi(X) = \bigcup_{1 \leq \xi < \lambda^+} \lD^0_\xi(X).
 \] 
(Here we use that closed sets are \( G_\delta \) by Fact~\ref{fct:booleancombinationsareFsigma}, and hence elements of \( \lP^0_2(X) \).) One can also easily check by induction on \( \xi \geq 1 \) that \( A \in \lS^0_\xi(X) \) if and only if \( A = \bigcup_{\alpha < \lambda} B_\alpha \) with \( B_\alpha \in \bigcup_{\xi' < \xi} \lP^0_{\xi'}(X) \) and, dually,  \( A \in \lP^0_\xi(X) \) if and only if \( A = \bigcap_{\alpha < \lambda} B_\alpha \) with \( B_\alpha \in \bigcup_{\xi' < \xi} \lS^0_{\xi'}(X) \). If \( \xi = \xi' +1 \), this further reduces to the fact that  \( A \in \lS^0_\xi(X) \) if and only if \( A = \bigcup_{\alpha < \lambda} B_\alpha \) with \( B_\alpha \in \lP^0_{\xi'}(X) \), and analogously for \( \lP^0_\xi(X) \). Arguing again by induction on \( \xi \geq 1 \), it is also easy to verify that all the above mentioned classes of sets are closed under finite unions, finite intersections, continuous preimages, and moreover:%
\footnote{Here we are again heavily relying on the fact that we assumed \( \AC_\lambda(\pre{\lambda}{2}) \), together with the fact that for each \( \xi' \geq 1 \) there are canonical surjections from \( \pre{\lambda}{2} \) onto the classes \( \lS^0_{\xi'}(X) \) and \( \lP^0_{\xi'}(X) \), and thus also onto the classes \( \bigcup_{\xi' < \xi} \lS^0_{\xi'} (X) \) and \( \bigcup_{\xi' < \xi} \lP^0_{\xi'}(X) \)  for all \( 1 \leq \xi < \lambda^+ \). Altogether, this ensures that e.g.\ a \(\lambda\)-sized union of sets which are \(\lambda\)-sized unions of element of one of those classes is still a \(\lambda\)-sized union of elements of the same class.}
\begin{itemizenew}
\item
\( \lS^0_\xi(X) \) is closed
under well-ordered unions of length at most \( \lambda \);
\item
\( \lP^0_\xi(X) \) is closed under well-ordered intersections of length at most \( \lambda \);
\item
\( \lD^0_\xi(X) \) is closed under complements (i.e.\ it is a Boolean algebra).
\end{itemizenew}
Unlike the regular case, 
however, the above classes are not closed under longer unions and intersections different from those specified in the above list. For example, if \( X \) is not discrete, then already taking countable intersections of open sets may result in proper closed (or even more complicated) sets because \( X \), being metrizable, is first-countable. (See~\cite{ACMRP} for more details.) This is partially balanced by the following technical proposition, which is specific to uncountable singular cardinals.

\begin{proposition} \label{prop:levelsBorel}
Assume that \( 2^{< \lambda} = \lambda\). Let \( X \) be a metrizable space and  \( 1 \leq \xi < \lambda^+ \).
\begin{enumerate-(i)}
\item \label{prop:levelsBorel-1}
If \( \xi > 1 \) is a successor ordinal, then every \( A \in \lS^0_\xi(X) \) can be written as \( A = \bigcup_{n \in \omega} A_n \) with \( A_n \in \lD^0_\xi(X) \) for all \( n \in \omega \), and dually for sets in \( \lP^0_\xi(X) \).
\item \label{prop:levelsBorel-2}
If \( \xi \) is limit, then every \( A \in \lS^0_\xi(X) \) can be written as \( A = \bigcup_{\alpha < \cf(\xi)} A_\alpha \) with \( A_\alpha \in \bigcup_{\xi'< \xi} \lD^0_{\xi'}(X) \subseteq \lD^0_\xi(X) \) for all \( \alpha < \cf(\xi) \), and dually for sets in \( \lP^0_\xi(X) \) (notice that \( \cf(\xi) < \lambda \) since \(\lambda\) is singular).
\item \label{prop:levelsBorel-3}
If \( \nu < \lambda \) and \( A_\alpha \in \lS^0_\xi(X) \) for \( \alpha < \nu \), then \( \bigcap_{\alpha < \nu } A_\alpha \in \lD^0_{\xi+1}(X) \). Dually, the union of less than \(  \lambda \)-many sets in \( \lP^0_\xi(X) \) belongs to \( \lD^0_{\xi+1}(X) \).
\end{enumerate-(i)}
If moreover \( X \) is \(\lambda\)-Polish and \( \mathrm{dim}(X) = 0 \), then part~\ref{prop:levelsBorel-1} holds also with \( \xi =1  \).
\end{proposition}

\begin{proof}
By induction on \( \xi \). If \( \xi =1 \) we just have to prove~\ref{prop:levelsBorel-3}. (The case of~\ref{prop:levelsBorel-1} when \( \mathrm{dim}(X) = 0 \) will be treated separately at the end of this proof.) Since each \( A_\alpha \) is open and \( X \) is metrizable, we can canonically write \( A_\alpha = \bigcup_{n \in \omega} B_\alpha^n \) with \( B_\alpha^n \) closed by Fact~\ref{fct:booleancombinationsareFsigma}. Then since
\[ 
\bigcap_{\alpha < \nu} A_\alpha = \bigcup_{s \in \pre{\nu}{\omega}} \, \bigcap_{\alpha < \nu} B^{s(\alpha)}_\alpha,
 \] 
we get that such set is a union of \( \omega^\nu \)-many closed sets. Since \( \lambda \) is strong limit when \( 2^{< \lambda} = \lambda \), we have \( \omega^\nu < \lambda \) and thus \( \bigcap_{\alpha < \nu} A_\alpha \in \lS^0_2(X) \). Since \( \bigcap_{\alpha < \nu} A_\alpha \in \lP^0_2(X) \) by definition, it follows \( \bigcap_{\alpha < \nu} A_\alpha \in \lD^0_2(X) \), as required. 

Now let \( \xi > 1 \) and assume that~\ref{prop:levelsBorel-1}--\ref{prop:levelsBorel-3} holds for all \( \xi' < \xi \). We first prove~\ref{prop:levelsBorel-1}, so assume that \( \xi = \xi'+1 \) is a successor ordinal. If \( A \in \lS^0_\xi(X) \), then there are \( B_\alpha \in \lP^0_{\xi'}(X) \) such that \( A= \bigcup_{\alpha < \lambda} B_\alpha \). For \( n \in \omega \), set \( A_n = \bigcup_{\alpha < \lambda_n} B_\alpha \), so that \( A = \bigcup_{n \in \omega} A_n \). By~\ref{prop:levelsBorel-3} applied to \( \lP^0_{\xi'}(X) \), we get that \( A_n \in \lD^0_{\xi'+1}(X) = \lD^0_\xi(X) \) and we are done. 

The case of~\ref{prop:levelsBorel-2} when \( \xi \) is limit is even simpler. Let \( (\xi_\alpha)_{ \alpha < \cf(\xi)} \) be a strictly increasing sequence cofinal in \(\xi\).  Given \( A \in \lS^0_\xi(X) \), let \( B_\beta \in \bigcup_{\xi' < \xi} \lP^0_{\xi'}(X) \) be such that \( A = \bigcup_{\beta < \lambda} B_\beta \). Set
\[ 
A_\alpha = \bigcup \{ B_\beta \mid B_\beta \in \lP^0_{\xi_\alpha}(X) \}.
 \] 
 Then \( A_\alpha \in \lS^0_{\xi_\alpha+1}(X) \subseteq \lD^0_{\xi_\alpha+2}(X) \subseteq  \bigcup_{\xi'< \xi} \lD^0_{\xi'}(X) \)  and \( A = \bigcup_{\alpha < \cf(\xi)} A_\alpha \), as desired.
 
 As for~\ref{prop:levelsBorel-3}, since if all \( A_\alpha \) are in \( \lS^0_\xi(X) \) then
  \( \bigcap_{\alpha < \nu} A_\alpha \in \lP^0_{\xi+1}(X) \) by definition, it is enough to prove that \( \bigcap_{\alpha < \nu} A_\alpha \in \lS^0_{\xi+1}(X) \). For all \( \alpha < \nu \), apply~\ref{prop:levelsBorel-1} or~\ref{prop:levelsBorel-2} (depending on whether \( \xi \) is a successor or a limit ordinal) to get \( B^\beta_\alpha \in \lD^0_\xi(X) \)  such that \( A_\alpha = \bigcup_{\beta < \mu} B^\beta_\alpha \), where \( \mu = \omega \) if \( \xi \) is successor and \( \mu = \cf(\xi) \) otherwise. (This can be done because at this point we already  proved~\ref{prop:levelsBorel-1} and~\ref{prop:levelsBorel-2} for level \( \xi \) and we assumed \( \AC_\lambda(\pre{\lambda}{2}) \).)
  Notice that since \(\lambda\) is singular, in all cases \( \mu < \lambda \). Then we again have
 \[ 
\bigcap_{\alpha < \nu} A_\alpha = \bigcup_{s \in \pre{\nu}{\mu}} \, \bigcap_{\alpha < \nu} B^{s(\alpha)}_\alpha.
 \] 
Since \( \mu,\nu < \lambda \)   we have \( \mu^\nu < \lambda \) because \(\lambda\) is strong limit. It then easily follows that since  \( \bigcap_{\alpha < \nu} B^{s(\alpha)}_\alpha \in \lP^0_\xi(X) \) by definition, we  have \( \bigcap_{\alpha < \nu} A_\alpha \in \lS^0_{\xi+1}(X) \), as desired.

Finally, we consider~\ref{prop:levelsBorel-1} for \( \xi = 1 \) under the extra assumption that \( X \) is a \(\lambda\)-Polish space with \( \mathrm{dim}(X) =0 \). By Proposition~\ref{prop:dim(X)=0}, we can assume that \( X \) is a closed subset of \( B(\lambda) \). Then every open \( A \subseteq X \) can be written as \( A = \bigcup_{n \in \omega} A_n \) with
\[ 
A_n = \bigcup \{ \Nbhd_s ( X ) \mid s \in \pre{n}{\lambda} \wedge \Nbhd_s ( X ) \subseteq A \},
 \] 
which is clearly a clopen subset of \( X \) because \( X \setminus A_n = \bigcup \{ \Nbhd_s(X) \mid s \in \pre{n}{\lambda} \wedge \Nbhd_s(X) \nsubseteq A \} \).
\end{proof}

Obviously, if \( |X| \leq \lambda \) then any of its subsets is trivially in \( \lD^0_2(X) \). To see that when \( |X| \nleq \lambda \)  the hierarchy does not collapse before \( \lambda^+ \) (i.e.\ that all inclusions in equation~\eqref{eq:Borelhierarchy} are proper) we use universal sets just as in the classical case. However, the situation here is slightly more delicate because the natural space of codes for sets in one of the given \(\lambda\)-Borel classes is the generalized Cantor space \( \pre{\lambda}{2} \), hence we need to assume that \( 2^{< \lambda} =  \lambda \) in order to have that such a coding space is \(\lambda\)-Polish itself. Granting this, one can  repeat \emph{mutatis mutandis} the arguments in~\cite[Chapter 22]{Kechris1995} and show that  if \( 2^{< \lambda} = \lambda\) then 
\begin{enumerate-(1)}
\item
each \( \lS^0_\xi(X) \) has a \( \pre{\lambda}{2} \)-universal set, that is a set \( U \in \lS^0_\xi (\pre{\lambda}{2} \times X) \) such that \( \lS^0_\xi(X) = \{ U_y \mid y \in \pre{\lambda}{2} \} \) for \( U_y = \{ x \in X \mid (y,x) \in U \} \), and similarly for \( \lP^0_\xi(X) \);
\item
in contrast, there is no \( \pre{\lambda}{2} \)-universal set for \( \lD^0_\xi(\pre{\lambda}{2}) \) (this follows from the fact that \( \lD^0_\xi \) is closed under continuous preimages and complements).
\end{enumerate-(1)}

Using also Corollary~\ref{cor:perfect2}, one finally obtains

\begin{theorem} \label{thm:noncollapseBorelhierarchy}
Assume that \( 2^{< \lambda} = \lambda\) and that \( X \) be a \(\lambda\)-Polish space with \( |X| > \lambda \). Then \( \lS^0_\xi(X) \neq \lP^0_\xi(X) \) for each \( 1 \leq \xi < \lambda^+\). Therefore for all \( 1 \leq \xi < \xi' < \lambda^+ \)
\[ 
\lD^0_\xi(X) \subsetneq \lS^0_\xi(X) \subsetneq \lD^0_{\xi'}(X),
 \] 
and similarly for \( \lP^0_\xi(X) \).
\end{theorem}

\begin{remark}
In particular, the above non-collapsing theorem applies to the space \( X = \pre{\lambda}{2} \approx B(\lambda) \approx C(\lambda) \).
In~\cite{AM} it is shown that such result holds unconditionally for the space \( X = \pre{\lambda}{2} \), even if the existence of universal sets is not granted when \( 2^{< \lambda} > \lambda \) (and thus a different argument is required). However, notice that if \( 2^{< \lambda} > \lambda \), then \( \pre{\lambda}{2} \) is not \(\lambda\)-Polish because of its weight, so in that case one is actually working with a space which does not fit the framework of the present paper.
\end{remark}

Another consequence of the existence, under \( 2^{< \lambda} = \lambda \), of universal sets%
\footnote{To be precise, the existence of universal sets grants \( |\lB(X)| \leq 2^\lambda \). For the lower bound, observe that \( |X| \geq \lambda \) is equivalent to \( \dens(X) = \lambda \) by Lemma~\ref{lem:densityvscardinality}. Then, a nontrivial argument involving Lemma~\ref{lem:r-spaces} shows that \( |\lS^0_\xi(X)| \geq 2^\lambda \). Notice that if instead \( |X| < \lambda \), then \( |\lB(X)| \leq |\pow(X)| < \lambda \) by \(  2^{< \lambda} = \lambda \).}
is e.g.\ that if \( X \) is a \(\lambda\)-Polish space such that \( |X| \geq \lambda \), then for all \( 1 \leq \xi < \lambda^+ \)
\[ 
|\lS^0_\xi(X)| = |\lP^0_\xi(X)| = |\lB(X)| = 2^\lambda.
 \] 
The same applies to \( \lD^0_\xi(X) \) if \( \xi > 1 \).

We conclude this section by observing that when \(2^{\aleph_0} \leq \lambda\), the classical Borel subsets of 
\(\lambda\)-Polish spaces appear at the smallest possible level in the \( \lambda \)-Borel hierarchy: 
this shows that there is a huge difference between our approach and the one considered e.g.\ 
in~\cite{Stone1962,Hansell1971,Stone1972,Hansell1974,Hansell1983}.

\begin{proposition} \label{prop:BorelvslambdaBorel}
Let \( X \) be a \(\lambda\)-Polish space.
\begin{enumerate-(i)}
\item \label{prop:BorelvslambdaBorel-1}
If \( \nu \) is an infinite cardinal such that \( 2^\nu \leq \lambda \), then \( \gBor{\nu^+}(X) \subseteq \lD^0_2(X) \).
\item \label{prop:BorelvslambdaBorel-2}
If \( 2^{< \lambda} = \lambda \) and \( |X| >\lambda \), then \( \gBor{\nu^+}(X) \subsetneq \lD^0_2(X) \) for all infinite \( \nu < \lambda \).
\end{enumerate-(i)}
\end{proposition}

\begin{proof}
\ref{prop:BorelvslambdaBorel-1}
By assumption, \( 2^\nu \) is a cardinal not grater than \(\lambda\): to simplify the notation, set \( \mu = 2^\nu \). Consider the collection \( \mathcal{A}_\nu \) of all subsets of \( X \) which can be written both as a \( \mu \)-sized intersection of open set and as a \( \mu \)-sized union of closed sets, so that clearly \( \mathcal{A}_\nu \subseteq \lD^0_2(X) \). The class \( \mathcal{A}_\nu \) contains all open sets by Fact~\ref{fct:booleancombinationsareFsigma}, and is closed under complements by definition. We claim that \( \mathcal{A}_\nu \) is also closed under \( \nu \)-sized intersections, so that \( \gBor{\nu^+}(X) \subseteq \mathcal{A}_\nu \subseteq \lD^0_2(X) \), as required. Let \( (B_\alpha)_{\alpha < \nu} \) be a family of sets in \( \mathcal{A}_\nu \) so that \( B_\alpha = \bigcap_{\beta < \mu} U_{\alpha,\beta} = \bigcup_{\beta < \mu}  C_{\alpha, \beta} \) with all the sets \( U_{\alpha,\beta} \) open and all the sets \( C_{\alpha,\beta} \) closed. Then \( \bigcap_{\alpha < \nu} B_\alpha = \bigcap_{(\alpha, \beta) \in \nu \times \mu}  U_{\alpha, \beta} \) and the latter can be viewed as a well-ordered intersection of open sets of length \( \mu \) because \( \nu \leq \mu \). Moreover
\[ 
\bigcap_{\alpha < \nu } B_\alpha = \bigcap_{\alpha < \nu } \bigcup_{\beta < \mu} C_{\alpha,\beta} = \bigcup_{s \in \pre{\nu}{\mu}} \bigcap_{\alpha < \nu} C_{\alpha,s(\alpha)}.
 \] 
 Since \( \bigcap_{\alpha < \nu} C_{\alpha,s(\alpha)} \) is closed and \( |\pre{\nu}{\mu}| = 2^\nu = \mu \) because \( \pre{\nu}{(\pre{\nu}{2})} \) is obviously in one-to-one correspondence with \( \pre{\nu}{2} \), this shows that \( \bigcap_{\alpha < \nu } B_\alpha \) is a \( \mu \)-sized union of closed sets as well.
 
\ref{prop:BorelvslambdaBorel-2}
For each \( \nu < \lambda \) we have \( \gBor{\nu^+}(X) \subseteq \lD^0_2(X) \) by part~\ref{prop:BorelvslambdaBorel-1} and \( 2^{< \lambda} = \lambda \).
Assume towards a contradiction that \( \gBor{\nu^+}(X) = \lD^0_2(X) \). Let \( A \in \lS^0_2(X) \setminus \lD^0_2(X) \): such a set exists by Theorem~\ref{thm:noncollapseBorelhierarchy}. Let \( (C_\alpha)_{\alpha < \lambda} \) be closed sets such that \( A = \bigcup_{\alpha < \lambda} C_\alpha \). Then by Proposition~\ref{prop:levelsBorel}\ref{prop:levelsBorel-3} we have \( A_n = \bigcup_{ \alpha < \lambda_n} C_\alpha \in \lD^0_2(X) = \gBor{\nu^+}(X) \) for each \( n \in \omega \). But then \( A = \bigcup_{n \in \omega} A_n \in \gBor{\nu^+}(X) =  \lD^0_2(X) \) by closure of \( \gBor{\nu^+}(X) \) under unions of size \( \nu \) and \( \nu \geq \omega \), a contradiction.
\end{proof}

In particular, if \( \lambda \geq 2^{\aleph_0} \) then the collection of classical (i.e.\ \( \omega_1 \)-)Borel subsets of \( X \) is contained in \( \lD^0_2(X) \), and if \( 2^{< \lambda} = \lambda \) then the inclusion is proper, i.e.\ there are already \( \lD^0_2(X) \)-sets which are not Borel in the classical sense.

\section{Changes of topology} \label{sec:changeoftopology}

In this section we address the problem of whether in our generalized context we have an 
analogue of the powerful technique of changes of topology developed
 e.g.\ in~\cite[Chapter 13]{Kechris1995}. This would amount to 
 prove that for every \(\lambda\)-Borel subset \( A \) of a \(\lambda\)-Polish space \( X \) 
 there is another \(\lambda\)-Polish topology on \( X \) which refines the original one, 
 yields the same \(\lambda\)-Borel sets, and turns \( A \) into a clopen set. Interestingly 
 enough, the argument from~\cite{Kechris1995} cannot be adapted to our context 
 because we miss the analogue of~\cite[Lemma 13.3]{Kechris1995}, whose proof requires that a product of \(\lambda\)-many \(\lambda\)-Polish spaces be \(\lambda\)-Polish as well: as already observed, this is 
 plainly false if \( \lambda > \omega = \cf(\lambda) \), where we even miss closure under
 uncountable products. To overcome this obstacle, we basically reverse the classical approach and first prove the technical Theorem~\ref{thm:changeoftopology}, from which we then derive a number of standard consequences including the result on changes of topology (Corollary~\ref{cor:changeoftopologyclopen}). It might be worth noticing that this alternative approach actually works quite uniformly and independently of \(\lambda\) and its cofinality:  with small and straightforward modifications, it can be successfully used in the case \( \lambda = \omega \) too, but also in the case of a regular \(\lambda\) (see~\cite[Section 4]{AgoMotSch}) and of a singular \(\lambda\) with uncountable cofinality (see~\cite[Section 6]{AgoMot}), provided as usual that \( 2^{< \lambda} = \lambda \).

We will need the following functions.%
\footnote{The functions \( \pi_\alpha \) should already be used in the (omitted) proof of the existence of universal sets for \( \lS^0_\alpha(X) \) --- see~\cite[Theorem 22.3]{Kechris1995}, where the analogues of such functions are denoted by \( y \mapsto (y)_n \).} 
Let \( \langle \cdot, \cdot \rangle \colon \lambda \times \lambda \to \lambda \) be the usual G\"odel pairing function, and for each \( \alpha \) consider the continuous and open surjection 
\( \pi_\alpha \colon \pre{\lambda}{2} \to \pre{\lambda}{2} \) 
defined by setting for \( \beta < \lambda \)
\begin{equation} \label{eq:pseudoprojection}
\pi_\alpha(x)(\beta) = x(\langle \alpha, \beta \rangle).
 \end{equation}
 Similarly, given \( s \in \pre{<\lambda}{2} \) we let \( \pi_\alpha(s) \in \pre{<\lambda}{2} \) be defined by setting \( \pi_\alpha(s)(\beta) = s(\langle \alpha,\beta \rangle) \) for all \( \beta \) such that \( \langle \alpha, \beta \rangle < \lh(s) \). (The fact that \( \pi_\alpha(s) \) is defined on an initial segment of \( \lambda \) follows from the fact that G\"odel's pairing function is monotone in the second coordinate.)


\begin{theorem} \label{thm:changeoftopology}
Assume that \( 2^{< \lambda} = \lambda\) and let \( X \) be a \(\lambda\)-Polish space. Then for every \( B \in \lB(X) \) there is a closed set \( F \subseteq B(\lambda) \) and a continuous \(\lambda\)-Borel isomorphism \( f \colon F \to X \) such that \( f^{-1}(B) \) is clopen relatively to \( F \). 
\end{theorem}

Noteworthy, the following proof ultimately relies on the fact that if \( 2^{< \lambda} = \lambda > \omega = \cf(\lambda) \) then \( B(\lambda) \approx \pre{\lambda}{2} \) (Theorem~\ref{thm:homeomorphictoCantor}\ref{thm:homeomorphictoCantor-2}).

\begin{proof} 
Let \( \mathcal{C} \) be the class of all \( B \subseteq X \) satisfying the conclusion of the theorem. Since it is clearly closed under complements, it is enough to show that it contains all closed sets and that it is closed under intersections of size \( \lambda \).

Assume \( B \subseteq X \) is closed, so that both \( B \) and \( X \setminus B \) are \(\lambda\)-Polish. By Proposition~\ref{prop:surjection} there are closed \( F_0, F_1 \subseteq B(\lambda) \) and continuous bijections \( f_0 \colon F_0 \to B \) and \( f_1 \colon F_1 \to X \setminus B \).

 Since \( B(\lambda) \) is strongly homogeneous, we can assume without loss of generality that \( F_i \subseteq \Nbhd_{\langle i \rangle} \): we claim that the closed set \( F = F_0 \cup F_1 \) and the continuous bijection \( f = f_0 \cup f_1 \)
 witness \( B \in \mathcal{C} \). 
 Indeed, we only need to check that \( f \) has a \(\lambda\)-Borel inverse: this follows from the fact that, by Remark~\ref{rmk:inverseofinjection}, \( f_i(U) \in \lS^0_2(X) \) for each relatively open \( U \subseteq F_i \).

Assume now that \( B_\alpha  \), \( \alpha < \lambda \), are subsets of \( X \) for which there are \( F_\alpha \) and \( f_\alpha \) witnessing \( B_\alpha \in \mathcal{C} \), \axioms{Si usa \( \AC_\lambda(\pre{\lambda}{2}) \) per scegliere gli \( F_\alpha \) e \( f_\alpha \).} and consider \( B = \bigcap_{\alpha < \lambda} B_\alpha \). Since \( B(\lambda) \approx \pre{\lambda}{2} \) by Theorem~\ref{thm:homeomorphictoCantor}, without loss of generality we may assume that \( F_\alpha \) is a closed subset of \( \pre{\lambda}{2} \). Consider the closed set
\[ 
F' = \bigcap_{\alpha < \lambda} \pi_\alpha^{-1}(F_\alpha),
 \] 
 and its closed subset
 \[ 
F'' = \{ x \in F' \mid (f_\alpha \circ \pi_\alpha)(x) = (f_{\alpha'} \circ \pi_{\alpha'})(x) \text{ for all } \alpha, \alpha' < \lambda \}.
 \] 
The function \( h \colon F'' \to X \colon x \mapsto (f_0 \circ \pi_0)(x) \) is clearly continuous and surjective. To check that it is injective, pick distinct \( x,x' \in F'' \). Then there is \( \alpha < \lambda \) such that \( \pi_\alpha(x) \neq \pi_{\alpha}(x') \), hence \( (f_\alpha \circ \pi_\alpha)(x) \neq (f_\alpha \circ \pi_\alpha)(x') \) because \( f_\alpha \) is injective. By definition of \( F'' \), it then follows that
\[ 
h(x) = (f_0 \circ \pi_0)(x) = (f_\alpha \circ \pi_\alpha)(x) \neq (f_\alpha \circ \pi_\alpha)(x') = (f_0 \circ \pi_0)(x') = h(x').
 \] 
Moreover, for every \( s \in \pre{<\lambda}{2} \) we have \( h(\Nbhd_s(F'')) = \bigcap_{\alpha < \kappa} f_\alpha(\Nbhd_{\pi_\alpha(s)}(F_\alpha)) \), 
hence \( h \) has \(\lambda\)-Borel inverse because by assumption all the maps \( f_\alpha \) have \(\lambda\)-Borel inverses.
 
Now notice that 
\[
 h^{-1}(B_\alpha) = F'' \cap (f_0 \circ \pi_0)^{-1}(B_\alpha) = F'' \cap (f_\alpha \circ \pi_\alpha)^{-1}(B_\alpha)
\]
is clopen in \( F'' \), thus
\[ 
h^{-1}(B) = h^{-1} \left( \bigcap\nolimits_{\alpha< \lambda} B_\alpha \right) = \bigcap\nolimits_{\alpha < \lambda} h^{-1}(B_\alpha)
 \] 
is closed in \( F'' \). Since \( F'' \) is \(\lambda\)-Polish, so are \( h^{-1}(B) \) and \( F'' \setminus h^{-1}(B) \). Arguing as above, there are closed \( G_0, G_1 \subseteq B(\lambda) \) with \( G_i \subseteq \Nbhd_{\langle i \rangle} \) (\( i \in \{ 0,1 \} \)) and continuous \(\lambda\)-Borel isomorphisms \( g_0 \colon G_0 \to B \) and \( g_1 \colon G_1 \to X \setminus B \). Therefore the closed set \( G = G_0 \cup G_1 \) and the continuous bijection \( g = h \circ (g_0 \cup g_1) \), which still has \(\lambda\)-Borel inverse, witness that \( B=\bigcap\nolimits_{\alpha< \lambda} B_\alpha \in \mathcal{C} \), as desired. 
\end{proof}


Notice that combining Theorem~\ref{thm:changeoftopology} with the argument in the second part of its proof (namely, the one leading to the construction of the closed set \( F'' \) and the continuous bijection \( h \)), we easily obtain the following generalization of it.

\begin{corollary} \label{cor:changeoftopology}
Let \(\lambda\) be such that \( 2^{< \lambda} = \lambda \), and let \( X \) be a \(\lambda\)-Polish space. Then for every family \( (B_\alpha)_{ \alpha < \lambda} \) of \(\lambda\)-Borel subsets of \( X \) there is a closed set \( F \subseteq B(\lambda) \) and a continuous \(\lambda\)-Borel isomorphism \( f \colon F \to X \) such that \( f^{-1}(B_\alpha) \) is clopen relatively to \( F \) for every \( \alpha < \lambda \).
\end{corollary}

Of course one can then derive all standard consequences of these facts. For example, the following corollary significantly extends both~\cite[Theorem 4]{Stone1962} and the corollary in~\cite[Section 3.6]{Stone1962} from (classical) Borel sets of a \(\lambda\)-Polish space to arbitrary \(\lambda\)-Borel sets, of course under our extra assumptions on \(\lambda\).

\begin{corollary} \label{cor:surjectionBorel}
Let \(\lambda\) be such that \( 2^{< \lambda} = \lambda \) and \( X \) be a \(\lambda\)-Polish space. Then for every \( B \in \lB(X) \) there is a closed set \( G \subseteq B(\lambda) \) and a continuous bijection \( g \colon G \to B \) (with \(\lambda\)-Borel inverse).
If moreover \( B \neq \emptyset \), then there is a continuous surjection \( h \colon B(\lambda) \to B \).
\end{corollary}

\begin{proof}
Let \( F \subseteq B(\lambda) \) and \( f \colon F \to X \) be as in Theorem~\ref{thm:changeoftopology}, and then set \( G = f^{-1}(B) \) and \( g = f \restriction G \).
\end{proof}

\begin{corollary}
Let \(\lambda\) be such that \( 2^{< \lambda} = \lambda \), and let \( X \) be a \(\lambda\)-Polish space. Then for every \( B \in \lB(X) \) there is a closed set \( H \subseteq X \times B(\lambda) \) such that for all \( x \in X \)
\[ 
x \in B \iff \exists y \, (x,y) \in H \iff \exists! y \, (x,y) \in H.
 \] 
\end{corollary}

\begin{proof}
Given \( G \) and \( g \) as in Corollary~\ref{cor:surjectionBorel},  set \( H = (\mathrm{graph}(g))^{-1} \). The set \( H \) is closed because \( G \) is closed in \( B(\lambda) \) and \( g \colon G \to X \) is a continuous function, thus \( \mathrm{graph}(g) \) is closed relatively to  \( G \times X \in \lP^0_1(B(\lambda) \times X) \).
\end{proof}

\begin{corollary} \label{cor:perfect3}
Assume that \(2^{< \lambda} =  \lambda\). 
Let \( X \) be a \(\lambda\)-Polish space and \( B \in \lB(X) \). Then either \( B \) is well-orderable and \( |B| \leq \lambda \), or else there is a continuous \(\lambda\)-Borel embedding \( h \) from \( \pre{\lambda}{2} \) into \( B \) with \(\lambda\)-Borel range (thus, in particular, \( |B| = 2^\lambda \)).
\end{corollary}

\begin{proof}
Let \( F \subseteq B(\lambda) \) and \( f \colon F \to X \) be as in Theorem~\ref{thm:changeoftopology}.  
Since \( f^{-1}(B) \) is clopen in \( F \) and thus closed in \( B(\lambda) \), it is a \(\lambda\)-Polish space.  By Corollary~\ref{cor:perfect2}, 
either \( f^{-1}(B) \) is well-orderable and \( |f^{-1}(B)| \leq \lambda \), in which case \( |B| = |f^{-1}(B)| \leq \lambda \) too, or else
there is  an embedding \( g \colon \pre{\lambda}{2} \to f^{-1}(B) \) with closed range: the composition \( h =  f \circ g \) is then as desired. 
\end{proof}

The above result will be substantially improved in Corollary~\ref{cor:analyticPSP} by showing that \( h \) can be taken to be a \emph{topological} embedding with a \emph{closed-in-\( X \)} range.

We finally come to the desired result on changes of topology. Applying Corollary~\ref{cor:changeoftopology} and pushing forward the topology on \( F \) along the \(\lambda\)-Borel isomorphism \( f \), we obtain that, as in the classical case, one can refine a \(\lambda\)-Polish topology in order to make clopen sets that were just \(\lambda\)-Borel, without changing the overall \(\lambda\)-Borel structure of the space.

\begin{corollary} \label{cor:changeoftopologyclopen}
Let \(\lambda\) be such that \( 2^{< \lambda} = \lambda \), and let \( X = (X,\tau) \) be a \(\lambda\)-Polish space. Then for every  family \( (B_\alpha)_{ \alpha < \lambda} \) of \(\lambda\)-Borel subsets of \( X \) there is a \(\lambda\)-Polish topology \( \tau' \) on \( X \) refining \( \tau \) such that \( \lB(X,\tau') = \lB(X,\tau) \) and  \( B_\alpha \) is \( \tau' \)-clopen for every \( \alpha < \lambda \). Moreover, such a \( \tau' \) can be taken so that \( \mathrm{dim}(X,\tau') = 0 \).
\end{corollary}


As for \(\lambda\)-Borel functions, applying Corollary~\ref{cor:changeoftopologyclopen} to the preimages of the basic open sets of \( Y \) we get:

\begin{corollary} \label{cor:changeoftopologyfunctions}
Assume that \( 2^{< \lambda} = \lambda \). Let \( (X,\tau) \) be \(\lambda\)-Polish and \( Y \) be a topological space of weight at most \( \lambda\). For every \(\lambda\)-Borel function \( f \colon (X,\tau) \to Y \) there is a \(\lambda\)-Polish topology \( \tau' \) on \( X \) refining \(\tau\) such that \( \lB(X,\tau') = \lB(X, \tau)  \) and  \( f \colon (X,\tau') \to Y \) is continuous. Moreover, such a \( \tau' \) can be taken so that \( \mathrm{dim}(X,\tau') = 0 \).
\end{corollary}

\begin{corollary} \label{cor:changeoftopologyfunctions-2} \axioms{\( DC_\omega(\pre{\lambda}{2} ) \)?}
Assume that \( 2^{< \lambda} = \lambda \) and let \( (X, \tau) \)  be a \(\lambda\)-Polish space. For every \(\lambda\)-Borel function \( f \colon (X,\tau) \to (X, \tau) \) there is a \(\lambda\)-Polish topology \( \tau' \) on \( X \) refining \(\tau\) such that \( \lB(X,\tau')= \lB(X,\tau) \) and \( f \colon (X,\tau') \to (X, \tau') \) is continuous. Moreover, such a \( \tau' \) can be taken so that \( \mathrm{dim}(X,\tau') = 0 \). 
\end{corollary}

\begin{proof}
Iteratively apply Corollary~\ref{cor:changeoftopologyfunctions} to get a chain of \(\lambda\)-Polish topologies \( (\tau_n)_{n \in \omega} \) such that \( \tau_0 = \tau \), \( \tau_{n+1} \supseteq \tau_n \), \( \lB(X,\tau_{n+1}) = \lB(X,\tau_n)  \), and \( f \colon (X,\tau_{n+1}) \to (X,\tau_n) \) is continuous. Since we are  dealing with just countably many \(\lambda\)-Polish topologies, the proof of~\cite[Lemma 13.3]{Kechris1995} shows that the topology \( \tau' \) generated by \( \bigcup_{n \in \omega} \tau_n \) is as required.
\end{proof}

Similarly, with a back and forth argument one easily obtains:

\begin{corollary}
Assume that \( 2^{< \lambda} = \lambda \) and let \( (X, \tau_X) \) and \( (Y, \tau_Y) \)  be  \(\lambda\)-Polish spaces. For every \(\lambda\)-Borel isomorphism \( f \colon (X, \tau_X) \to (Y, \tau_Y) \) there are \(\lambda\)-Polish topologies \( \tau'_X \) and \( \tau'_Y \) on \( X \) and \( Y \), respectively, which refine the original topologies, generate the same \(\lambda\)-Borel sets, and such that \( f \colon (X,\tau'_X) \to (Y, \tau'_Y) \) is a homeomorphism. Moreover, such  \( \tau'_X  \) and \( \tau'_Y \) can be taken so that \( \mathrm{dim}(X,\tau'_X)= \mathrm{dim}(Y,\tau'_Y) = 0 \). Finally, if \( (X,\tau_X) = (Y,\tau_Y) \) we can have \( \tau'_X = \tau'_Y \).
\end{corollary}

Clearly, one can also (simultaneously) work with any \(\lambda\)-sized family of \(\lambda\)-Borel functions.

\section{Structural properties} \label{sec:structuralproperties}

We now turn to structural properties. It is convenient to introduce a general notion of \(\lambda\)-pointclass.

\begin{defin}
A \markdef{boldface \(\lambda\)-pointclass} is an operator \( \boldsymbol{\Gamma} \) assigning to each \(\lambda\)-Polish space \( X \) a collection of subsets \( \boldsymbol{\Gamma}(X) \subseteq \pow(X) \) which is closed under continuous preimages, that is, if \( f \colon X \to Y \) is continuous and \( B \in \boldsymbol{\Gamma}(Y) \) then \( f^{-1}(B) \in \boldsymbol{\Gamma}(X) \). The dual \( \check{\boldsymbol{\Gamma}} \) of \( \boldsymbol{\Gamma} \) is the boldface \(\lambda\)-pointclass defined by \( \check{\boldsymbol{\Gamma}}(X) = \{  X \setminus A \mid A \in \boldsymbol{\Gamma}(X) \} \). The ambiguous \( \lambda \)-pointclass \( \boldsymbol{\Delta}_{\boldsymbol{\Gamma}} \) associated to \( \boldsymbol{\Gamma} \) is defined by \( \boldsymbol{\Delta}_{\boldsymbol{\Gamma}}(X) = \boldsymbol{\Gamma}(X) \cap \check{\boldsymbol{\Gamma}}(X) \). 
\end{defin}

Examples of boldface \(\lambda\)-pointclasses are the classes \( \boldsymbol{\Gamma}(X) = \lS^0_\xi(X) \), for an arbitrary \( 1 \leq \xi < \lambda^+ \). In this case, \( \check{\boldsymbol{\Gamma}}(X) = \lP^0_\xi(X) \), and \( \boldsymbol{\Delta}_{\boldsymbol{\Gamma}}(X) = \lD^0_\xi(X) \).

We list here a number of important structural properties (compare them with those defined in~\cite[Section 22.C]{Kechris1995}).

\begin{defin}
Let \( \boldsymbol{\Gamma} \) be a boldface \(\lambda\)-pointclass. 
\begin{enumerate-(a)}
\item
\( \boldsymbol{\Gamma} \) has the \markdef{separation property} if for every \(\lambda\)-Polish space \( X \) and \( A,B \in \boldsymbol{\Gamma}(X) \) with \( A \cap B = \emptyset \), there is \( C \in \boldsymbol{\Delta}_{\boldsymbol{\Gamma}}(X) \) separating \( A \) from \( B \), i.e.\ such that \( A \subseteq C \) and \( C \cap B = \emptyset \).
\item
\( \boldsymbol{\Gamma} \) has the \markdef{\( \lambda \)-generalized separation property} if for every \(\lambda\)-Polish space \( X \) and  any sequence of sets \( A_\alpha \in \boldsymbol{\Gamma}(X) \), \( \alpha < \lambda \), with \( \bigcap_{\alpha < \lambda} A_\alpha = \emptyset \) there is a sequence of sets \( B_\alpha \in \boldsymbol{\Delta}_{\boldsymbol{\Gamma}}(X) \) with \( A_\alpha \subseteq B_\alpha \) and \( \bigcap_{\alpha<\lambda} B_\alpha = \emptyset \).
\item
\( \boldsymbol{\Gamma} \) has the \markdef{reduction property} if for every \(\lambda\)-Polish space \( X \) and for any \( A,B \in \boldsymbol{\Gamma}(X) \) there are disjoint sets \( A^*, B^* \in \boldsymbol{\Gamma}(X) \) such that \( A^* \subseteq A \), \( B^* \subseteq B \), and \( A^* \cup B^* = A \cup B \). (If this happens, we say that the pair \( (A^*, B^*) \) reduces the pair \( (A,B) \).)
\item
\( \boldsymbol{\Gamma} \) has the \markdef{\(\lambda\)-generalized reduction property} if for every \(\lambda\)-Polish space \( X \) and  any sequence of sets \( A_\alpha \in \boldsymbol{\Gamma}(X) \), \( \alpha < \lambda \), there is a sequence of pairwise disjoint sets \( A^*_\alpha \in \boldsymbol{\Gamma}(X) \) with \( A^*_\alpha \subseteq A_\alpha \) and \( \bigcup_{\alpha< \lambda} A^*_\alpha = \bigcup_{\alpha< \lambda} A_\alpha \).
\item
\( \boldsymbol{\Gamma} \) has the \markdef{ordinal \(\lambda\)-uniformization property} if for every \(\lambda\)-Polish space \( X \) and for any \( R \in \boldsymbol{\Gamma}( X \times \lambda) \) (where \( X \times \lambda \) is endowed with the product of the topology of \( X \) and the discrete topology on \( \lambda \)) there is a uniformization of \( R \) in \( \boldsymbol{\Gamma} \), that is, a set \( R^* \subseteq R \) such that \( R^* \in \boldsymbol{\Gamma}(X \times \lambda) \) and \( R^* \) is the graph of a function with domain the projection \( \p(R) = \{ x \in X \mid \exists \alpha < \lambda \, (x,\alpha) \in R \} \) of \( R \) on its first coordinate.
\end{enumerate-(a)}
\end{defin}

We first prove the following Proposition~\ref{prop:regularityproperties}, which is a higher analogue of~\cite[Theorem 22.16]{Kechris1995}. The natural generalization of the original argument would require that the classes \( \lS^0_\xi(X) \) are closed under intersections of length smaller than \( \lambda \), which is false when \(\lambda\) is singular. To overcome this difficulty, we will use a nontrivial adaptation of that proof which heavily relies on our Proposition~\ref{prop:levelsBorel}.

\begin{defin} \label{def:lambdareasonable}
A boldface \(\lambda\)-pointclass \( \boldsymbol{\Gamma} \) is called \markdef{\(\lambda\)-reasonable} if the following property holds:
\begin{quotation}
Suppose that \( I \) is a set of cardinality at most \( \lambda \) and \( X \) is \(\lambda\)-Polish. Let  \( (A_i)_{i \in I} \) be a sequence of subsets of \( X \). Then \( A_i \in \boldsymbol{\Gamma}(X) \) for all \( i \in I \) if and only if
\[ 
A = \{ (x,i) \in X \times I \mid x \in A_i \} \in \boldsymbol{\Gamma}(X \times I),
 \] 
where \( X \times I \) is endowed with the product of the topology of \( X \) and the discrete topology on \( I \).
\end{quotation}
\end{defin}

Notice that if a boldface \(\lambda\)-pointclass \( \boldsymbol{\Gamma} \supseteq \lD^0_1 \) is closed under unions of length \( \lambda \), then \( \boldsymbol{\Gamma} \) is \(\lambda\)-reasonable. Moreover, if \( \boldsymbol{\Gamma} \) is \(\lambda\)-reasonable then so are \( \check{\boldsymbol{\Gamma}}  \) and \( \boldsymbol{\Delta}_{\boldsymbol{\Gamma}} \) because 
\[ 
(X \times I ) \setminus \{ (x,i) \in X \times I \mid x \in A_i \}  = \{ (x,i) \in X \times I \mid x \in X \setminus A_i \} .
 \] 
It follows that all of \( \lS^0_\xi \), \( \lP^0_\xi \), and \( \lD^0_\xi \) are \( \lambda \)-reasonable.%
\footnote{Examples of boldface \(\lambda\)-pointclasses which are not \(\lambda\)-reasonable are the ones defined by setting \( \boldsymbol{\Gamma}(X) = \lS^0_\xi(X) \cup \lP^0_\xi(X) \) for any \( 1 \leq \xi < \lambda^+ \), or even just \( \boldsymbol{\Gamma}(X) = \{ \emptyset,X \} \).}

\begin{proposition} \label{prop:regularityproperties}
Let \(\lambda\) be such that \( 2^{< \lambda} = \lambda \). For every \( 1 < \xi < \lambda^+ \), the boldface 
\(\lambda\)-pointclass \( \lS^0_\xi \) has the ordinal \(\lambda\)-uniformization property. The same is true for \( \lS^0_1 \) if we restrict the attention to \(\lambda\)-Polish spaces \( X \) with \( \mathrm{dim}(X) = 0 \). 
\end{proposition}

\begin{proof}
We first consider the case where \( \xi = \xi'+1 \) is a successor ordinal. Let \( X \) be \(\lambda\)-Polish and \( R \in \lS^0_\xi(X \times \lambda) \). By definition, there are \( R_\alpha \in \lP^0_{\xi'}(X \times \lambda) \) such that \( R = \bigcup_{\alpha < \lambda} R_\alpha \). Set
\[ 
Q = \{ (x,\gamma,\alpha) \in X \times \lambda \times \lambda \mid (x,\gamma) \in R_\alpha \}
 \] 
and notice that \( Q \in \lP^0_{\xi'}(X \times \lambda \times \lambda) \) because \( \lP^0_{\xi'} \) is \(\lambda\)-reasonable. Order \( \lambda \times \lambda \) by setting \( (\gamma', \alpha') \prec (\gamma,\alpha) \) if and only if \( \langle \gamma', \alpha' \rangle < \langle  \gamma,\alpha \rangle \), where \( \langle \cdot, \cdot \rangle \) is again the G\"odel pairing function. Define then \( Q^* \subseteq X \times \lambda \times \lambda \) by setting
\[ 
(x,\gamma,\alpha) \in Q^*  \iff  (x,\gamma,\alpha) \in Q \wedge \forall (\gamma',\alpha') \prec (\gamma,\alpha) \, [(x,\gamma',\alpha') \notin Q ].
 \] 
Since there are less than  \( \lambda \)-many pairs \( (\gamma',\alpha') \prec (\gamma,\alpha) \) and the intersection of fewer than \( \lambda \)-many sets from \( \lS^0_{\xi'}(X \times \lambda \times \lambda) \) is in \( \lD^0_\xi(X \times \lambda \times \lambda ) \) by Proposition~\ref{prop:levelsBorel}\ref{prop:levelsBorel-3}, it follows that \( Q^* \in \lD^0_\xi(X \times \lambda \times \lambda) \subseteq \lS^0_\xi(X \times \lambda \times \lambda) \). Thus the set \( R^* \subseteq X \times \lambda \) defined by
\[ 
(x,\gamma) \in R^*  \iff  \exists \alpha < \lambda \, [(x,\gamma,\alpha) \in Q^*]
 \] 
is in \( \lS^0_\xi(X \times \lambda) \), and it clearly uniformizes \( R \).

If \(\xi\) is limit, we use instead Proposition~\ref{prop:levelsBorel}\ref{prop:levelsBorel-2}. Given \( R \in \lS^0_\xi(X \times \lambda) \), let \( (R_\alpha)_{\alpha < \cf(\xi)} \) and \( s \colon \cf(\xi) \to \xi \) be such that \( R_\alpha \in \lD^0_{s(\alpha)}(X \times \lambda) \) and \( R = \bigcup_{\alpha < \cf(\xi)} R_\alpha \). Without loss of generality, we may assume that \( s \) is non-decreasing. For \( \alpha < \cf(\xi) \), let \( Q^\alpha \subseteq X \times \lambda \times \lambda \) be defined by
\[ 
Q^\alpha = \{ (x,\gamma,\alpha') \in X \times \lambda \times \lambda \mid  \alpha' \leq \alpha \wedge (x,\gamma) \in R_{\alpha'} \},
 \] 
and set 
\[
Q  = \{ (x,\gamma,\alpha) \in X \times \lambda \times \lambda \mid (x,\gamma) \in R_\alpha \} = \bigcup_{\alpha<\cf(\xi)} Q^\alpha. 
\]
Since \( \lD^0_{s(\alpha')}(X) \subseteq \lD^0_{s(\alpha)}(X) \) for \( \alpha' \leq \alpha \) because
\( s \) is non-decreasing,  and since the boldface \(\lambda\)-pointclasses \( \lD^0_{s(\alpha)} \) and \( \lD^0_\xi \) are \(\lambda\)-reasonable, each set \( Q^\alpha \) is \( \lD^0_{s(\alpha)} \), while \( Q \) is \( \lD^0_\xi \). Define \( Q^* \subseteq X \times \lambda \times \lambda \) by
\begin{align*}
(x, \gamma,\alpha) \in Q^* \iff (x,\gamma,\alpha) \in Q & \wedge \forall \alpha' < \alpha \, \forall \gamma' < \lambda \, [(x, \gamma', \alpha') \notin Q^\alpha] \\
& \wedge \forall \gamma' < \gamma \, [(x, \gamma',\alpha) \notin Q^\alpha].
 \end{align*}
Since for a fixed pair \( (\gamma, \alpha) \) the part after the first conjunction in the right term defines a set in \( \lP^0_{s(\alpha)+1}(X) \) and \( s(\alpha)+1 < \xi \) because \(\xi\) is limit, by 
\( \lP^0_{s(\alpha)+1}(X) \subseteq \lD^0_\xi(X) \) and the 
\(\lambda\)-reasonability of \( \lD^0_\xi \) we easily get \( Q^* \in \lD^0_\xi(X \times \lambda \times \lambda ) \). As before, it follows that the set \( R^* \) consisting of those \( (x,\gamma) \in X \times \lambda \) such that \( (x,\gamma,\alpha) \in Q^* \) for some \( \alpha < \cf(\xi) \) is the desired uniformization of  \( R \) in \( \lS^0_\xi \). 

Assume now that \( \mathrm{dim}(X) = 0 \) and \( \xi = 1 \). None of the two previous arguments can work in this case because \( \lD^0_1 \) is in general  not closed under infinite intersections, so a further refinement is in order. By Proposition~\ref{prop:dim(X)=0}, without loss of generality we may assume that \( X \subseteq B(\lambda) \). A set  \( A \subseteq X \) is \( n \)-clopen if there is \( S \subseteq \pre{n}{\lambda} \) such that \( A = \bigcup \{ \Nbhd_s(X) \mid s \in S \} \). It is easy to check that \( n \)-clopen sets are closed under complements and \emph{arbitrary} unions and intersections. Moreover, if \( m \leq n \) then every \( m \)-clopen set is also \( n \)-clopen. Let now \( R \subseteq X \times \lambda \) be open. For each \( \gamma < \lambda \) and \( n \in \omega \), the set
\[ 
R^\gamma_n = \bigcup \{ \Nbhd_s(X) \mid s \in \pre{n}{\lambda} \wedge \Nbhd_s(X) \times \{ \gamma \} \subseteq R \}
 \] 
is \( n \)-clopen, so that  the set \( R_n = \bigcup_{\gamma < \lambda} R^\gamma_n \) is \( n \)-clopen as well. By the closure properties of \( n \)-closed sets, also the sets \( Q^\gamma_n  \subseteq X \) defined by
\[ 
x \in Q^\gamma_n \iff x \in R^\gamma_n \wedge \forall m < n \, (x \notin R_m) \wedge \forall \gamma' < \gamma \, (x \notin R^{\gamma'}_n)
 \] 
are \( n \)-clopen, hence the set 
\begin{align*}
Q^* & = \{ (x,\gamma,n) \in X  \times \lambda \times \omega \mid x \in Q^\gamma_n \} \\
& = \bigcup_{\substack{n \in \omega \\ \gamma < \lambda}} Q^\gamma_n  \times \{ \gamma \} \times \{ n \}
 \end{align*} 
is open. Then, the open set \( R^* \subseteq X \times \lambda \) defined by
\[ 
(x,\gamma) \in R^* \iff \exists n \in \omega \, (x,\gamma,n) \in Q^*
 \] 
unifomizes \( R \).
\end{proof}

The following is the analogue of~\cite[Proposition 22.15]{Kechris1995}. We omit the proof since it is (almost) identical to the original one.

\begin{proposition} \label{prop:relationsamongregularityproperties}
Let \( \boldsymbol{\Gamma} \) be a boldface \(\lambda\)-pointclass.
\begin{enumerate-(i)}
\item
If \( \boldsymbol{\Gamma} \) has the reduction property, then \( \check{\boldsymbol{\Gamma}} \) has the separation property.
\item
If \( \boldsymbol{\Gamma} \) is closed under unions of length \( \lambda \) and has the \(\lambda\)-generalized reduction property, then \( \check{\boldsymbol{\Gamma}} \) has the \(\lambda\)-generalized separation property.
\item
If \( \boldsymbol{\Gamma} \) is \(\lambda\)-reasonable, then \( \boldsymbol{\Gamma} \) has the \(\lambda\)-generalized reduction property if and only if \( \boldsymbol{\Gamma} \) has the ordinal \(\lambda\)-uniformization property.
\item
Assume that \( \boldsymbol{\Gamma}(\pre{\lambda}{2}) \) admits a \( \pre{\lambda}{2} \)-universal set, that is, a set  \( U \in \boldsymbol{\Gamma}(\pre{\lambda}{2} \times \pre{\lambda}{2}) \) such that \( \boldsymbol{\Gamma}(\pre{\lambda}{2}) = \{ U_y \mid y \in \pre{\lambda}{2} \} \).
Then \( \boldsymbol{\Gamma} \) cannot have both the reduction and the separation properties.
\end{enumerate-(i)}
\end{proposition}

Finally, using the fact that the boldface \(\lambda\)-pointclasses \( \lS^0_\xi \) are \(\lambda\)-reasonable and admit \( \pre{\lambda}{2} \)-universal sets if \(2^{< \lambda} = \lambda\), we can combine Propositions~\ref{prop:regularityproperties} and~\ref{prop:relationsamongregularityproperties} to obtain:

\begin{theorem} \label{thm:regularityproperties}
Let \( \lambda \) be such that \( 2^{< \lambda} = \lambda \) and \( 1 < \xi < \lambda^+ \). Then \( \lS^0_\xi \) has the ordinal \(\lambda\)-uniformization property and the (\( \lambda \)-generalized) reduction property, but it does not have the separation property. The class \( \lP^0_\xi \) has instead the (\( \lambda \)-generalized) separation property but not the reduction property. If we restrict the attention to \(\lambda\)-Polish spaces \( X \) with \( \mathrm{dim}(X) = 0 \), the same holds for \( \xi =1 \).
\end{theorem}

\section{Failure of \(\lambda\)-Borel determinacy} \label{sec:noBoreldeterminacy}

An important tool in classical descriptive set theory is constituted by infinite, two-players, zero-sum, perfect information games of length \(\omega\), also called Gale-Stewart games. These are games of the form \( G^\omega_X(A) \), where \( X \) is a nonempty set and \( A \subseteq \pre{\omega}{X} \). Two players \( \pI \) and \( \pII \) alternatively pick elements \( x_0 \), \( x_1 \), \( x_2 \), \dots from \( X \) for \(\omega\)-many turns, so that in the end of the run a sequence \( (x_i)_{i \in \omega} \in \pre{\omega}{X} \) has been played; \( \pI \) wins if \( (x_i)_{i \in \omega} \in A \), otherwise \( \pII \) wins. Winning strategies for either \( \pI \) or \( \pII \) are defined in the obvious way, and a set \( A \subseteq \pre{\omega}{X} \) is said \markdef{determined} if one (and only one) of the players has a winning strategy in the corresponding game \( G^\omega_X(A) \). Notice that if we equip \( \pre{\omega}{X} \) with the product of the discrete topology on \( X \), we can consider the topological complexity of the payoff set \( A \subseteq \pre{\omega}{X} \).

One of the earliest determinacy results is the following Gale-Stewart theorem (see e.g.\ \cite[Theorem 20.1]{Kechris1995}). 

\begin{proposition} \label{prop:closeddeterminacy}
Let \( X \) be a well-orderable set. If \( A \subseteq \pre{\omega}{X} \) is closed or open, then \( G^\omega_X(A) \) is determined.
\end{proposition}

The hypothesis that \( X \) is well-orderable cannot be removed --- the assertion that \( G^\omega_X(A) \) is determined \emph{for every \( X \)} and every closed set \( A \subseteq \pre{\omega}{X} \) is equivalent to the full \( \AC \).
However, in the choiceless setting one can work with ``non-deterministic'' analogues of strategies, called quasistrategies (see~\cite[Section 20.B]{Kechris1995} for the definition and more details). Using this notion, Proposition~\ref{prop:closeddeterminacy} can be restated for all sets \( X \) as follows:

\begin{proposition} \label{prop:closedquasideterminacy}
Let \( X \) be any set. If \( A \subseteq \pre{\omega}{X} \) is closed or open, then \( G^\omega_X(A) \) is \markdef{\emph{quasidetermined}} i.e.\ one (and only one) of the players has a winning quasistrategy.
\end{proposition}

It is easy to see that when \( X \) is well-ordered by \( \precsim \), then from any winning quasistrategy \(\sigma\) for one of the players in the game \( G^\omega_X(A) \) one can canonically extract a winning strategy by picking the \( \precsim \)-least move compatible with \(\sigma\) in each round. (Actually, for this argument to work it is enough that the set of \emph{all moves that can be made by the given player} can be well-ordered; this is relevant for the ``games with rules'' briefly discussed at the end of this section.)
This simple observation shows that Proposition~\ref{prop:closedquasideterminacy} implies Proposition~\ref{prop:closeddeterminacy}. 

The use of games in classical descriptive set theory relies on the fact that one can sometimes formalize a given problem via a suitable game of the form \( G^\omega_X(A) \), for some \( X \) and \( A \): when such a game is determined, one then obtain useful information on (if not the solution to) the given problem. For example, the regularity of Borel subsets of the classical Baire space \( \pre{\omega}{\omega} \) in terms of Baire category, Lebesgue measurability, perfect set property, and so on, can be deduced via games from Borel determinacy, a theorem proved by Martin (see e.g.\ \cite[Theorem 20.5]{Kechris1995}) which greatly extend Proposition~\ref{prop:closeddeterminacy} when \( X \) is a countable set. 

 \begin{theorem} \label{thm:Boreldeterminacy}
If \( A \subseteq \pre{\omega}{\omega} \) is (\( \omega_1 \)-)Borel, then \( G^\omega_\omega(A) \) is determined.
\end{theorem}\

It is also worth noticing that if \( \RR \) is well-orderable, then there are sets \( A \subseteq \pre{\omega}{2} \) which are not determined.

Moving to our generalized context, there are at least two possible directions one might take: considering games%
\footnote{Usually, in games longer than \(\omega\) one stipulates that at limit levels it is \( \pI \) who has to play next. However, different choices make no difference for what concerns Proposition~\ref{prop:noclopendeterminacyinlonggames}.}
\( G^\mu_X(A) \) of length \( \mu > \omega \) or, in view of Proposition~\ref{prop:closeddeterminacy}, considering games of the form \( G^\omega_\nu(A) \) with \( \nu > \omega \). In the first case, it is natural to endow the space \( \pre{\mu}{X} \) with the bounded topology. However, in this setup there is no hope to get analogues of Proposition~\ref{prop:closeddeterminacy} and Theorem~\ref{thm:Boreldeterminacy} as soon as \( \RR \) is well-orderable.

\begin{proposition}[Folklore] \label{prop:noclopendeterminacyinlonggames}
Let \( \mu > \omega  \) and equip \( \pre{\mu}{X} \) with the bounded topology. Assume that \( \RR \) is well-orderable and that \( X \) has at least two elements. Then there is a  clopen set \( C \subseteq \pre{\mu}{X} \) such that \( G^\mu_X(C) \) is not determined.
\end{proposition}

\begin{proof}
Let \( A \subseteq \pre{\omega}{2} \) be such that \( G^\omega_2(A) \) is not determined. 
Fix distinct \( x_0, x_1 \in X \). By identifying \( i \in \{ 0,1 \} \) with the corresponding \( x_i \), without loss of generality we can assume that \(  \pre{\omega}{2} \subseteq \pre{\omega}{X} \) Set \( C = \bigcup_{s \in S} \Nbhd_s(\pre{\mu}{X}) \) where \( S \) is the union of \( A \subseteq \pre{\omega}{X} \) with all finite sequences \( s \in \pre{<\omega}{X} \) such that \( s \notin \pre{< \omega}{2} \) and the smallest \( i < \lh(s) \) such that \( s(i) \notin \{ 0,1 \} \) is odd. It is straightforward to check that \( G^\mu_X(C) \) is equivalent to \( G^\omega_2(A) \), that is, a player has a winning strategy in the first game if and only if (s)he has a winning strategy in the latter one. Thus \( G^\mu_X(C) \) is not determined by the choice of \( A \). Moreover, \( C \) is clearly open with respect to the bounded topology, but it is also closed because \( \pre{\mu}{X} \setminus C = \bigcup_{t \in T} \Nbhd_t(\pre{\mu}{X}) \) where \( T \) is the union of \( \pre{\omega}{2} \setminus A \) with all finite sequences \( t \in \pre{<\omega}{X} \) such that \( t \notin \pre{< \omega}{2} \) and the smallest \( i < \lh(t) \) such that \( t(i) \notin \{  0,1 \} \) is even.
\end{proof}

\begin{remark}
One might argue that the problem arises because we considered the bounded topology on \( \pre{\mu}{X} \). But even when considering the (\( < \omega \)-supported) product topology on such a space, the set \( C \) constructed in the proof of Proposition~\ref{prop:noclopendeterminacyinlonggames} is not very complicated. If e.g.\ \( \nu \geq 2^{\aleph_0} \), then \( C \) is in the class \( \nu\text{-}\boldsymbol{\Delta}^0_2 \) relatively to product topology \( \tau_p \), that is, \( C \) is both a \( \nu \)-sized union of \( \tau_p \)-closed sets and a \( \nu \)-sized intersection of \( \tau_p \)-open sets.
\end{remark}

The second option seems more promising, as Proposition~\ref{prop:closeddeterminacy} applies: even when \( \nu > \omega \), the game \( G^\omega_\nu(A) \) is determined whenever \( A \subseteq \pre{\omega}{\nu} \) is closed or open. However, if \( \RR \) is well-orderable we cannot go much further.

\begin{proposition} \label{prop:noBoreldeterminacyheightomega}
Let \( \nu \in \Cn \) be such that \( 2^{\aleph_0} \leq \nu \). Then there is a \( \nu\text{-}\boldsymbol{\Delta}^0_2(\pre{\omega}{\nu}) \) set \( D \) such that \( G^\omega_\nu(D) \) is not determined.
\end{proposition}

\begin{proof}
The proof is similar to that of Proposition~\ref{prop:noclopendeterminacyinlonggames}.
Let \( A \subseteq \pre{\omega}{\omega} \) be such that \( G^\omega_\omega(A) \) is not determined. Let \( D \) be the set \( A \cup \bigcup_{s \in S} \Nbhd_s(\pre{\omega}{\nu}) \) where \( S \subseteq \pre{<\omega}{\nu} \) consists of those \( s \notin \pre{<\omega}{\omega}  \) such that the smallest \( i < \lh(s) \) such that \( s(i) \geq \omega \) is odd. Then \( G^\omega_\nu(D) \) is equivalent to \( G^\omega_\omega(A) \), hence it is not determined. Moreover \( D \) is the union of at most \( 2^{\aleph_0} \)-many singletons (the elements of \( A \)) with an open set, thus \( D \in \nu\text{-}\boldsymbol{\Sigma}^0_2(\pre{\omega}{\nu}) \) because \( \nu \geq 2^{\aleph_0} \). Similarly, the complement of \( D \) is again a union of at most \( \nu \)-many singletons (the elements in \( \pre{\omega}{\omega} \setminus A \)) with the open set \( \bigcup_{t \in T} \Nbhd_t(\pre{\omega}{\nu}) \), where \( T \subseteq \pre{<\omega}{\nu} \) consists of those \( t \notin \pre{<\omega}{\omega}  \) such that the smallest \( i < \lh(t) \) such that \( t(i) \geq \omega \) is even. Thus \( \pre{\omega}{\nu} \setminus D \in \nu\text{-}\boldsymbol{\Sigma}^0_2(\pre{\omega}{\nu}) \) as well, hence \( D \in \nu\text{-}\boldsymbol{\Delta}^0_2(\pre{\omega}{\nu}) \).
\end{proof}


Stepping back to our uncountable \(\lambda\), the previous results show that we cannot hope for any form of ``\(\lambda\)-Borel determinacy'' for subsets of \( B(\lambda) \) or \( \pre{\lambda}{2} \). More in detail, when \( A \subseteq B(\lambda) \) the Gale-Stewart Theorem (Proposition~\ref{prop:closeddeterminacy}) for games of the form $G^\omega_\lambda(A)$ gives us determinacy for closed sets, but Proposition \ref{prop:noBoreldeterminacyheightomega} already gives us \( \lD^0_2(B(\lambda)) \)-sets which are not determined. If we instead consider the generalized Cantor space \( \pre{\lambda}{2} \) and assume \( 2^{< \lambda} = \lambda \) (which is the condition turning \( \pre{\lambda}{2} \) into a \(\lambda\)-Polish space), then by Proposition~\ref{prop:noclopendeterminacyinlonggames} there are even clopen games $G^\lambda_2(A)$ which are not determined.

To consider spaces like $C(\lambda)$ or $V_{\lambda+1}$, we need to extend our definition of games by adding extra rules in the form of a tree of ``legal positions''. This means that besides the length \( \mu \) of the game and the set \( X \) of all possible moves, we also fix a \( < \mu \)-closed%
\footnote{A \( \mu \)-tree \( T \subseteq \pre{<\mu}{X} \) is \( < \mu \)-closed if for every limit ordinal \( \alpha < \mu \) and every \( t \in \pre{\alpha}{X} \), if \( t \restriction \beta \in T \) for every \( \beta < \alpha \) then \( t \in T \).} 
\( \mu \)-tree \( T \subseteq \pre{<\mu}{X} \) and require the players to ensure that every partial play in a given run of the game, which is an element of \( \pre{<\mu}{X} \), belongs to \( T \) --- the first player that breaks this prescription loses immediately. Given a payoff set \( A \subseteq [T] \), we then obtain the game ``with rules'' \( G^\mu _X(T,A) \). (This is just a jargon to make the game look anthropomorphic: every game of this form can be turned into a game of the form \( G^\mu_X(A') \), where \( A' \) is obtained from \( A \) by adding a suitable open set to it.)
For example, all the games in Section~\ref{chapter:lambda-PSP}, if properly formalized, are of this sort. 

This formalism allows us to consider natural games on \( C(\lambda) \) and \( V_{\lambda+1} \) too.  For example, if $A\subseteq C(\lambda)$, then we can consider the game on \( \lambda \) of length $\omega$ in which we require that each move $x_i$ belongs to $\lambda_i$, for every  \( i \in \omega \); in order words, we consider the game \( G^\omega_\lambda(T,A) \) where \( T = \pre{<\omega}{(\vec \lambda)} \) is the \(\omega\)-tree on \( \lambda \) such that \( [T] = C(\lambda) \).  
Then Proposition \ref{prop:noBoreldeterminacyheightomega} holds with the same proof, and if \( 2^{\aleph_0} \leq \lambda \) (in particular: if \(\lambda\) is \(\omega\)-inaccessible) there are \( \lD^0_2(C(\lambda)) \)-subsets of \( C(\lambda) \) which are not determined, while all closed subsets of $C(\lambda)$ are determined by Proposition~\ref{prop:closeddeterminacy}. 

Under the assumption \( \beth_\lambda = \lambda \), given $A\subseteq V_{\lambda+1}$ we can think of two possible games with payoff set \( A \). The first one is the game of length $\lambda$ in which we require that the $\alpha$-th move $x_\alpha$ is a subset of $V_{\alpha+1}\setminus V_\alpha$; then player \( \pI \) wins if and only if $\bigcup_{\alpha<\lambda}x_\alpha\in A$. In this case the proof of Proposition~\ref{prop:noclopendeterminacyinlonggames} yields that there are non-determined clopen games. As an alternative, one could consider the game of length $\omega$ in which the $i$-th move $x_i$ is a subset of $V_{\lambda_{i}}\setminus V_{\lambda_{i-1}}$ (where we set $V_{\lambda_{-1}}=\emptyset$), and player \( \pI \) wins if and only if $\bigcup_{i \in \omega}x_i \in A$. In this case,
all closed sets \( A \subseteq V_{\lambda+1} \) are determined by Proposition~\ref{prop:closeddeterminacy}, but
there are \( \lD^0_2(V_{\lambda+1}) \)-subsets of \( V_{\lambda+1} \) which are not determined by Proposition~\ref{prop:noBoreldeterminacyheightomega}.

The latter suggests an alternative game for subsets of \( \pre{\lambda}{2} \), under the assumption \( 2^{< \lambda} = \lambda \). Consider the game of length $\omega$ in which the $i$-th move is an element of $\pre{\lambda_i}{2}$, and player \( \pI \) wins if and only if the concatenation of those \( \omega \)-many moves belongs to $A$. Then we have closed determinacy (Proposition~\ref{prop:closeddeterminacy}), but also \( \lD^0_2(\pre{\lambda}{2}) \)-sets which are not determined (Proposition \ref{prop:noBoreldeterminacyheightomega}).

Summing up, the above discussion shows that in the generalized context there is no natural analogue of Theorem~\ref{thm:Boreldeterminacy}, irrespective of the chosen space and of the game on it. For this reason, results asserting that certain games are determined, like Corollary~\ref{cor:determinacyPSPgames}, become informative.


%
%
%

\chapter{\( \lambda \)-analytic and \(\lambda\)-projective sets} \label{sec:lambda-analytic}

\section{Definition and  examples} \label{sec:defanalytic}

The following is the natural generalization of~\cite[Definition 14.1]{Kechris1995} to arbitrary infinite cardinals. By its equivalent reformulation from Proposition~\ref{prop:charanalytic}\ref{prop:charanalytic-4}, when \( X \) is a (classical) Polish space it also coincides with the notion of \(\lambda\)-Souslin sets (see e.g.\ \cite[Section 31.B]{Kechris1995}).

\begin{defin} \label{def:lambdaanalytic}
Let \( X \) be a \(\lambda\)-Polish space. A set \( A \subseteq X \) is \markdef{\(\lambda\)-analytic} if it is 
a continuous image of some \(\lambda\)-Polish space \( Y \). 
\end{defin}

The collection of all \(\lambda\)-analytic subsets of \( X \) is denoted by \( \lS^1_1(X) \).
As usual, when \( X \) is clear from the context we remove it from the notation above and we sometime say that ``\( A \) is \( \lS^1_1 \)'' instead of ``\( A \) is \(\lambda\)-analytic''. We also set
\begin{align*}
\lP^1_1(X) & = \{ A \subseteq X \mid X \setminus A \in \lS^1_1(X) \} \\
\lD^1_1(X) & = \lS^1_1(X) \cap \lP^1_1(X).
\end{align*}
Sets in \(\lP^1_1 (X) \) are called \markdef{\(\lambda\)-coanalytic}, while sets in \( \lD^1_1(X) \) are called \markdef{\( \lambda \)-bianalytic}.

It is straightforward to check that if \( Y \subseteq X \) are both \(\lambda\)-Polish, then
\[ 
\lS^1_1(Y) = \{  A \cap Y \mid A \in \lS^1_1(X) \},
 \] 
and similarly for the classes \( \lP^1_1(Y) \) and \( \lD^1_1(Y) \).

By Corollary~\ref{cor:surjectionBorel}, if \( 2^{< \lambda} = \lambda\) then
\begin{equation} \label{eq:borelvsanalytic}
\lB(X) \subseteq \lS^1_1(X),
 \end{equation}
 hence \( \lB(X) \subseteq \lD^1_1(X) \) because \( \lB(X) \) is closed under complements.
 
This enables us to prove several characterizations of \(\lambda\)-analytic sets. The proof is basically identical to the classical case \( \lambda  = \omega \) (using of course the results proved so far, including e.g.\ Proposition~\ref{prop:surjection} and Corollaries~\ref{cor:changeoftopologyclopen} and~\ref{cor:changeoftopologyfunctions}), so it will be omitted. In what follows we denote by \( \p(F) \) the projection on the first-coordinate of the set \( F \subseteq X \times Y \).
 
\begin{proposition} \label{prop:charanalytic}
Let \( X \) be a \(\lambda\)-Polish space. For any \( \emptyset \neq A \subseteq X \) the following are equivalent:
\begin{enumerate}[leftmargin=2pc, label = \upshape (\arabic*)]
\item
\( A \) is \(\lambda\)-analytic;
\item
\( A \) is a continuous image of a closed (or even just \( G_\delta \)) subset of some \(\lambda\)-Polish space \( Y \);
\item \label{prop:charanalytic-3}
\( A \) is a continuous image of \( B(\lambda) \) (equivalently, of \( C(\lambda) \)), or even just of a closed or \( G_\delta \) subset of it;
\item \label{prop:charanalytic-4}
\( A = \p(F) \) for some closed \( F \subseteq X \times B(\lambda) \);
\item \label{prop:charanalytic-5}
\( A = \p(G) \) for some \( G \subseteq X \times Y \) with \( Y \) a \(\lambda\)-Polish space and \( G \) a \( G_\delta \) set.
\end{enumerate}
If \( 2^{< \lambda} = \lambda \) we can further add
\begin{enumerate}[leftmargin=2pc, label = \upshape (\arabic*), resume]
\item \label{prop:charanalytic-6}
\( A \) is a continuous image of \( \pre{\lambda}{2} \);
\item
\( A = \p(B) \) for some \( \lambda \)-Borel \( B \subseteq X \times Y \), where \( Y \) is any \(\lambda\)-Polish space;
\item
\( A \) is a continuous image of a \(\lambda\)-Borel subset of some Polish space \( Y \);
\item \label{prop:charanalytic-9}
\( A \) is a \(\lambda\)-Borel image of a \(\lambda\)-Borel subset of some Polish space \( Y \).
\end{enumerate}
\end{proposition}
 
The various equivalent definitions of \(\lambda\)-analytic sets from Proposition~\ref{prop:charanalytic} allow us to obtain more information on them. For example,
by~\ref{prop:charanalytic-9} we have that if \( 2^{< \lambda} = \lambda \), then the notion of \(\lambda\)-analytic set depends just on the \(\lambda\)-Borel structure of the space \( X \). Indeed, if \( \tau \) and \( \tau' \) are different \(\lambda\)-Polish topologies on \( X \) such that \( \lB(X,\tau) = \lB(X,\tau') \) and \( A \subseteq X \), then \( A \) is \(\lambda\)-analytic with respect to \(\tau\) if and only if it is analytic with respect to \( \tau' \). The same holds for  \(\lambda\)-coanalytic and \(\lambda\)-bianalytic sets. 
Moreover, by~\ref{prop:charanalytic-4} a set \( A \subseteq B(\lambda) \) is \(\lambda\)-analytic if and only if there is an \(\omega\)-tree \( T \subseteq \pre{< \omega}{(\lambda \times \lambda)} \) on \( \lambda \times \lambda \) such that%
\footnote{Formally, we should have written \( A = \p([T]) \) instead of \( A = \p[T] \). However, when dealing with projections of the body of a tree \( T \), we prefer to systematically omit the parentheses around \( [T] \) in order to simplify the notation.}
\begin{equation} \label{eq:normalformforanalytic}
A = \p[T].
 \end{equation}
Clearly, if \( 2^{< \lambda} =  \lambda \) we can also replace \( B(\lambda) \) with \( \pre{\lambda}{2} \) and notice that \( A \subseteq \pre{\lambda}{2} \) is \(\lambda\)-analytic if and only if there is a \(\lambda\)-tree \( T \subseteq \pre{< \lambda}{(2 \times 2)} \) such that \( A = \p[T] \). Similar considerations hold for subsets of (at most countable) products of copies of \( B(\lambda) \) or \( \pre{\lambda}{2} \).  Equation~\eqref{eq:normalformforanalytic} provides a sort of ``normal form'' representation of \(\lambda\)-analytic sets and will often be our working definition.

\begin{remark}
Inspired by the usual definition given in the regular case (see e.g.\ \cite{AM, Friedman:2011nx,   Motto-Ros:2011qc,  LS, HytKul2015,   FKK2016, MMS2016,  HytKul2018, Mildenberger2019} and the references contained therein), one might be tempted to define \( \lambda \)-analytic subsets of \( \pre{\lambda}{2} \) as the sets of the form \( \p[T] \) for \( T \) a subtree of \( \pre{< \lambda}{(2 \times \lambda)} \) or, equivalently, as the projections of closed subsets of \( \pre{\lambda}{2} \times \pre{\lambda}{\lambda} \). However, besides the fact that  if \(\lambda\) is singular the space \( \pre{\lambda}{\lambda} \) arguably fall in a different setup because it has weight greater than \( \lambda \), one can also prove that the class obtained in this way trivializes.
Indeed, when \( 2^{< \lambda} = \lambda > \cf(\lambda) = \omega \)  there is
 a \emph{continuous} surjection from any \( \pre{\mu}{\lambda} \) with \( \mu > \omega \) onto any nonempty \( A \subseteq \pre{\lambda}{2} \), so that all subsets  of \( \pre{\lambda}{2} \) are projections of closed subsetes of \( \pre{\lambda}{2} \times \pre{\lambda}{\lambda} \). (Indeed, by Theorem~\ref{thm:homeomorphictoCantor} we can work with nonempty sets \( A \subseteq B(\lambda) \). Fix any \( y_0 \in A \) and define \( f \colon \pre{\mu}{\lambda} \to B(\lambda) \) by \( f(x) = x \restriction \omega \) if \( x \restriction \omega \in A \) and \( f(x) = y_0 \) otherwise. Then \( f \) is clearly continuous and onto \( A \).)
 All of this adds evidence to the fact that, in the general case, the ``right'' analogue of the Baire space is \( \pre{\cf(\lambda)}{\lambda} \). Notice that when \(\lambda\) has countable cofinality this results in the space \( B(\lambda) = \pre{\omega}{\lambda} \) we are currently using, while when \( \lambda \) is regular it is the usual space \( \pre{\lambda}{\lambda} \) considered in the existing literature.
\end{remark}
 
Another consequence of the reformulation~\ref{prop:charanalytic-4} from Proposition~\ref{prop:charanalytic} 
is that if
\(2^{< \lambda} = \lambda\) and \( |X| > \lambda \), then the inclusion in~\eqref{eq:borelvsanalytic} is proper, that is, there are \(\lambda\)-analytic subsets of \( X \) which are not \(\lambda\)-Borel: this follows from the existence of \( \pre{\lambda}{2} \)-universal sets for \( \lS^1_1(X) \), which can be obtained as in the classical case as a suitable projection of a universal set for closed subsets of \( X \times B(\lambda) \). This also implies that for most \(\lambda\)-Polish spaces \( X \) there are sets \( C \subseteq X \) which are complete for \( \lambda \)-analytic sets, i.e.\ such that for every \( A \in \lS^1_1(B(\lambda)) \) there is a continuous function \( f \colon B(\lambda) \to X \) such that \( f^{-1}(C) = A \). 

An example of such a complete \(\lambda\)-analytic set is the collection of codes for ill-founded \( \omega \)-trees on \(\lambda\), which is defined as follows. 
Fix a bijection 
\begin{equation} \label{eq:sigmatr} 
\sigma_{\mathrm{Tr}} \colon \pre{<\omega}{\lambda} \to \lambda
 \end{equation} 
 such that for all \( \alpha < \lambda \) there is \( n \in \omega \) for which \( \lh(\sigma_{\mathrm{Tr}}^{-1}(\beta)) \leq n \) for all \( \beta < \alpha \). Using \( \sigma_{\mathrm{Tr}} \), every set \( A \subseteq \pre{< \omega}{\lambda} \) can be coded into an element \( x_A \in \pre{\lambda}{2} \) by setting \( x_A(\alpha) = 1 \iff \sigma_{\mathrm{Tr}}^{-1}(\alpha) \in A \). Vice versa, every \( x \in \pre{\lambda}{2} \) codes the set \( T_x = \{ s \in \pre{<\omega}{\lambda} \mid x(\sigma_{\mathrm{Tr}}(s)) = 1 \} \). Let
 \begin{equation} \label{eq:Trlambda}
\mathrm{Tr}_\lambda = \{ x \in \pre{\lambda}{2} \mid T_x \text{ is a tree} \}.
 \end{equation}
It is not hard to check that \( \mathrm{Tr} \) is closed in \( \pre{\lambda}{2} \), hence it is \(\lambda\)-Polish whenever \( 2^{< \lambda} =  \lambda \).
Consider the set
\[ 
\mathrm{WF}_\lambda = \{ x \in \mathrm{Tr} \mid [T_x] = \emptyset \}
 \] 
of (codes for) well-founded \(\omega\)-trees on \(\lambda\), and let \( \mathrm{IF}_\lambda = \mathrm{Tr}_\lambda \setminus \mathrm{WF}_\lambda \) be the set of (codes for) ill-founded \(\omega\)-trees.

\begin{proposition} \label{prop:IF}
Let \(\lambda\) be such that \( 2^{< \lambda} = \lambda \).
The set \( \mathrm{IF}_\lambda \) is complete for \(\lambda\)-analytic sets.
\end{proposition}

\begin{proof}
To see that \( \mathrm{IF}_\lambda \) is \(\lambda\)-analytic, consider the set \( F \subseteq \mathrm{Tr}_\lambda \times B(\lambda) \) consisting of those pair \( (x,y) \) such that for all \( n < m \) we have \( \sigma_{\mathrm{Tr}}^{-1}(y(n)) \subsetneq \sigma_{\mathrm{Tr}}^{-1}(y(m)) \) and \( x(y(n)) = x(y(m)) = 1 \). Then \( F \) is closed and \( \mathrm{IF}_\lambda = \p(F) \).

Given \( A \in \lS^1_1(B(\lambda)) \), let \( T \subseteq \pre{<\omega}{(\lambda \times \lambda)} \) be a tree such that \( A = \p[T] \), so that
\( x \in A \) if and only if the corresponding section tree \(T(x) \) is ill-founded. Let \( f \colon B(\lambda) \to \mathrm{Tr}_\lambda \) be the function sending \( x \) to the (unique) code for \( T(x) \). By the choice of \( \sigma_{\mathrm{Tr}} \), the map \( f \) is continuous. Indeed, let \( x \in B(\lambda) \) and consider the basic neighborhood \( \Nbhd_{f(x) \restriction \alpha} \) of \( f(x) \), for any \( \alpha < \lambda \). Let \( n \in \omega \) be such that \( \lh(\sigma_{\mathrm{Tr}}^{-1}(\beta)) \leq n \) for all \( \beta < \alpha \): then \( f(\Nbhd_{x \restriction n}) \subseteq \Nbhd_{f(x) \restriction \alpha} \). Finally, \( f^{-1}(\mathrm{IF}_\lambda) = A \) by the choice of \( T \).
\end{proof}

As in the classical case, this allows us to move from \(\omega\)-trees to linear orders. Identify \( \pre{\lambda \times \lambda}{2} \) with \( \pre{\lambda}{2} \) using the G\"odel pairing function \( \langle \cdot , \cdot \rangle \colon \lambda \times \lambda \to \lambda \). Given \( x \in \pre{\lambda \times \lambda}{2} \) and \( \alpha, \beta < \lambda \), set \( \alpha \preceq_x \beta \iff x(\alpha,\beta) = 1 \). It is not hard to check that
\[ 
\mathrm{LO}_\lambda = \{ x \in \pre{\lambda \times \lambda}{2} \mid {\preceq_x} \text{ is a linear order on } \lambda \}
 \] 
of (codes for) linear orders on \(\lambda\) is closed in \( \pre{\lambda \times \lambda}{2} \), and thus it is \(\lambda\)-Polish if \(2^{< \lambda} = \lambda\). Let
\[ 
\mathrm{WO}_\lambda = \{ x \in \mathrm{LO}_\lambda \mid {\preceq_x} \text{ is well-founded} \},
 \] 
and \( \mathrm{NWO}_\lambda = \mathrm{LO}_\lambda \setminus \mathrm{WO}_\lambda \). Then \( \mathrm{NWO}_\lambda \) is clearly \(\lambda\)-analytic, hence \( \mathrm{WO}_\lambda \) is \(\lambda\)-coanalytic. Consider the map sending each \( x \in \mathrm{Tr}_\lambda \) into the linear ordering \( \trianglelefteq_x \) on \(\lambda\) defined by setting \( \alpha \trianglelefteq_x \beta \) if and only if \( x(\alpha) = x( \beta) = 1 \) and \( 
\sigma_{\mathrm{Tr}}^{-1}(\alpha) \leq_{\mathrm{KB}} \sigma_{\mathrm{Tr}}^{-1}(\beta) \) (where \( \leq_{\mathrm{KB}} \) is the usual Kleene-Brouwer ordering), or \( x(\alpha) = 1  \) and \(x( \beta) = 0 \), or \( x(\alpha) = x( \beta) = 0 \) and \( \alpha \leq \beta \). Such map is continuous and \( T_x \in \mathrm{WF}_\lambda \iff {\trianglelefteq_x} \in \mathrm{WO}_\lambda \). It follows that~\cite[Theorem 27.12]{Kechris1995} generalizes to: 

\begin{proposition} \label{prop:NWO}
Let \(\lambda\) be such that \( 2^{< \lambda} = \lambda \).
Then \( \mathrm{NWO}_\lambda \) is a complete \(\lambda\)-analytic set.
\end{proposition}
 
 \begin{remark} \label{rmk:IFcompletenessfromcofomega}
 Propositions~\ref{prop:IF} and~\ref{prop:NWO} work precisely because we assumed \( \cf(\lambda) = \omega \). If \( \cf(\lambda) > \omega \), then the coding spaces \( \mathrm{Tr}_\lambda \) and \( \mathrm{LO}_\lambda \) can still be defined, and they turn out to be closed subsets of the generalized Cantor space \( \pre{\lambda}{2} \). But the sets \( \mathrm{IF}_\lambda \) and \( \mathrm{NWO}_\lambda \) would then become open subsets of such spaces.
 Indeed, let \( x \in \mathrm{IF}_\lambda \), and let \( (\alpha_n)_{n \in \omega} \) be a sequence of ordinals smaller than \(\lambda\) such that \( (\sigma_{\mathrm{Tr}}^{-1}(\alpha)n))_{n \in \omega} \) is an infinite branch through \( T_x \). By \( \cf(\lambda) > \omega \), we have that \( \alpha = \sup_{n \in \omega} \alpha_n < \lambda \): therefore, \( \Nbhd_{x \restriction \alpha} \cap \mathrm{Tr}_\lambda \) is a basic neighborhood of \( x \) contained in \( \mathrm{IF}_\lambda \), as desired. 
 A similar argument shows that \( \mathrm{NWO}_\lambda \) is open in \( \mathrm{LO}_\lambda \).
 \end{remark}


 
\section{Basic properties} \label{sec:basifactsanalytic}
 
We now check some standard closure properties of \(\lambda\)-analytic sets.
 
\begin{proposition}\label{prop:closurepropertiesofanalytic}
\begin{enumerate-(i)}
\item \label{prop:closurepropertiesofanalytic-i}
If \( X \) is \(\lambda\)-Polish, then \( \lS^1_1(X) \) is closed under well-ordered unions of length at most \( \lambda \) and countable intersections. If we further assume that \( 2^{< \lambda} =  \lambda \), then \( \lS^1_1(X) \) is also closed under well-ordered intersections of length at most \( \lambda \).
\item \label{prop:closurepropertiesofanalytic-ii}
If \( X \) and \( Y \) are \(\lambda\)-Polish spaces, \( f \colon X \to Y \) is \(\lambda\)-Borel, \( A \in \lS^1_1(X) \), and \( B \in \lS^1_1(Y) \), then both \( f^{-1}(B) \in \lS^1_1(X) \) and \( f(A) \in \lS^1_1(Y) \).
\end{enumerate-(i)}
\end{proposition}

\begin{proof}
For the first part, let \( A_\alpha \in \lS^1_1(X)\) and \( f_\alpha \colon Y_\alpha \to X \)
be witness of this. Using~\ref{prop:charanalytic-3} of Proposition~\ref{prop:charanalytic} we can e.g.\ assume that \( Y_\alpha = B(\lambda) \) for all \( \alpha < \lambda \). Let \( Y \) be the sum of the spaces \( Y_\alpha \), which is thus a \(\lambda\)-Polish space, and let \( f = \bigcup_{\alpha < \lambda} f_\alpha \): then \( f \colon Y \to X \) is continuous and \( f(Y) = \bigcup_{\alpha < \lambda} A_\alpha \), hence \( \bigcup_{\alpha < \lambda} A_\alpha \in \lS^1_1(X) \). 

As for countable intersections, let \( f_n \colon Y_n \to X \) be witnesses of \( A_n \in \lS^1_1(X) \) with e.g.\ \( Y_n = B(\lambda) \) for all \( n \in \omega \).
 Let
\[ 
Y = \left\{ (x_n)_{n \in \omega} \in \prod\nolimits_{ n \in \omega} Y_n \mid f_n(x_n) = f_m(x_m) \text{ for all } n,m \in \omega \right\}.
 \] 
Since \( \prod_{n \in \omega} Y_n \) is \(\lambda\)-Polish and \( Y \) is closed in it, it is \(\lambda\)-Polish as well and \( f \colon Y \to X \colon (x_n)_{n \in \omega} \mapsto f_0(x_0) \) is a continuous function witnessing \( \bigcap_{n \in \omega} A_n \in \lS^1_1(X) \).

Now let \( A_\alpha \), \( Y_\alpha \) and \( f_\alpha \) be as in the first paragraph, with all the sets \( A_\alpha \) nonempty (this is not restrictive, as otherwise \( \bigcap_{\alpha<\lambda} A_\alpha = \emptyset\) would trivially be \(\lambda\)-analytic). By the equivalent definition~\ref{prop:charanalytic-6} in Proposition~\ref{prop:charanalytic}, if \(2^{< \lambda} = \lambda\) we can assume that \( Y_\alpha = \pre{\lambda}{2} \) for all \( \alpha < \lambda \). Let
\[ 
Y = \{ x \in \pre{\lambda}{2} \mid (f_\alpha \circ \pi_\alpha)(x) = (f_\beta \circ \pi_\beta)(x) \text{ for all } \alpha, \beta < \lambda \},
 \] 
where the maps \( \pi_\alpha \) are defined as in~\eqref{eq:pseudoprojection}.
Since all the functions \( f_\alpha \circ \pi_\alpha \) are continuous, the set \( Y \) is closed, and hence \(\lambda\)-Polish. Moreover, \( f = (f_0 \circ \pi_0) \restriction Y \) is continuous and witnesses \( \bigcap_{\alpha < \lambda} A_\alpha \in \lS^1_1(X) \).

The second part of the proposition can easily be proved as in the classical case (using the first part), see e.g.\ \cite[Proposition 14.4]{Kechris1995}.
\end{proof}

\begin{corollary}\label{cor:closurepropertiesofanalytic}
The boldface \(\lambda\)-pointclass \( \lP^1_1 \) is closed under well-ordered intersections of length at most \(  \lambda \), countable unions, and \( \lambda \)-Borel preimages. If we further assume that \( 2^{< \lambda} =  \lambda \), then \( \lP^1_1 \) is also closed under well-ordered unions of length at most \( \lambda \).

The boldface \(\lambda\)-pointclass \( \lD^1_1 \) is closed under countable unions, countable intersections, complements, and \( \lambda \)-Borel preimages. If \( 2^{< \lambda} = \lambda \), then \( \lD^1_1 \) is also closed under well-ordered unions and intersections of length at most \( \lambda \), i.e.\ each \( \lD^1_1(X) \) is a \( \lambda^+ \)-algebra on \( X \).
\end{corollary}

The normal form for \( \lambda \)-analytic sets \( A \subseteq B(\lambda) \) from equation~\eqref{eq:normalformforanalytic} allows us to prove a higher version of Lusin-Sierpi\'nski theorem~\cite[Theorem 25.16]{Kechris1995} stating that \( A \) can be expressed as both a \( \lambda^+ \)-long union or a \( \lambda^+ \)-long intersection of \(\lambda\)-Borel sets.

\begin{theorem} \label{thm:longunionsintersections}
If \( A \subseteq B(\lambda) \) is \(\lambda\)-analytic, then \( A = \bigcup_{\xi < \lambda^+} A_\xi = \bigcap_{\xi < \lambda^+} B_\xi \) with \( A_\xi, B_\xi \in \lB \), and the same is true for \(\lambda\)-coanalytic sets.
\end{theorem}

\begin{proofsketchcite}{Kechris1995}{Theorem 25.16}
Let \( T \subseteq \pre{<\omega}{(\lambda \times \lambda)} \) be a tree such that \( A = \p[T] \). For \( \xi < \lambda^+ \) and \( s \in \pre{<\omega}{\lambda} \) set
\[ 
C^\xi_s = \{ x \in B(\lambda) \mid \rho_{T(x)}(s) \leq \xi \}.
 \] 
Arguing by induction on \( \xi < \lambda^+ \), one easily sees that each \( C^\xi_s \) is \(\lambda\)-Borel. Indeed, \( C^0_s = \{ x \in B(\lambda) \mid \forall \alpha < \lambda \, (x \restriction (\lh(s)+1), s {}^\smallfrown{} \alpha) \notin T \} \) is closed, and if \( \xi > 0 \) then \( C^\xi_s = \bigcap_{\alpha < \lambda} \bigcup_{\xi' < \xi} C^{\xi'}_{s {}^\smallfrown{} \alpha} \).

Then \( A = \bigcap_{\xi < \lambda^+} B_\xi \) with \( B_\xi = B(\lambda) \setminus C^\xi_\emptyset \). On the other hand, \( A = \bigcup_{\xi  < \lambda^+} A_\xi \) with
\begin{align*}
A_\xi & = \{ x \in B(\lambda) \mid [T(x)] \neq \emptyset \wedge \rho(T(x)) \leq \xi \} \\
& = \{ x \in B(\lambda) \mid \rho_{T(x)}(\emptyset) > \xi \wedge \forall s \in \pre{<\omega}{\lambda} \, (\rho_{T(x)}(s) \neq \xi) \} \\
& = \Big\{  x \in B(\lambda) \mid x \notin C^\xi_\emptyset \wedge \forall s \in \pre{<\omega}{\lambda} \, \Big(x \in C^\xi_s \Rightarrow x \in \bigcup\nolimits_{\xi' < \xi} C^{\xi'}_s\Big) \Big\}. \qedhere 
\end{align*}
\end{proofsketchcite}

As a consequence, using Corollary~\ref{cor:perfect3} we also get information on the cardinality of \(\lambda\)-coanalytic sets (under a small amount of additional choice%
\footnote{The choiceless version of Corollary~\ref{cor:sizeofcoanalytic} reads as follows: Let \(\lambda\) be such that \( 2^{< \lambda} = \lambda \) and \( X \) be a \(\lambda\)-Polish space. For every \( A \in \lP^1_1(X) \), either \( A \)  can be written as a \( \lambda^+ \)-sized union of well-orderable sets of size at most \( \lambda \), or there is a continuous injection from \( \pre{\lambda}{2} \) into \( A \) (and thus \( |A| = 2^\lambda \)).}).

\begin{corollary}[\( \AC_{\lambda^+}(\pre{\lambda}{2}) \)] \label{cor:sizeofcoanalytic} \axioms{\( \AC_{\lambda^+}(\pre{\lambda}{2}) \)}
Let \(\lambda\) be such that \( 2^{< \lambda} = \lambda \). For every \( A \in \lP^1_1(B(\lambda)) \), either \( A \) is well-orderable and \( |A| \leq \lambda^+ \) or there is a continuous injection from \( \pre{\lambda}{2} \) into \( A \) (and thus \( |A| = 2^\lambda \)). 
\end{corollary}

Of course a similar result holds for \(\lambda\)-analytic sets, but for such sets we will actually prove a better result in Corollary~\ref{cor:analyticPSP}.

In view of Theorem~\ref{prop:closurepropertiesofanalytic}\ref{prop:closurepropertiesofanalytic-ii} and Theorem~\ref{thm:Borelisomorphism}, if \(\lambda\) is \(\omega\)-inaccessible then both Theorem~\ref{thm:longunionsintersections} and Corollary~\ref{cor:sizeofcoanalytic} can be extended to arbitrary \(\lambda\)-Polish spaces.

\section{Lusin separation theorem and Souslin's theorem} \label{subsec:Lusinseparation}

Replacing \( \pre{\omega}{\omega} \) and Borel sets with, respectively, \( B(\lambda) \) and \( \lambda \)-Borel sets in the proof of~\cite[Theorem 14.7]{Kechris1995} we get the analogue of the Lusin separation theorem.

\begin{theorem} \label{thm:lusinseparation}
Let \( X \) be \(\lambda\)-Polish and \( A,B \subseteq X \) be disjoint \(\lambda\)-analytic sets. Then there is a \(\lambda\)-Borel set \( C \subseteq X \) separating \( A \) from \( B \), i.e.\ such that \( A \subseteq C \) and \( C \cap B = \emptyset \).
\end{theorem}

\begin{remark} \label{rmk:constructivelusinseparation}
If one follows the ``constructive proof'' of~\cite[Theorem 2E.1]{Moschovakis2009} instead of the proof of~\cite[Theorem 14.7]{Kechris1995}, then one can actually exhibit a somewhat canonical \(\lambda\)-Borel set separating any two disjoint \(\lambda\)-analytic subsets of \( B(\lambda) \). Moreover, this can be extended to subsets of an arbitrary \(\lambda\)-Polish space using Proposition~\ref{prop:surjection} and Remark~\ref{rmk:inverseofinjection} in the obvious way. The possibility of finding a canonical witness to Theorem~\ref{thm:lusinseparation} might be useful in choiceless settings.
\end{remark}

\begin{remark} \label{rmk:lusinseparation}
Since \(\lambda\)-analytic subsets of a \(\lambda\)-Polish space are closed under finite intersections, the notion of \(\lambda\)-analicity can be relativized to \(\lambda\)-analytic spaces. More precisely, if \( Y \in \lS^1_1(X) \) then \( A = A' \cap Y \) for some \( A' \in \lS^1_1(X) \) if and only if \( A \in \lS^1_1(X) \) (and thus \( A \) is a continuous image of a \(\lambda\)-Polish space itself). This allows us to transfer Theorem~\ref{thm:lusinseparation} to the broader setup of \markdef{\(\lambda\)-analytic spaces}, that is, of those spaces which are homeomorphic to a \(\lambda\)-analytic subset of a \(\lambda\)-Polish space. The same applies to most of the results related to separation theorems for \(\lambda\)-analytic sets.
\end{remark}

Since in Proposition~\ref{prop:closurepropertiesofanalytic}\ref{prop:closurepropertiesofanalytic-i} we proved that \( \lS^1_1 \) is closed under well-ordered unions of length at most \( \lambda \), we also get:

\begin{corollary} \label{cor:separationforlambdasets}
Let \( X \) be \(\lambda\)-Polish and \( (A_\alpha)_{\alpha < \lambda} \) be a sequence of pairwise disjoint \(\lambda\)-analytic subsets of \( X \). Then there are pairwise disjoint \(\lambda\)-Borel sets \( B_\alpha \) with \( B_\alpha \supseteq A_\alpha \) for all \( \alpha < \lambda \).
\end{corollary}

Finally, taking \( B = X \setminus A \) in Theorem~\ref{thm:lusinseparation} when \( A \in \lD^1_1(X) \), we get the analogue of Souslin's theorem.

\begin{theorem} \label{thm:souslin}
Assume%
\footnote{Interestingly enough, the condition  \(2^{< \lambda} = \lambda\) is required to prove \( \lB(X) \subseteq \lD^1_1(X) \), while the reverse inclusion \( \lD^1_1(X) \subseteq \lB(X) \) holds unconditionally.} 
that
 \( 2^{< \lambda	} = \lambda \) and let \( X \) be \(\lambda\)-Polish. Then \( \lB(X) = \lD^1_1(X) \).
\end{theorem}

Arguing as in the classical case, we can then apply this result to get a converse to Proposition~\ref{prop:borelvsgraph}.

\begin{theorem}\label{thm:borelvsgraph}
Assume that \( 2^{< \lambda} = \lambda \).
Let \( X, Y \) be \(\lambda\)-Polish spaces, and let \( f \colon X \to Y \). Then the following are equivalent:
\begin{enumerate}[label={\upshape (\arabic*)}, leftmargin=2pc]
\item \label{thm:borelvsgraph-1}
\( f \) is \(\lambda\)-Borel;
\item \label{thm:borelvsgraph-2}
\( \mathrm{graph}(f) \in \lB(X \times Y) \).
\end{enumerate}
In particular, if \( f \) is a \( \lambda \)-Borel bijection, then \( f \) is a \(\lambda\)-Borel isomorphism (i.e.\ \( f^{-1} \) is \(\lambda\)-Borel as well).

If we assume \( 2^{< \lambda} = \lambda \), then we can add
\begin{enumerate}[label={\upshape (\arabic*)}, leftmargin=2pc,resume]
\item \label{thm:borelvsgraph-3}
\( \mathrm{graph}(f) \in \lS^1_1(X \times Y) \).
\end{enumerate}
to the above list of equivalent conditions.
\end{theorem}

\begin{remark} \label{rmk:borelvsgraph}
It is easy to see that Theorem~\ref{thm:borelvsgraph} holds also for functions \( f \colon A \to B \) with \( A \in \lB(X) \) and \( B \in \lB(Y) \).
\end{remark}

Further consequence that can be proved \emph{mutatis mutandis} as in the classical case \( \lambda = \omega \) are the following ones.

\begin{theorem} \label{thm:injectiveBorelimage}
Let \( X,Y \) be \(\lambda\)-Polish spaces and \( f \colon X \to Y \) be \(\lambda\)-Borel. If \( A \subseteq X \) is \(\lambda\)-Borel and \( f \restriction A \) is injective, then \( f(A) \) is \(\lambda\)-Borel and \( f \restriction A \) is a \(\lambda\)-Borel isomorphism between \( A \) and \( f(A) \). 

\end{theorem}

Together with Corollary~\ref{cor:surjectionBorel}, this gives:

\begin{corollary} \label{cor:Borelvsinjimages}
Assume that \( 2^{< \lambda} = \lambda \) and
let \( X \) be \(\lambda\)-Polish. A set \( A \subseteq X \) is \(\lambda\)-Borel if and only if it is a continuous (equivalently, \(\lambda\)-Borel) injective image of a closed subset of \( B(\lambda) \).
\end{corollary}

Using again Proposition~\ref{prop:surjection} and Remark~\ref{rmk:inverseofinjection}, Corollary~\ref{cor:perfect1}, and the \(\lambda\)-Borel version of the usual Cantor-Schr\"oder-Bernstein argument granted by Theorem~\ref{thm:injectiveBorelimage}, we finally get the following isomorphism theorem.

\begin{theorem} \label{thm:Borelisomorphism}
Assume that \(\lambda\) is \(\omega\)-inaccessible. If \( X \) is a \(\lambda\)-Polish space with \( |X| > \lambda \), then \( X \) is \(\lambda\)-Borel isomorphic to \( B(\lambda) \). Moreover, if \( X,Y \) are \(\lambda\)-Polish spaces, then \( X \) and \( Y \) are \(\lambda\)-Borel isomorphic if and only if \( |X| = |Y| \). 
\end{theorem}

Together with the fact that by Corollary~\ref{cor:perfect1} either \( X \) is well-orderable and \( |X| \leq \lambda \), or else \( |X| = |B(\lambda)| = \lambda^\omega \), this completely classifies \(\lambda\)-Polish spaces up to \(\lambda\)-Borel isomorphism when \(\lambda\) is \(\omega\)-inaccessible. Moreover, it allows us to rewrite Corollary~\ref{cor:Borelvsinjimages} as follows.

\begin{corollary}
Assume that \( 2^{< \lambda} = \lambda \) and
let \( X \) be \(\lambda\)-Polish. A set \( A \subseteq X \) is \(\lambda\)-Borel if and only if either \( |A| \leq \lambda \) or else \( A \) is a \(\lambda\)-Borel injective image of the whole \( B(\lambda) \). 

\end{corollary}

\section{Standard \(\lambda\)-Borel spaces} \label{sec:standardBorel}

\subsection{\(\lambda\)-measurable spaces and functions}

A \markdef{\(\lambda\)-measurable space} is a pair \( X = (X, \mathcal{B}) \) with \( \mathcal{B} \)  a \( \lambda^+ \)-algebra on \( X \). 
A function \( f \) between two \( \lambda \)-measurable spaces \( (X,\mathcal{B}) \) and \( ( Y , \mathcal{C}) \) is \markdef{\( \lambda \)-measurable} if \( f^{-1}(C) \in \mathcal{B} \) for all \( C \in \mathcal{C} \).
Two \(\lambda\)-measurable spaces \( (X, \mathcal{B}) \) and \( (Y, \mathcal{C}) \) are \markdef{isomorphic} if there is a bijection \( f \colon X \to Y \) such that both \( f \) and \( f^{-1} \) are \(\lambda\)-measurable, that is,  \( B \in \mathcal{B} \) if and only if \( f(B) \in \mathcal{C} \) for all \( B \subseteq X \).

\begin{defin}
A \(\lambda\)-measurable space \( X = (X, \mathcal{B}) \) is called \markdef{\(\lambda\)-Borel space} if \( \mathcal{B} \) is \(  \lambda \)-generated and separates points, i.e.\ for all distinct \( x,y \in X \) there is \( B \in \mathcal{B} \) such that \( x \in B \) while \( y \notin B \).
\end{defin}

If \( (X,\mathcal{B}) \) is a \(\lambda\)-Borel space, any \( Y \subseteq X \) can be canonically equipped with the \( \lambda^+ \)-algebra \( \mathcal{B} \restriction Y  = \{ B \cap Y \mid B \in \mathcal{B} \} \), which turns it into a \(\lambda\)-Borel space as well. 
Moreover,
every metrizable space \( X \) of weight at most \( \lambda\) can be canonically turned into a \(\lambda\)-Borel space by setting \( \mathcal{B} = \lB(X) \). For this reason, if \( (X, \mathcal{B}) \) is an arbitrary \(\lambda\)-Borel space, the elements of \( \mathcal{B} \) are called \( \lambda \)-Borel sets of \( X \). Similarly, \(\lambda\)-measurable functions between \(\lambda\)-Borel spaces are called \( \lambda \)-Borel (measurable). The notions of \(\lambda\)-Borel embedding and \(\lambda\)-Borel isomorphism are defined accordingly. Notice that all the terminology and notation is consistent with the one introduced at the beginning of Section~\ref{sec:Borel}.

Naturally adapting the proof of~\cite[Proposition 12.1]{Kechris1995} and recalling that if \( 2^{< \lambda} = \lambda\)  then the product topology and the bounded topology on \( \pre{\lambda}{2} \) generate the same \(\lambda\)-Borel sets, one easily obtain the following characterization of \(\lambda\)-Borel spaces.

\begin{proposition} \label{prop:Borelspace}
Assume that \( 2^{< \lambda} = \lambda\) and let \( (X,\mathcal{B}) \) be a \( \lambda \)-measurable space. The following are equivalent:
\begin{enumerate-(1)}
\item \label{prop:Borelspace-1}
\( (X, \mathcal{B}) \) is a \(\lambda\)-Borel space;
\item
\( \mathcal{B} = \lB(X,\tau) \) for some metrizable topology \(\tau\) on \( X \) of weight at most \( \lambda \);
\item
\( (X, \mathcal{B}) \) is isomorphic to \( (Z, \lB(Z)) \) for some \( Z \subseteq Y \) with \( Y \) a \(\lambda\)-Polish space;
\item \label{prop:Borelspace-4}
\( (X, \mathcal{B}) \) is isomorphic to \( (Z, \lB(Z)) \) for some \( Z \subseteq \pre{\lambda}{2} \).
\item \label{prop:Borelspace-5}
\( (X, \mathcal{B}) \) is isomorphic to \( (Z, \lB(Z)) \) for some subspace \( Z \) of the topological space \( (\pre{\lambda}{2}, \tau_p ) \), where \( \tau_p \) is the product topology.
\end{enumerate-(1)}
\end{proposition}

Similarly, adapting the classical argument one obtains the following analogue of~\cite[Theorem 12.2]{Kechris1995}, originally due to Kuratowski.

\begin{theorem}
Let \( X \) be a \(\lambda\)-measurable space and \( Y \) be \(\lambda\)-Polish. If \( Z \subseteq X \) and \( f \colon Z \to Y \) is \(\lambda\)-measurable, then there is a 
\(\lambda\)-measurable function \( \hat{f} \colon X \to Y \) extending \( f \).
\end{theorem}

It follows that any partial \(\lambda\)-Borel function between \(\lambda\)-Polish spaces can be extended to a total \(\lambda\)-Borel function.

\begin{defin}
A \(\lambda\)-measurable space \( X = (X,\mathcal{B}) \) is a \markdef{standard \(\lambda\)-Borel space} if there is a \(\lambda\)-Polish topology \( \tau \) on \( X \) such that \( \mathcal{B} = \lB(X,\tau) \).
\end{defin}

With an abuse of terminology, we will also say that a topological space \( (X,\tau) \) is (standard) \(\lambda\)-Borel if so is the \(\lambda\)-measurable space \( (X, \lB(X,\tau)) \), even if \( \tau \) itself was not \(\lambda\)-Polish. 

\begin{example} \label{xmp:lambda^lambdawithproduct}
Assume that \( 2^{< \lambda} = \lambda \) and equip \( \pre{\lambda}{\lambda} \) with the \emph{product} topology \( \tau_p \). The space \( (\pre{\lambda}{\lambda}, \tau_p) \) is not \(\lambda\)-Polish because it is not even first-countable (hence neither metrizable), yet its weight is \( \lambda \) and, being Hausdorff, \( \tau_p \) separates points. It follows that \( (\pre{\lambda}{\lambda}, \mathcal{B}) \), where \( \mathcal{B} = \lB(\pre{\lambda}{\lambda}, \tau_p) \), is a \(\lambda\)-Borel space: we claim that it is standard. Indeed, equip \( \pre{\lambda \times \lambda}{2} \) with the product topology. Then the map \( \pre{\lambda}{\lambda} \to \pre{\lambda \times \lambda}{2} \) sending \( z \in \pre{\lambda}{\lambda} \) to the characteristic function of its graph is a topological embedding, and one can easily compute that its range is \(\lambda\)-Borel. Identifying \( \pre{\lambda \times \lambda}{2} \) with \( \pre{\lambda}{2} \) in the obvious way,
 we get that \( (\pre{\lambda}{\lambda}, \mathcal{B}) \) is \(\lambda\)-Borel isomorphic to a \(\lambda\)-Borel subset of 
 \( (\pre{\lambda}{2},\tau_p) \), thus it is standard \(\lambda\)-Borel because the product and the bounded topology on \( \pre{\lambda}{2} \) give rise to the same \(\lambda\)-Borel sets under \( 2^{< \lambda} = \lambda \), 
 and each \( \lambda \)-Borel subset of \( \pre{\lambda}{2} \) can be equipped with a \(\lambda\)-Polish topology generating its \(\lambda\)-Borel structure by Corollary~\ref{cor:changeoftopologyclopen}.
 (More generally, see condition~\ref{prop:charstandardBorel-5} in the ensuing Proposition~\ref{prop:charstandardBorel}).
 
 In contrast, if we equip \( \pre{\lambda}{\lambda} \) with the bounded topology \( \tau_b \), we get a space which is completely metrizable but has weight \(  \lambda^{\cf(\lambda)} > \lambda \), hence it is not \(\lambda\)-Polish again. But this time one can actually check that \( (\pre{\lambda}{\lambda}, \tau_b) \) is not even (standard) \(\lambda\)-Borel. Indeed, any \(\lambda\)-generated \(\lambda ^+ \)-algebra has size at most \( 2^\lambda \), while there are \( 2^{(\lambda^{\cf(\lambda)})} = 2^{2^\lambda}  \) distinct \( \tau_b \)-open sets, hence \( |\lB(\pre{\lambda}{\lambda}, \tau_b)| = 2^{2^\lambda} > 2^\lambda \). 
\end{example}

The only nontrivial implications in the next result are \ref{prop:charstandardBorel-3}~\( \Rightarrow \)~\ref{prop:charstandardBorel-2} and \ref{prop:charstandardBorel-1}~\( \Rightarrow \)~\ref{prop:charstandardBorel-4}: the former follows from Corollary~\ref{cor:changeoftopologyclopen}, while the latter follows from Proposition~\ref{prop:Borelspace} and Theorem~\ref{thm:injectiveBorelimage}.

\begin{proposition} \label{prop:charstandardBorel}
Let \(\lambda\) be such that \( 2^{< \lambda} = \lambda \), and let \( X \) be a \(\lambda\)-Borel space. The following are equivalent:
\begin{enumerate-(1)}
\item \label{prop:charstandardBorel-1}
\( X \) is a standard \(\lambda\)-Borel space;
\item \label{prop:charstandardBorel-2}
\( X \) is \(\lambda\)-Borel isomorphic to some $\lambda$-Polish space \( Y \) (we can additionally require \( \mathrm{dim}(Y) = 0 \));
\item \label{prop:charstandardBorel-3}
\( X \) is \(\lambda\)-Borel isomorphic to a \(\lambda\)-Borel subset of some \(\lambda\)-Polish space \( Y \)  (with \( \mathrm{dim}(Y) = 0 \), if needed);
\item \label{prop:charstandardBorel-4}
\( X \) is \(\lambda\)-Borel isomorphic to a \(\lambda\)-Borel subset of \( \pre{\lambda}{2} \);
\item \label{prop:charstandardBorel-5}
\( X \) is \(\lambda\)-Borel isomorphic to a \(\lambda\)-Borel subset of \( (\pre{\lambda}{2}, \tau_p) \).
\end{enumerate-(1)}
\end{proposition}

Propositions~\ref{prop:Borelspace} and~\ref{prop:charstandardBorel} in particular imply that (standard) \(\lambda\)-Borel spaces are closed under (\(\lambda\)-Borel) subspaces,
but one can show that they are also closed under products of size \(\lambda\). Recall that the product of \( \lambda \)-many \(\lambda\)-measurable spaces \( (X_i, \mathcal{B}_i)\) is the \(\lambda\)-measurable space \( \big(  \prod_{i < \lambda} X_i, \mathcal{B}_\Pi \big) \), where \( \mathcal{B}_\Pi \) is the \(\lambda^+ \)-algebra on the Cartesian product \( \prod_{i < \lambda} X_i \) generated by the direct product%
\footnote{Recall that the direct product of the \( \lambda^+ \)-algebras \( \mathcal{B}_i \) is the collection of the sets of the form \( \prod_{i < \lambda} A_i \subseteq \prod_{i < \lambda} X_i \) where \( A_i \in \mathcal{B}_i \) for all \( i < \lambda \). It is a \( \lambda^+ \)-complete Boolean algebra if equipped with the \emph{pointwise operations}, but it is not closed under unions, intersections, and complements relatively to the ambient space \( \prod_{i < \lambda} X_i \).}
of the \( \lambda^+ \)-algebras \( \mathcal{B}_i \). Notice that if each \( \mathcal{B}_i \) is generated by a corresponding topology \( \tau_i \) on \( X \), then the resulting \( \mathcal{B}_\Pi \) on \( X \) coincides with the \( \lambda^+ \)-algebra generated by some/any product topology%
\footnote{Here we can equivalently take any \( \mu \leq \lambda \) and consider the \(  < \mu \)-supported product topology. Even considering the box topology would give rise to the same \(\lambda^+\)-algebra.}
 of the \( \tau_i \); notice however that even if all the \( \tau_i \) are \(\lambda\)-Polish, such a product needs not to be \(\lambda\)-Polish.

\begin{proposition} \label{prop:productofstandardBorelspaces}
Assume \( 2^{< \lambda} = \lambda \). Suppose that  \( (X_i, \mathcal{B}_i) \) are (standard) \(\lambda\)-Borel spaces for \( i <\lambda \).  Then the product \(\lambda\)-measurable space \( \big( \prod_{i < \lambda} X_i, \mathcal{B}_\Pi \big) \) is (standard) \(\lambda\)-Borel.
\end{proposition}

\begin{proof}
By~\ref{prop:Borelspace-5} of Proposition~\ref{prop:Borelspace}, each \( (X_i, \mathcal{B}_i) \) is \(\lambda\)-Borel isomorphic to a subspace \( A_i \) of \( (\pre{\lambda}{2}, \tau_p ) \) via some \( f_i \). Thus the product \( \prod_{i < \lambda} f_i \) of such maps witness that \( \big( \prod_{i < \lambda} X_i, \mathcal{B}_\Pi \big) \) is \(\lambda\)-Borel isomorphic to \( \prod_{i < \lambda} A_i  \subseteq \prod_{i < \lambda} \pre{\lambda}{2}  = \pre{\lambda}{(\pre{\lambda}{2})}\), where the latter is equipped with the induced product \(\lambda\)-Borel structure. However, such \( \lambda^+ \)-algebra is induced by the \( < \omega \)-supported product of \( \tau_p \) on each copy of \( \pre{\lambda}{2} \), which turns \( \pre{\lambda}{(\pre{\lambda}{2})} \) into a topological space homeomorphic to \( (\pre{\lambda}{2}, \tau_p ) \). Thus \( \big( \prod_{i < \lambda} X_i, \mathcal{B}_\Pi \big) \) is \(\lambda\)-Borel isomorphic to a  subspace of \( (\pre{\lambda}{2}, \tau_p ) \), and hence we are done by~\ref{prop:Borelspace-5} of Proposition~\ref{prop:Borelspace}.

For the case of standard \(\lambda\)-Borel spaces, it is enough to use Proposition~\ref{prop:charstandardBorel} instead of  Proposition~\ref{prop:Borelspace} and notice that if all the sets \( A_i \) were  in \( \lB(\pre{\lambda}{2}, \tau_p) \), then \( \prod_{i < \lambda} A_i \) is by definition in the product of \(\lambda\)-many copies of \( \lB(\pre{\lambda}{2}, \tau_p) \).
\end{proof}

By Proposition~\ref{prop:charstandardBorel}, if \( 2^{< \lambda} = \lambda\) then Corollary~\ref{cor:perfect1} and Theorem~\ref{thm:Borelisomorphism} extend to standard \(\lambda\)-Borel spaces as well (including \(\lambda\)-Borel subsets of \(\lambda\)-Polish spaces). 

\begin{theorem} \label{thm:Borelisomorphism2}
Assume that \(2^{< \lambda} = \lambda\). If \( X \) is a standard \(\lambda\)-Borel space,
then either \( X \) is well-orderable and \( |X| \leq \lambda \), or else
\( X \) is \(\lambda\)-Borel isomorphic to \( B(\lambda) \) (and thus has size \( \lambda^\omega = 2^\lambda \)). Moreover, if \( X,Y \) are standard \(\lambda\)-Borel spaces, then \( X \) and \( Y \) are \(\lambda\)-Borel isomorphic if and only if \( |X| = |Y| \). 

\end{theorem}

In particular we get:

\begin{corollary}
Assume that \(2^{< \lambda} = \lambda\) and let  \( X \) be a \(\lambda\)-Polish space. If  \( B  \subseteq X \) is \(\lambda\)-Borel, then \( |B| > \lambda \) if and only if \( B \) is \(\lambda\)-Borel isomorphic to \( B(\lambda) \).
\end{corollary}

This result is similar in spirit to~\cite[Theorem 13]{Stone1962}, but it is different: on the one hand we have the additional assumption \( 2^{< \lambda} = \lambda > \cf(\lambda) = \omega \) and we reach the weaker conclusion that the spaces be \(\lambda\)-Borel isomorphic (rather than (\( \omega_1 \)-)Borel isomorphic); but on the other hand we can apply it to arbitrary \(\lambda\)-Borel sets and not just (classical) Borel sets.

Proposition~\ref{prop:charstandardBorel} allows us to transfer many notion and results from the realm of \(\lambda\)-Polish spaces to that of standard \(\lambda\)-Borel spaces in a natural way. For example a subset of a standard \(\lambda\)-Borel space \(  (X,\mathcal{B}) \) is called \markdef{\(\lambda\)-analytic} if and only if it is \(\lambda\)-analytic with respect to some \(\lambda\)-Polish topology \(\tau\) on \( X \) for which \( \lB(X,\tau) =  \mathcal{B} \): when \( 2^{< \lambda} = \lambda \), the definition depends only on the \(\lambda\)-Borel structure of \( X \) and not on the chosen \(\tau\). 
Equivalently, a nonempty \( A \subseteq X \) is \(\lambda\)-analytic if and only if it is the image of \( B(\lambda) \) through a \(\lambda\)-Borel measurable function.
It readily follows that all separation theorems seen so far, together with their consequences, hold in this broader context as well. For example: 
\begin{itemizenew}
\item
a subset of a standard \(\lambda\)-Borel set is \(\lambda\)-Borel if and only if it is \( \lambda \)-bianalytic,
\item
 a function between standard \(\lambda\)-Borel spaces \( X \) and \( Y \) is \(\lambda\)-Borel if and only it its graph is \(\lambda\)-Borel in \( X \times Y \),
 \end{itemizenew} 
 and so on.

\subsection{The Effros \(\lambda\)-Borel space}

In the classical setting, a widely considered standard Borel space is the so-called Effros Borel space. This can be straightforwardly generalized as follows. Given a topological space, endow the collection \( F(X) = \lP^0_1(X) \) of all closed subsets of \( X \) with the 
\(\lambda\)-algebra \( \mathcal{B}_F(X) \) generated by the sets of the form
\begin{equation} \label{eq:EffrosBorel}
B_U(X) = \{ F \in F(X) \mid F \cap U \neq \emptyset \},
 \end{equation}
where \( U \) varies over open subsets of \( X \). Moreover, if \( X \) has weight at most \(\lambda\), then in~\eqref{eq:EffrosBorel} it is enough to consider the open sets \( U \) in a fixed \(\lambda\)-sized basis. This shows that under such hypothesis \( \mathcal{B}_F(X) \) is \(\lambda\)-generated. Moreover, it separate points: indeed, if \( F, F' \in F(X) \) are distinct with \( F \nsubseteq F' \), then \( B_{X \setminus F'}(X) \) contains \( F \) but not \( F' \). This shows that \( (F(X), \mathcal{B}_F(X)) \) is a \(\lambda\)-Borel space, that in analogy with the classical case, will be called \markdef{Effros \(\lambda\)-Borel space} (of \( X \)) and denoted by \( F(X) \).

The next result is the analogue of~\cite[Theorem 12.6]{Kechris1995}. However, the proof is necessarily different, as spaces \( X \) such that \( |X| \nleq 2^{\aleph_0} \) cannot have a (metrizable) compactification. 
We provide an alternative argument, which incidentally works also in the classical case 
\( \lambda = \omega \).

\begin{theorem} \label{thm:EffrosBorel}
Assume that \(2^{< \lambda} = \lambda\). If \( X \) is a \(\lambda\)-Polish space, then its Effros \(\lambda\)-Borel space \(  F(X) \) is standard.
\end{theorem}

\begin{proof}
We first prove the result when \( X = B(\lambda) \), following~\cite[Exercise 12.10]{Kechris1995}. Recall from~\cite[Proposition 2.4]{Kechris1995} that the map  \( F \mapsto T_F = \{ x \restriction n \mid x \in F \wedge n \in \omega \}  \) is a bijection between \( F(B(\lambda)) \) and pruned \(\omega\)-trees \(T \subseteq \pre{<\omega}{\lambda} \).
Recall also the \(\lambda\)-Polish space \( \mathrm{Tr}_\lambda \) of (codes for) \(\omega\)-trees on \(\lambda\) from equation~\eqref{eq:Trlambda}, and the bijection \( \sigma_{\mathrm{Tr}} \colon \pre{<\omega}{\lambda} \to \lambda \) from equation~\eqref{eq:sigmatr} used to define it.
 Consider the subspace \( \mathrm{PTr}_\lambda \) of (codes for) pruned \(\omega\)-trees: it is easy to see that \( \mathrm{PTr}_\lambda \in \lP^0_2(\mathrm{Tr}_\lambda) \), so it is a \(\lambda\)-Borel subset of \( \pre{\lambda}{2} \). We claim that the map \( F \mapsto T_F \) is a \(\lambda\)-Borel isomorphism between \( F(B(\lambda)) \) and \( \mathrm{PTr}_\lambda \), thus witnessing that \( F(B(\lambda)) \) is standard by conditions~\ref{prop:charstandardBorel-4} of Proposition~\ref{prop:charstandardBorel}. It is clear that for every \( s \in \pre{<\omega}{\lambda} \),
\[ 
F \in B_{\Nbhd_s}(B(\lambda)) \iff s \in T_F \iff x_{T_F}(\sigma_{\mathrm{Tr}}(s)) = 1.
 \] 
Since the sets of the form \( \{ x \in \pre{\lambda}{2} \mid x(\alpha) = i \} \) with \( \alpha < \lambda \) and \( i \in \{ 0,1 \} \) constitute a subbasis for the product topology on \( \pre{\lambda}{2} \), and since the product and the bounded topology on \( \pre{\lambda}{2} \) generate the same \(\lambda\)-Borel structure because we assumed  \(2^{< \lambda} = \lambda\), it easily follows that \( F \mapsto T_F \) is a \(\lambda\)-Borel isomorphism.

We now move to an arbitrary \(\lambda\)-Polish space \( X \), which can be assumed to be nonempty because otherwise \( F(X) = \{ \emptyset \} \) would trivially be \(\lambda\)-Polish.  By Corollary~\ref{cor:surjection2}, there is a continuous open surjection \( g \colon B(\lambda) \to X \). We claim that the map
\begin{equation} \label{eq:mapG}
G \colon F(X) \to F(B(\lambda)), \quad F \mapsto g^{-1}(F)
 \end{equation}
is a \(\lambda\)-Borel isomorphism between \( F(X) \) and a \(\lambda\)-Borel subset of \( F(B(\lambda)) \), which implies that \( F(X) \) is standard by Proposition~\ref{prop:charstandardBorel} again.

First observe that \( G \) is injective: if \( F_0, F_1 \in F(X) \) are distinct and \( x \in F_i \setminus F_{1 - i } \), then every \( y \in g^{-1}(x) \), which exists because \( g \) is surjective, will belong to \( G(F_i) \) but not to \( G(F_{1-i}) \), witnessing that they are different. 

We next show that the range of \( G \) is \(\lambda\)-Borel.
%
Given \( s \in \pre{< \omega}{\lambda} \), let \( V_s = g^{-1}(g(\Nbhd_s)) \), which is open by the choice of \( g \). We claim that the range of \( G \) consists of those \( F' \in F(B(\lambda)) \) such that
\begin{equation} \label{eq:defrangeG}
\forall s \in \pre{< \omega}{\lambda} \, (F' \cap \Nbhd_s \neq \emptyset \iff F' \cap V_s \neq \emptyset ),
 \end{equation}
which proves that it is \(\lambda \)-Borel with respect to the Effros \(\lambda\)-Borel structure of \( B(\lambda) \). 
Let \( F' \) be in the range of \( G \), i.e.\ \( F' = g^{-1}(F) \) for some \( F \in F(X) \). Then by surjectivity of \( g \)  we have, for every \( s \in \pre{<\omega}{\lambda} \):
\[ 
g^{-1}(F) \cap \Nbhd_s \neq \emptyset \iff F \cap g(\Nbhd_s) \neq \emptyset \iff g^{-1}(F) \cap V_s \neq \emptyset.
 \] 
 Thus equation~\eqref{eq:defrangeG} holds for such an \( F' \).
Conversely, assume that \( F' \) is not in the range of \( G \). Since \( g \) is open and surjective, 
if \( g^{-1}(A) \) is closed than so is \( A \), because \( X \setminus A = g(B(\lambda) \setminus g^{-1}(A)) \). Applying this with \( A = g(F') \), one easily gets
 that \( F'' = g^{-1}(g(F')) \supsetneq F' \) --- the inclusion must be proper because otherwise \( g(F') \) would be closed and \( F =  F'' = G(g(F')) \) would be in the range of \( G \), against the choice of \( F' \). Pick \(s \in \pre{<\omega}{\lambda} \) such that \( F'' \cap \Nbhd_s \neq \emptyset \) but \( F' \cap \Nbhd_s = \emptyset \). By definition of \( V_s \), we have \( F' \cap V_s \neq \emptyset \) because
\[ 
F' \cap V_s \neq \emptyset \iff g(F') \cap g(\Nbhd_s ) \neq \emptyset \iff F'' \cap \Nbhd_s \neq \emptyset.
 \] 
Hence \( s \) witnesses that~\eqref{eq:defrangeG} is not satisfied by \( F' \), as desired.

Finally, we address \(\lambda\)-measurability of the map \( G \) and of its inverse, starting from the latter. If \( U \subseteq X \) is open, then by surjectivity of \( g \) we have that for every \( F \in F(X) \)
\[ 
F \cap U \neq \emptyset \iff g^{-1}(F) \cap g^{-1}(U) \neq \emptyset, 
 \] 
so that the \( G \)-image of the basic \(\lambda\)-Borel set \( B_U(X) \) is the (trace on the range of \( G \) of the) basic \(\lambda\)-Borel set \( B_{g^{-1}(U)}(B(\lambda)) \). Conversely, fix any \( s \in \pre{<\omega}{\lambda} \) and \( F \in F(X) \). Since the set \( g^{-1}(F) \), being in the range of \( G \), satisfies~\eqref{eq:defrangeG}, by surjectivity of \( g \) we have that
\[ 
g^{-1}(F) \cap \Nbhd_s \neq \emptyset \iff g^{-1}(F) \cap V_s \neq \emptyset \iff F \cap g(\Nbhd_s) \neq \emptyset.
 \] 
Since \( g(\Nbhd_s) \) is open in \( X \), this shows that the \( G \)-preimage of \( B_{\Nbhd_s}(B(\lambda))  \) is \( B_{g(\Nbhd_s)}(X) \).
Thus \( G \) and \( G^{-1} \) are both \(\lambda\)-Borel measurable, and \( G \) is a \(\lambda \)-Borel isomorphism between \( F(X) \) and the range of \( G \), which is a \(\lambda\)-Borel subset of \( F(B(\lambda)) \). 
\end{proof}

\begin{example}
Consider again a universal \(\lambda\)-Polish space \( \mathfrak{U}_\lambda \) (see the end of Section~\ref{sec:def-example-Polish}). The space \( F(\mathfrak{U}_\lambda) \) can be equipped with its Effros \(\lambda\)-Borel structure, which is standard by Theorem~\ref{thm:EffrosBorel}; thus \( F(\mathfrak{U}_\lambda) \) can be construed as the standard \(\lambda\)-Borel space%
\footnote{A different way of forming a space of (codes for) \(\lambda\)-Polish spaces was presented in~\cite{AM,MottoRos2017h}.}
 of all \(\lambda\)-Polish spaces. Similarly, by Proposition~\ref{prop:dim(X)=0} the Effros \(\lambda\)-Borel space \( F(B(\lambda)) \) can be viewed as the standard \(\lambda\)-Borel space of all \(\lambda\)-Polish spaces with \( \mathrm{dim}(X) = 0 \).
\end{example}

Theorem~\ref{thm:EffrosBorel} unlocks a number of interesting and useful facts, whose proof goes essentially unchanged. For example, 
 we have the following analogue of the Kuratowski-Ryll-Nardzewski's Selection Theorem for \( F(X) \).

\begin{theorem} \label{thm:selectionforF(X)}
Let \( \lambda \) be such that \( 2^{< \lambda} = \lambda \), and  let \( X \) be a \(\lambda\)-Polish space. Then there is a sequence \( (\sigma^X_\alpha)_{\alpha < \lambda} \) of \(\lambda\)-Borel functions \( \sigma^X_\alpha \colon F(X) \to X \) such that for every nonempty \( F \in F(X) \) the set \( \{ \sigma^X_\alpha(F) \mid \alpha < \lambda \} \) is dense in \( F \).
\end{theorem}

We present a slight variation of the classical proof which relies on the constructions carried out in the proof of Theorem~\ref{thm:EffrosBorel}.

\begin{proofcompare}{Kechris1995}{Theorem 12.13}
Consider first the case \( X = B(\lambda) \). Fix any \( y_0 \in B(\lambda) \), and consider the map \( \sigma \colon F(B(\lambda)) \to B(\lambda) \) defined by \( \sigma(\emptyset) = y_0 \) and 
\[ 
\sigma(F) = \text{ the leftmost branch of } T_F 
\] 
if \( F \neq \emptyset \). It is easy to see that \( \sigma \) is \(\lambda\)-Borel, and that \( \sigma(F) \in F \) when \( F \neq \emptyset \). Next fix a bijection \( f \colon \pre{<\omega}{\lambda} \to \lambda \). For each \( \alpha < \lambda \), the map from \( F(X) \) into itself sending \( F \) to \( F \cap \Nbhd_{f^{-1}(\alpha)} \) is \(\lambda\)-Borel because \( \Nbhd_{f^{-1}(\alpha)} \) is clopen. The maps \( \sigma^{B(\lambda)}_\alpha \colon F(B(\lambda)) \to B(\lambda) \) can then be defined by setting \( \sigma^{B(\lambda)}_\alpha (F) = \sigma (F \cap \Nbhd_{f^{-1}(\alpha)}) \) if \( F \cap \Nbhd_{f^{-1}(\alpha)} \neq \emptyset \), and \( \sigma^{B(\lambda)}_\alpha (F) = \sigma(F) \) otherwise.

We now move to an arbitrary \(\lambda\)-Polish space \( X \). Without loss of generality we may assume \( X \neq \emptyset \). Let \( G \colon F(X) \to F(B(\lambda)) \) be the map defined in~\eqref{eq:mapG}, where \( g \colon B(\lambda) \to X \) is an open continuous surjection and \( G(F) = g^{-1}(F) \) for every \( F \in F(X) \). Then each \( \sigma^X_\alpha = g \circ \sigma^{B(\lambda)}_\alpha \circ G \) is clearly \(\lambda\)-Borel. Fix \( \emptyset \neq F \in F(X) \) and \( U \subseteq X \) open such that \( F \cap U \neq \emptyset \). Then \( G(F) \cap g^{-1}(U) \neq \emptyset \) as well because \( g \) is surjective. Let \( \alpha < \lambda \) be such that \( \sigma_\alpha^{B(\lambda)}(G(F)) \in G(F) \cap g^{-1}(U) \):
then \( \sigma^X_\alpha(F) \in F \cap U \). This shows that \( \{ \sigma^X_\alpha(F) \mid \alpha < \lambda \} \) is dense in \( F \), hence we are done.
\end{proofcompare}

Recall that given an equivalence relation \( E \) on \( X \), a \markdef{selector} for \( E \) is a map \( s \colon X \to X \) such that \( s(x) \mathrel{E} x \) and \( s(x) = s(y) \) whenever \( x \mathrel{E} y \), while a \markdef{transversal} for \( E \) is a set \( T \subseteq X \) that meets every equivalence class in exactly one point. Notice that if \( s \) is a selector for \( E \), then \( T = \{ x \in X \mid s(x) = x \} \) is a transversal for \( X \); conversely, if \( T \) is a transversal for \( E \) then the map \( s \) sending \( x \) to the unique \( y \in T \) with \( y \mathrel{E} x \) is a selector for \( X \). These observations work in the definable context as well. By the previous arguments, if \( X \) is  a standard \(\lambda\)-Borel space then every \( \lambda \)-Borel sector \( s \) for \( E \) yields a \(\lambda\)-Borel transversal \( T \) for \( E \), while if \( E \) is \(\lambda\)-Borel (as a subset of \( X^2 \)) then every \(\lambda\)-Borel transversal \( T \) for \( E \) yields the \(\lambda\)-Borel selector \( s \) for \( E \) defined by setting \( s(x) = y \) if and only if \( x \mathrel{E} y \) and \( y \in T \).

For a set \( A \subseteq X \), the (\markdef{\( E \)-})\markdef{saturation} \( [A]_E \) of \( A \) is defined by
\[ 
[A]_E = \{ x \in X \mid \exists y \in A \, (x \mathrel{E} y) \}.
 \] 
The following is a direct consequence of Theorem~\ref{thm:selectionforF(X)} and is obtained by adapting the proof of~\cite[Theorem 12.16]{Kechris1995} in the obvious way.

\begin{theorem}  \label{thm:selector}
Let \(\lambda\) be such that \( 2^{< \lambda} = \lambda \), \( X \) be a \(\lambda\)-Polish space, and \( E \) be an equivalence relation on \( X \) such that every equivalence class is closed and the saturation of any open set is \(\lambda\)-Borel. Then \( E \) admits a \(\lambda\)-Borel selector, and thus also a Borel transversal.
\end{theorem}

\begin{proofsketchcite}{Kechris1995}{Theorem 12.16}
By our assumptions, the map \( X \to F(X) \) sending \( x \) to \( [x]_E \) is well-defined and \(\lambda\)-Borel, hence the map \( s \colon X \to X \) defined by \( s(x) = \sigma_0^X([x]_E) \), where \( \sigma_0^X \) is as in Theorem~\ref{thm:selectionforF(X)}, is a \(\lambda\)-Borel selector. 
\end{proofsketchcite}

The condition that the saturation of open sets is \(\lambda\)-Borel can easily be replaced by the condition that the saturation of closed sets is \(\lambda\)-Borel. An important special case of 
Theorem~\ref{thm:selector} is the following (see~\cite[Theorem 12.17]{Kechris1995}). A \markdef{\(\lambda\)-Polish group} is a topological group whose topology is \(\lambda\)-Polish.

\begin{theorem} \label{thm:cosets}
Let \(\lambda\) be such that \( 2^{< \lambda} = \lambda \), \( G \) be a \(\lambda\)-Polish group, and \( H \subseteq G \) be a closed subgroup. Then there is a \(\lambda\)-Borel selector for the equivalence relation whose classes are the (left) cosets of \( H \). In particular, there is a \(\lambda\)-Borel set meeting every (left) coset in exactly one point.
\end{theorem}

\begin{proof}
Theorem~\ref{thm:selector} can be applied because every (left) coset \( gH \) is closed and the saturation of an open set \( U \) is the open set \( UH = \bigcup_{h \in H} Uh \).
\end{proof}

The following is a standard application of Theorem~\ref{thm:cosets}: we report its short proof for the reader's convenience.

\begin{corollary}
Let \(\lambda\) be such that \( 2^{< \lambda} = \lambda \), and let \( G \) and \( H \) be \(\lambda\)-Polish groups. If \( \varphi \colon G \to H \) is a continuous homomorphism, then \( \varphi(G) \) is \(\lambda\)-Borel in \( H \).
\end{corollary}

\begin{proof}
Apply Theorem~\ref{thm:cosets} to get a \(\lambda\)-Borel set \( T \) meeting every (left) coset of the the closed subgroup \( \varphi^{-1}(\mathrm{id}_H) \) of \( G \) in exactly one point. Then \( \varphi \restriction T \) is injective, hence \( \varphi(T) \) is \(\lambda\)-Borel by Theorem~\ref{thm:injectiveBorelimage}. Moreover \( \varphi(T) = \varphi(G) \), hence the result follows.
\end{proof}

From Theorem~\ref{thm:cosets} we also obtain the following (weak) analogue of Miller's theorem~\cite[Theorem 15.14]{Kechris1995}.

\begin{theorem}
Let \(\lambda\) be such that \( 2^{< \lambda}=\lambda \), \( G \) be a \(\lambda\)-Polish group, and \( X \) be a \(\lambda\)-Polish space. If \( (g,x) \mapsto g \cdot x \) is a continuous action of \( G \) on \( X \), then for every \( x \in X \) its orbit \( \{ g \cdot x \mid g \in G \} \) is \(\lambda\)-Borel.
\end{theorem}

\begin{proofsketchcite}{Kechris1995}{Theorem 15.14}
Since the action is continuous, the stabilizer \( G_x = \{ g \in G \mid g \cdot x = x \} \) of each point \( x \in X \) is a closed subgroup of \( G \). Let \( T_x \) be a \(\lambda\)-Borel set meeting every left coset of \( G_x \) in exactly one point (Theorem~\ref{thm:cosets}). The map \( g \mapsto g \cdot x \) turns out to be a \(\lambda\)-Borel bijection between \( T_x \) and the orbit \( \{ g \cdot x \mid g \in G \} \), hence the latter is \(\lambda\)-Borel by  Theorem~\ref{thm:injectiveBorelimage}.
\end{proofsketchcite}

\section{Other separation theorems for \(\lambda\)-analytic sets} \label{sec:otherseparationtheorems}

Some of the results of this section are stated for \(\lambda\)-Polish spaces in order to minimize the cardinal arithmetic assumptions on \(\lambda\).
However, if \( 2^{< \lambda} = \lambda \) they can be extended to standard \(\lambda\)-Borel spaces as well, whenever this makes sense.


\subsection{Invariant Souslin's theorem}

We begin with the invariant version of Theorem~\ref{thm:lusinseparation}.
An equivalence relation \( E \) on a \(\lambda\)-Polish (or even just standard \(\lambda\)-Borel) space \( X \) is called \(\lambda\)-analytic if \( E \in \lS^1_1(X \times X) \); similarly, \( E \) is \(\lambda\)-Borel if \( E \in \lB(X \times X) \). A set \( A \subseteq X \) is \markdef{\( E \)-invariant} if 
\( A =  [A]_E \). Notice that if both \( E \) and \( A \) are \(\lambda\)-analytic, then so is \( [A]_E \).

\begin{theorem}
Let \( E \) be a \(\lambda\)-analytic equivalence relation on a \(\lambda\)-Polish space \( X \). If \( A,B \subseteq X \) be disjoint \(\lambda\)-analytic \( E \)-invariant sets, then there is an \( E \)-invariant \(\lambda\)-Borel set \( C \subseteq X \) separating \( A \) from \( B \).
\end{theorem}

The following argument is pretty standard, and is reported only for the reader's convenience.

\begin{proof}
Recursively%
\footnote{The proof seems to require an axiom of dependent choices, namely \( \DC_\omega(\pre{\lambda}{2}) \), to be carried out. However, once a canonical representation witnessing \( E \in \lS^1_1(X^2) \) is fixed, the sets \( A_i \) and \( C_i \) can all be determined in a canonical way by Remark~\ref{rmk:constructivelusinseparation}.}
define the sequences \( (A_i)_{i \in \omega} \) and \( ( C_i)_{i \in \omega} \) by letting \( A_0 = A \), \( A_{i+1} = [C_i]_E \), and \( C_i \) be any \(\lambda\)-Borel set separating the disjoint \(\lambda\)-analytic sets \( A_i \) and \( B \) (such \( C_i \) exists by Theorem~\ref{thm:lusinseparation}). Then let \( C = \bigcup_{i \in \omega} C_i = \bigcup_{i \in \omega} A_i \).
\end{proof}

\subsection{Novikov's separation theorem} \label{subsec:Novikov}

The following separation theorem is the analogue of~\cite[Theorem 28.5]{Kechris1995}, due to Novikov;
its proof is a fairly nontrivial adptation of the original argument.

\begin{theorem} \label{thm:novikovsepartion}
Assume that \(2^{< \lambda}=\lambda\).
The class \( \lS^1_1 \) has the generalized separation property, that is: For every \(\lambda\)-Polish space \( X \) and every sequence \( (A_\alpha)_{\alpha < \lambda} \) of \(\lambda\)-analytic subsets of \( X \), if \( \bigcap_{\alpha< \lambda} A_\alpha  = \emptyset \) then there is a sequence \( (B_\alpha)_{\alpha < \lambda} \) of \(\lambda\)-Borel %
\footnote{Recall that when \( \boldsymbol{\Gamma}(X) = \lS^1_1(X) \), then \( \boldsymbol{\Delta}_{\boldsymbol{\Gamma}}(X) = \lD^1_1(X) = \lB(X) \) by  the generalized Souslin's Theorem~\ref{thm:souslin}.}
 sets such that \( B_\alpha \supseteq A_\alpha \) for all \( \alpha < \lambda \) and \( \bigcap_{\alpha < \lambda} B_\alpha = \emptyset \).
\end{theorem}

\begin{proof}
We can assume without loss of generality that \( A_\alpha \neq \emptyset \) for all \( \alpha < \lambda \). Thus by~\ref{prop:charanalytic-3} of Proposition~\ref{prop:charanalytic} there are continuous surjections \( f_\alpha \colon C(\lambda) \to A_\alpha \): for each \( \alpha < \lambda \) and \( s \in \pre{<\omega}{(\vec{\lambda})} =  \bigcup_{n \in \omega} \big( \prod_{i<n} \lambda_i \big) \), set \( P^{(\alpha)}_s = f_\alpha(\Nbhd_s ( C(\lambda))) \) and notice that each \( P^{(\alpha)}_s \) is a \(\lambda\)-analytic set. 
A sequence \( c = (s_\alpha)_{\alpha < \lambda} \) is called a configuration of level  \( n \in \omega \) if \( s_\alpha \in \prod_{i < n} \lambda_i \) for \( \alpha < \lambda_{n-1} \) and \( s_\alpha = \emptyset \) otherwise, where by convention \( \lambda_{-1} = 0 \). In particular, there is a unique configuration \( c_0 \) of level \( 0 \), namely the one with \( s_\alpha = \emptyset \) for all \( \alpha < \lambda \). A configuration \( c' = (s'_\alpha)_{\alpha < \lambda} \) of level \( n+1 \) extends a configuration \( c = (s_\alpha)_{\alpha < \lambda} \) of level \( n \), in symbols \( c' \supseteq c \), if \( s_\alpha \subseteq s'_\alpha \) for all \( \alpha < \lambda \). Since we assumed \( 2^{< \lambda} = \lambda \), the set \( \mathscr{C}_n \) of configurations of level \( n \) is canonically well-orderable, and in fact there are less than \(  \lambda \)-many configurations of level \( n+1 \) extending a given configuration \( c \) of level \( n \). Finally, say that a configuration \( c = (s_\alpha)_{\alpha < \lambda} \) (of any level) is bad if the conclusion of the theorem fails for the family \( \big(P^{(\alpha)}_{s_\alpha}\big)_{\alpha < \lambda } \), that is, there is no sequence \( (B_\alpha)_{\alpha < \lambda} \) of \(\lambda\)-Borel sets such that \( B_\alpha \supseteq  P^{(\alpha)}_{s_\alpha} \) and \( \bigcap_{\alpha< \lambda} B_\alpha = \emptyset \). (Notice that \( \bigcap_{\alpha < \lambda} P^{(\alpha)}_{s_\alpha} = \emptyset \) always holds because \( P^{(\alpha)}_{s_\alpha} \subseteq A_\alpha \) for every \( \alpha < \lambda \).)

\begin{claim}
Fix any \( n \in \omega \). Every bad configuration \( c = (s_\alpha)_{\alpha < \lambda} \in \mathscr{C}_n\)  can be extended to a bad configuration \( c' \in \mathscr{C}_{n+1} \).
\end{claim}

\begin{proof}[Proof of the claim]
Suppose not, and for every  \( c  \subseteq c' \in \mathscr{C}_{n+1} \), let \( \big(B^{(c')}_\alpha \big)_{\alpha < \lambda } \) be a family of \(\lambda\)-Borel sets witnessing that \( c' \) is not bad. For every \( \alpha < \lambda_{n} \) and \( t \in \prod_{i < n+1} \lambda_i \) such that \( s_\alpha \subseteq t \), let 
\[
D_{\alpha,t} = \bigcap \Big\{ B^{(c')}_\alpha \mid c' = (s'_\gamma)_{\gamma < \lambda} \in \mathscr{C}_{n+1} \wedge c' \supseteq c \wedge s'_\alpha = t \Big\}
\]
and
\[ 
D_\alpha = \bigcup \bigg\{  D_{\alpha,t} \mid s_\alpha \subseteq t \in \prod_{i < n+1} \lambda_i \bigg\}.
 \] 
 If instead \( \lambda_n \leq \alpha < \lambda \), set
 \[ 
D_\alpha = \bigcap \{ B^{(c')}_\alpha \mid c' \in \mathscr{C}_{n+1} \wedge  c'\supseteq c\}.
 \] 
Since \( c \) has less than \( \lambda \)-many extensions in \( \mathscr{C}_{n+1} \), it follows that all the sets \( D_{\alpha,t} \) and \( D_\alpha \) are \(\lambda\)-Borel. Moreover, if \( \alpha < \lambda_n \) then \( D_{\alpha,t} \supseteq P^{(\alpha)}_t \), thus \[ 
D_\alpha = \bigcup \bigg\{  D_{\alpha,t} \mid s_\alpha \subseteq t \in \prod_{i < n+1} \lambda_i \bigg\} \supseteq \bigcup \bigg\{  P^{(\alpha)}_t  \mid s_\alpha \subseteq t \in \prod_{i < n+1} \lambda_i \bigg\} = P^{(\alpha)}_{s_\alpha} 
\] 
for all \( \alpha < \lambda_n \).
Similarly, \( D_\alpha \supseteq P^{(\alpha)}_{s_\alpha} \) for every \( \lambda_n \leq \alpha < \lambda \) because if \( c' = (c'_\gamma)_{\gamma < \lambda} \in \mathscr{C}_{n+1} \) then \( s'_\alpha = \emptyset = s_\alpha \).
We claim that \( \bigcap_{\alpha < \lambda } D_\alpha = \emptyset \), contradicting the assumption that \( c \) was bad. Indeed, if \( x \in \bigcap_{\alpha < \lambda} D_\alpha \), then for every \( \alpha < \lambda_n \) there is \( t_\alpha \supseteq s_\alpha  \) with \( t_\alpha \in \prod_{i < n+1} \lambda_i \) such that \( x \in D_{\alpha,t_\alpha} \). Consider the configuration \( c' = (s'_\alpha)_{\alpha < \lambda} \in \mathscr{C}_{n+1}\) with \( s'_\alpha = t_\alpha \) if \( \alpha < \lambda_n \) and \( s'_\alpha = \emptyset \) otherwise: it clearly extends \( c \), and is such that \( x \in D_{\alpha,t_\alpha} \subseteq  B^{(c')}_\alpha \) for all \( \alpha < \lambda_n \), and also \( x \in D_\alpha \subseteq B^{(c')}_\alpha \) for all \( \lambda_n \leq \alpha < \lambda \).
Thus \( x \in \bigcap_{\alpha < \lambda}B^{(c')}_\alpha \), against the choice of the sequence \( (B^{(c')}_\alpha)_{\alpha < \lambda} \).
\end{proof}

Towards a contradiction, assume that the conclusion of the theorem fails or, equivalently, that the (unique) configuration \( c_0 \) of level \( 0 \) is bad.
Recursively apply the claim to find%
\footnote{Since each set of configurations \( \mathscr{C}_n \) is canonically well-orderable, at each step we can just take the configuration with the desired properties which appears first with respect to such well-ordering. This shows that no choice is needed here.}
 a sequence \( (c_n)_{n \in \omega} \) of configurations \( c_n = (s^n_\alpha)_{\alpha < \lambda} \) such that \( c_n \) is of level \( n \), \( c_n \) is bad, and \( c_{n+1} \supseteq c_n \), for every \( n \in \omega \). For each \( \alpha < \lambda \) set \( x_\alpha = \bigcup_{n \in \omega} s^n_\alpha \in C(\lambda) \) and \( p_\alpha = f_\alpha(x_\alpha) \). Since \( p_\alpha \in A_\alpha \) and \( \bigcap_{\alpha < \lambda } A_\alpha = \emptyset \), there are \( \alpha_0, \alpha_1 < \lambda \) such that \( p_{\alpha_0} \neq p_{\alpha_1} \): let \( U_0,U_1 \subseteq X \) be open neighborhoods of \( p_{\alpha_0} \) and \( p_{\alpha_1} \), respectively, such that \( U_0 \cap U_1 = \emptyset \). 
By continuity of the maps \( f_{\alpha_i} \), there is \( m \in \omega \) such that \( \lambda_{m-1} > \alpha_0, \alpha_1 \) and \( f_{\alpha_i}(\Nbhd_{x_{\alpha_i} \restriction m} ( C(\lambda))) = P^{(\alpha_i)}_{x_{\alpha_i} \restriction m} \subseteq U_i \). It follows that the sequence \( (B_\alpha)_{\alpha < \lambda} \) with \( B_{\alpha_i} = U_i \) for \( i \in \{ 0,1 \} \) and \( B_\beta = X \) for all \( \beta \) different from \( \alpha_0 \) and \( \alpha_1 \) witnesses that \( c_m \) is not bad, a contradiction.
\end{proof}

Theorem~\ref{thm:novikovsepartion} can be reformulated in many different ways. For example, dualizing the statement and recursively disjointifying  the resulting \(\lambda\)-Borel sets we get a weaker form of the \(\lambda\)-generalized reduction property for \( \lP^1_1 \).

\begin{corollary} \label{cor:novikovseparation1}
Let \(\lambda\) be such that \( 2^{< \lambda} = \lambda \), and let \( X \) be a \(\lambda\)-Polish space. If \( (C_\alpha)_{\alpha < \lambda} \) is a sequence of \(\lambda\)-coanalytic sets covering \( X \), then it can be refined to a \(\lambda\)-Borel \emph{weak} partition of \( X \), that is, there are pairwise disjoint (possibly empty) \(\lambda\)-Borel sets \( B_\alpha \), \( \alpha < \lambda \), such that \( \bigcup_{\alpha < \lambda} B_\alpha = X \) and \( B_\alpha \subseteq C_\alpha \) for all \( \alpha < \lambda \).
\end{corollary}

Clearly, Corollary~\ref{cor:novikovseparation1} in turn implies Theorem~\ref{thm:novikovsepartion} via dualization again, so it is equivalent to it. Moreover, from Corollary~\ref{cor:novikovseparation1} one can derive with a standard argument the following result. It is a weaker form of the ordinal  \(\lambda\)-uniformization property for \( \lP^1_1 \), and can trivially be extended to \(\lambda\)-coanalytic subsets of \( X \times \lambda \) whose projection on the first coordinate is \(\lambda\)-analytic (and thus \(\lambda\)-Borel).

\begin{corollary} \label{cor:novikovseparation2}
Let \( \lambda \) be such that \( 2^{< \lambda} = \lambda \) and \( X \) be a \(\lambda\)-Polish space. If \( C \subseteq X \times \lambda \) (where \(\lambda\) is endowed with the discrete topology) is \(\lambda\)-coanalytic and such that for all \( x \in X \) there is \( \alpha < \lambda \) for which \( (x,\alpha) \in C \), then there is a \(\lambda\)-Borel function \( f \colon X \to \lambda \) such that \( (x,f(x)) \in C \) for all \( x \in X \).
\end{corollary}

\begin{proof}
For each \( \alpha < \lambda \), let \( C_\alpha = \{ x \in X \mid (x,\alpha) \in C \} \). Apply Corollary~\ref{cor:novikovseparation1} and then set \( f(x) = \alpha \) if and only if \( x \in B_\alpha \).
\end{proof}

Since \( \lP^1_1 \) is clearly a \(\lambda\)-reasonable (in the sense of Definition~\ref{def:lambdareasonable}) boldface \(\lambda\)-pointclass, one can easily verify that in turn Corollary~\ref{cor:novikovseparation2} implies Corollary~\ref{cor:novikovseparation1}, hence it is equivalent to Theorem~\ref{thm:novikovsepartion} again.


As observed, all results seen so far in this subsection remain true if we work with standard \(\lambda\)-Borel spaces rather than \(\lambda\)-Polish spaces, as we are assuming \( 2^{< \lambda} = \lambda \) anyway. Thus
combining Corollary~\ref{cor:novikovseparation1} with (the analogue for standard \(\lambda\)-Borel spaces of) Theorem~\ref{thm:lusinseparation}, one can prove an analogue of~\cite[Theorem 28.7]{Kechris1995} in our new setup. For technical reasons (see the proof of Theorem~\ref{thm:compactsections}), our version is slightly stronger than the original formulation, which corresponds to the special case where \( \mathcal{U} \) is a basis  for the whole topology of \( Y \) of size at most \(\lambda\). 

Let \( \mathcal{V} \) be a collection of open subsets of a topological space \( X = (X,\tau) \). A \markdef{basis for \( \mathcal{V} \)} is a collection \( \mathcal{U} \subseteq \tau \) such that each \( V \in \mathcal{V} \) is a union of elements of \( \mathcal{U} \). In particular, \( \mathcal{U} \) is a basis for \( X \) if and only if it is a basis for \( \mathcal{V} = \tau \). Moreover, if \( \mathcal{U} \) is a basis for \( \mathcal{V} \), \( \mathcal{V}' \subseteq \mathcal{V} \), and \( \mathcal{U}' \supseteq \mathcal{U} \), then \( \mathcal{U}' \) is a basis for \( \mathcal{V} ' \).

\begin{theorem} \label{thm:rectangles}
Let \(\lambda\) be such that \( 2^{< \lambda} = \lambda \), and let \( X,Y \) be \(\lambda\)-Polish spaces. Let \( A \subseteq X \times Y \) be a \(\lambda\)-Borel set such that every vertical section 
\[ 
A_x = \{  y \in Y \mid (x,y) \in A \} 
\] 
is open. If \( \mathcal{U} = \{U_\alpha \mid \alpha < \lambda \} \) is a basis for \( \mathcal{V} = \{ A_x \mid x \in X \} \), then
\( A \) can be written as a union of rectangles
\[ 
A = \bigcup_{\alpha < \lambda } (B_\alpha \times U_\alpha)
 \] 
with \( B_\alpha \in \lB(X) \) for all \( \alpha < \lambda \). In particular, the result applies when \( \mathcal{U} \) is a basis for \( Y \).
\end{theorem}

\begin{proofcompare}{Kechris1995}{Theorem 28.7}
By the choice of \( \mathcal{U} \), \( A = \bigcup_{\alpha < \lambda} (X_\alpha \times U_\alpha) \) where \( X_\alpha = \{  x \in X \mid U_\alpha \subseteq A_x \} \). Since \( X_\alpha \) is \(\lambda\)-coanalytic, so is each \( Z_\alpha = X_\alpha \times U_\alpha \). Applying (the analogue for standard \(\lambda\)-Borel spaces of) Corollary~\ref{cor:novikovseparation1} to the covering of the standard \(\lambda\)-Borel space \( A \) constituted by the sets \( Z_\alpha \), we get a family of \(\lambda\)-Borel sets \( A_\alpha \subseteq Z_\alpha \) such that \( A = \bigcup_{\alpha < \lambda} A_\alpha \). Let \( S_\alpha \) be the projection on \( X \) of \( A_\alpha \), so that \( S_\alpha \subseteq X_\alpha \). Since \( S_\alpha \) is \(\lambda\)-analytic and \( X_\alpha \) is \(\lambda\)-coanalytic, by Theorem~\ref{thm:lusinseparation} there is a \(\lambda\)-Borel set \( B_\alpha \) such that \( S_\alpha \subseteq B_\alpha \subseteq X_\alpha \). Since \( A_\alpha \subseteq B_\alpha \times U_\alpha \subseteq Z_\alpha \), the sets \( B_\alpha \) are as required.
\end{proofcompare}

Applying Corollary~\ref{cor:changeoftopologyclopen} to the sets \( B_\alpha \) from Theorem~\ref{thm:rectangles} after considering as \( \mathcal{U} \) any basis for \( Y \) of size at most \(\lambda\), one easily gets:

\begin{corollary} \label{cor:Borelwithclosedsections}
Let \( \lambda \) be such that \( 2^{< \lambda} = \lambda \), and let \( X,Y \) be \(\lambda\)-Polish spacs, with \(\tau\) the topology of \( X \). Let \( A \subseteq X \times Y \) be a \(\lambda\)-Borel set such that all its vertical sections \( A_x \) are open (respectively, closed). Then there is a \(\lambda\)-Polish topology \( \tau' \supseteq \tau \) such that \( \lB(X,\tau') = \lB(X,\tau) \) and \( A \) is open (respectively, closed) as a subset of \( (X,\tau') \times Y \). Moreover, we can further require \( \mathrm{dim}(X,\tau') = 0 \). 
\end{corollary}

An important application of Corollary~\ref{cor:Borelwithclosedsections} is the following.

\begin{lemma}
Let \( \lambda \) be such that \( 2^{< \lambda} = \lambda \), and let \( X = (X,\tau) \) be a \(\lambda\)-Polish space. If \( E \) is a \(\lambda\)-Borel equivalence relation on \( X \) all of whose equivalence classes are closed, then there is a \(\lambda\)-Polish topology \( \tau' \supseteq \tau \) such that \( \lB(X,\tau') = \lB(X,\tau) \) and \( E \) is closed as a subset of \( (X, \tau') \times (X, \tau') \).
\end{lemma}

\begin{corollary} \label{cor:closedclasses-closedrelation}
Let \( \lambda \) be such that \( 2^{< \lambda} = \lambda \), and let \( E \) be a \(\lambda\)-Borel equivalence relation on a standard \(\lambda\)-Borel space \( (X, \mathcal{B}) \). The following are equivalent:
\begin{enumerate-(1)}
\item \label{cor:closedclasses-closedrelation-1}
there is a \(\lambda\)-Polish topology \( \tau \) on \( X \) such that \( \mathcal{B} = \lB(X,\tau) \) and \( E \) is \( \tau \times \tau \)-closed;
\item \label{cor:closedclasses-closedrelation-2}
there is a \(\lambda\)-Polish topology \( \tau \) on \( X \) such that \( \mathcal{B} = \lB(X,\tau) \) and each \( E \)-equivalence class is \(\tau\)-closed.
\end{enumerate-(1)}
\end{corollary}

Equivalence relations satisfying condition~\ref{cor:closedclasses-closedrelation-1} in Corollary~\ref{cor:closedclasses-closedrelation} are called \markdef{potentially closed}.
Thus Corollary~\ref{cor:closedclasses-closedrelation} says that a \(\lambda\)-Borel equivalence relation on a standard \(\lambda\)-Borel space is potentially closed if and only if all its equivalence classes are ``\emph{simultaneously} potentially closed''. 

If a potentially closed equivalence relation \( E \) is induced by an action of a \(\lambda\)-Polish group \( G \) such that the action is continuous when \( X \) is equipped with a topology \(\tau\) as in~\ref{cor:closedclasses-closedrelation-2} of Corollary~\ref{cor:closedclasses-closedrelation}, then \( E \) is \markdef{smooth}, i.e.\ there is a \(\lambda\)-Borel function \( f \colon X \to Y \) with \( Y \) standard \(\lambda\)-Borel such that for all \( x_0,x_1 \in X \)
\[ 
x_0 \mathrel{E} x_1 \iff f(x_0) = f(x_1).
 \] 
 This follows from the next, more general result.
 
 \begin{proposition} \label{prop:smoothnessunderappropriatehypotheses}
 Let \( \lambda^{< \lambda} = \lambda \), and let \( X \) be a \(\lambda\)-Polish space. Let \( E \) be an equivalence relation on \( X \) induced by a continuous action \( (g,x) \mapsto g \cdot x \) a \(\lambda\)-Polish group \( G \), that is, \( x \mathrel{E} y \iff \exists g \in G \, (g \cdot x = y) \) for every \( x,y \in X \). If all \( E \)-equivalence classes are \( G_\delta \), then \( E \) is smooth. Furthermore, if all \( E \)-equivalence classes are closed, then \( E \) admits a \(\lambda\)-Borel  selector (equivalently: a \(\lambda\)-Borel transversal).
 \end{proposition}
 
 \begin{proof}
 Set \( Y = F(X) \) and consider the map \( f \) sending \( x \) to \( \mathrm{cl}( [x]_E) \). Then \( f \) is \(\lambda\)-Borel because for every open \( U \subseteq X \) and every \( x \in X \), 
\[ 
\mathrm{cl}([x]_E) \cap U \neq \emptyset \iff
[x]_E \cap U \neq \emptyset \iff \exists \alpha < \lambda \, (g_\alpha \cdot x \in U),
 \] 

where \( \{ g_\alpha \mid \alpha < \lambda \} \) is any dense subset of \( G \) fixed in advance.  Obviously, if \( x_0 \mathrel{E} x_1 \) then \( [x_0]_E = [x_1]_E \), and hence \( f(x_0) = f(y_1) \). Vice versa, if \( \mathrm{cl}([x_0]_{E}) = \mathrm{cl}([x_1]_{E}) \) then both \( [x_0]_{E} \) and \( [x_1]_E \) are dense \( G_\delta \) subsets of the \(\lambda\)-Polish space \( \mathrm{cl}([x_0]_{E}) \), hence they must intersect by the Baire category theorem for complete metric spaces, and therefore \( x_0 \mathrel{E} x_1 \).

Assume now that all \( E \)-equivalence classes are closed. The fact that \( E \) is a induced by a continuous action of a \(\lambda\)-Polish group implies that the saturation of any open subset of \( X \) is still open, hence we are done by Theorem~\ref{thm:selector}. 
\end{proof}

Conversely, an arbitrary smooth equivalence relation \( E \) on a standard \(\lambda\)-Borel space \( X \) is necessarily potentially closed. 
Indeed, let \( f \colon X \to Y \) witness that \( E \) is smooth, and let \( \tau_X \) and \( \tau_Y \) be \(\lambda\)-Polish topology on \( X \) and \( Y \), respectively, generating their \(\lambda\)-Borel structure. Apply Corollary~\ref{cor:changeoftopologyfunctions} to refine \( \tau_X \) to a \(\lambda\)-Polish topology \( \tau'_X \) still generating the same \(\lambda\)-Borel structure on \( X \) but turning \( f \) into a continuous function. Then \( E \) is \( \tau'_X \)-closed, and so are all its equivalence classes.


\subsection{Monotone and positive \(\lambda\)-Borel sets}

A set \( A \subseteq \pow(\lambda) \) is \markdef{monotone} if it is upward closed with respect to inclusion, i.e.\ if \( y \in A \) whenever \( x \subseteq y \subseteq \lambda \) for some \( x \in A \). Assume that \( 2^{< \lambda} = \lambda\) and identify \( \pow(\lambda) \) with \( \pre{\lambda}{2} \) through characteristic functions, so that \( \pow(\lambda) \) is \(\lambda\)-Polish as well. Consider the induced \(\lambda\)-Borel structure \( \lB(\pow(\lambda) ) \) on \( \pow(\lambda) \). It is not hard to see that \( \lB(\pow(\lambda) ) \) is obtained by closing under unions and intersections of size \( \lambda \) the sets of the form
\begin{align*}
U_\alpha & = \{ x \subseteq \lambda \mid \alpha \in x \} \\
\hat{U}_\alpha & = \{ x \subseteq \lambda \mid \alpha \notin x \},
\end{align*}
for \( \alpha < \lambda \). Indeed, through the above mentioned identification the sets \( U_\alpha \) and \( \hat{U}_\alpha \) correspond to the subbasis of the product topology on \( \pre{\lambda}{2} \), which generates the same \(\lambda\)-Borel sets as the bounded topology because we are assuming \(2^{< \lambda} = \lambda\). Notice also that since on \( \pre{\lambda}{2} \) the product topology is coarser than the bounded topology, the sets \( U_\alpha \) and \( \hat{U}_\alpha \) are open subsets of \( \pow(\lambda) \).

The \(\lambda\)-Borel subsets of \( \pow(\lambda) \) obtained from the sets \( U_\alpha \) using only \(\lambda\)-sized unions and intersections are called \emph{positive} \( \lambda \)-Borel sets. (The name comes from the fact that the variable ``\( x \)'' is used only positively in their definitions.) The following result generalizes~\cite[Theorem 28.11]{Kechris1995}.

\begin{theorem} 
Assume that \( 2^{< \lambda} = \lambda \).
For every \( \lambda \)-Borel set \( A \subseteq \pow(\lambda) \), \( A \) is monotone if and only if \( A \) is positive.
\end{theorem}

One direction is obvious: since the sets \( U_\alpha \) are monotone, and monotone sets are closed under (arbitrary) intersections and unions, all positive sets are trivially monotone.
As in the countable case, the other direction is instead derived by a stronger separation theorem, whose classical version~\cite[Theorem 28.12]{Kechris1995} is due to Dyck.

\begin{theorem} 
Assume that \( 2^{< \lambda} = \lambda \). Let \( A,B \subseteq \pow(\lambda) \) be disjoint \(\lambda\)-analytic sets with \( A \) monotone. Then there is a positive \(\lambda\)-Borel set \( C \subseteq \pow(\lambda) \) separating \( A  \) from \(  B \).
\end{theorem}

\begin{proofcompare}{Kechris1995}{Theorem 28.12}
Let \( f,g \colon B(\lambda) \to \pow(\lambda) \) be continuous functions such that \( f(B(\lambda)) = A \) and \( g(B(\lambda) ) = B \). For each \( s \in \pre{< \omega}{\lambda} \), let \( P_s=f(\Nbhd_s) \) e \( Q_s = g(\Nbhd_s) \). Call a pair \( (s,t ) \in \pre{< \omega}{\lambda} \times \pre{<\omega}{\lambda} \) bad if \( P_s \) cannot be separated from \( Q_t \) by a positive \(\lambda\)-Borel set. Notice that if \( (s,t) \) is bad then there are \( \alpha, \beta \in \lambda \) such that \( (s {}^\smallfrown{} \alpha, t {}^\smallfrown{} \beta) \) is bad as well: indeed, if \( R_{\alpha,\beta} \) is a positive \(\lambda\)-Borel set separating \( P_{s {}^\smallfrown{} \alpha} \) from \( Q_{t {}^\smallfrown{} \beta} \), then \( \bigcup_{\alpha < \lambda} \bigcap_{\beta < \lambda} R_{\alpha,\beta}  \) is a positive \(\lambda\)-Borel set separating \( P_s \) from \( Q_t \).

Assume towards a contradiction that the conclusion of the theorem is false, i.e.\ that \( (\emptyset,\emptyset) \) is bad. Recursively construct \( x,y \in B(\lambda) \) such that \( (x \restriction n, y \restriction n) \) is bad for all \( n \in \alpha \). Since \( f(x) \in A \) and \( g(y) \in B \), and thus \( g(y) \notin A \) because \( A \cap B = \emptyset \), it follows from the monotonicity of \( A \) that \( f(x) \not\subseteq g(y) \). Let \( \alpha < \lambda \) be such that \( \alpha \in f(x) \setminus g(y) \), so that \( f(x) \in U_\alpha \) and \( g(y) \in \hat{U}_\alpha \). Since \( U_\alpha, \hat{U}_\alpha \) are open and the maps \( f \) and \( g \) are continuous, there is \( k \in \omega \) large enough so that \( f(\Nbhd_{x \restriction k}) = P_{x \restriction k} \subseteq U_\alpha \) and \( g(\Nbhd_{y \restriction k}) = Q_{y \restriction k} \subseteq \hat{U}_\alpha \). Then \( U_\alpha \) is a positive \(\lambda\)-Borel (in fact, open) set separating \( P_{x \restriction k} \) from \( Q_{y \restriction k} \), contradicting the fact that \( (x \restriction k, y \restriction k) \) was bad.
\end{proofcompare}

\section{The \(\lambda\)-projective hierarchy} \label{sec:lambda-projective}

In this short section we introduce all relevant definitions and basic facts concerning the 
\(\lambda\)-projective hierarchy. We omit all proofs because they follow the ideas used in the classical case, possibly with the modifications and clarifications pointed out in Sections~\ref{sec:defanalytic} and~\ref{sec:basifactsanalytic}.
Also, we stick to the realm of \(\lambda\)-Polish spaces to avoid further assumptions on \(\lambda\), but when \( 2^{<\lambda} = \lambda \) everything can be transferred to arbitrary standard \(\lambda\)-Borel spaces by Proposition~\ref{prop:charstandardBorel}.

\begin{defin} \label{def:lambda-projective}
 For \( n \geq 1 \),
recursively define  the \(\lambda\)-pointclass \( \lS^1_{n+1} \) as follows. 
Let \( X \) be \(\lambda\)-Polish and \( A \subseteq X \). Then \( A \in  \lS^1_{n+1}(X) \) if and only if there is some \(\lambda\)-Polish space \( Y \) and a continuous function \( f \colon Y \to X \) such that \( A = f(B) \) for some \( B \subseteq Y \) with \( Y \setminus B \in \lS^1_n(Y) \).

We also let \( \lP^1_{n+1} \) and \( \lD^1_{n+1} \) be the dual and the ambiguous \(\lambda\)-pointclasses, respectively, associated to \( \lS^1_{n+1} \), that is: \( \lP^1_{n+1}(X) = \{ X \setminus A \mid A \in \lS^1_{n+1}(X) \} \) and \( \lD^1_{n+1}(X)  = \lS^1_{n+1}(X) \cap \lP^1_{n+1}(X) \). 
\end{defin}

Since \( \emptyset \in \lS^1_1(Y) \) for every \(\lambda\)-Polish \( Y \), we trivially have \( \lS^1_1(X) \subseteq \lS^1_2(X) \), hence arguing by induction we get that for every \( 1 \leq n < m < \omega \)
\[ 
\lS^1_n(X), \lP^1_n(X) \subseteq \lD^1_m(X) \subseteq \lS^1_m(X), \lP^1_m(X).
 \] 
 In particular, 
 \[ 
 \bigcup_{n \geq 1} \lS^1_n(X) = \bigcup_{n \geq 1 } \lP^1_n(X) = \bigcup_{n \geq 1 } \lD^1_n(X).
 \]
If we further assume \( 2^{< \lambda} = \lambda \), then each class \( \lS^1_n(X) \) and \( \lP^1_n(X) \), for any \( n \geq 1 \), has a \( \pre{\lambda}{2} \)-universal set, hence if \( |X| > \lambda \) the previous inclusions are proper.

\begin{defin}
Let \( X \) be a \(\lambda\)-Polish space. A set \( A \subseteq X \) is \markdef{\(\lambda\)-projective} if \( A \in \bigcup_{n \geq 1} \lS^1_n(X) \).
\end{defin}

Arguing as in the first part of the proof of Proposition~\ref{prop:closurepropertiesofanalytic}, one can easily show by induction on \( n \geq 1 \) that all the classes \( \lS^1_n(X) \), \( \lP^1_n(X) \), and \( \lD^1_n(X) \) are closed under finite unions and intersections. Moreover they contain all closed sets, hence using the standard argument one sees that they are closed under continuous preimages, i.e.\ they are all boldface \(\lambda\)-pointclasses. Using these easy facts, as in the classical case one can then provide a number of equivalent reformulation of the class \( \lS^1_{n+1} \) from Definition~\ref{def:lambda-projective}.

\begin{proposition} \label{prop:equivalentreformulationprojective}
Let \( X \) be \(\lambda\)-Polish and \( A \subseteq X \). For each \( n \geq 1 \) the following are equivalent:
\begin{enumerate-(1)}
\item \label{prop:equivalentreformulationprojective-1}
\( A \in \lS^1_{n+1}(X) \), i.e.\ \( A \) is a continuous image of some \( B \in \lP^1_n(Y) \) for some \(\lambda\)-Polish space \( Y \);
\item \label{prop:equivalentreformulationprojective-2}
\( A \) is a continuous image of some \( B \in \lP^1_n(B(\lambda)) \) (or, equivalently, \( B \in \lP^1_n(C(\lambda)) \));
\item \label{prop:equivalentreformulationprojective-3}
\( A = \p(C) \) with \( C \in \lP^1_n(X \times B(\lambda)) \);
\item \label{prop:equivalentreformulationprojective-4}
\( A = \p(C) \) with \( C \in \lP^1_n(X \times Y) \) for some \(\lambda\)-Polish space \( Y \).
\end{enumerate-(1)}
If \( 2^{< \lambda} = \lambda \), we can further replace \( B(\lambda) \) with \( \pre{\lambda}{2} \) in~\ref{prop:equivalentreformulationprojective-2} and~\ref{prop:equivalentreformulationprojective-3}, or replace continuous functions with \(\lambda\)-Borel functions in~\ref{prop:equivalentreformulationprojective-1} and~\ref{prop:equivalentreformulationprojective-2}.
\end{proposition}

Finally, the arguments used in Proposition~\ref{prop:closurepropertiesofanalytic} allow one to prove analogous closure properties for all levels of the \(\lambda\)-projective hierarchy (again by induction on \( n \geq 1 \)).

\begin{proposition}
Let \( n \geq 1 \).
\begin{enumerate-(i)}
\item
The boldface \( \lambda \)-pointclass \( \lS^1_n \) is closed under well-ordered unions of length at most \( \lambda \), countable intersections, and \(\lambda\)-Borel images and preimages. If \( 2^{< \lambda} = \lambda \), then it is also closed under well-ordered intersections of length at most \( \lambda \).
\item
The boldface \( \lambda \)-pointclass \( \lP^1_n \) is closed under well-ordered intersections of length at most \( \lambda \), countable unions, and \(\lambda\)-Borel preimages. If \( 2^{< \lambda} = \lambda \), then it is also closed under well-ordered unions of length at most \( \lambda \).
\item
The boldface \( \lambda \)-pointclass \( \lD^1_n \) is closed under countable unions, countable intersections, complements, and \(\lambda\)-Borel preimages. If \( 2^{< \lambda} = \lambda \), then it is also closed under well-ordered unions and intersections of length at most \( \lambda \), i.e.\ each \( \lD^1_n(X) \) is a \(\lambda^+ \)-algebra on \( X \).
\end{enumerate-(i)}
\end{proposition}

Notice that the non-collapse of the \(\lambda\)-projective hierarchy on \(\lambda\)-Polish spaces larger than \( \lambda \) under the assumption \( 2^{< \lambda} = \lambda \) implies the failure of the missing closure properties, namely: 
\( \lS^1_n \) and \( \lP^1_n \) are not closed under complements, and
\( \lP^1_n \) and \( \lD^1_n \) are not closed under continuous (hence \(\lambda\)-Borel) images.

\section{\( \kappa \)-Souslin sets} \label{sec:Souslinsets}

\begin{defin} \label{def:kappaSouslin}
Let \( X \) be a \(\lambda\)-Polish space and \( \kappa \geq \lambda \) be a cardinal. A set \( A \subseteq X \) is \markdef{\( \kappa \)-Souslin} if it is a continuous image of some \( \kappa \)-Polish space \( Y \).
\end{defin}

Thus \(\lambda\)-Souslin sets are precisely the \(\lambda\)-analytic set. Moreover, if \( \lambda \leq \kappa \leq \kappa' \) then all \( \kappa \)-Souslin sets are also \( \kappa' \)-Souslin. Thus, in particular, the collection of \( \kappa \)-Souslin sets always contains all \(\lambda\)-Borel sets if \( 2^{< \lambda} = \lambda \).

Using the results from Section~\ref{subsec:Polishforothercofinalities}, it is not hard to see that the notion of a \( \kappa \)-Souslin set can be reformulated in several different ways (compare the next proposition with Proposition~\ref{prop:charanalytic}).

\begin{proposition} \label{prop:charkappaSouslin}
Let \( X \) be a \(\lambda\)-Polish space and \( \kappa \geq \lambda \) be a cardinal. For any  \( \emptyset \neq A \subseteq X \) the following are equivalent:
\begin{enumerate-(1)}
\item \label{prop:charkappaSouslin-1}
\( A \) is \( \kappa \)-Souslin;
\item \label{prop:charkappaSouslin-2}
\( A \) is a continuous image of a closed (or even just \( G_\delta \)) subset of some \(\kappa\)-Polish space \( Y \);
\item \label{prop:charkappaSouslin-3}
\( A \) is a continuous image of \( B(\kappa) = \pre{\omega}{\kappa} \), or even just of a closed or \( G_\delta \) subset of it;
\item \label{prop:charkappaSouslin-4}
\( A = \p(F) \) for some closed \( F \subseteq X \times B(\kappa) \);
\item \label{prop:charkappaSouslin-5}
\( A = \p(G) \) for some \( G \subseteq X \times Y \) with \( Y \) a \(\kappa\)-Polish space and \( G \) a \( G_\delta \) set.
\end{enumerate-(1)}
\end{proposition}

The reformulation~\ref{prop:charkappaSouslin-4} from Proposition~\ref{prop:charkappaSouslin} allows one to find a canonical representation of \( \kappa \)-analytic subsets of \( B(\lambda) \) which is analogous to the one in equation~\eqref{eq:normalformforanalytic}, namely: \( A \subseteq B(\lambda) \) is \( \kappa \)-Souslin if and only if there is an \(\omega\)-tree \( T \subseteq \pre{<\omega}{(\lambda \times \kappa)} \) on \( \lambda \times \kappa \) such that
\begin{equation} \label{eq:representationofkappasoulsin}
A = \p[T].
 \end{equation}
 
Arguing as in Proposition~\ref{prop:closurepropertiesofanalytic} and recalling that \(\lambda\)-Borel sets are \( \kappa \)-Souslin if \( 2^{< \lambda}= \lambda \), one easily gets the following closure properties. (The extra choice assumptions are needed to collect continuous functions \( f \colon B(\kappa) \to X \) witnessing that the sets we are dealing with are \( \kappa \)-Souslin.)

\begin{proposition} \label{prop:closurepropertieskappaSouslin}
Let \( \kappa \geq \lambda \). 
\begin{enumerate-(i)}
\item \label{prop:closurepropertieskappaSouslin-1} \axioms{\( \AC_\mu(\pre{\kappa}{2}) \)}
Let \( \omega \leq \mu \leq \kappa \) and
assume \( \AC_\mu(\pre{\kappa}{2}) \). Then the collection of \( \kappa \)-Souslin sets is closed under
 well-ordered unions of length at most \( \mu \) and under countable intersection.  
\item \label{prop:closurepropertieskappaSouslin-2}
The collection of \( \kappa \)-Souslin sets is closed under continuous images and preimages; in particular, it is a boldface \(\lambda\)-pointclass.
If \( 2^{< \lambda} = \lambda \), then it is also closed under  \(\lambda\)-Borel images and preimages.
\end{enumerate-(i)}
\end{proposition}

Finally, using the obvious modifications of the proofs of Theorems~\ref{thm:lusinseparation} and~\ref{thm:souslin} (see also~\cite[Theorems 2E.1 and 2E.2]{Moschovakis2009}), one easily gets:

\begin{proposition} \label{prop:separationforkappaSouslin}
Let \( \kappa \geq \lambda \), \( X \) be  \(\lambda\)-Polish, and \( A,B \subseteq X \) be disjoint \( \kappa \)-Souslin sets. Then there is a \( \kappa^+ \)-Borel set \( C \subseteq X \) separating \( A \) from \( B \).

In particular, if \( A \subseteq X \) is such that both \( A \) and \( X \setminus A \) are \( \kappa \)-Souslin, then \(A \) is \( \kappa^+ \)-Borel.
\end{proposition}

Assuming a bit more choice than usual, one can prove that all sets in the first levels of the \(\lambda\)-projective hierarchy are \( \lambda^+ \)-Souslin, and sometimes even \( \lambda^{++} \)-Borel. This generalizes a theorem from the classical setup due to Shoenfield (see~\cite[Theorem 2D.3]{Moschovakis2009}). 

\begin{theorem}[\( \AC_{\lambda^+}(\pre{\lambda}{2}) \)] \label{thm:coanalyticaresouslin} \label{thm:coanalyticsetsarelambda^+-Souslin} \axioms{\( \AC_{\lambda^+}(\pre{\lambda}{2}) \)}
Assume that \(2^{<\lambda} = \lambda \), and let \( X \) be a \(\lambda\)-Polish space. Then every \(\lambda\)-coanalytic subset of \( X \), and thus also every set in \( \lS^1_2(X) \), is \(\lambda^+\)-Souslin.
\end{theorem}

\begin{proof}
Let \( A \in \lP^1_1(X) \). By Theorem~\ref{thm:longunionsintersections} there are \(\lambda\)-Borel sets \( A_\xi \), \( \xi < \lambda^+ \), such that \( A = \bigcup_{\xi < \lambda^+} A_\xi \), and without loss of generality they might be assumed to be nonempty. 
Since \( \lB(X) \subseteq \lS^1_1(X) \), for each \( \xi < \lambda^+ \) there is a continuous map \( f_\xi \colon B(\lambda) \to X \) with range \( A_\xi \). Consider the closed set \( C \subseteq B(\lambda^+) \) defined by
\[ 
C = \{ x \in B(\lambda^+) \mid \forall n > 0 \, (x(n) < \lambda) \}, 
 \] 
so that each \( x \in C \) is of the form \( \xi {}^\smallfrown{} x^- \) for some \( \xi \in \lambda^+ \) and \( x^- \in B(\lambda) \). Then the map \( f \colon C \to X \) defined by \( f(x) = f_{x(0)}(x^-) \) is continuous and such that \( f(C) = \bigcup_{\xi < \lambda^+} A_\xi = A \). Therefore, using~\ref{prop:charkappaSouslin-3} of Proposition~\ref{prop:charkappaSouslin} we get that the set \( A \) is \( \lambda^+ \)-Souslin.

Since \( \lS^1_2 \) sets are continuous images of \( \lP^1_1 \) sets, they are \( \lambda^+ \)-Souslin as well by Proposition~\ref{prop:closurepropertieskappaSouslin}\ref{prop:closurepropertieskappaSouslin-2}.
\end{proof}

\begin{corollary}[\( \AC_{\lambda^+}(\pre{\lambda}{2}) \)]
Assume that \( 2^{< \lambda} = \lambda \), and let \( X \) be a \(\lambda\)-Polish space. Then every set in \( \lD^1_2(X) \) is \( \lambda^{++} \)-Borel. 
\end{corollary}

\begin{proof}
By Theorem~\ref{thm:coanalyticsetsarelambda^+-Souslin} and Proposition~\ref{prop:separationforkappaSouslin}.
\end{proof}

%

\section{Some observations on absoluteness} \label{sec:absolutenessnew}

Our proof of Theorem~\ref{thm:coanalyticaresouslin} is somewhat unsatisfactory, as it would be desirable to have a ``definable'' representation of \( \lambda \)-coanalytic sets as \( \lambda^+ \)-Souslin sets to be able to deal with (the higher analogues of) scales and Shoenfield's absoluteness. In contrast, our proof heavily relies on the possibility of choosing \( \lambda^+ \)-many representations for certain analytic sets: this is in general not possible in the definable context. One of the most common definable Souslin representation of coanalytic sets is called Shoenfield tree, and in the following we see how the theory of Shoenfield trees and Shoenfield's absoluteness generalize in our setting. All the notations have been introduced in Section~\ref{subsec:trees}. As~\eqref{eq:normalformforanalytic} is the fundamental starting point of such analysis, we will work only on (at most countable copies of) $B(\lambda)$. The same analysis can be carried on \emph{mutatis mutandis} for $\pre{\lambda}{2}$.

\begin{defin}
 Let $A\in \lP^1_1(B(\lambda))$, and let $T\subseteq\pre{<\omega}{(\lambda\times\lambda)}$ be a tree such that $B(\lambda)\setminus A=\p[T]$, as granted by \eqref{eq:normalformforanalytic}. For any $s\in\pre{<\omega}{\lambda}$, let $R(s)=\{t\in T(s) \mid \sigma_{\mathrm{Tr}}(t)\in\lambda_{\lh(s)}\}$. The \markdef{Shoenfield tree} of $A$ with respect to $T$ is the 
\( \omega \)-tree $S$ on $\lambda\times\pre{<\lambda}{(\lambda^+)}$ consisting of all $(s,u)\in \pre{<\omega}{(\lambda\times\pre{<\lambda}{(\lambda^+)})}$ such that for every $i \leq j<\lh(s)$
 \begin{enumerate-(a)}
 \item
 $u(i)\in\pre{\lambda_i}{(\lambda^+)}$;
 \item
 $u(i)\subseteq u(j)$;
 \item
 the restriction of $u(i) \circ \sigma_{\mathrm{Tr}}$ to $R(s\restriction i)$ is an order-preserving map between $(R(s \restriction i),\leq_{\mathrm{KB}})$ and $(\lambda^+,{\leq})$. 
 \end{enumerate-(a)}
 If $A$ is a \(\lambda\)-coanalytic subset of $\pre{n}{B(\lambda)}$ for some $n \leq \omega$, the Shoenfield tree $S$ on $\pre{n}{\lambda}\times\pre{<\lambda}{(\lambda^+)}$ is defined in the expected way.  
\end{defin}

The idea is the same as the classical Shoenfield tree, that can be found e.g.\ in~\cite[Chapter 25]{Jech2003}: the left side of $S$ is building some $x\in B(\lambda)$, while the right side of $S$ is trying to build level-by-level an order-preserving function from $(T(x),{\leq_{\mathrm{KB}}})$ to \( (\lambda^+,{\leq}) \). In this way, for any \( x \in B(\lambda) \) we have that if there is \( g \) such that $(x,g) \in [S]$, then $g$ witnesses that $T(x)$ is well-founded (so that \( x \in A \)), while if there is no such \( g \), then $T(x)$ is ill-founded (i.e.\ \( x \notin A \)).

There is a substantial difference with the classical case \( \lambda = \omega \), though: a branch of the original Shoenfield tree still provides $x\in\pre{\omega}{\omega}$ and an order-preserving $g \colon T(x) \to \omega_1$ (where \( T(x) \) is endowed with \( \leq_{\mathrm{KB}} \)), but since $T(x)\subseteq\pre{<\omega}{\omega}$, we can code $g$ with some $g' \colon \omega \to \omega_1$, and therefore the resulting Shoenfield tree can be construed as a tree on $\omega\times\omega_1$. In the generalize case, instead, the witness $g$ is a function from \( \omega \) to $\pre{<\lambda}{(\lambda^+)}$, and in general it is no longer possible to code such an object as an element of $\pre{\omega}{(\lambda^+)}$ (or, even worse, as an element of \( \pre{\omega}{\nu} \) with \( \nu \in \Cn \), in case \( \pre{< \lambda}{(\lambda^+)} \) is not well-orderable); in other words, we cannot define $S$ directly on products of ordinals. For this reason, it is worth checking whether the original proof still holds in the generalized context.

\begin{proposition} \label{prop:Shoenfieldtree}
Let $A\in \lP^1_1(B(\lambda))$, and let $T\subseteq\pre{<\omega}{(\lambda\times\lambda)}$ be such that $B(\lambda)\setminus A=\p[T]$. Let $S$ be the Shoenfield tree of $A$ with respect to $T$. Then \( A=p[S] \).
\end{proposition}

\begin{proof}
Let $x\in A$. Then, the section tree $T(x)$ is well-founded, and this implies that the restriction of $\leq_{\mathrm{KB}}$ to $T(x)$ is a well-ordering of length less than \( \lambda^+ \). Let $\rho \colon T(x) \to \lambda^+$ be the rank function $\rho_{<_{\mathrm{KB}\restriction T(x)}}$ associated to (the strict part of) \( \leq_{\mathrm{KB}} \restriction T(x) \). For every $i \in \omega$, $\rho_i=\rho\restriction T(x\restriction i)$ is an order-preserving mpap between $(T(x\restriction i),\leq_{\mathrm{KB}})$ and $(\lambda^+,{\leq})$. 
Let
\[
 w(i)=\rho_i \circ \sigma_{\mathrm{Tr}}^{-1} \colon \sigma_{\mathrm{Tr}}(T(x\upharpoonright i)) \to \lambda^+.
\]
and let $u(i) \colon \lambda_i \to \lambda^+$ be defined by \( u(i)(\alpha) = w(i)(\alpha) \) if $\alpha < \lambda_i$ is such that \( \alpha \in \sigma_{\mathrm{Tr}}(T(x\restriction i)) \), and $U(i)(\alpha) = 0$ otherwise. 
Then $u(i) \circ \sigma_{\mathrm{Tr}}$ coincides with $\rho_i$ on $R(x\restriction i)\subseteq T(x\restriction i)$, and it is therefore order-preserving. Let $u=(u(i))_{i\in\omega}$. Then 
\( (x,u) \) is a branch of $S$ witnessing $x\in p[S]$. 

 Vice versa, let $x\in p[S]$, and let \( u = (u(i))_{i \in \omega} \) be such that $(x,u)\in [S]$. Since $u(i)\subseteq u(j)$ for every $i<j$, then $\bar u = \bigcup_{i \in \omega} u(i)$ is a function from $\bigcup_{i\in\omega}\lambda_i=\lambda$ to $\lambda^+$, and, by construction, the restriction of $\bar u \circ \sigma_{\mathrm{Tr}}$ to \( T(x) \) is an order-preserving map between $(T(x),\leq_{\mathrm{KB}})$ and $(\lambda^+,{\leq})$. Therefore $T(x)$ is well-founded and $x\in A$.
\end{proof}

Proposition~\ref{prop:Shoenfieldtree} easily extends to the case where $A\in \lP^1_1(\pre{n}{B(\lambda)})$, with $n$ finite or countable.

If $\pre{<\lambda}{(\lambda^+)}$ is well-orderable, then in the construction of the Shoenfield tree we can substitute the elements of $\pre{<\lambda}{(\lambda^+)}$ that appear in the tree with an ordinal, and therefore recover \( \kappa \)-Souslinity for a suitable cardinal \( \kappa \geq \lambda^+ \).

\begin{corollary}
Suppose that $\pre{<\lambda}{(\lambda^+)} = \kappa$ for some \( \kappa \in \Cn \). Then every $\lambda$-coanalytic set in $B(\lambda)$ is $\kappa$-Souslin.
\end{corollary} 

In particular, if $|\pre{<\lambda}{(\lambda^+)}|=\lambda^+$ (that holds e.g.\ under $\GCH$), then we recover Theorem~\ref{thm:coanalyticaresouslin}. However, notice that this fails to be a ``definable'' result too, as it uses a bijection from $\pre{<\lambda}{(\lambda^+)}$ to $\lambda^+$. 

The Shoenfield tree was introduced to prove Shoenfield absoluteness (see e.g., \cite[Theorem 25.20]{Jech2003}), and we can exploit Proposition~\ref{prop:Shoenfieldtree} to prove a similar result.
Recall that, given $M\subseteq N$ sets or classes and $\upvarphi(v_1,\dots,v_n)$ a formula in the language of set theory, possibly with parameters in $M$, we say that $\upvarphi$ is \markdef{absolute between $M$ and $N$} if and only if for any $a_1,\dotsc,a_n\in M$, 
\[
M \models \upvarphi[a_1/v_1,\dots,a_n/v_n] \quad \iff \quad N \models \upvarphi[a_1/v_1,\dots,a_n/v_n]. 
\]
We say that \( \upvarphi \) is \markdef{upward absolute} if the direction from left to right in the displayed equivalence holds, and that \( \upvarphi \) is \markdef{downward absolute} if the direction from right to left holds. 
If $N=V$, we do not mention $N$. 

For example, the definition of $B(\lambda)$ as in Definition~\ref{xmp:lambda-Polish}\ref{xmp:lambda-Polish-2} is absolute for all models of large enough fragments of $\ZF$ that contain $\lambda$, where with the phrase ``models of large enough fragments of $\ZF$'' we mean that we need that the model satisfies enough axioms so that ordinals are absolute and $B(\lambda)$ can be defined from $\lambda$. In the following, with an abuse of terminology, we will simply say ``models of $\ZF$'' to mean ``models of large enough fragments of $\ZF$'', and the exact fragments of \( \ZF \) to be considered will be clear from the context.

In descriptive set theory it is common to talk about absoluteness of definable sets, via their definition. If $A$ is a $\lambda$-projective subset of $B(\lambda)$, then by definition $A$ is constructed via simpler sets. For example, if $A$ is $\lambda$-coanalytic, then it is the complement of the projection of a closed subset of \( B(\lambda) \times B(\lambda) \). The given construction of \( A \) yields a formula $\upvarphi(v)$ in the language of set theory, possibly with parameters in $B(\lambda)$, such that $A=\{x\in B(\lambda) \mid V \models \upvarphi[x/v]\}$. If $M$ is a model for which $B(\lambda)$ is absolute, that contains the relevant parameters, and such that $\upvarphi(v)$ is absolute for $M$, then $\{x\in M \mid M \models \upvarphi[x/v]\}=A\cap M$: when this happens, with some abuse of terminology we will say that $A$ is absolute for $M$.%
\footnote{The abuse comes from the fact that $A$ can be defined using different formulas and different parameters, and therefore $A$ can be absolute for $M$ according to one definition but not accoding to some other one.}


By equation~\eqref{eq:normalformforanalytic}, if $A\in \lS^1_1(B(\lambda))$ then there exists an \(\omega\)-tree $T\subseteq\pre{<\omega}{(\lambda\times\lambda)}$ such that $A=\p[T]$, namely, such that $x\in A$ if and only if $T(x)$ is ill-founded, for every \( x \in B(\lambda) \). This is a definition for $A$ that is absolute for any model of $\ZF$ that contains $\lambda$ and $T$. The notion of complement is also absolute for such a model, and therefore also $\lambda$-conalytic subsets of $B(\lambda)$ are absolute for models of $\ZF$ that contain $\lambda$ and the relevant tree \( T \) in their definition as a complement of a \(\lambda\)-analytic set.

\begin{proposition} \label{prop:downward}
Let $A\in \lS^1_2(B(\lambda))$, let $C\in \lP^1_1(B(\lambda)\times B(\lambda))$ be such that $A=p(C)$, and let $T \subseteq \pre{<\omega}{(\lambda^3)}$ be an \(\omega\)-tree such that $C=(B(\lambda)\times B(\lambda))\setminus p[T]$. Then $A$ is upward absolute for all models of $\ZF$ that contain $\lambda$ and $T$.
\end{proposition}

\begin{proof}
Let $\upvarphi(u,v,\lambda,T)$ be the first-order formula in the language of set theory formalizing the assertion ``\( u \in B(\lambda) \wedge v \in B(\lambda)  \wedge T(u,v)\) is well-founded'', so that for every \( x,y \in B(\lambda) \) we have that 
\( \upvarphi[x/u,y/v,\lambda,T] \) holds (in \( V \)) if and only if \( (x,y) \in C \). Then the formula $\exists v \, \upvarphi(u,v,\lambda,T)$ is a definition of $A$ witnessing that it is upward absolute for all models of \( \ZF \) containing \( \lambda \) and $T$. 
\end{proof}

The previous proof shows why $A$ is possibly not downward absolute: if $x\in A\cap M$, then it is possible that any $y \in B(\lambda)$ witnessing ``$x\in A$'' in \( V \), i.e.\ such that $V \models \upvarphi[x/u,y/y,\lambda,T]$, does not belong to $M$. However, if $M$ contains all ingredients needed for the construction of the Shoenfield tree, we recover also downward absoluteness.

\begin{proposition} \label{prop:absoluteness}
Let $A\in \lS^1_2(B(\lambda))$, let $C\in \lP^1_1(B(\lambda)\times B(\lambda))$ be such that $A=p(C)$, and let $T \subseteq \pre{<\omega}{(\lambda^3)}$ be an \(\omega\)-tree such that $C=(B(\lambda)\times B(\lambda))\setminus p[T]$. 
Then $A$ is absolute for all models of $\ZF$ that contain $\lambda$, $\vec{\lambda} = (\lambda_i)_{i \in \omega}$, $T$, and $\pre{<\lambda}{(\lambda^+)}$.
\end{proposition}

\begin{proof}
 Let $M$ be a model of $\ZF$ such that $\lambda, \vec{\lambda}, T,\pre{<\lambda}{(\lambda^+)} \in M$. Since $M$ contains $\lambda$, then $M$ contains also $\sigma_{\mathrm{Tr}}$. Reading the definition of the Shoenfield tree $S$ of $C$ with respect to $T$, it is clear that $S\in M$: since $T\in M$, then $T(s)\in M$  for any $s\in\pre{<\omega}{\lambda}$, and since $\vec{\lambda}\in M$, then $R(s)\in M$: this proves that the definition of $S$ is absolute for $M$. Since $\pre{<\lambda}{(\lambda^+)}\in M$, then also $\pre{<\omega}{(\pre{<\lambda}{(\lambda^+)})}\in M$, and therefore $S\subseteq M$, and in fact $S\in M$. 

By Proposition~\ref{prop:Shoenfieldtree}, for every \( x,y \in B(\lambda) \) we have that $(x,y)\in C$ if and only if $S(x,y)$ is ill-founded, and therefore $x\in A$ if and only if $S(x)$ is ill-founded.%
\footnote{Notice that since \( C \subseteq B(\lambda) \times B(\lambda) \), its Shoenfield tree \( S \) is an \(\omega\)-tree on \( \lambda \times \lambda \times \pre{<\lambda}{(\lambda^+)} \), so it makes sense to consider both its section trees \( S(x) \) obtained by fixing only its first coordinate \( x \), and its section trees \( S(x,y) \) obtained by fixing its first two coordinates \( x \) and \( y \).}
Since \( S \in M \) and well-foundedness is absolute for $M$, this yields absoluteness of \( A \), as desired.
\end{proof}

Note that if $\lambda=\omega$, then any model $M$ of $\ZF$ that contains $\omega_1$ also contains $\omega$, $\vec{\omega}$ (which in this case can be taken as the sequence of all natural numbers), and $\pre{<\omega}{\omega_1}$, and this explains why the hypotheses of the usual Shoenfield's Absoluteness Theorem are much weaker than those used in its counterpart for \( \lambda > \omega \), i.e.\ in  Proposition~\ref{prop:absoluteness}. Indeed, in the classical case \( \lambda = \omega \) it is enough to prove that $S$ is absolute for $M$ (and not necessarily an element of $M$), and therefore we can even drop the assumption $\omega_1\in M$, and reach the usual formulation of the original Shoenfield's Absoluteness Theorem.

We also remark that Proposition~\ref{prop:absoluteness} is not meant to be optimal: for instance, in~\cite[Theorem 1.4]{Laver1997} there is an example of absoluteness of a certain type of \( \lP^1_2(V_{\lambda+1}) \) sets,  assuming very large cardinals, for an inner model that does not satisfy the hypothesis of Proposition~\ref{prop:absoluteness}.


%
%
%

\chapter{Uniformization results and \(\lambda\)-coanalytic ranks} \label{chapter:uniformization}

\section{Uniformization for \(\lambda\)-analytic and \(\lambda\)-Borel sets} \label{sec:uniformization}

Recall that a \markdef{uniformization} of \( A \subseteq X \times Y \) is a set \( A^* \subseteq A \) such that for all \( x \in \p(A) \) there is exactly one \( y \in Y \) with \( (x,y) \in A^* \). 
The existence of a uniformization of \( A \subseteq X \times Y \) is indeed equivalent to the  existence of  a  \markdef{uniformizing function} for \( A \), that is, a function \( f \colon \p(A) \to Y \) such that \( f(x) \) belongs to the vertical section \( A_x  \) of \( A \) at \( x \), for all \( x \in \p(A) \).
Indeed, \( A^* \) is a uniformization of \( A \) if and only if \( A^* \) is the graph of a uniformizing function for \( A \).
Notice also that if \( X \) and \( Y \) are standard \(\lambda\)-Borel spaces and \( A^* \in \lB(X \times Y) \), then \( \p(A) \) is \(\lambda\)-Borel as well by Theorem~\ref{thm:injectiveBorelimage} (here the key fact is that the projection function \( \p \) is continuous, \( \p(A^*) = \p(A) \), and \( \p \restriction A^* \) is injective because \( A^* \) is a uniformization of \( A \)). By Theorem~\ref{thm:borelvsgraph} and Proposition~\ref{prop:charstandardBorel}, under \( 2^{< \lambda} = \lambda \) we have that the set \( A \) has a \(\lambda\)-Borel uniformization if and only if \( \p(A) \) is \(\lambda\)-Borel and there is a \(\lambda\)-Borel uniformizing function \( f \colon \p(A) \to Y \).

In Section~\ref{sec:structuralproperties} we already proved that if \( 2^{< \lambda} = \lambda \), \( X \) is a \(\lambda\)-Polish space, and \( Y \) is any \(\lambda\)-Borel space of size at most \( \lambda \), then every \(\lambda\)-Borel subset of \( X \times Y \) has a \(\lambda\)-Borel uniformization --- this is basically the ordinal \(\lambda\)-uniformization property, once we notice that without loss of generality the space \( Y \) can be given the discrete topology. This implies that under the same assumption on \(\lambda\), if \( X,Y \) are standard \(\lambda\)-Borel spaces and \( A \subseteq X \times Y \) is a \(\lambda\)-Borel set whose projection on the second coordinate has size%
\footnote{By Proposition~\ref{prop:charanalytic}, such projection is a \(\lambda\)-analytic subset of \( Y \). It will follow from Corollary~\ref{cor:analyticPSP} that either it is well-orderable and of size at most \( \lambda \), or else it has size \( 2^\lambda \).}
at most \( \lambda \), then \( A \) has a \(\lambda\)-Borel uniformization.

In this section we prove some other uniformization results for \(\lambda\)-Borel sets, mirroring corresponding classical facts. All proofs share the fact that, thanks to the results in Section~\ref{sec:basicpropertiesforpolish}, one can without loss of generality work with (often closed) subsets of \( X \times B(\lambda) \) --- see in particular Theorems~\ref{thm:changeoftopologyforuniformizations}, \ref{thm:ctblsections}, and~\ref{thm:compactsections}. 
As usual, to minimize the requirements on \(\lambda\) some of the results are stated for \(\lambda\)-Polish spaces only, but they work for arbitrary standard \(\lambda\)-Borel spaces (whenever this makes sense) as soon as we assume \( 2^{< \lambda} = \lambda \).


\subsection{Jankov-von Neumann uniformization}

The first and easiest uniformization result is the analogue of the Jankov-von Neumann uniformization theorem~\cite[Theorem 18.1]{Kechris1995}, and actually it applies to arbitrary \(\lambda\)-analytic sets. Given a standard \(\lambda\)-Borel space \( X \), let \( \mathcal{B}_\lambda(\lS^1_1) (X) \) be the \(\lambda\)-algebra generated by the \(\lambda\)-analytic subsets of \( X \). If \( Y \) is a \(\lambda\)-Borel space, we say that a function \( f \colon X \to Y \) is \markdef{\( \mathcal{B}_\lambda(\lS^1_1) \)-measurable} if \( f^{-1}(B) \in \mathcal{B}_\lambda(\lS^1_1)(X) \) for every \( B \in \lB(Y) \). Since the projection of a \(\lambda\)-analytic set is still \(\lambda\)-analytic, and hence in \( \mathcal{B}_\lambda(\lS^1_1) \), the proof of Proposition~\ref{prop:borelvsgraph} shows that if a \(\lambda\)-analytic set \( P \subseteq X \times Y \) admits a \( \mathcal{B}_\lambda(\lS^1_1) \)-measurable uniformizing function, then \( P \) has also a \( \mathcal{B}_\lambda(\lS^1_1) \) uniformization.

\begin{theorem} \label{thm:Jankov-vonNeumann}
Let \( X,Y \) be \(\lambda\)-Polish spaces. Then for every \(\lambda\)-analytic \( P \subseteq X \times Y \) there is a \( \mathcal{B}_\lambda(\lS^1_1) \)-measurable uniformizing function for \( P \). In particular, \( P \) has a \( \mathcal{B}_\lambda(\lS^1_1) \) uniformization.
\end{theorem} 

\begin{proofsketchcite}{Kechris1995}{Theorem 18.1}
We can assume \( P \neq \emptyset \), as otherwise there is nothing to prove. 
Using~\ref{prop:charanalytic-3} of Proposition~\ref{prop:charanalytic},
let \( h \colon B(\lambda) \to X \times Y \) be continuous and such that \( h(B(\lambda)) = P \). Let also \( g \colon F \to X \) be a continuous bijection with \( F \subseteq B(\lambda) \) closed (Proposition~\ref{prop:surjection}). By Remark~\ref{rmk:inverseofinjection} we can assume that \( g^{-1} \) is \(\lambda\)-Borel measurable. Let \( P' \subseteq B(\lambda) \times B(\lambda) \) be the closed set defined by
\[ 
(x,y) \in P' \iff x \in F \wedge \p(h(y)) = g(x),
 \] 
where as usual \( \p \colon X \times Y  \to X \) is the projection on the first coordinate. If \( f' \colon \p(P') \to B(\lambda) \) is a \( \mathcal{B}_\lambda(\lS^1_1) \)-measurable uniformizing function for \( P' \), then the function \( f \colon \p(P) \to Y \) sending \( x \in \p(P) \) to the projection on \( Y \) of \( h(f'(g^{-1}(x))) \) is the desired \( \mathcal{B}_\lambda(\lS^1_1) \)-measurable uniformizing function for \( P \).

The previous paragraph shows that without loss of generality we can assume that \( X = Y = B(\lambda) \) and that \( P \in \lP^0_1(X \times Y) \). Let \( T \) be an \(\omega\)-tree on \( \lambda \times \lambda \) such that \( P = [T] \). Then the map sending each \( x \in \p(P) \) to the leftmost branch of the section tree \( T(x) \) is a \( \mathcal{B}_\lambda(\lS^1_1) \)-measurable uniformizing function for \( P \).
\end{proofsketchcite}

Theorem~\ref{thm:Jankov-vonNeumann} cannot be improved: there are even closed sets \( P \subseteq X \times Y \) such that \( \p(P) = X \), yet there is no \(\lambda\)-Borel uniformization of \( P \) (see Example~\ref{xmp:notunifomizable}).

\begin{corollary}
Let \( X,Y \) be \(\lambda\)-Polish spaces. Every \(\lambda\)-Borel function \( f \colon X \to Y \) admits a \( \mathcal{B}_\lambda(\lS^1_1) \)-measurable right-inverse, that is, a \( \mathcal{B}_\lambda(\lS^1_1) \)-measurable map \( g \colon f(X) \to X \) such that \( (f \circ g)(y) = y \) for all \( y \in f(X) \). 
\end{corollary}

\subsection{A general criterion}

 We now move to
 \(\lambda\)-Borel uniformizations of \(\lambda\)-Borel sets. This is a highly nontrivial task, because as in the classical case there are even very simple \(\lambda\)-Borel sets which do not admit a \(\lambda\)-Borel uniformization.

 \begin{example} \label{xmp:notunifomizable}
If \( 2^{< \lambda} = \lambda \), then there is a closed \( F \subseteq B(\lambda) \times B(\lambda) \) such that \( \p(F) = B(\lambda) \) but there is no \(\lambda\)-Borel uniformization for \( F \). Indeed, let \( C_0, C_1 \subseteq B(\lambda) \) be disjoint \(\lambda\)-coanalytic sets that cannot be separated by a \(\lambda\)-Borel set (the existence of such sets is granted by Theorem~\ref{thm:structuralpropertiesforcoanalytic}, which will be proved later on). 
For \( i \in \{ 0,1 \} \) let \( F_i \subseteq B(\lambda) \times B(\lambda) \) be closed such that \( X \setminus C_i = \p(F_i) \), and let \( F = F_0 \cup F_1 \). It is not hard to check that \( F \) is as required because if \( f \) were a \(\lambda\)-Borel uniformizing function for \( F \), then the  \(\lambda\)-Borel set \( f^{-1}(F_1) \) would separate \( C_0 \) from \( C_1 \).
 \end{example}

We
  first establish a quite general criterion along the lines of Theorem~\ref{thm:changeoftopology}  for the existence of a \(\lambda\)-Borel uniformization of a given \(\lambda\)-Borel set. It works also in the classical case \( \lambda = \omega \) and in that setup it might be folklore, although we could not trace it back in the literature.
If \( B \) is a subset of a topological space \( X \), we say that \( B \) has an isolated point if there is \( x \in B \) which is isolated with respect to the relative topology on \( B \) inherited from \( X \).
 
 \begin{theorem} \label{thm:changeoftopologyforuniformizations}
Let \( X,Y \) be \(\lambda\)-Polish spaces, and let \(  A \subseteq X \times Y \) be a nonempty \(\lambda\)-Borel set. Then the following are equivalent:
\begin{enumerate-(1)}
\item \label{thm:changeoftopologyforuniformizations-1}
\( A \) has a \(\lambda\)-Borel uniformization;
\item \label{thm:changeoftopologyforuniformizations-2}
there is a continuous \(\lambda\)-Borel isomorphism \( h \colon H \to X \times Y \) with \( H \subseteq B(\lambda) \) closed such that the set
\begin{equation*} \label{eq:changeoftopologyforuniformizations}
F = \{ (x,z) \in X \times B(\lambda) \mid z \in H  \wedge h(z) \in A \wedge \p(h(z)) = x \}
 \end{equation*}
is closed and all its nonempty vertical sections \( F_x \) have an isolated point;
\item \label{thm:changeoftopologyforuniformizations-3}
there is a closed set \( F \subseteq X \times B(\lambda) \) and a continuous \(\lambda\)-Borel isomorphism \( f \colon F \to A \) such that \( \p(w) = \p(f(w)) \) for all \( w \in F \) and all nonempty vertical sections \( F_x \) of \( F \) have an isolated point;
\item \label{thm:changeoftopologyforuniformizations-4}
there is a continuous surjection \( g \colon B(\lambda) \to Y \) and a closed set \( F \subseteq X \times B(\lambda) \) such that \( (\id_X \times \, g) (F) = A \) and all nonempty vertical sections \( F_x \) of \( F \) have an isolated point;
\item \label{thm:changeoftopologyforuniformizations-5}
the topology of \( X \times Y \) can be refined to a \(\lambda\)-Polish topology \( \tau' \) generating the same \(\lambda\)-Borel sets such that \( A \) is \( \tau' \)-closed and all its nonempty vertical sections \( A_x \), when identified with the subspace \( \{ x \} \times A_x \) of \( (X \times Y, \tau') \), have an isolated point.
\end{enumerate-(1)}
\end{theorem}

Before proving Theorem~\ref{thm:changeoftopologyforuniformizations}, we need a technical lemma. By adapting in the obvious way the proof of~\cite[Theorem 18.11]{Kechris1995}, one obtains a generalization of Luzin's theorem on the set of unicity of a Borel set.

\begin{lemma} \label{lem:setofuniqueness}
Assume that \( 2^{ < \lambda} = \lambda \), let \( X,Y \) be standard \(\lambda\)-Borel spaces, and let \( R \subseteq X \times Y \) be \(\lambda\)-Borel. Then the set
\[ 
\{ x \in X \mid \exists ! y \, (x,y) \in R \}
 \] 
is \(\lambda\)-coanalytic.
\end{lemma}

Combining this with Corollary~\ref{cor:novikovseparation1}, we obtain the following technical uniformization result, which will be subsumed by Theorem~\ref{thm:changeoftopologyforuniformizations} once the proof of the latter will be complete.

\begin{proposition} \label{prop:isolatedgivesuniformization}
Assume that \( 2^{< \lambda} = \lambda \), let \( X,Y \) be \(\lambda\)-Polish spaces, and let
 \( A \in \lB(X \times Y) \).  
Suppose that there is a family \( \{ U_\alpha \mid \alpha < \lambda \} \) of \(\lambda\)-Borel subsets of \( X \times Y \) such that for every \( x \in \p(A) \) there is \( \alpha < \lambda \) such that \( U_\alpha \cap (\{x \} \times A_x) \) is a singleton.
Then there is a  \( \lambda \)-Borel uniformization \( A^* \) of \( A \).
\end{proposition}

\begin{proof}
For each \( \alpha < \lambda \), consider the set \( R_\alpha = A \cap U_\alpha \in \lB(X \times Y, \tau) \).
By Lemma~\ref{lem:setofuniqueness}, the set \( C_\alpha = \{ x \in X \mid \exists! y \, (x,y) \in R_\alpha \} \) is \(\lambda\)-coanalytic. By our hypotheses, we have \( \p(A) = \bigcup_{\alpha < \lambda} C_\alpha \), hence \( \p(A) \) is \(\lambda\)-coanalytic by Corollary~\ref{cor:closurepropertiesofanalytic}; since obviously \( \p(A) \in \lS^1_1(X) \), then \( \p(A) \) is \(\lambda\)-Borel by Theorem~\ref{thm:souslin}. It follows that \( (C_\alpha)_{\alpha < \lambda} \) is a covering of the standard \(\lambda\)-Borel space \( \p(A) \) with \(\lambda\)-coanalytic sets. By Corollary~\ref{cor:novikovseparation1} we can refine such a covering to a weak partition of \( \p(A) \) in \(\lambda\)-Borel sets \( (B_\alpha)_{\alpha < \lambda} \) such that \( B_\alpha \subseteq C_\alpha \). Let \( A_\alpha = \{ (x,y) \in R_\alpha \mid x \in B_\alpha \}  = R_\alpha \cap (B_\alpha \times Y)\subseteq A \), so that \( A_\alpha \in \lB(X \times Y, \tau) \). Since \( B_\alpha \subseteq C_\alpha \), for each \( x \in B_\alpha \) there is a unique \( y \in Y \) such that \( (x,y) \in A_\alpha \): it follows that \( A^* = \bigcup_{\alpha < \lambda} A_ \alpha \) is a \(\lambda\)-Borel uniformization of \( A \), as required. 
\end{proof}

An important special case in which Proposition~\ref{prop:isolatedgivesuniformization} can be applied is when each nonempty vertical section \( A_x \) has an isolated point (with respect to the topology of \( Y \)). 
We are now ready to prove our characterization of the existence of \(\lambda\)-Borel uniformizations of a given \(\lambda\)-Borel set.

\begin{proof}[Proof of Theorem~\ref{thm:changeoftopologyforuniformizations}]
\ref{thm:changeoftopologyforuniformizations-1} \( \Rightarrow \) \ref{thm:changeoftopologyforuniformizations-2}. 
Let \( A^* \) be a \(\lambda\)-Borel uniformization of \( A \). By applying Corollary~\ref{cor:changeoftopology} to the family of \(\lambda\)-Borel sets \( \{ A^* , A \setminus A^* \} \), there is a closed set \( H \subseteq B(\lambda) \) and a continuous \(\lambda\)-Borel isomorphism \( h \colon H \to X \times Y \) such that \( h^{-1}(A^*) \) and \( h^{-1}(A \setminus A^* ) \) are both clopen in \( H \), and hence closed in \( B(\lambda) \). Since \( h(z) \in A \) if and only if \( z \in  h^{-1}(A^*) \cup h^{-1}(A \setminus A^* ) \), it follows that the set \( F \) in condition~\ref{thm:changeoftopologyforuniformizations-2} is closed. Given any \( x \in \p(F) = p(A) \), let \( y \in Y \) be the unique point such that \( (x,y) \in A^* \). Let also \( z \in H \) be the unique point such that \( h(z) = (x,y) \). Then \( z \in F_x \subseteq H \). Since \( h^{-1}(A^*) \) is open relatively to \( H \) and since for all \( z' \in F_x \setminus \{ z \} \) we have \( h(z') \notin A^* \) (because \( A^* \) was a uniformization of \( A \)), we have that \( h^{-1}(A^*) \cap F_x \) is a relatively open set isolating \( z \) in \( F_x \).

\ref{thm:changeoftopologyforuniformizations-2} \( \Rightarrow \) \ref{thm:changeoftopologyforuniformizations-3}.
Let \( h \colon H \to X \times Y \) and \( F \) be as in~\ref{thm:changeoftopologyforuniformizations-2}. Let \( f \colon F \to A \) be defined by \( f(x,z) = h(z) \) for every \( (x,z) \in F \). Using also Theorem~\ref{thm:borelvsgraph}, it is easy to verify that \( f \) is as required.  

\ref{thm:changeoftopologyforuniformizations-2} \( \Rightarrow \) \ref{thm:changeoftopologyforuniformizations-4}.
Let \( h \colon H \to X \times Y \) and \( F \) be as in~\ref{thm:changeoftopologyforuniformizations-2}. Let \( r \colon B(\lambda) \to H \) be a retraction onto \( H \) and \( \p_2 \colon X \times Y \to Y \) be the projection on the second coordinate. Then \( g = \p_2 \circ h \circ r \) and the same \( F \) witness~\ref{thm:changeoftopologyforuniformizations-4}.

\ref{thm:changeoftopologyforuniformizations-2} \( \Rightarrow \) \ref{thm:changeoftopologyforuniformizations-5}.
Let \( h \colon H \to X \times Y \) and \( F \) be as in~\ref{thm:changeoftopologyforuniformizations-2}. Set also
\[ 
F' = \{ (x,z) \in X \times B(\lambda) \mid z \in H \wedge \p(h(z)) = x \}.
 \] 
Then \( F ' \) is closed, \( F \subseteq F' \), and the map \( f' \colon F' \to X \times Y \) sending \( (x,z) \in F' \) to \( h(z) \) is a continuous \(\lambda\)-Borel isomorphism such that for all \( x \in \p(F) = \p(A) \)  it holds \( f' (\{ x \} \times F_x) = \{ x \} \times A_x \), so that in particular \( f'(F) = A \). It easily follows that the push forward \( \tau' \) of the topology of \( F' \) along \( f' \) is as desired.

We are left with proving that each of~\ref{thm:changeoftopologyforuniformizations-3}, \ref{thm:changeoftopologyforuniformizations-4}, and~\ref{thm:changeoftopologyforuniformizations-5} implies~\ref{thm:changeoftopologyforuniformizations-1}.
The implication~\ref{thm:changeoftopologyforuniformizations-5} \( \Rightarrow \) \ref{thm:changeoftopologyforuniformizations-1} directly follows from Proposition~\ref{prop:isolatedgivesuniformization},
once we let \( \{ U_\alpha \mid \alpha < \lambda \} \) be any \(\lambda\)-sized basis for \( \tau' \).
For the remaining implications \ref{thm:changeoftopologyforuniformizations-3} \( \Rightarrow \) \ref{thm:changeoftopologyforuniformizations-1} and \ref{thm:changeoftopologyforuniformizations-4} \( \Rightarrow \) \ref{thm:changeoftopologyforuniformizations-1} we first apply Proposition~\ref{prop:isolatedgivesuniformization} to the set \( F \subseteq X \times B(\lambda) \) with \( \{ U_\alpha \mid \alpha < \lambda \} \) any \(\lambda\)-sized basis for the product topology on \( X \times B(\lambda) \) to obtain a \(\lambda\)-Borel uniformization of \( F^* \). Since both the restriction of the function \( f \) from~\ref{thm:changeoftopologyforuniformizations-3} and the restriction of the function \( \id_X \times \, g \) from~\ref{thm:changeoftopologyforuniformizations-4} to the \(\lambda\)-Borel set \( F^* \) are necessarily injective \(\lambda\)-Borel functions, it follows from Theorem~\ref{thm:injectiveBorelimage} that their range \( A^* \subseteq A \) is \(\lambda\)-Borel, and by construction it is a uniformization of \( A \).
\end{proof}

\subsection{\(\lambda\)-Borel uniformizations for sets with countable sections} \label{subsec:ctblsections}

We next move to \(\lambda\)-Borel sets \( A \subseteq X \times Y \) all of whose vertical sections \( A_x \) are countable. The following result naturally generalizes the Luzin-Novikov's theorem (see~\cite[Theorem 18.10]{Kechris1995}). Our proof is necessarily different from that of~\cite[Theorem 18.10]{Kechris1995} because, as already observed, we lack a Baire category notion matching well with that of  \(\lambda\)-Borel sets. Also, notice that part~\ref{thm:ctblsections-3} is obviously implied by part~\ref{thm:ctblsections-4}, but we do not know how to obtain the latter without first proving the former. 

\begin{theorem} \label{thm:ctblsections}
Assume that \( 2^{< \lambda} = \lambda \). Let \( X \) be a standard \(\lambda\)-Borel space, and let \( Y \) be a \(\lambda\)-Polish space. Let \( A \subseteq X \times Y \) be a \(\lambda\)-Borel set all of whose vertical sections are countable. 
\begin{enumerate-(i)}
\item \label{thm:ctblsections-1}
There is a \(\lambda\)-Borel uniformization  of \( A \), hence in particular \( \p(A) \) is \(\lambda\)-Borel.
\item \label{thm:ctblsections-2}
The map from \( \p(A) \) to \( F(Y) \) sending \( x \in \p(A) \) to \( \mathrm{cl}(A_x) \) is \(\lambda\)-Borel.
\item \label{thm:ctblsections-3}
There is a sequence \( (\varsigma^A_n)_{n \in \omega} \) of \(\lambda\)-Borel functions \( \varsigma^A_n \colon\p(A) \to Y \) such that for all \( x \in \p(A) \) the set \( \{ \varsigma^A_n(x) \mid n \in \omega \} \) is dense in \( A_x \).
\item \label{thm:ctblsections-4}
There is a sequence \( (\varrho^A_n)_{n \in \omega} \) of \(\lambda\)-Borel functions \( \varrho^A_n \colon\p(A) \to Y \) such that \( A_x = \{ \varrho^A_n(x) \mid n \in \omega \}  \) for all \( x \in \p(A) \).
\item \label{thm:ctblsections-5}
The set \( A \) can be written as \( A = \bigcup_{n \in \omega} P_n \), where the sets \( P_n \) are pairwise disjoint \(\lambda\)-Borel graphs of partial (necessarily \(\lambda\)-Borel) functions.
\end{enumerate-(i)}
\end{theorem}

\begin{proof}
Without loss of generality, we may assume that \( X \) is \(\lambda\)-Polish too.

\ref{thm:ctblsections-1}
It is enought to show that~\ref{thm:changeoftopologyforuniformizations-5} of Theorem~\ref{thm:changeoftopologyforuniformizations} is satisfied.
Apply Corollary~\ref{cor:changeoftopologyclopen} within the \(\lambda\)-Polish space \( X \times Y \) to turn \( A \) into a \( \tau' \)-closed set, for a suitable \(\lambda\)-Polish topology \( \tau' \). Each nonempty vertical section \( A_x \), when equipped with the topology induced by \( \tau' \) via the identification of \( A_x \) with \( \{ x \} \times A_x \), is a countable (classical!) Polish spaces because \( \{ x \} \times A_x \) is \( \tau' \)-closed, and hence it surely has isolated points. 
%

\ref{thm:ctblsections-2}
For any open \( U \subseteq Y \) and \( x \in X \) we have 
\[ 
\mathrm{cl}(A_x) \cap U \neq \emptyset \iff A_x \cap U \neq \emptyset \iff x \in \p(A \cap (X \times U)).
 \] 
Since \( A \cap (X \times U) \) is a \(\lambda\)-Borel set with countable vertical sections, it follows that \( \p(A \cap (X \times U)) \) is \(\lambda\)-Borel by part~\ref{thm:ctblsections-1}, hence we are done.

\ref{thm:ctblsections-3}
We first observe that it is enough to consider the case where \( A \) is a (closed) subset of \( X \times B(\lambda) \). 
Since we already proved~\ref{thm:ctblsections-1}, we can use~\ref{thm:changeoftopologyforuniformizations-3} of Theorem~\ref{thm:changeoftopologyforuniformizations} to get a closed set \( F \subseteq X \times B(\lambda) \) and a continuous \(\lambda\)-Borel isomorphism \( f \colon F \to A \) such that \( \p(w) = \p(f(w)) \) for every \( w \in F \), so that in particular \( \p(F) = \p(A) \) and \( F \) has countable vertical sections too. If \( (\varsigma^F_n)_{n \in \omega } \) are as in~\ref{thm:ctblsections-3} with respect to the set \( F \), then it is enough to let \( \varsigma^A_n(x) = \p_2(f(x,\varsigma^F_n(x))) \), 
where \( \p_2 \) is again the projection on the second coordinate, to get the desired result.


So without loss of generality we may assume that \( Y = B(\lambda) \) and \( A \subseteq X \times B(\lambda) \) be closed.
Fix any bijection \( \sigma \colon \pre{< \omega}{\lambda} \to \lambda \) and for each infinite \( \beta < \omega_1 \) fix a strict well-ordering \( <_\beta \) of \(\omega\) of order type \(\beta\).
Let \( G \subseteq X \times B(\lambda) \) be the set of pairs \( (x,z) \in X \times B(\lambda) \) such that \( x \in \p(A) \), 
\[ 
\forall n,m \in \omega \, (n <_\beta m \iff z(n) < z(m)) 
\] 
for some (necessarily unique) \( \beta < \omega_1 \), and for all \( s \in \pre{< \omega}{\lambda} \)
\[ 
A_x \cap \Nbhd_s \neq \emptyset \iff \exists n \in \omega \, (z(n) = \sigma(s)).
 \] 
Clearly \( G \) is a \(\lambda\)-Borel set by parts~\ref{thm:ctblsections-1} and~\ref{thm:ctblsections-2} and \( \omega_1 \leq \lambda \) (notice that \( A_x = \mathrm{cl}(A_x) \) since \( A \) is closed). Moreover,
since each \( A_x \) is countable then for every \( x \in \p(A) \) there are exactly \( \aleph_0 \)-many \( \alpha < \lambda \) such that \( A_x \cap \Nbhd_{\sigma^{-1}(\alpha)} \neq \emptyset \), hence the set of such ordinals \(\alpha\) will be well-ordered by the ordering of \( \lambda \) in order type \( \beta \) for a unique \( \omega \leq \beta < \omega_1 \). It follows that there is a unique \( z \in B(\lambda) \) such that \( (x,z) \in G \). Thus \( G \) is the graph of a \(\lambda\)-Borel function \( g \colon \p(A) \to B(\lambda) \).
For each \( \alpha < \lambda \) let \( f_\alpha \colon \p(A \cap (X \times \Nbhd_{\sigma^{-1}(\alpha)})) \to B(\lambda)\) be a \(\lambda\)-Borel uniformizing function for the \(\lambda\)-Borel set \(  A \cap (X \times \Nbhd_{\sigma^{-1}(\alpha)}) \), which exists by part~\ref{thm:ctblsections-1} because \( A \cap (X \times \Nbhd_{\sigma^{-1}(\alpha)}) \subseteq A \) is \(\lambda\)-Borel and has countable vertical sections. Define \( \varsigma^A_n \) by setting \( \varsigma_n^A(x) = f_{ g(x)(n)}(x) \) for each \( x \in \p(A) \).
By definition of \( G \), for each \( x \in \p(A) \) the ordinal \( \alpha = g(x)(n) \) is such that \( A_x \cap \Nbhd_{\sigma^{-1}(\alpha)} \neq \emptyset \), thus \( f_\alpha \) is defined on \( x \) and \( \varsigma^A_n(x) = f_\alpha(x) \in  A_x \cap \Nbhd_{\sigma^{-1}(\alpha)}  \). This shows in particular that each map \( \varsigma^A_n \) is defined on the entire \( \p(A) \). To see that it is \(\lambda\)-Borel, it is enough to observe that for every \( s \in \pre{<\omega}{\lambda} \) and \( x \in \p(A) \) we have \( \varsigma^A_n(x) \in \Nbhd_s \) if and only if \( \exists \alpha < \lambda \, (g(x)(n) = \alpha \wedge f_\alpha(x) \in \Nbhd_s ) \), and that the latter condition is \(\lambda\)-Borel because the functions \( g \) and \( f_\alpha \) are all \(\lambda\)-Borel. Finally, to show that \( \{ \varsigma^A_n(x) \mid n \in \omega \} \) is dense in \( A _x \) notice that if \( A_x \cap \Nbhd_s \neq \emptyset \), then by construction there is \( n \in \omega \) such that \( g(x)(n) = \sigma(s) \), and hence \( \varsigma^A_n(x) \in A_x \cap \Nbhd_s  \).

\ref{thm:ctblsections-4}
Arguing as in the first paragraph of the proof of~\ref{thm:ctblsections-3},
without loss of generality we can again assume that \( A \) be closed. 
(Actually, having that all vertical sections \( A_x \) are closed would be enough for the ensuing argument.)
Let \( (\varsigma^A_n)_{n \in \omega} \) be as in part~\ref{thm:ctblsections-3}, and fix a bijection \( h \colon 2^{\aleph_0} \to \pre{\omega}{\omega} \).
Given \( \alpha < 2^{\aleph_0} \) let \(\vartheta^A_\alpha \colon \p(A) \to Y \) be defined by setting for \( x \in \p(A) \)
\[ 
\vartheta^A_\alpha(x) = 
\begin{cases}
\lim_{n \to \infty} \varsigma^A_{h(\alpha)(n)}(x) & \text{if } (\varsigma^A_{h(\alpha)(n)}(x))_{n \in \omega} \text{ is a Cauchy sequence} \\
\varsigma^A_0(x) & \text{otherwise}.
\end{cases}
 \] 
Then each \( \vartheta^A_\alpha \) is \(\lambda\)-Borel and for each \( x \in \p(A) \)
\[ 
A_x = \{ \vartheta^A_\alpha(x) \mid \alpha < 2^{\aleph_0} \}
 \] 
because \( A_x \) is closed and \( \{ \varsigma^A_n(x) \mid n  \in \omega \} \) is dense in it.
Now for every  \( e \colon \omega \to 2^{\aleph_0} \) let \( X_e \) be the set of those \( x \in \p(A) \) such that for every \( \alpha < 2^{\aleph_0} \) there is \( n \in \omega \) such that \( \vartheta^A_\alpha(x) = \vartheta^A_{e(n)}(x) \). Then \( X_e \) is \(\lambda\)-Borel 

because \(  2^{\aleph_0} < \lambda \), all the maps \( \vartheta^A_\alpha \) are \(\lambda\)-Borel, and \( \p(A) \) is \(\lambda\)-Borel by part~\ref{thm:ctblsections-1}. Moreover, \( \p(A) = \bigcup \{ X_e \mid e \in \pre{\omega}{(2^{\aleph_0})}  \} \) because each \( A_x \) is countable, and for \( x \in X_e \) we have \( A_x = \{ \vartheta^A_{e(n)}(x) \mid n \in \omega \} \). Since there are less than \( \lambda \)-many functions \( e \colon \omega \to 2^{\aleph_0} \) by our hypothesis on \(\lambda\), we can easily refine the covering \( (X_e)_{e \in \pre{\omega}{(2^{\aleph_0})}} \) of \( \p(A) \) to a weak partition \( (B_e)_{e \in \pre{\omega}{(2^{\aleph_0})}} \) consisting of \(\lambda\)-Borel sets  satisfying \( B_e \subseteq X_e \). Then it is enough to set, for each \( x \in \p(A) \),  \( \varrho^A_n(x) = \vartheta^A_{e(n)}(x)\) where \( e \in \pre{\omega}{(2^{\aleph_0})} \) is unique such that \( x \in B_e \).

\ref{thm:ctblsections-5}
Let \( (\varrho^A_n)_{n \in \omega} \) be as in part~\ref{thm:ctblsections-4} and set 
\[ 
P_n = \mathrm{graph}(\varrho^A_n) \setminus \bigcup_{m <n} \mathrm{graph}(\varrho^A_m). \qedhere
 \] 
\end{proof}

Recall that an equivalence relation \( E \) is called \markdef{countable} if all its equivalence classes have size at most \( \aleph_0 \). In particular, if \( E \) is a countable equivalence relation on a \(\lambda\)-Polish space, then all its \( E \)-equivalence classes are at most \( F_\sigma \). It follows that if such a class \( [x]_E \) is also \( G_\delta \), as it happens in~\ref{prop:ctblBorelER-3} of Proposition~\ref{prop:ctblBorelER}, then it is a transfinite countable difference of open subsets of the (classical!) Polish space \( \mathrm{cl}([x]_E) \) by~\cite[Theorem 22.27]{Kechris1995}, and hence the same is true with respect to the whole space \( X \).

\begin{proposition} \label{prop:ctblBorelER} 
Assume that \( 2^{< \lambda} = \lambda \), and let \( E \) be a countable \(\lambda\)-Borel equivalence relation on a standard \(\lambda\)-Borel space \( X \). Then the following are equivalent:
\begin{enumerate-(1)}
\item \label{prop:ctblBorelER-1}
\( E \) is potentially closed;
\item \label{prop:ctblBorelER-2}
there is a \(\lambda\)-Polish topology on \( X \) generating its \(\lambda\)-Borel structure such that
all \( E \) equivalence classes are closed;
\item \label{prop:ctblBorelER-3}
there is a \(\lambda\)-Polish topology on \( X \) generating its \(\lambda\)-Borel structure such that
all \( E \) equivalence classes are \( G_\delta \);
\item \label{prop:ctblBorelER-4}
\( E \) is smooth;
\item \label{prop:ctblBorelER-5}
\( E \) has a \(\lambda\)-Borel selector (equivalently: a \(\lambda\)-Borel transversal).
\end{enumerate-(1)}
\end{proposition}

In particular, if \( 2^{< \lambda} = \lambda \) and \( E \) is a countable \(\lambda\)-Borel equivalence relation on a standard \(\lambda\)-Borel space \( X \) such that all its equivalence classes are finite (and hence closed), then \( E \) has a \(\lambda\)-Borel selector.

\begin{proof}
Conditions~\ref{prop:ctblBorelER-1} and~\ref{prop:ctblBorelER-2} are equivalent by Corollary~\ref{cor:closedclasses-closedrelation}, while \ref{prop:ctblBorelER-2}~\( \Rightarrow \)~\ref{prop:ctblBorelER-3} is obvious. The implication \ref{prop:ctblBorelER-3}~\( \Rightarrow \)~\ref{prop:ctblBorelER-4} can be proved as in Proposition~\ref{prop:smoothnessunderappropriatehypotheses}, once we notice that we already know that the map \( x \mapsto \mathrm{cl}([x]_E) \) is \(\lambda\)-Borel because we can apply Theorem~\ref{thm:ctblsections}\ref{thm:ctblsections-2} with \( A = E \). For the implication \ref{prop:ctblBorelER-2}~\( \Rightarrow \)~\ref{prop:ctblBorelER-5}, we use the same observation and compose the \(\lambda\)-Borel map \( x \mapsto \mathrm{cl}([x]_E) = [x]_E \) with the \(\lambda\)-Borel function \( \sigma^X_0 \colon F(X) \to X \) from Theorem~\ref{thm:selectionforF(X)} to get the desired \(\lambda\)-Borel selector. Finally, \ref{prop:ctblBorelER-5}~\( \Rightarrow \)~\ref{prop:ctblBorelER-4} is obvious, while \ref{prop:ctblBorelER-4}~\( \Rightarrow \)~\ref{prop:ctblBorelER-1} follows from the argument at the end of Subsection~\ref{subsec:Novikov}.
\end{proof}

%

Another consequence of Theorem~\ref{thm:ctblsections}\ref{thm:ctblsections-5} is the analogue of the celebrated Feldman-Moore theorem. To deduce it from the mentioned result, we cannot simply adapt to our context the standard argument from e.g.\ \cite[Theorem 7.1.4]{Gao2009}, as our spaces have uncountable weight. To overcome this difficulty, we will exploit the following observation, which can be used as an alternative also in the classical setup \( \lambda = \omega \).

\begin{lemma} \label{lem:partial bijections}
Let \( X \) be a \(\lambda\)-Polish space. If \( Q \in \lB(X \times X) \) is the graph of an injective partial function \( f \) from \( X \) into itself, then \( Q \) can be written as \( Q = \bigcup_{k \in \omega} R_k \), where the sets \( R_k \) are pairwise disjoint, \(\lambda\)-Borel, and such that the two projections \( \p(R_k) \) and \( p_2(R_k) \) of \( R_k \) are either disjoint or equal.
\end{lemma}

\begin{proof}
By Theorems~\ref{thm:injectiveBorelimage} and~\ref{thm:borelvsgraph} (together with Remark~\ref{rmk:borelvsgraph}), the function \( f \) is a \(\lambda\)-Borel isomorphism between the \(\lambda\)-Borel sets \( \p(Q) = \dom(f) \) and \( \p_2(Q) = \ran(f) \).
Let \( B_0 = \p(Q)  \), and
\[ 
B_{k+1} =
\begin{cases}
B_k \cap f(B_k) & \text{if \( k \) is even}, \\
B_k \cap f^{-1}(B_k) & \text{if \( k \) is odd}.
\end{cases}
 \] 
Notice that \( B_k \supseteq B_{k+1} \), for every \( k \in \omega \).
Finally, let \( B_{\infty} = \bigcap_{k \in \omega} B_k \). Using Theorem~\ref{thm:injectiveBorelimage} again, it is easy to prove by induction on \( k \in \omega \) that each \( B_k \) is \(\lambda\)-Borel, so that \( B_\infty \in \lB(X) \) too. Let \( R_0 = Q \cap (B_\infty \times X) \) and \( R_{k+1} = Q \cap ((B_k \setminus B_{k+1}) \times X ) \): we claim that the sets \( R_k \), which are clearly \(\lambda\)-Borel, are as required.

Since the sets \( B_k \setminus B_{k+1} \) together with \( B_\infty \) form a weak partition of \( B_0 = \p(Q) \), we get that the sets \( R_k \) are pairwise disjoint and that \( Q = \bigcup_{k \in \omega} R_k \). 
By construction, \( \p(R_0) = B_\infty \) and \( \p(R_{k+1}) = B_k \setminus B_{k+1} \), and moreover \( \p_2(R_k) = f(\p(R_k)) \) for every \( k \in \omega \). 

It follows that, by construction, \( \p_2(R_0) = B_{\infty} \), so that \( \p_2(R_0) = \p(R_0) \). Indeed, assume first that \( y \in B_{\infty} \). Then \( y \in B_{k+1} \) for every even \( k \in \omega \), and so \( y \in f(B_k)  \) for every such \( k \). In particular, \( y \in \ran(f) \): let \(x \in \dom(f) \) be the only element such that \( f(x) = y \). Then \( x \in B_k \) for every even \( k \in \omega \), hence \( x \in B_{\infty} \) because the sets \( B_k \) are decreasing with respect to inclusion. Conversely, if \( y \in \p_2(R_0) \subseteq \ran(f) \) it means that the only \( x \) such that \( f(x) = y \) belongs to \( \p(R_0)  = B_\infty \), and in particular to \(B_{k+1} \) for every odd \( k \in \omega \). This means that \( y = f(x) \in B_k \) for every such \( k \), and hence \( y \in B_\infty \).

Finally, fix any \( k \in \omega \). In order to prove that \( \p(R_{k+1}) \cap \p_2(R_{k+1}) = \emptyset \) we need to check that \( f(B_k \setminus B_{k+1}) \cap (B_k \setminus B_{k+1}) =\emptyset \). 
Notice that \( B_k \setminus B_{k+1}  = B_k \setminus f(B_k) \) if \( k \) is even, and \( B_k \setminus B_{k+1} = B_k \setminus f^{-1}(B_k) \) if \( k \) is odd. Therefore, if \( k \) is even we get \( f(B_k \setminus f(B_k)) \cap (B_k \setminus f(B_k)) = \emptyset \) because \( f(B_k \setminus f(B_k)) \subseteq f(B_k) \); if \( k \) is odd, then \( f(B_k \setminus f^{-1}(B_k)) \cap (B_k \setminus f^{-1}(B_k)) = \emptyset \) because \( f(B_k \setminus f^{-1}(B_k)) \cap B_k = \emptyset \).
\end{proof}

Of course, the above lemma holds also in the context of standard \(\lambda\)-Borel spaces whenever we assume \( 2^{< \lambda} = \lambda \). 
We are now ready to prove the generalized Feldman-Moore theorem.

\begin{theorem} \label{thm:feldman-moore}
Assume that \( 2^{< \lambda} = \lambda \), and let  \( X \) be a standard \(\lambda\)-Borel space. Then \( E \) is a countable \(\lambda\)-Borel equivalence relation on \( X \) if and only if \( E \) is induced by a \(\lambda\)-Borel action%
\footnote{Indeed, our proof shows that we can even assume that \( G \) acts on \( X \) by \(\lambda\)-Borel automorphisms.} 
of a (discrete) countable group \( G \) on \( X \).
\end{theorem}



\begin{proof} 
One direction is obvious, so let us assume that \( E \) is a countable \(\lambda\)-Borel equivalence relation on \( X \). Without loss of generality, we can assume that \( X \) is a \(\lambda\)-Polish space. By Theorem~\ref{thm:ctblsections}\ref{thm:ctblsections-5} applied with \( A = E \), we get that \( E = \bigcup_{n \in \omega} P_n \) with \( \{ P_n \mid n \in \omega \} \) a family of pairwise disjoint \(\lambda\)-Borel sets which are graphs of partial functions from \( X \) into itself. For every \( n,m \in \omega \), consider the \(\lambda\)-Borel set \( Q_{n,m} = P_n \cap P_m^{-1} \). 
Then \( E = \bigcup_{n,m \in \omega} Q_{n,m} \), and the sets \( Q_{n,m} \) are still pairwise disjoint graphs of partial function \( f_{n,m} \) from \( X \) into itself. 
Moreover, each \( f_{n,m} \) is injective because \( Q_{n,m} \subseteq P^{-1}_m \) and the latter is the graph of a partial function. 

Next apply Lemma~\ref{lem:partial bijections} to each \( Q_{n,m} \) and write it as \( Q_{n,m} = \bigcup_{k \in \omega} R_{n,m,k} \), for suitable \(\lambda\)-Borel sets \( R_{n,m,k} \). 
By Theorems~\ref{thm:injectiveBorelimage} and~\ref{thm:borelvsgraph} (together with Remark~\ref{rmk:borelvsgraph}), each \( R_{n,m,k} \) is the graph of a \( \lambda \)-Borel isomorphism \( h_{n,m,k} \) between the \(\lambda\)-Borel sets \( A_{n,m,k} = \p(R_{n,m,k}) \) and \( B_{n,m,k} = \p_2(R_{n,m,k}) \), and either \( A_{n,m,k} = B_{n,m,k} \) or \( A_{n,m,k} \cap B_{n,m,k} \). In the first case, we let \( g_{n,m,k} \colon X \to X \) be defined by setting \( g_{n,m,k}(x) = h_{n,m,k}(x) \) if \( x \in A_{n,m,k} \) and \( g_{n,m,k}(x) = x \) otherwise; in the second case, we instead let \( g_{n,m,k} \colon X \to X \) be defined by \( g_{n,m,k}(x) = h_{n,m,k}(x) \) if \( A_{n,m,k} \), \( g_{n,m,k}(x) = h_{n,m,k}^{-1}(x) \) if \( x \in B_{n,m,k} \), and \( g_{n,m,k}(x) = x \) otherwise. Notice that each \( g_{n,m,k} \) is a well-defined \(\lambda\)-Borel isomorphism between \( X \) and itself, and moreover
\[ 
x \mathrel{E} y \iff g_{n,m,k}(x) = y \text{ for some \( n,m,k \in \omega \)}.
 \] 
 
Let \( G \) be the set of \(\lambda\)-Borel automorphisms obtained by closing under composition the family \( \{ g_{n,m,k} \mid n,m,k \in \omega \} \) and adding the identity map on \( X \) if necessary. Then \( G \) is a countable group with respect to the composition operation, and its natural action on \( X \) induces precisely the equivalence relation \( E \).
\end{proof}

Among the many important consequences of the Feldman-Moore theorem, we would like to mention at least Slaman-Steel's ``marker lemma'', that can now be generalized as follows.

\begin{lemma} \label{lem:markerlemma} 
Assume that \( 2^{< \lambda} = \lambda \), and let \( E \) be a countable \(\lambda\)-Borel equivalence relation on a standard \(\lambda\)-Borel space \( X \). Let
\( B = \{ x \in X \mid [x]_E \text{ is infinite} \} \).
Then there is a countable decreasing sequence \( B \supseteq S_0 \supseteq S_1 \supseteq S_2 \supseteq \dotsc \) of \(\lambda\)-Borel sets such that \( [S_n]_E = B \) for every \( n \in \omega \), yet \( \bigcap_{n \in \omega} S_n = \emptyset \).
\end{lemma} 

\begin{proofsketchcite}{Gao2009}{Lemma 7.1.5 and Exercise 7.1.4}
By Propositions~\ref{prop:charstandardBorel} and~\ref{prop:surjection} (together with Remark~\ref{rmk:inverseofinjection}), without loss of generality we can assume that \( X \) is a closed subset of \( B(\lambda) \), and in particular a \(\lambda\)-Polish space.
By Theorem~\ref{thm:feldman-moore} there is a countable group \( G = \{ g_n \mid n \in \omega \} \) and a \(\lambda\)-Borel action \( (g,x) \to g \cdot x \) of \( G \) on \( X \) inducing \( E \).
Therefore the \( E \)-invariant set \( B \) is \(\lambda\)-Borel, as
\[ 
x \in B \iff \forall n  \exists m  \forall k \leq n \, (g_m \cdot x \neq g_k \cdot x).
 \] 
By Theorem~\ref{thm:ctblsections}\ref{thm:ctblsections-2}, the map \( x \mapsto \mathrm{cl}([x]_E) \) is \(\lambda\)-Borel, hence so is the map \( f \) sending \( x \in X \) to \( \sigma_0^X(\mathrm{cl}([x]_E)) \), where \( \sigma_0^X \) is as in Theorem~\ref{thm:selectionforF(X)}.
Consider the \( E \)-invariant \(\lambda\)-Borel set \( A = \{ x \in X \mid f(x) \in [x]_E \} \). Then \( A \supseteq X \setminus B \), and \( f \) is a \(\lambda\)-Borel selector for the restriction of \( E \) to \( A \). For every \( n \in \omega \), let
\[ 
x \in S_n \iff (x \in A \cap B \wedge \forall m < n \, (g_m \cdot x \neq f(x))) \vee (x \notin A \wedge x \restriction n = f(x) \restriction n).
 \] 
Then the sets \( S_n \) are as required.
\end{proofsketchcite}

Let us also mention yet another corollary of Theorem~\ref{thm:ctblsections}.
A map \( f \colon X \to Y \) is called \markdef{countable-to-\( 1 \)} if \( f^{-1}(y) \) is (at most) countable for every \( y \in Y \). 

\begin{corollary}
Let \( \lambda \) be such that \( 2^{< \lambda} = \lambda \), and let \( X,Y \) be standard \(\lambda\)-Borel spaces. If a \(\lambda\)-Borel map \( f \colon X \to Y \) is countable-to-\( 1 \), then \( \ran(f) \) is \(\lambda\)-Borel and \( f \) admits a \(\lambda\)-Borel right-inverse.
\end{corollary}

\begin{proof}
Apply Theorem~\ref{thm:ctblsections}\ref{thm:ctblsections-1} to the set \( A = \{ (y,x) \in Y \times X \mid f(x) = y \} \), which is \(\lambda\)-Borel by Theorem~\ref{thm:borelvsgraph} and has countable vertical sections because \( f \) is countable-to-\( 1 \).
\end{proof}

Along the same lines, Theorem~\ref{thm:ctblsections}\ref{thm:ctblsections-2} also implies the following fact.

\begin{corollary}
Let \(\lambda\) be such that \( 2^{< \lambda} = \lambda \), \( X \) be a standard \(\lambda\)-Borel space, and \( Y \) be a \(\lambda\)-Polish space. Let \( f \colon X \to F(Y) \) be such that \( |f(x)| \leq \aleph_0 \) for all \( x \in X \). Then \( f \) is \(\lambda\)-Borel if and only if the set \( F = \{ (x,y) \in X \times Y \mid y \in f(x) \} \) is \(\lambda\)-Borel.
\end{corollary}

\subsection{\(\lambda\)-Borel uniformizations for sets with compact sections}\label{subsec:cpctsections}

We finally consider \(\lambda\)-Borel sets with compact vertical sections.
The following result is the analogue of~\cite[Theorem 28.8]{Kechris1995}. Our proof is again necessarily different, this time because non-separable (complete) metric spaces cannot have a metrizable compactification.
The argument we are going to present works also when \( \lambda = \omega \), hence it gives an alternative proof of~\cite[Theorem 28.8]{Kechris1995} which does not pass through compactification.

\begin{theorem} \label{thm:compactsections}
Assume that \( 2^{< \lambda} =  \lambda \). Let \( X \) be a standard \(\lambda\)-Borel space, and \( Y \) be a \(\lambda\)-Polish space. If \( A \subseteq X \times Y \) is a \(\lambda\)-Borel set all of whose vertical sections \( A_x \) are compact, then the map \( x \mapsto A_x \) is \(\lambda\)-Borel as a map from \( X \) to \( K(Y) \subseteq F(Y) \). Equivalently, a map \( f \colon X \to K(Y) \) is \(\lambda\)-Borel if and only if the set \( F = \{ (x,y) \in X \times Y \mid y \in f(x)  \}\) is \(\lambda\)-Borel.  
\end{theorem}

\begin{proof}
We first show that we can restrict the attention to the special case \( Y = B(\lambda) \). Indeed, if \( Y \) is an arbitrary nonempty \(\lambda\)-Polish space, then by Proposition~\ref{prop:surjectionpreservingcompact} there is a continuous surjection \( f \colon B(\lambda) \to Y \) with the property that the preimage of any compact set is still compact. It follows that 
\[ 
A' = \{ (x,y) \in X \times B(\lambda) \mid (x,f(y)) \in A \} 
\] 
is a \(\lambda\)-Borel set with compact sections. Suppose that \( x \mapsto A'_x \) is \(\lambda\)-Borel as a map from \( X \) to \( K(B(\lambda)) \). Since \( A'_x = f^{-1}(A_x) \) and \( f \) is surjective,  for any open \( U \subseteq Y \) we have \( A_x \cap U \neq \emptyset \) if and only if \( A'_x \cap f^{-1}(U) \neq \emptyset \), thus the function \( x \mapsto A_x \) is \(\lambda\)-Borel as well.

So assume \( Y = B(\lambda) \) from now on. For each \( 0 \neq n \in \omega \), 
let \( ( S^n_\alpha )_{\alpha < \lambda } \) be an enumeration without repetitions of 
all cofinite subsets of \( \pre{n}{\lambda} \), and for each \( \alpha < \lambda \) let 
\( U^n_\alpha = \bigcup_{s \in S^n_\alpha} \Nbhd_s \). We first claim that 
\( \mathcal{U} = \{ U^n_\alpha \mid 0 \neq n \in \omega \wedge \alpha < \lambda \} \) is a basis for 
\( \mathcal{V} = \{ B(\lambda) \setminus A_x \mid x \in X \} \). Indeed, given \( x \in X \) 
and \( y \notin A_x \), there is \( 0 \neq k \in \omega \) such that 
\( \Nbhd_{y \restriction k} \cap A_x =\emptyset \) because \( A_x \) is closed. Let 
\( S = \{ s \in \pre{k}{\lambda} \mid \Nbhd_s \cap A_x = \emptyset \} \): since \( A_x \) is 
compact and \( \Nbhd_t \cap \Nbhd_{t'} = \emptyset \) if \( \lh(t) = \lh(t') \) and 
\( t \neq t' \), it follows that \( S \) is cofinite and thus \( S = S^k_\alpha \) for some 
\( \alpha < \lambda \). By the choice of \( k \) and \( S \), it follows that \( y \in U^k_\alpha \subseteq B(\lambda) \setminus A_x \), 
as desired.

By Theorem~\ref{thm:rectangles} there are \(\lambda\)-Borel sets \( B^n_\alpha \subseteq X \) such that 
\[ 
(X \times B(\lambda)) \setminus A = \bigcup_{\substack{0 \neq n \in \omega \\ \alpha < \lambda}} (B^n_\alpha \times U^n_\alpha ) .
\]
Fix an arbitrary \( s \in \pre{<\omega}{\lambda} \). We claim that for all \( x \in X \)
\begin{multline} \label{eq:compactsections}
A_x \cap \Nbhd_s \neq \emptyset \iff \forall n \geq \lh(s) \,  \exists m \geq n  \, \exists t \in \pre{m}{\lambda} \, \Big[t \supseteq s \wedge \exists \alpha < \lambda \, (x \in B^m_\alpha) \\ 
 \wedge \forall k \leq m\,  \forall \beta < \lambda \, (x \in B^k_\beta \Rightarrow t \restriction k \notin S^k_\beta)\Big].
 \end{multline}
Since the right part of equation~\eqref{eq:compactsections} clearly defines a \(\lambda\)-Borel set, the map \( x \mapsto A_x \) will then be \(\lambda\)-Borel, as desired.

Fix any \( x \in X \). To prove the above equivalence, assume first that there is \( y \in A_x \cap \Nbhd_s \) and fix any \( n \geq \lh(s) \). Since \( \Nbhd_{y \restriction n} \) is closed but not compact, there is \( z \in B(\lambda) \setminus A_x \) such that \( z \restriction n = y \restriction n \). 
Let \( 0 \neq m \in \omega \) and \( \alpha < \lambda \) be such that \( (x,z) \in B^m_\alpha \times U^m_\alpha \), so that \( z \restriction m \in S^m_\alpha \), and set \( t = y \restriction m \): we claim that \( m \) and \( t \) are as required. First of all, observe that \( m > n \), as otherwise \( y \in U^m_\alpha \) (because \( y \restriction m = z \restriction m \) if \( m \leq n \)) and thus \( (x,y) \in A \cap (B^m_\alpha \times U^m_\alpha) \), a contradiction. 
It follows that \( t \supseteq s \) because both sequences are initial segments of \( y \) and \( \lh(t) = m > n \geq \lh(s) \). Moreover, \( x \in B^m_\alpha \) for the chosen \(\alpha\), so also the second conjunct in the formula between the square brackets is satisfied. Finally, if \( x \in B^k_\beta \) and \( t \restriction k \in S^k_\beta \) for some \( k \leq m \) and \( \beta < \lambda \), then by \( y \supseteq t \) we would have again \( (x,y) \in A\cap (B^k_\beta \times U^k_\beta) \), a contradiction. Thus \( x \in B^k_\beta \Rightarrow t \restriction k \notin S^k_\beta \) for all such \( k \) and \( \beta \), as required.

Conversely, assume now that the right hand side of the equivalence~\eqref{eq:compactsections} is satisfied. We call \( t \) a witness if the part in square brackets of the formula is satisfied with \( m = \lh(t) \), that is, \( t \supseteq s \),  \( x \in B^{\lh(t)}_\alpha \) for some \( \alpha < \lambda \), and \( x \in B^k_{\beta} \Rightarrow t \restriction k \notin S^k_\beta \) for all \( k \leq \lh(t) \) and \( \beta < \lambda \). Let 
\[ 
T_s = \{ t \restriction k \mid  k \leq \lh(t) \wedge  t \text{ is a witness} \} 
\] 
be the tree generated by all witnesses. Clearly \( T_s \) is infinite because our condition requires that there are witnesses of arbitrarily high length. 

\begin{claim}
Each level of \( T_s \) is finite. 
\end{claim}

\begin{proof}[Proof of the claim]
Fix any \( 0 \neq n \in \omega \). Let \( m \geq n \) be smallest such that \( x \in B^m_\alpha \) for some \( \alpha < \lambda \), which exists because there must be witnesses of arbitrarily high length. 
Notice that if \( t \) is any witness of length at least \( m \), then \( t \restriction m \notin S^m_\alpha \) by \( x \in B^m_\alpha \). It follows that every \( u \in T_s \) of length \( n \) is of the form \( t \restriction n \) for some \( t \in \pre{m}{\lambda} \setminus S^m_\alpha \), and since \( S^m_\alpha \) is cofinite there are only finitely many such sequences \( u \).
\end{proof}

Since \( T_s \) is an infinite finitely branching \(\omega\)-tree, by K\"onig's lemma there is an infinite branch \( y \in [T_s] \). Clearly, \( y \in \Nbhd_s \) because all sequences in \( T_s \) are compatible with \( s \) because any witness extends \( s \) by definition: we claim that \( (x,y) \in A \), so that \( y \) witnesses \( A_x \cap \Nbhd_s \neq \emptyset \), as required. If not, there are \( 0 \neq k \in \omega \) and \( \beta < \lambda \) such that \( (x,y) \in B^k_\beta \times U^k_\beta \). Since \( y \in [T_s ] \), there is a witness \( t \) of length at least \( k \) such that \( y \restriction k = t \restriction k \). Since \( x \in B^k_\beta \), by definition of ``witness'' it follows that \( t \restriction k \notin S^k_\beta \). On the other hand, \( y \in U^k_\beta \) is equivalent to \( y \restriction k \in S^k_\beta \), and this is then a contradiction. 
\end{proof}

We are now ready to prove the analogue of Theorem~\ref{thm:ctblsections} for \(\lambda\)-Borel  sets \( A \) with compact vertical sections. 
The statements and their proofs are quite similar to the previous ones, but with the following significant differences. For \(x \in \p(A) \):
\begin{itemizenew}
\item
\( A_x \) is now automatically closed because it is compact;
\item
\( A_x \) can now be uncountable, but still we have \( |A_x| \leq 2^{\aleph_0} \) and \( \dens(A_x) \leq \aleph_0 \) because \( A_x \) is compact metrizable.
\end{itemizenew}
This in particular means that we will still be able to uniformly enumerate dense subsets of the \( A_x \) with countably many \(\lambda\)-Borel functions, but we will need continuum many of them to uniformly enumerate the whole sections.
Notice also that the analogue of Theorem~\ref{thm:ctblsections}\ref{thm:ctblsections-2} is already proved in Theorem~\ref{thm:compactsections}, and will thus be omitted.

\begin{theorem} \label{thm:compactsections2}
Assume that \( 2^{< \lambda} =  \lambda \). Let \( X \) be a standard \(\lambda\)-Borel space, \( Y \) be a \(\lambda\)-Polish space, and \( A \subseteq X \times Y \) be a \(\lambda\)-Borel set all of whose vertical sections are compact. 
\begin{enumerate-(i)} 
\item \label{thm:compactsections2-1}
There is a \(\lambda\)-Borel uniformization of \( A \), hence in particular \( \p(A) \) is \(\lambda\)-Borel. 
\item \label{thm:compactsections2-2}
There is a sequence \( (\varsigma^A_n)_{n \in \omega} \) of \(\lambda\)-Borel functions \( \varsigma^A_n \colon \p(A) \to Y \) such that for all \( x \in \p(A) \) the set \( \{ \varsigma^A_n(x) \mid n \in \omega \} \) is dense in \( A_x \).
\item \label{thm:compactsections2-3}
There is a sequence \( (\varrho^A_\alpha)_{\alpha < 2^{\aleph_0}} \) of \(\lambda\)-Borel functions \( \varrho^A_\alpha \colon \p(A) \to Y \) such that
\( A_x = \{ \varrho_\alpha^A(x) \mid \alpha < 2^{\aleph_0} \} \)
 for all \( x \in \p(A) \).
 \item \label{thm:compactsections2-4}
 The set \( A \) can be written as \( A = \bigcup_{\alpha < 2^{\aleph_0}} P_\alpha \), where the sets \( P_\alpha \) are pairwise disjoint \(\lambda\)-Borel graphs of partial (necessarily \(\lambda\)-Borel) functions.
\end{enumerate-(i)}
\end{theorem}

\begin{proof}
\ref{thm:compactsections2-1}
The map \( x \mapsto A_x \) is \(\lambda\)-Borel by Theorem~\ref{thm:compactsections}. Consider a function \( \sigma^Y_0 \colon F(Y) \to Y \) as in Theorem~\ref{thm:selectionforF(X)}. The map \( f \colon X \to Y \) sending \( x \) to \( \sigma^Y_0(A_x) \) is a \(\lambda\)-Borel uniformizing function for \( A \).

\ref{thm:compactsections2-2} 

Arguing as in  the proof of Theorem~\ref{thm:compactsections}, one easily sees that it is enough to consider the case \( Y = B(\lambda) \). Moreover, for \( x \in \p(A) \) there are exactly \( \aleph_0 \)-many \( s \in \pre{<\omega}{\lambda} \) such that \( A_x \cap \Nbhd_s \neq \emptyset \) because \( A_x \) is compact. Thus we can repeat the argument from Theorem~\ref{thm:ctblsections}\ref{thm:ctblsections-3}, using the fact that \( A_x \cap \Nbhd_{\sigma^{-1}(\alpha)} \), being a closed subset of the compact set \( A_x \), is compact itself.

%
%

\ref{thm:compactsections2-3}
Follow the proof of Theorem~\ref{thm:ctblsections}\ref{thm:ctblsections-4} until the construction of the functions \( (\vartheta^A_\alpha)_{\alpha < 2^{\aleph_0}} \). (This is possible because each \( A_x \), being compact, is also closed.)  Since for all \( x \in \p(A) \) we already have
\[ 
A_x = \{ \vartheta^A_\alpha(x) \mid \alpha < 2^{\aleph_0} \},
 \] 
 we can simply set \( \varrho^A_\alpha = \vartheta^A_\alpha \).
%

\ref{thm:compactsections2-4}
Let \( (\varrho^A_\alpha)_{\alpha < 2^{\aleph_0}} \) be as in part~\ref{thm:compactsections2-3}, and set 
\[ 
P_\alpha = \mathrm{graph}(\varrho^A_\alpha) \setminus \bigcup_{\beta < \alpha} \mathrm{graph}(\varrho^A_\beta) . \qedhere
\]
\end{proof}

Similarly to the case of \(\lambda\)-Borel sets with countable vertical sections, once we have Theorems~\ref{thm:compactsections} and~\ref{thm:compactsections2}
 we can then derive a number of standard consequences, whose proofs are omitted because they are essentially the same of the corresponding results in Section~\ref{subsec:ctblsections}. First, as an easy application to \(\lambda\)-Borel equivalence relations we get:

\begin{corollary} \label{cor:compactclassesimpliessmooth}
Assume \( 2^{< \lambda} = \lambda \). If \( X \) is \(\lambda\)-Polish and \( E \) is a \(\lambda\)-Borel equivalence relation on \( X \) all of whose equivalence classes are compact, then \( E \) has a \(\lambda\)-Borel selector (equivalently: a \(\lambda\)-Borel transversal). In particular, \( E \) is smooth.
\end{corollary}


%
%

Since \( 2^{\aleph_0} < \lambda \) when \( 2^{< \lambda} = \lambda \), the following can be seen as a form  of the Feldman-Moore theorem for \(\lambda\)-Borel equivalence relations with compact equivalence classes. 

\begin{theorem} \label{thm:weakFeldmanMoore}
Assume that \( 2^{< \lambda} =  \lambda \).
Let \( X \) be a  \(\lambda\)-Polish space, and let \( E \) be a \(\lambda\)-Borel equivalence relation on \( X \) with compact equivalence classes. Then \( E \) is induced by a \(\lambda\)-Borel action of a (discrete) continuum-sized group \( G \) on \( X \).
\end{theorem}

There is an interesting phenomenon that arises when considering \(\lambda\)-Borel equivalence relations induced by a \(\lambda\)-Borel action of an \emph{uncountable} (discrete) group. Given a cardinal \( \kappa \leq \lambda \), say that an equivalence relation \( E \) on a standard \(\lambda\)-Borel space \( X \) is \markdef{hyper-\( \kappa \)-small} if \( E \) can be written as an increasing union \( \bigcup_{\alpha < \kappa} E_\alpha \) of \( \lambda \)-Borel equivalence relations on \( X \) that are  \markdef{\( \kappa \)-small}, that is: for each \( \alpha < \kappa \) there is \( \mu < \kappa \) such that all \( E_\alpha \)-equivalence classes have size at most \( \mu \). 
Notice that hyper-\( \aleph_0 \)-smallness is a strengthening of what in classical descriptive set theory is called ``hyperfiniteness''.

\begin{lemma} \label{lem:hypersmall} 
Let \( \omega < \kappa \leq \lambda \) be a cardinal. If the equivalence relation \( E \) on the standard \(\lambda\)-Borel space \( X \) is induced by a \(\lambda\)-Borel action of a group \( G \) of size \( \kappa \), then \( E \) is hyper-\( \kappa \)-small.
\end{lemma}

\begin{proof}
Enumerate \( G \) as \( G = \{ g_\alpha \mid \alpha < \kappa \} \). For every \( \alpha < \kappa \), let \( G_\alpha \) be the subgroup of \( G \) generated by \( \{ g_\beta \mid \beta < \alpha \} \), and let \( E_\alpha \) be the equivalence relation on \( X \) induced by the restriction of the action of \( G \) to \( G_\alpha \). Then the sequence \( (E_\alpha)_{\alpha < \kappa} \) witnesses that \( E \) is hyper-\( \kappa \)-small.
\end{proof}

Combining this with Theorem~\ref{thm:weakFeldmanMoore}, we then get:

\begin{corollary} \label{cor:hypersmall} 
Assume that \( 2^{< \lambda} =  \lambda \).
Let \( X \) be a  \(\lambda\)-Polish space, and let \( E \) be a \(\lambda\)-Borel equivalence relation on \( X \) with compact equivalence classes. Then \( E \) is hyper-\( 2^{\aleph_0} \)-small.

In particular, if \( \mathsf{CH} \) holds, then every \( E \) as above can be written as an \( \omega_1 \)-increasing union of countable \(\lambda\)-Borel equivalence relations.
\end{corollary}


Finally, for \( X \) a topological space call
a map \( f \colon X \to Y \), where \( X \) is a topological space, \markdef{compact-to-\( 1 \)} if \( f^{-1}(y) \) is compact for every \( y \in Y \). 

\begin{corollary}
Let \( \lambda \) be such that \( 2^{< \lambda} = \lambda \). Let \( X \) be a \(\lambda\)-Polish space, and let \( Y \) be a standard \(\lambda\)-Borel space. If a \(\lambda\)-Borel map \( f \colon X \to Y \) is compact-to-\( 1 \), then \( \ran(f) \) is \(\lambda\)-Borel and \( f \) admits a \(\lambda\)-Borel right-inverse.
\end{corollary}


It is quite natural to conjecture that results along the lines of Theorems~\ref{thm:ctblsections}, \ref{thm:compactsections}, and~\ref{thm:compactsections2} could be obtained for \(\lambda\)-Borel sets with sections of size at most \( \lambda \). The natural adaptation of the classical techniques for the ``small section'' uniformization results from~\cite[Sections 18.C and 35.H]{Kechris1995} seems not to be of help here, mainly because we are still missing a Baire category theory for \(\lambda\)-Polish spaces matching well with the structure of \(\lambda\)-Borel sets. However, let us notice that the conjecture, if confirmed, would subsume all the results in Sections~\ref{subsec:ctblsections} and~\ref{subsec:cpctsections} (with the notable exception of Theorems~\ref{thm:feldman-moore} and~\ref{thm:weakFeldmanMoore}) because in our generalized context most of the ``smallness notions'' imply having small size: this includes e.g.\ the case of compact sections, \( K_\lambda \) sections (i.e.\ sections which are \(\lambda\)-sized unions of compact sets), \( < \lambda \)-Lindel\"of sections (i.e.\ sections which are \( \mu \)-Lindel\"of for some \( \mu < \lambda \), where \( \mu \) possibly varies with the section), \
sections which are \(\lambda\)-sized unions of closed \(\lambda\)-Lindel\"of sets (see Corollary~\ref{cor:lambdaLindelofissmall}), and so on.
The above discussion shows that settling the conjecture is an important problem in the area, as it would yield one of the strongest possible uniformization result for \(\lambda\)-Borel sets.

\section{Boundedness theorems and ranks for \(\lambda\)-coanalytic sets} \label{sec:lambdacoanalyticranks}

In this section we quickly review some important results related to the existence of \(\lambda\)-coanalytic ranks on \(\lambda\)-coanalytic sets, which roughly correspond to the material in Sections 34.A--34.C and 35.A--35.E of~\cite{Kechris1995}. All proofs in the generalized setting are natural adaptations of the corresponding classical arguments and will thus be omitted for the sake of conciseness. The only notable exception to this is the proof of Theorem~\ref{thm:Pi11isranked}, in which one needs to justify why the analogue of the classical argument can still be carried out despite the fact that the space \( \pre{\lambda}{\lambda} \) of functions from \( \lambda \) into itself is not \(\lambda\)-Polish when equipped with the bounded topology.

First we notice that the proof of~\cite[Theorem 31.1]{Kechris1995}, due to Kunen, yields the following boundedness theorem for \(\lambda\)-analytic well-founded (strict) relations.

\begin{theorem} \label{thm:rankwellfoundedrelations}
Let \(\lambda\) be such that \( 2^{< \lambda} = \lambda \), and let \( X \) be a \(\lambda \)-Polish space. If \( \prec \) is a \(\lambda\)-analytic well-founded (strict) relation on \( X \), then \( \rho(\prec) < \lambda^+ \).
\end{theorem}

The same argument basically shows that the following generalized Kunen-Martin theorem holds if we assume \( 2^{< \lambda} = \lambda \) (see also~\cite[Theorem 31.5]{Kechris1995}).

\begin{theorem}\label{thm:rankwellfoundedrelations2}
Let \(\lambda\) be such that \( 2^{< \lambda} = \lambda \), let \( X \) be a \(\lambda\)-Polish space, and let \( \kappa \geq \lambda \) be a cardinal. If \( \prec \) is a \( \kappa \)-Souslin well-founded (strict) relation on \( X \), then \( \rho(\prec) < \kappa^+ \).
\end{theorem}

In particular, under \( 2^{< \lambda} = \lambda \) there are no \(\lambda\)-analytic well-orders of \( B(\lambda) \) or of any \( \lambda \)-Polish space \( X \) with \( |X| > \lambda \). Similarly, by Theorem~\ref{thm:coanalyticaresouslin} we get that  under \(  \AC_{\lambda^+}(\pre{\lambda}{2}) \) there are no \( \lS^1_2 \) well-orderings of such spaces if the generalized continuum hypothesis fails at \(\lambda\), i.e.\ \( 2^\lambda \neq \lambda^+ \).
In contrast, we will show that, as in the classical case, in the constructible universe \( L \) there is a \( \lS^1_2 \) well-ordering of any given \(\lambda\)-Polish space (Proposition~\ref{prop:definablewellorderinL}).

Arguing as in~\cite[Theorem 31.2 and Exercise 31.3]{Kechris1995} and using Propositions~\ref{prop:IF} and~\ref{prop:NWO}, one then gets the following boundedness theorems for well-founded \(\omega\)-trees on \(\lambda\) and well-orders of \(\lambda\). 

\begin{theorem} \label{thm:analyticboundedness}
Let \( \lambda \) be such that \( 2^{< \lambda} = \lambda \). If \( A \subseteq \mathrm{WF}_\lambda \) is \(\lambda\)-analytic, then \( \sup \{ \rho(T_x) \mid x \in A \} < \lambda^+ \).

Similarly, if \( A \subseteq \mathrm{WO}_\lambda \) is \(\lambda\)-analytic, then \( \sup \{\mathrm{ot}(\preceq_x) \mid x \in A \} < \lambda^+ \).
\end{theorem}

We now move to the notion of ranked (or normed, or prewellordered) classes. The notions involved are exactly those presented in~\cite[Section 34.B]{Kechris1995}, except that we work with classes of subsets of \(\lambda\)-Polish spaces rather than just (classical) Polish spaces.
We will just recall the basic definitions and facts, referring the reader to the mentioned monograph for more details.

Let \( \boldsymbol{\Gamma} \) be a boldface \(\lambda\)-pointclass, \( X \) a \(\lambda\)-Polish space, and \( A  \subseteq X \). A \markdef{\( \boldsymbol{\Gamma} \)-rank} on \( A \) is a function \( \varphi \colon A \to \On \) for which there are relations \( {\leq^{\boldsymbol{\Gamma}}_\varphi}, {\leq^{\check{\boldsymbol{\Gamma}}}_\varphi} \subseteq X^2 \) in \( \boldsymbol{\Gamma} \) and \( \check{\boldsymbol{\Gamma}} \), respectively, such that for all \( x \in X \) and \( y \in A \)
\[ 
x \in A \wedge \varphi(x) \leq \varphi(y) \iff x \leq^{\boldsymbol{\Gamma}}_\varphi y \iff x \leq^{\check{\boldsymbol{\Gamma}}}_\varphi y.
 \] 
If \( \boldsymbol{\Gamma} \) is also closed under finite unions and intersections, then one can prove that:
\begin{equation} \label{eq:alternativedefofranks}
\begin{minipage}[c]{0.87\linewidth}
A map \( \varphi \colon A \to \On \) is a \( \boldsymbol{\Gamma} \)-rank if and only if there exists \( \leq^{\check{\boldsymbol{\Gamma}}}_\varphi \) as above, together with a binary relation \( {<^{\check{\boldsymbol{\Gamma}}}_\varphi} \subseteq X^2 \) in \( \check{\boldsymbol{\Gamma}} \) such that for all \( x \in X \) and \( y \in A \)
\[ 
x \in A \wedge \varphi(x) < \varphi(y) \iff x <^{\check{\boldsymbol{\Gamma}}}_\varphi y.
 \] 
\end{minipage} 
 \end{equation}
 A \( \boldsymbol{\Gamma} \)-rank is \markdef{regular} if its range is downward closed in \( \On \).
Finally, we say that the boldface \(\lambda\)-pointclass \( \boldsymbol{\Gamma} \) is \markdef{ranked} if for every \(\lambda\)-Polish space \( X \), each set  in  \( \boldsymbol{\Gamma}(X) \) admits a \( \boldsymbol{\Gamma} \)-rank.

\begin{theorem} \label{thm:Pi11isranked}
Let \(\lambda\) be such that \( 2^{< \lambda}=\lambda \).
Then the boldface \(\lambda\)-pointclass \( \lP^1_1 \) is ranked (with ranks ranging in \(\lambda^+ \)).
\end{theorem}

The proof is identical to that of~\cite[Theorem 34.4]{Kechris1995}. The only subtlety is due to the fact that we have to use the space \( \pre{\lambda}{\lambda} \), which would not be \(\lambda\)-Polish (nor standard \(\lambda\)-Borel) if, as common sense would suggest, we endow it with the bounded topology.


\begin{proofcompare}{Kechris1995}{Theorem 34.4}
By Proposition~\ref{prop:NWO}, it is enough to show that the map \( \mathrm{WO}_\lambda \to \lambda^+ \) sending each \( x \in \mathrm{WO}_\lambda \) to its order type \( \mathrm{ot}({\preceq_x}) \) is a \( \lP^1_1 \)-rank.  
It is easy to see that the relations
\begin{align*}
x \leq^{\lS^1_1} y \iff \exists z \in \pre{\lambda}{\lambda} \, [ \forall \alpha,\beta < \lambda \, ( \alpha \preceq_x \beta \iff z(\alpha) \preceq_y z(\beta))]
\end{align*}
and
\begin{align*}
x <^{\lS^1_1} y \iff \exists \epsilon < \lambda \, \exists z  \in \pre{\lambda}{\lambda} \, [ & \forall \alpha < \lambda \, (z(\alpha) \preceq_y \epsilon \wedge z(\alpha) \neq \epsilon) \\
&\wedge \forall \alpha,\beta < \lambda \, ( \alpha \preceq_x \beta \iff z(\alpha) \preceq_y z(\beta))]
\end{align*} 
are as in~\eqref{eq:alternativedefofranks}. Indeed, when \( \pre{\lambda}{\lambda} \) is equipped with the \(\lambda\)-Borel structure \( \mathcal{B} = \lB(\pre{\lambda}{\lambda}, \tau_p) \) induced by the \emph{product topology \( \tau_p \)}, it is a standard \(\lambda\)-Borel space (Example~\ref{xmp:lambda^lambdawithproduct}), and in both formulas the parts in square brackets define sets which are \(\lambda\)-Borel in the (standard \(\lambda\)-Borel) product space \( (\pre{\lambda}{\lambda}, \mathcal{B}) \times (\mathrm{LO}_\lambda)^2 \). It follows that \( \leq^{\lS^1_1} \) and \( <^{\lS^1_1} \) are \(\lambda\)-analytic, as required.
\end{proofcompare}

If \( A \subseteq B(\lambda) \) is \( \lP^1_1(X) \), a natural way to define a \( \lP^1_1 \)-rank on it is the following (we are basically combining the proofs of Propositions~\ref{prop:IF} and~\ref{prop:NWO} with that of Theorem~\ref{thm:Pi11isranked}, see~\cite[Exercise 34.7]{Kechris1995}). Let \( T \subseteq \pre{<\omega}{(\lambda \times \lambda)}\) be any tree such that \( B(\lambda) \setminus A = \p[T] \), so that \( x \in A \iff T(x) \) is well-founded. Then the map \( \varphi \colon A \to \lambda^+ \) defined by \( \varphi(x) = \rho_{T(x)} (\emptyset) \) is a \( \lP^1_1 \)-rank on \( A \).

%

From Theorem~\ref{thm:Pi11isranked}, the existence of \( \pre{\lambda}{2} \)-universal sets of \( \lS^1_1 \) and \( \lP^1_1 \), and their closure properties (Proposition~\ref{prop:closurepropertiesofanalytic} and Corollary~\ref{cor:closurepropertiesofanalytic}), one derives in the usual way some structural properties of \(\lambda\)-(co)analytic sets (see e.g.\ the proof of~\cite[Theorem 35.1]{Kechris1995}).

\begin{theorem} \label{thm:structuralpropertiesforcoanalytic}
Let \(\lambda\) be such that \( 2^{< \lambda} = \lambda \).
\begin{enumerate-(i)}
\item \label{thm:structuralpropertiesforcoanalytic-1}
The class \( \lP^1_1 \) has the reduction property but not the separation property, while \( \lS^1_1 \) has the separation property%
\footnote{This also follows from Theorems~\ref{thm:lusinseparation} and~\ref{thm:souslin}.} 
but not the reduction property.
\item \label{thm:structuralpropertiesforcoanalytic-2}
The class \( \lP^1_1 \) has the ordinal \(\lambda\)-uniformization property and the \(\lambda\)-generalized reduction property.
\item \label{thm:structuralpropertiesforcoanalytic-3}
The class \( \lS^1_1 \) has the \(\lambda\)-generalized separation property.
\end{enumerate-(i)}
\end{theorem}

\begin{remark}	
As in the classical setting, Theorem~\ref{thm:structuralpropertiesforcoanalytic}\ref{thm:structuralpropertiesforcoanalytic-1} implies that Theorem~\ref{thm:lusinseparation} cannot be extended to \(\lambda\)-coanalytic spaces (compare this with Remark~\ref{rmk:lusinseparation}). Indeed, if \( X \) is \(\lambda\)-Polish and \( B_0, B_1 \in \lP^1_1(X) \) are disjoint sets that cannot be separated by a \(\lambda\)-Borel set, then \( Y = B_0 \cup B_1 \) is a \(\lambda\)-coanalytic space whose \(\lambda\)-analytic%
\footnote{Here we are naturally extending the notion of being \(\lambda\)-analytic to subsets of arbitrary subspaces \( Y  \subseteq X\) of a \(\lambda\)-Polish space \( X\) as follows: A set \( A \subseteq Y \) is \(\lambda\)-analytic if and only if \( A = A' \cap Y \) for some \( A' \in \lS^1_1(X) \).}
 subsets \( B_0 \) and \( B_1 \) cannot be separated by a set in \( \lB(Y) \). 
\end{remark}

The reduction property for \( \lP^1_1 \) implies that we can parametrize sets in \( \lD^1_1 = \lB \) in a \( \lD^1_1 \) fashion using a \( \lP^1_1 \) set of codes (see the proof of~\cite[Theorem 35.5]{Kechris1995}).

\begin{theorem} \label{thm:borelcodes}
Let \(\lambda\) be such that \( 2^{< \lambda}= \lambda \), and let \( X \) be a \(\lambda\)-Polish space. Then there is a \(\lambda\)-coanalytic set \( D \subseteq \pre{\lambda}{2} \) and two sets \( S \in \lS^1_1(\pre{\lambda}{2} \times X) \) and \( P \in \lP^1_1(\pre{\lambda}{2} \times X) \) such that \( S_d = P_d \) for all \( d \in D \), and \( \{ S_d \mid d \in D \} = \lD^1_1(X) \).
\end{theorem}

Other consequences of the fact that the boldface \(\lambda\)-pointclass \( \lP^1_1 \) is ranked are the following higher analogues of~\cite[Theorems 35.22 and 35.23]{Kechris1995}, which yield in particular corresponding boundedness theorems for \( \lP^1_1 \)-ranks.

\begin{theorem}  \label{thm:otherboundedness}
Let \(\lambda\) be such that \( 2^{< \lambda} = \lambda \), and let \( X \) be a \(\lambda\)-Polish space. If \( A \subseteq X \) is \( \lP^1_1 \) but not \(\lambda\)-Borel, then for every \( \lP^1_1 \)-rank \( \varphi \colon A \to \On \) and every \(\lambda\)-analytic \( B \subseteq A \) there is \( x_0 \in A \) with \( \varphi(x) \leq \varphi(x_0) \) for all \( x \in B \).
\end{theorem}

\begin{theorem} \label{thm: boundednessforanalytic}
Let \(\lambda\) be such that \( 2^{< \lambda} = \lambda \), and let \( X \) be a \(\lambda\)-Polish space.  Let \( A \in \lP^1_1(X) \),  and let \( \varphi \colon A \to \On \) be a regular \( \lP^1_1 \)-rank. Then we have that \( \varphi(A) = \sup \{ \varphi(x) \mid x \in A \} \leq \lambda^+ \), and that \( A \) is \(\lambda\)-Borel if and only if \( \varphi(A) < \lambda^+ \).

If \( \psi \colon A \to \lambda^+ \) is any \( \lP^1_1 \)-rank and \( B \subseteq A \) is \( \lS^1_1 \), then \( \sup \{ \psi(x) \mid x \in B \} < \lambda^+ \).
\end{theorem}

A standard consequence of this is that there is no \( \lS^1_1 \) set \( A \subseteq \mathrm{WO}_\lambda \) of size greater than \( \lambda \) with the property that \( x \) and \( y \) have different order types for all distinct \(x,y \in A \).

\begin{remark}
By Remark~\ref{rmk:IFcompletenessfromcofomega}, all the rank theory as developed so far works only when \( \cf(\lambda) = \omega \) and \( 2^{<\lambda} = \lambda \). Notably, boundedness theorems for \(\lambda\)-analytic sets along the lines of Theorem~\ref{thm: boundednessforanalytic} have been obtained in~\cite[Proposition 7.6 and Corollary 7.7]{DzaVaa} for \emph{arbitrary} strong limit singular cardinals. However, if \( \cf(\lambda) > \omega \) one cannot longer use ranks ranging into ordinal numbers, but rather use their variant ranging in certain trees ``ordered'' by a suitable quasi-order.
\end{remark}

Theorem~\ref{thm: boundednessforanalytic} also provides an alternative proof of the fact that every \(\lambda\)-coanalytic subset \( A \) of a \(\lambda\)-Polish space \( X \) can be written as a well-ordered union of \( \lambda^+ \)-many \(\lambda\)-Borel sets. Indeed, let \( \varphi \colon A \to \lambda^+ \) be a \( \lP^1_1 \)-rank on \( A \): then \( A = \bigcup_{\xi < \lambda^+} A_\xi \), where \( A_\xi \) is the \(\lambda\)-Borel set \( \{ x \in A \mid \varphi(x) \leq \xi \} \).

Other standard consequences of the fact that \( \lP^1_1 \) is ranked are the so-called reflection theorems. Let \( \boldsymbol{\Gamma} \) be a boldface \(\lambda\)-pointclass, and \( X \) be a \(\lambda\)-Polish space. Let \( \Phi \) be a property for subsets of \( X \), i.e.
 \( \Phi \subseteq \pow(X) \). Then \( \Phi \) is \markdef{\( \boldsymbol{\Gamma} \) on \( \boldsymbol{\Gamma } \)} if for any \(\lambda\)-Polish space \( Y \) and any \( A \in \boldsymbol{\Gamma}(Y \times X) \) the set of \( y \in Y \) such that  \( A_y \) 
 has property \( \Phi \) (that is, \( A_y \in \Phi \)) is in \( \boldsymbol{\Gamma} \). 

\begin{theorem}[First Reflection Theorem] \label{thm:firstreflection}
Let \( \lambda \) be such that \( 2^{< \lambda} = \lambda \), and let \( X \) be a \(\lambda\)-Polish space. If \( \Phi \subseteq \pow(X) \) is a \( \lP^1_1 \) on \( \lP^1_1 \) property, then for every \( A \in \lP^1_1(X) \) with property \( \Phi \) there exists \( B \subseteq A \), still with property \( \Phi \), such that \( B \in \lD^1_1(X) = \lB(X) \).
\end{theorem}

Let \( \boldsymbol{\Gamma} \) and \( X \) be as above. Let \( \Psi \) be a property for pair of subsets of \( X \), i.e.\ \( \Psi \subseteq \pow(X) \times \pow(X) \). We say again that \( \Psi \) is \markdef{\( \boldsymbol{\Gamma} \) on \( \boldsymbol{\Gamma} \)} if for any \(\lambda\)-Polish spaces \( Y,Z \) and any \( A \in \boldsymbol{\Gamma}(Y \times X) \) and \( B \in \boldsymbol{\Gamma}(Z \times X) \), the set of \( (y,z) \in Y \times Z \) such that the corresponding pair \( (A_y, B_z) \) of vertical sections has property \( \Psi \) is in \( \boldsymbol{\Gamma} \). Moreover, we say that \( \Psi \) is \markdef{upward monotone} if for all pairs \( (A,B) \) with property \( \Psi \) and all \( A \subseteq A' \subseteq X \) and \( B \subseteq B' \subseteq X \), the pair \( (A',B') \) has property \( \Psi \) as well. Finally, \( \Psi \) is \markdef{downward  continuous in the second variable} if given \( A, B_n \subseteq X \) such that \( B_n \supseteq B_{n+1} \) for all \( n \in \omega \), if each pair \( (A,B_n) \) has property \( \Psi \) then \( \big(A, \bigcap_{n \in \omega} B_n \big) \) has property \( \Psi \) as well.

\begin{theorem}[Second Reflection Theorem] \label{thm:secondreflection}
Let \( \lambda \) be such that \( 2^{< \lambda} = \lambda \), and let \( X \) be a \(\lambda\)-Polish space. If \( \Psi \subseteq \pow(X) \times \pow(X) \) is a \( \lP^1_1 \) on \( \lP^1_1 \) property which is both upward monotone and downward continuous in the second variable, then for every \( A \in \lP^1_1 (X) \) such that \( (A, X \setminus A ) \) has property \( \Psi \) there exists \( B \subseteq A \) such that \( (B, X \setminus B) \) still has property \( \Psi \) and \( B \in \lD^1_1(X) = \lB(X) \).
\end{theorem}

The proofs of these theorems are basically the same as those of~\cite[Theorems 35.10 and 35.16]{Kechris1995}, and as usual will thus be omitted. As in the classical setup, the reflection theorems admit a dual formulation and have a number of standard consequences which go through also in the generalized setting --- we refer the interested reader to~\cite[Section 35.C]{Kechris1995} for more information on this.



%
%
%

\chapter{\( \lambda \)-Perfect Set Property} \label{chapter:lambda-PSP}

The classical Perfect Set Property (briefly, \( \mathrm{PSP} \)) for a set \( A \subseteq X \), where \( X \) is a Polish space, states that either \( |A| \leq \omega \), or else there is an embedding of \( \pre{\omega}{2} \) into \( A \). Since the Cantor space \( \pre{\omega}{2} \) is compact, in the latter case the range of the embedding is necessarily closed in \( X \); moreover, to have such an embedding it is enough to construct a continuous injective function \( f \colon  \pre{\omega}{2} \to A \).

In the generalized setting, all previous variants of the \( \mathrm{PSP} \) are not necessarily equivalent to each other (see e.g.\ the discussion in~\cite[Section 5]{AgoMotSch}),  so let us list the most relevant options in our setup.

\begin{defin} \label{def:PSP}
Let \( X \) be a \(\lambda\)-Polish space. 
\begin{enumerate-(a)}
\item
A set \( A \subseteq X \) has the \markdef{\(\lambda\)-Perfect Set Property} (briefly, \markdef{\lPSP}) if either \( |A| \leq \lambda \), or \( \pre{\lambda}{2} \) embeds into \( A \) as a closed-in-\( X \) set. 
\item
A set \( A \subseteq X \) has the \markdef{weak \(\lambda\)-Perfect Set Property} (briefly, \markdef{\lwPSP}) if either \( |A| \leq \lambda \), or there is continuous injection \( f \colon \pre{\lambda}{2} \to A \). 
\item
A set \( A \subseteq X \) has the \markdef{\(\lambda\)-Borel \(\lambda\)-Perfect Set Property} (briefly, \markdef{\lbPSP}) if either \( |A| \leq \lambda \), or there is \(\lambda\)-Borel injection \( f \colon \pre{\lambda}{2} \to A \).
\end{enumerate-(a)}
\end{defin}

We write \( \lPSP(A) \) to say that the set \( A \) has the \( \lPSP \), and similarly for \lwPSP and \lbPSP. Analogously, if \( \boldsymbol{\Gamma} \) is a boldface \(\lambda\)-pointclass we write \( \lPSP(\boldsymbol{\Gamma}) \) (respectively, \( \lwPSP(\boldsymbol{\Gamma}) \) or \( \lbPSP(\boldsymbol{\Gamma}) \)) to say that all sets in \( \boldsymbol{\Gamma}(X) \) have the \( \lPSP \) (respectively, the \( \lwPSP \) or the \( \lbPSP) \), for every \(\lambda\)-Polish space \( X \).

Since \( \pre{\lambda}{2} \) is not \(\lambda\)-Polish if \( 2^{< \lambda} > \lambda \), it makes also sense to consider variants of the above properties where \( \pre{\lambda}{2} \) is replaced with \( B(\lambda) \) or, equivalently, with \( C(\lambda) \). When considering such variants, we will systematically add a \( * \) to the notation. For instance, \( \lPSP^*(A) \) is the statement: Either \( |A| \leq \lambda \), or else \( B(\lambda) \) (equivalently, \( C(\lambda) \)) embeds into \( A \) as a closed-in-\( X \) set. 
(By Corollary~\ref{cor:lambda-perfect}, if \(\lambda\) is \(\omega\)-inaccessible then the second alternative is equivalent to requiring that \( A \) contains a \(\lambda\)-perfect subspace of the ambient \(\lambda\)-Polish space \( X \).)
%
%
%
Clearly, such starred versions coincide with the original ones  if \( 2^{< \lambda} =  \lambda\), as in that case \( \pre{\lambda}{2} \approx B(\lambda) \) by Theorem~\ref{thm:homeomorphictoCantor}.

Trivially, if $X$ is a $\lambda$-Polish space and $|X|\leq\lambda$ then all its subsets have the \lPSP. When $2^{<\lambda} = \lambda$, using the axiom of choice \( \AC \) it is possible to build in any $\lambda$-Polish space of size greater than \( \lambda \) a subset without the \lPSP. The proof is the same as the classical Bernstein's proof (see~\cite[Proposition 11.4(a)]{Kanamori}). 

\begin{proposition}[\( \AC \)]
 \label{prop:counterexPSP}
 Let $\lambda$ be such that \( 2^{< \lambda} = \lambda \), and let $X$ be a $\lambda$-Polish space such that $|X|>\lambda$. Then there exists $A\subseteq X$ without the \( \lPSP \).
\end{proposition}

\begin{proofsketchcite}{Kanamori}{Proposition 11.4(a)}
There are $2^\lambda$-many $\lambda$-perfect subsets of $X$, and since by Corollary~\ref{cor:lambda-perfect} in each one of those we can embed $B(\lambda)$, they  all have cardinality $\lambda^{\omega}$. Since $\lambda^{\omega}=2^\lambda$ because we assumed \( 2^{< \lambda} = \lambda \),  we can follow the construction in~\cite[Proposition 11.4(a)]{Kanamori} to build a set $A\subseteq X$ that does not contain any $\lambda$-perfect set and has size $|A|=2^\lambda > \lambda$.  
\end{proofsketchcite}

\section{\( \lambda \)-\( \mathrm{PSP} \) for \(\lambda\)-analytic sets} \label{sec:PSPforanalytic}

The proof of the next result is reminiscent of those of Theorem~\ref{thm:perf} and Theorem~\ref{thm:decomp}, and is strictly related to~\cite[Theorems 5 and 22]{Stone1962} (although the settings are slightly different and the arguments more involved). Notice also that Remark~\ref{rmk:not-omega-inaccessible} shows that the cardinal assumption on \(\lambda\) cannot be removed: if \( \lambda \) is not \(\omega\)-inaccessible, then there are even closed subsets of \( B(\lambda) \) which do not have the \( \lPSP^* \) (nor the \( \lwPSP^* \)).

\begin{theorem} \label{thm:analyticPSP}
Let \(\lambda\) be \(\omega\)-inaccessible, and let \( X \) be a \(\lambda\)-Polish space. Then every \(\lambda\)-analytic subset \( A \) of \( X \) has the \( \lPSP^* \), and thus either \( |A| \leq \lambda \) or else \( |A| = \lambda^\omega \)


In particular, if \( 2^{< \lambda} = \lambda \) this applies to every \( A \in \lB(X) \).
\end{theorem}

\begin{proof}
Let \( F \subseteq X \times B(\lambda) \) be a closed set such that \( A = \p (F) \),
and let \( \mathcal{B} = \{ U_\alpha \mid \alpha < \lambda \} \) be a basis for \( X \times B(\lambda) \). Let \( S = \{ \alpha< \lambda \mid |\p(U_\alpha \cap F)| \leq \lambda \} \) and set \( C = \bigcup_{\alpha \in S} U_\alpha \). 
Set also \( F' = F \setminus C \) and \( A' = \p(F') \). Notice that \( A' \subseteq A \) and \( |A \setminus A'| \leq \lambda \) because 
\[ 
A \setminus A' \subseteq \p (C \cap F) = \p \bigg( \bigcup_{\alpha \in S} (U_\alpha \cap F) \bigg) = \bigcup_{\alpha \in S} \p(U_\alpha \cap S) ,
\]
and that if \( V \subseteq X \times B(\lambda) \) is open and \( V \cap F' \neq \emptyset \), then \( |\p(V \cap F')| \nleq \lambda \).

If \( A' = \emptyset \) then \( |A| = |A \setminus A'| \leq \lambda \), and hence \( A \) has the \( \lPSP^* \).
We can thus assume without loss of generality that \(  A' \neq \emptyset \) and show that \( C(\lambda) \approx B(\lambda) \) embeds into \( A' \subseteq A \) as a closed-in-\( X \) set. To this aim, fix  compatible complete metrics \( d, d' \leq 1 \) on \( X \) and \( F' \),
 respectively. (Here we use that \( F' \) is \(\lambda\)-Polish because it is closed in \( X \times B(\lambda) \).)
 
We will build  \( \lambda \)-schemes  \( \{ B_s \mid s \in \pre{\omega}{\lambda} \} \) and  \( \{ B'_s \mid s \in \pre{\omega}{\lambda} \} \) on \( X \) and \( F' \), respectively, together with
a family \( \{ r_s \mid s \in \pre{<\omega}{\lambda} \} \) of positive reals  
such that for all \( s \in \pre{< \omega}{\lambda} \) and distinct \( \beta, \beta' < \lambda \):
\begin{enumerate-(1)}
\item \label{analytic1}
\( \p(B'_s) \subseteq B_s \), \( B'_\emptyset = F' \), and \( B_\emptyset = X \); 
\item \label{analytic2}
\( \mathrm{cl}(B_{s {}^\smallfrown{} \beta}) \subseteq B_s \) and \( \mathrm{cl}(B'_{s {}^\smallfrown{}  \beta}) \subseteq B'_s \);
\item \label{analytic3}
if \( B'_s \neq \emptyset \) (so that also \( B_s \neq \emptyset \) by~\ref{analytic1}), then \( B'_{s {}^\smallfrown{} \beta} \neq \emptyset  \) if and only if \( B_{s {}^\smallfrown{} \beta} \neq \emptyset  \) if and only if \( \beta <  \lambda_{\lh(s)} \);
\item \label{analytic4}
both \( \mathrm{diam}(B_s) \) and \( \mathrm{diam}(B'_s) \) are at most \( 2^{-\lh(s)} \) (where the two diameters refer to \( d \) and \( d' \), respectively);
\item \label{analytic4half}
\( B_s \) is open in \( X \) and \( B'_s \) is open in \( F' \);
\item \label{analytic5}
\( B_{s {}^\smallfrown{} \beta} \cap  B_{s {}^\smallfrown{} \beta'}  = \emptyset \);
\item \label{analytic6}
the distance between \( B_{s {}^\smallfrown{} \beta} \) and \( B_{s {}^\smallfrown{} \beta'} \) is at least \(  r_s \) (if they are both nonempty).
\end{enumerate-(1)}

Let \( f \colon B(\lambda) \to X \) and \( f' \colon B(\lambda) \to F'\)  be the (partial) continuous functions induced by the \(\lambda\)-schemes  \( \{ B_s \mid s \in \pre{\omega}{\lambda} \} \) and \( \{ B'_s \mid s \in \pre{\omega}{\lambda} \} \), respectively. By~\ref{analytic2} and~\ref{analytic3} we have that both \( f \) and \( f' \) have domain \( C(\lambda) \). By~\ref{analytic1}
we have \( f = \p \circ f' \), hence in particular the range of \( f \) is contained in \( \p (F') = A' \). Conditions~\ref{analytic4half} and~\ref{analytic5} ensure that \( f \) is an embedding, while~\ref{analytic6}, together with~\ref{analytic2}, ensures that its range is closed in \( X \).

It remains to construct \( B_s \), \( B'_s \), and \( r_s \) by recursion on \( \mathrm{lh}(s) \).
Let \( B_\emptyset = X \) and \( B'_\emptyset = F' \). Next assume that \( B_s \) and \( B'_s \) have been defined so that the above conditions~\ref{analytic1}--\ref{analytic6} are satisfied. 
If \( B'_s = \emptyset \), it is enough to let \( B_{s {}^\smallfrown{} \beta} = B'_{s {}^\smallfrown{} \beta} = \emptyset \) for all \( \beta < \lambda \) (and \( r_s \) be any real greater than \( 0 \)), hence let us consider the case \( B'_s \neq \emptyset \). For \(  \lambda_{\lh(s)} \leq \beta < \lambda \), let again \( B_{s {}^\smallfrown{}  \beta} = B'_{s {}^\smallfrown{}  \beta} =  \emptyset \).
Next notice that \( |\p (B'_s)| \nleq \lambda \) by definition of \( F' \), hence \( \dens( \p (B'_s ) ) = \lambda \) by Lemma~\ref{lem:densityvscardinality} applied to the metrizable space \(  \p(B'_s) \).
Apply Lemma~\ref{lem:r-spaces} to \( \p (B'_s) \) with \( \nu = \lambda_{\lh(s)} \), and enumerate the resulting \( r \)-spaced set (for some \( r > 0 \)) as \( ( x_\beta)_{\beta < \lh(s)} \). Set \( r_s = r/3 \), and let \( B_{s {}^\smallfrown{} \beta} \) be an open ball centered in \( x_\beta \) with radius at most \(   \min \{ 2^{-(\lh(s)+2)},r_s\} \) and small enough to ensure \( \mathrm{cl}(B_{s {}^\smallfrown{}  \beta}) \subseteq B_s \). Now for each \( \beta < \lh(s) \) pick \( y_\beta\in B(\lambda) \) such that \( (x_\beta, y_\beta) \in B'_s \), and let \( B'_{s {}^\smallfrown{}  \beta} \) be an open ball (relatively to \( F' \)) centered in \( (x_\beta, y_\beta) \) with radius at most \( 2^{-(\lh(s)+2)} \) and such that \( \mathrm{cl}(B'_{s {}^\smallfrown{} \beta}) \subseteq B'_s \cap \p^{-1}(B_{s {}^\smallfrown{}  \beta}) \) (This is possible because the latter is an open-in-\( F' \) neighborhood of \( (x_\beta, y_\beta) \).) This concludes the construction and the proof.
\end{proof}

\begin{corollary} \label{cor:analyticPSP}
Let \(\lambda\) be such that \( 2^{< \lambda} = \lambda \), and let \( X \) be a \(\lambda\)-Polish space. Then every \(\lambda\)-analytic subset \( A \) of \( X \) has the \lPSP, and thus either \( |A| \leq \lambda \) or else \( |A| = 2^\lambda \).
\end{corollary}

 In particular, the previous corollary applies to all \(\lambda\)-Borel subsets of \( X \). Thus it improves Theorem~\ref{cor:perfect3}, which in the new terminology states that  all \(\lambda\)-Borel subset of \( X \) satisfy an intermediate version of the property lying in between \( \lwPSP \) and the full \( \lPSP \) (but still equivalent to them by Corollary~\ref{cor:equivPSP} below).


A map \( f \colon X \to Y \) is called \markdef{\( \lambda \)-to-\( 1 \)} if \( f^{-1}(y) \) has size at most \( \lambda \) for every \( y \in Y \).

\begin{corollary} \label{cor:equivPSP}
Let $\lambda$ be $\omega$-inaccessible, $X$ be a $\lambda$-Polish space, and $A\subseteq X$. The following are equivalent:
 \begin{enumerate-(1)}
  \item \label{cor:equivPSP-1}
  $B(\lambda)$ embeds into $A$ as a closed-in-$X$ set;
	\item \label{cor:equivPSP-2}
	there is a continuous injection from $B(\lambda)$ into $A$;
	\item \label{cor:equivPSP-3}
	there is a $\lambda$-to-\( 1 \) continuous function from $B(\lambda)$ into $A$.
 \end{enumerate-(1)}
Moreover, if $2^{<\lambda} = \lambda$ we can weaken ``continuous'' to ``\(\lambda\)-Borel'' in items~\ref{cor:equivPSP-2} and~\ref{cor:equivPSP-3}, and we can replace $B(\lambda)$ with $\pre{\lambda}{2}$ in all occurrences.
\end{corollary}

\begin{proof}
 It is clear that each condition trivially implies the next one, so let $f \colon  B(\lambda)\to A$ be continuous and $\lambda$-to-\( 1 \). The set $f(B(\lambda))\subseteq A$ is \(\lambda\)-analytic, and since $|B(\lambda)|>\lambda$ and $f$ is $\lambda$-to-\( 1 \), then $|f(B(\lambda))|\nleq\lambda$. Since $f(B(\lambda))$ has the \( \lPSP^* \)  by Theorem~\ref{thm:analyticPSP}, we have that $B(\lambda)$ embeds into $f(B(\lambda))$ (and therefore in $A$) as a closed-in-$X$ set.

For the ``moreover'' part we can use the same argument, using also Proposition~\ref{prop:charanalytic} and Theorem~\ref{thm:homeomorphictoCantor} when needed.
\end{proof}

Corollary~\ref{cor:equivPSP} also shows that \( \lPSP^* \) is equivalent to \( \lwPSP^* \) when \( \lambda \) is \(\omega\)-inaccessible. If moreover \( 2^{< \lambda} = \lambda \), then all of \( \lPSP^{(*)} \), \( \lwPSP^{(*)} \), and \( \lbPSP^{(*)} \) are equivalent to each other, and no distinction among the variants of the \(\lambda\)-Perfect Set Property has to be made.

A further consequence of Corollary~\ref{cor:equivPSP} is that the \lPSP{} can be transferred from a single \(\lambda\)-Polish space of size greater than \( \lambda \) to all other \(\lambda\)-Polish spaces.

\begin{corollary} \label{cor:transferPSP}
Assume that \( 2^{< \lambda} = \lambda \) and let \( \boldsymbol{\Gamma} \) be a boldface \(\lambda\)-pointclass closed under \(\lambda\)-Borel preimages. Suppose that there is some \(\lambda\)-Polish space \( X \) with \( |X| > \lambda \) such that \( \lPSP(A) \) for all \( A \in \boldsymbol{\Gamma}(X) \). Then \( \lPSP(\boldsymbol{\Gamma}) \) holds, that is, for every \(\lambda\)-Polish space \( Y \) and every \( B \in \boldsymbol{\Gamma}(Y) \) we have \( \lPSP(B) \).
\end{corollary}

\begin{proof}
If \( |Y| \leq \lambda \) there is nothing to prove, so assume \( |Y| > \lambda \). Then by Theorem~\ref{thm:Borelisomorphism} there is a a \(\lambda\)-Borel isomorphism \( f \colon X \to Y \). Let \( B \in \boldsymbol{\Gamma}(Y) \). Then \( A = f^{-1}(B) \) is in \( \boldsymbol{\Gamma}(X) \) by the hypothesis on \( \boldsymbol{\Gamma}\), hence \( \lPSP(A) \). If \( |A| \leq \lambda \) then \( |B| \leq \lambda \) as well because \( f \) is a bijection. If instead there is a continuous injection \( g \colon \pre{\lambda}{2} \to A \), then \(f \circ g \colon \pre{\lambda}{2} \to B \) is a \(\lambda\)-Borel injection, hence we are done by Corollary~\ref{cor:equivPSP}.
\end{proof}


\section{A $\lambda$-coanalytic sets without \( \lPSP \)} \label{sec:noPSPcoanalytic}

It is natural after Theorem \ref{thm:analyticPSP} to ask whether $\lambda$-coanalytic sets need to have the \lPSP; or, more in general, if there is a $\lambda$-projective set without the \lPSP, and how low it is in the $\lambda$-projective hierarchy. In the classical case, there is no conclusive answer, since there are models with two extremes: in the constructible  universe $L$ there is a coanalytic set without the $\mathrm{PSP}$ (\cite[Theorem 13.12]{Kanamori}), while under 
 \( \ZF+ \AD \)
all sets of reals have the $\mathrm{PSP}$ (\cite[Theorem 27.9]{Kanamori}). We are going to prove now that a similar phenomenon arises also in the context of $\lambda$-Polish spaces (see Theorem~\ref{thm:noPSPcoanalytic} and Corollary~\ref{cor:PSPallPolishSpaces}). 

Suppose that $2^{<\lambda}=\lambda$. We already observed in Section~\ref{sec:defanalytic} that linear orders on $\lambda$ can be coded by elements of $\pre{\lambda}{2}$, and Proposition~\ref{prop:NWO} shows that the set  $\mathrm{WO}_\lambda$ of codes of well-founded linear orders on $\lambda$ is a complete $\lambda$-coanalytic set. 
The same coding procedure can be applied to arbitrary binary relations on \(\lambda\): every \( E \subseteq \lambda \times \lambda \) can be coded by \( x_E \in \pre{\lambda \times \lambda}{2} \) setting \( x_E(\alpha,\beta) = 1 \iff \alpha \mathrel{E} \beta \), and, conversely, every \( x \in \pre{\lambda \times \lambda}{2} \) codes the binary relation \( E_x = \{ (\alpha,\beta) \in \lambda \times \lambda \mid x(\alpha,\beta) = 1 \} \). In this section, structures with domain \(\lambda\) and a binary relation on it will be conceived as codes for \(\lambda\)-sized models of set theory, that is, structures of the form \( (M, {\in}) \) for \( M \) a set of cardinality \( \lambda \). For this reason, we will denote by \( \in_x \) (rather than \( E_x \)) the relation coded by \( x \in \pre{\lambda \times \lambda}{2} \). Also, the space \( \pre{\lambda \times \lambda}{2} \) will again be tacitly identified with \( \pre{\lambda}{2} \) using the G\"odel pairing function, and will thus be a \(\lambda\)-Polish space as soon as \( 2^{< \lambda} = \lambda \).

The following observation follows from~\cite[Proposition 8.9(d)]{AM}, once we recall that on \( \pre{\lambda \times \lambda}{2} \approx \pre{\lambda}{2} \) the product topology is finer than the bounded topology.

\begin{fact} \label{fct:LopezEscobar} 
If \( \upvarphi(x_1, \dotsc, x_n) \) is a first-order formula in the language of set theory
and \( \alpha_1, \dotsc, \alpha_n \in \lambda \), then the set
\[
\{ x \in \pre{\lambda \times \lambda}{2} \mid (\lambda,{\in_x}) \models \upvarphi[\alpha_1/x_1, \dotsc, \alpha_n/x_n] \},
\]
is a \(\lambda\)-Borel subset of \( \pre{\lambda \times \lambda}{2} \). 
\end{fact}
 Consider the following subsets of \( \pre{\lambda \times \lambda}{2} \):
\begin{align*}
\mathrm{E}_\lambda & =  \{x\in\pre{\lambda \times \lambda}{2} \mid {\in_x} \text{ is extensional}\} \\
\mathrm{W}_\lambda & =  \{x\in\pre{\lambda \times \lambda}{2} \mid {\in_x} \text{ is well-founded}\} \\
\mathrm{EW}_\lambda & = \mathrm{E}_\lambda \cap \mathrm{W}_\lambda.
\end{align*}
All three sets are cleary invariant under isomorphism, that is, if \( x \) belongs to one of the sets and \( y \in \pre{\lambda \times \lambda}{2} \) is such that \( (\lambda,\in_x) \cong (y,\in_y) \), then \( y \) belongs to the same set as well.

\begin{lemma} \label{lem:extensional+well-founded}
The set \( \mathrm{E}_\lambda \) is \(\lambda\)-Borel, while \( \mathrm{W}_\lambda \) and \( \mathrm{EW}_\lambda \) are complete \(\lambda\)-coanalytic sets.
\end{lemma}

\begin{proofsketch}
An easy computation, involving Fact~\ref{fct:LopezEscobar} when necessary, provides the upper bounds on the complexity of the three sets at hand. The completeness of \( \mathrm{W}_\lambda \) and \( \mathrm{EW}_\lambda \) for \(\lambda\)-coanalytic sets follows from the proof of Proposition~\ref{prop:NWO}, once we notice that linear orders are necessarily extensional and \( \mathrm{LO}_\lambda \cap \mathrm{W}_\lambda = \mathrm{WO}_\lambda \).
\end{proofsketch}

Recall from Section~\ref{subsec:transitivity} that every extensional, well-founded relation collapses to a unique isomorphic transitive model of set theory, its Mostowski collapse, via the Mostowski collapsing function. For \( x \in \mathrm{EW}_\lambda \),  we denote by \( \mathrm{tr}(\lambda,\in_x) \) the Mostowski collapse of \( (\lambda,{\in_x}) \), and by \( \pi_x \) the corresponding collapsing function, which is an isomorphism between \((\lambda,{\in_x}) \) and \( \mathrm{tr}(\lambda,\in_x) \). 
Conversely, every transitive model of set theory \( (M,{\in}) \) with \( |M| = \lambda \) can be coded by some element of \( \mathrm{EW}_\lambda \): just fix any bijection \( \pi \colon \lambda \to M \), and define \( x \in \mathrm{EW}_\lambda \) by letting \( x(\alpha,\beta) = 1 \iff \pi(\alpha) \in \pi(\beta) \). (Notice that for such an \( x \in \mathrm{EW}_\lambda \), we have \( \pi_x = \pi \).)
In this framework, the elements of \( \mathrm{EW}_\lambda \) must thus be conceived as codes for the transitive models of set theory of size \( \lambda \). Of course, codes are not unique: there are \( 2^\lambda \)-many codes (one for each bijection \( \pi \colon \lambda \to M \)) for any \( (M,\in) \) as above.

\begin{remark} \label{rmk:collapsemisleading}
Note that the name ``collapse'', in this case, is quite misleading: the ``collapse'' of $\alpha<\lambda$ can be an ordinal greater than $\lambda$. For example, suppose that \( x \in \mathrm{EW}_\lambda \) and that $\alpha$ is an ordinal in $(\lambda,\in_x)$, i.e., that the set 
\[ 
\mathrm{Pred}_x(\alpha) = \{ \beta < \lambda \mid \beta \in_x \alpha \} 
\] 
of $\in_x$-predecessors of $\alpha$ is linearly ordered. It is possible that  
$\mathrm{ot}(\mathrm{Pred}_x(\alpha),\in_x) = \gamma>\lambda$, and therefore $\pi_x(\alpha)$ will go up to $\gamma$. In fact, $\On^{\mathrm{tr}(\lambda,\in_x)}$ can be any ordinal between $\lambda$ and $\lambda^+$. In a certain sense, $\mathrm{tr}(\lambda,\in_x)$ is curled inside $\lambda$, and the ``collapse'' unfurls it.
\end{remark}

When $x\in \mathrm{EW}_\lambda$, we can recursively define the ``$\in_x$-transitive closure'' \( \mathrm{trcl}_x(\alpha) \) of any $\alpha\in\lambda$, but we will use the following equivalent definition that will make the future complexity analysis more transparent:
\begin{align*}
 \beta\in\mathrm{trcl}_x(\alpha)\quad\text{if and only if}\quad\exists s\in\pre{<\omega}{\lambda} \,
[ &s(0)=\beta \wedge {s(\lh(s)-1)=\alpha} \wedge{}  \\ 
 &\ \forall n<\lh(s)-1\, (s(n)\in_x s(n+1)) ].
\end{align*}
Contrary to $\mathrm{Pred}_x(\alpha)$, it is possible that there is no element of $\lambda$ that codes $\mathrm{trcl}_x(\alpha)$, i.e., that there is no $\beta<\lambda$ such that $\forall \gamma<\lambda \, (\gamma\in_x\beta \iff \gamma\in\mathrm{trcl}_x(\alpha))$, as possibly $\mathrm{tr}(\lambda,\in_x)$ does not satisfy enough axioms of $\ZFC$ for the transitive closure to live inside the model.
This would be an issue in the proof of Theorem~\ref{thm:noPSPcoanalytic}. Luckily, in the situation of interest we will have that \( \mathrm{trcl}_x(\alpha) = \mathrm{Pred}_x(\alpha) \) (see Remark~\ref{rmk:restrictions}), and this will allow us to overcome the mentioned difficulty. 


Recall from Section~\ref{subsec:transitivity} that \( H_\lambda \) is the collection of all sets whose transitive closure has size smaller than \( \lambda \), and that \( |H_\lambda| = \lambda \) under the assumption \( 2^{< \lambda} = \lambda \). 

\begin{lemma} \label{lem:collapse}
Assume that \( 2^{< \lambda} = \lambda \), and let $b\in H_\lambda$ and \( \alpha < \lambda \). Then the set 
\[
P_{\alpha,b} = \{x\in \mathrm{EW}_\lambda \mid  \pi_x(\alpha)=b\}
\] 
is $\lambda$-analytic relatively to \( \mathrm{EW}_\lambda \). 
\end{lemma}

\begin{proof}
Observe that \( \pi_x(\alpha)=b \) if and only if there is a function $f \in \pre{\lambda}{( H_\lambda)}$ such that for every $\beta,\gamma < \lambda$
\begin{enumerate-(1)}
\item \label{lem:collapse-1}
\( f \restriction \mathrm{trcl}_x(\alpha) \) is a bijection between \( \mathrm{trcl}_x(\alpha) \) and \( \mathrm{trcl}(b) \);
\item \label{lem:collapse-2}
if \( \beta , \gamma \in \mathrm{trcl}_x(\alpha) \), then \( \beta \in_x \gamma \iff f(\beta) \in f(\gamma) \);
\item \label{lem:collapse-3}
if \( \beta \notin \mathrm{trcl}_x(\alpha) \), then \( f(\beta) = 0 \).
\end{enumerate-(1)}
The three conditions above define a subset \( C \) of \( \pre{\lambda}{(H_\lambda)} \times \pre{\lambda \times \lambda}{2} \). Equip \( \pre{\lambda}{(H_\lambda)} \) with the product of discrete topology on \( H_\lambda \), so that it becomes a topological space homeomorphic to \( (\pre{\lambda}{\lambda}, \tau_p) \). By Example~\ref{xmp:lambda^lambdawithproduct}, the space \( \pre{\lambda}{(H_\lambda)} \) is then standard \(\lambda\)-Borel, and \( C \) is easily seen to be \(\lambda\)-Borel in the product \( \pre{\lambda}{(H_\lambda)} \times \pre{\lambda \times \lambda}{2} \). Since \( P_{\alpha,b} = \p(C) \cap \mathrm{EW}_\lambda \), the result easily follows.
\end{proof}

For technical reason, it would be desirable to have that the set \( P_{\alpha,b} \) from Lemma~\ref{lem:collapse} is \(\lambda\)-Borel (relatively to the space of codes under consideration) rather then merely \(\lambda\)-analytic. This is what happens in the classical case \( \lambda = \omega \), and the same would be true if \(\lambda\) were regular because we could then replace the function \( f \in \pre{\lambda}{(H_\lambda)} \) with a function in \( \pre{<\lambda}{(H_\lambda)} \) (which is a set of cardinality \(\lambda\), for regular cardinals \(\lambda\) satisfying \( 2^{< \lambda} = \lambda \)) doing essentially the same job. But in the singular case, this does not work for those \( x \in \mathrm{EW}_\lambda \) for which the set of \( \in_x \)-predecessors of \( \alpha \) is unbounded in \( \lambda \), let alone with the fact that \( \pre{< \lambda}{(H_\lambda)} \) would anyway be of cardinality larger than \(\lambda\). To overcome these difficulties, we restrict the set of codes \( \mathrm{EW}_\lambda \) to a set of ``good'' codes \( \mathrm{BC}_\lambda \).

\begin{defin}
 We say that $x \in \mathrm{EW}_\lambda$ has \markdef{bounded collapse} if for every $b\in H_\lambda$ and \( \alpha < \lambda \), if $\pi_x(\alpha) = b$ then \( \mathrm{trcl}_x(\alpha) \) is bounded in \(\lambda\). 
We denote by $\mathrm{BC}_\lambda$ the sets of all $x\in\mathrm{EW}_\lambda$ that have bounded collapse.
\end{defin}

\begin{lemma}\label{lem:BC}
Assume that \( 2^{< \lambda} = \lambda \).
\begin{enumerate-(i)}
\item \label{lem:BC-1}
$\mathrm{BC}_\lambda$ is a $\lambda$-coanalytic subset of \( \pre{\lambda \times \lambda}{2} \).
\item  \label{lem:BC-2}
For every $b\in H_\lambda$ and $\alpha<\lambda$, the set
\[
P'_{\alpha,b}  = \{x\in \mathrm{BC}_\lambda \mid \pi_x(\alpha)=b\} 
\]
is $\lambda$-Borel relatively to $\mathrm{BC}_\lambda$, and therefore it is $\lambda$-coanalytic in $\pre{\lambda \times \lambda}{2}$.
 \end{enumerate-(i)}
\end{lemma}

\begin{proof}
\ref{lem:BC-1}
The definition of $\mathrm{BC}_\lambda$ is: 
	  \begin{equation*}
		 x\in \mathrm{EW}_\lambda\wedge\forall\alpha<\lambda\, \forall b\in H_\lambda\, [ x \notin P_{\alpha,b} \vee
		\exists\gamma<\lambda\, \forall\delta<\lambda\, (\delta\in\mathrm{trcl}_x(\alpha) \Rightarrow \delta<\gamma)].
		\end{equation*} 
By Lemma~\ref{lem:collapse}, the formula \( x \notin P_{\alpha,b} \) defines a \(\lambda\)-coanalytic set, while the part on the right of the disjunction defines a \(\lambda\)-Borel set. Therefore $\mathrm{BC}_\lambda$ is the intersection of the $\lambda$-coanalytic set \( \mathrm{EW}_\lambda \) (Lemma~\ref{lem:extensional+well-founded}), and the intersection of $\lambda$-many $\lambda$-coanalytic sets, thus it is \(\lambda\)-coanalytic itself by Corollary~\ref{cor:closurepropertiesofanalytic}.
	  
\ref{lem:BC-2}
Suppose that \( x \in \mathrm{BC}_\lambda \). If \( \pi_x(\alpha) = b \), then this can be witnessed by a function \( f \in \pre{\delta}{(H_\kappa)} \), for \( \delta  = \sup ( \mathrm{trcl}_x(\alpha)) < \lambda \) and \( \kappa < \lambda \) such that \( b \in H_\kappa \), satisfying conditions~\ref{lem:collapse-1}--\ref{lem:collapse-3} from the proof of Lemma~\ref{lem:collapse} for all \( \beta,\gamma < \delta \). Therefore the quantifier ``\( \exists f \in \pre{\lambda}{(H_\lambda)} \)'' in the proof of Lemma~\ref{lem:collapse} can now be turned into ``\( \exists \delta < \lambda \, \exists \kappa < \lambda \, \exists f \in \pre{\delta}{(H_\kappa)} \)'', which easily corresponds to a union of size \(\lambda\) because, having assumed \( 2^{< \lambda} = \lambda \), we have that \( |\pre{\delta}{(H_\kappa)}| = | \pre{\delta}{(2^{<\kappa})} | < \lambda \). We thus conclude that \( P'_{\alpha,b} \) is \(\lambda\)-Borel relatively to \( \mathrm{BC}_\lambda \).
\end{proof}

Notice that the set \( \mathrm{BC}_\lambda \) is no longer closed under isomorphism: this is not an issue as long as we can still code all transitive models of set theory of size \(\lambda\).

\begin{lemma} \label{lem:existenceofBCcodes}
Assume that \( 2^{< \lambda} = \lambda \). 
For any transitive model of set theory $(M,\in)$ such that $|M|=\lambda$, there is $x\in \mathrm{BC}_\lambda$ such that $\mathrm{tr}(\lambda,\in_x) = (M,\in)$.
\end{lemma}

\begin{proof}
By the discussion before Remark~\ref{rmk:collapsemisleading}, it is enough to find a bijection \( \pi \colon \lambda \to M  \) such that for every \( b \in M \cap H_\lambda \), the set \( \{ \pi^{-1}(c) \mid c \in \mathrm{trcl}(b) \} \) is bounded in \( \lambda \). For ease of exposition, we will in fact define the inverse of \( \pi \), which will be denoted by \( \varrho \).

First of all, note that $|M\cap H_\lambda|=\lambda$. If $M\subseteq H_\lambda$ this is obvious. Otherwise, there is an $a\in M$ such that $|\mathrm{trcl}(a)|=\lambda$. By the Axiom of Foundation, there must be an $a'\in M$ such that $|\mathrm{trcl}(a')|=\lambda$, but for all $b\in a'$, and therefore all $b\in\mathrm{trcl}(a')$, $b\in H_\lambda$. By transitivity of $M$, $\mathrm{trcl}(a')\subseteq M$, so $|M\cap H_\lambda|\geq|\mathrm{trcl}(a')|=\lambda$ again.
Moreover, we can suppose that also $|M\setminus H_\lambda|$ has cardinality $\lambda$ (otherwise the proof is even easier, and is left to the reader). 

Since the sequence \( \vec \lambda = (\lambda_i)_{i \in \omega} \) is cofinal in $\lambda$, we have that $H_\lambda=\bigcup_{i\in\omega}H_{\lambda_i}$ and $|H_{\lambda_i}|=2^{<\lambda_i} < \lambda$. (In particular, \( M \cap (H_{\lambda_{i+1}} \setminus H_{\lambda_i}) \neq \emptyset \) because of the cardinality of \( M \cap H_\lambda \).) We can thus find a bijection $\varrho_1$ between $M\cap H_\lambda$ and the set $A = \{ \alpha < \lambda \mid \alpha \text{ is even} \}$ such that for every $i\in\omega$, if \( a \in M \cap H_{\lambda_i} \) and \( b \in M \cap (H_\lambda \setminus H_{\lambda_i}) \) then \( \varrho_1(a) < \varrho_1(b) \). Then we let \( \varrho_2 \) any bijection between \( M \setminus H_\lambda \) and \( \lambda \setminus A\), and finally we let \( \pi = \varrho^{-1} \) for \( \varrho = \varrho_1 \cup \varrho_2 \).

Given any $b\in M\cap H_\lambda$, there exists an $i\in\omega$ such that $b \in
 H_{\lambda_i}$, so that also \( c \in H_{\lambda_i} \) for every \( c \in \mathrm{trcl}(b) \). By choice of \( \varrho_1 \), the set \( \{ \pi^{-1}(c) \mid c \in b \} \) is then bounded in \(\lambda\) by any \( \varrho_1(d) \) with \( d \in M \cap (H_{\lambda_{i+1}} \setminus H_{\lambda_i}) \).
\end{proof}

Before moving to the main result of the section, we need to compute the complexity of a few other sets. For \( x,y \in \pre{\lambda\ \times \lambda}{2} \), we write \( x \cong y \) to abbreviate \( (\lambda, \in_x) \cong (\lambda , \in_y) \).

\begin{lemma} \label{lem:technical}
\begin{enumerate-(i)}
\item \label{lem:technical-1}
The isomorphism relation \( \cong \) 
 is \(\lambda\)-analytic as a subset of \( (\pre{\lambda \times \lambda}{2})^2 \).

\item \label{lem:technical-2}
For every $\alpha<\lambda$, the set  $\{(x,y)\in \mathrm{BC}_\lambda\times\pre{\lambda \times \lambda}{2} \mid \pi_x(\alpha)=y\}$ is $\lambda$-Borel relatively to $\mathrm{BC}_\lambda\times\pre{\lambda \times \lambda}{2}$.	

\item \label{lem:technical-3}
The set $\{(x,y)\in \mathrm{BC}_\lambda\times\pre{\lambda \times \lambda}{2}\mid y\in \mathrm{tr}(\lambda,\in_x)\}$ is $\lambda$-Borel relatively to $\mathrm{BC}_\lambda\times\pre{\lambda \times \lambda}{2}$.
 \end{enumerate-(i)}
\end{lemma}

\begin{proof}
\ref{lem:technical-1}
This follows from the fact that $x \cong y$ if and only if there is a bijective map $f\in\pre{\lambda}{\lambda}$ such that \( \forall \alpha,\beta < \lambda \, (\alpha \in_x \beta \iff f(x) \in_y f(y)) \), together with the fact that \( \pre{\lambda}{\lambda} \) is a standard \(\lambda\)-Borel space when equipped with the \(\lambda\)-Borel structure induced by the product topology (Example~\ref{xmp:lambda^lambdawithproduct}).

\ref{lem:technical-2}
Here we cannot directly use Lemma~\ref{lem:collapse} because $y\notin H_\lambda$.
However,
$\pi_x(\alpha)=y$ if and only if 
\begin{multline*}
\forall \delta < \lambda \, [ \delta \in_x \alpha \Rightarrow \exists \beta,\gamma < \lambda \, \exists i \in \{ 0,1 \} \, (y(\beta,\gamma) = i \wedge \pi_x(\delta) = ((\beta,\gamma),i))] \\
\wedge \forall \beta,\gamma < \lambda \, \forall i \in \{ 0,1 \} \, [y(\beta,\gamma) = i \Rightarrow \exists \delta < \lambda \, (\delta \in_x \alpha \wedge \pi_x (\delta)= ((\beta,\gamma),i))].
\end{multline*}
Since each \( ((\beta,\gamma),i) \) belongs to \( H_\lambda \), the sets \( P'_{\delta,((\beta,\gamma),i)} \) are \(\lambda\)-Borel relatively to \( \mathrm{BC}_\lambda \) by Lemma~\ref{lem:BC}\ref{lem:BC-2}, and thus an easy computation along the previous formula yields the desired result.

\ref{lem:technical-3}
Use part~\ref{lem:technical-2} and the fact that for \( y \in \pre{\lambda \times \lambda}{2} \), we have
\[
y\in \mathrm{tr}(\lambda,\in_x) \iff \exists\alpha<\lambda\, (\pi_x(\alpha)=y). \qedhere
\]
\end{proof}

Recall from Proposition~\ref{prop:constructibility} the sentence \( \upsigma_0 \) and the formula \( \upvarphi_0(u,v) \). The following observation will be used in the proof of Theorem~\ref{thm:noPSPcoanalytic}.

\begin{remark} \label{rmk:restrictions}
Let \( \lambda \leq \delta < \lambda^+ \), and let \( x \in \mathrm{BC}_\lambda \) be such that $\mathrm{tr}(\lambda,{\in_x})=L_\delta$. (Such an \( x \) exists by Lemma~\ref{lem:existenceofBCcodes} and the fact that \( |L_\delta| = |\delta| = \lambda \).)
Let \( \gamma < \delta \) be a limit ordinal.
Since \( L_\gamma \in L_\delta \), there is \( \beta < \lambda \) such that \( \pi_x(\beta) = L_\gamma \). 
It follows that \( \pi_x \) maps \( \mathrm{trcl}_x(\beta) \) onto \( L_\gamma \), and
since \( L_\gamma \) is transitive and \( \pi_x \) is an isomorphism, we get that \( \mathrm{trcl}_x(\beta) = \mathrm{Pred}_x(\beta) \).
Denote by $(\lambda,{\in_x}) \restriction \beta$ the substructure of \( (\lambda, \in_x) \) with domain \( \mathrm{Pred}_x(\beta) \).
By the previous discussion, \( (\lambda,\in_x) \restriction \beta \) is isomorphic to \( L_\gamma \) via \( \pi_x  \), so that in particular \( (\lambda,\in_x) \restriction \beta \models \upsigma_0 \). 
Conversely, if \( \beta < \lambda \) is such that \( (\lambda,{\in_x}) \restriction \beta \models \upsigma_0 \), where \( (\lambda,{\in_x}) \restriction \beta \) is as above, then by choice of \( \upsigma_0 \) 
the image of \( \mathrm{Pred}_x(\beta) \) under \( \pi_x \) is \( L_\gamma \), for some limit ordinal \( \gamma \). Since \( L_\gamma \) is transitive and \( \pi_x \) is an isomorphism, this implies that \( \mathrm{Pred}_x(\beta) = \mathrm{trcl}_x(\beta) \), and thus \( \pi_x(\beta) = L_\gamma \). Therefore \( \gamma < \delta \) because \( \pi_x(\beta) \in \mathrm{tr}(\lambda,{\in_x}) = L_\delta \).
Summing up: for a given \( \beta < \lambda \), we have that \( \pi_x(\beta) = L_\gamma \) for some limit ordinal \( \gamma < \delta \) if and only if \( (\lambda,{\in_x}) \restriction \beta \models \upsigma_0 \).

Notice also that for every \( \beta < \lambda \), the condition ``\( (\lambda,\in_x) \restriction \beta \models \upsigma_0 \)'' is \(\lambda\)-Borel. Indeed, let \( \upsigma'_0(w) \) be obtained by bounding every quantifier in \( \upsigma_0 \) with a fresh variable \( w \), that is, every quantification of the form ``\( \exists v \)'' appearing in \( \upsigma_0 \) is replaced by ``\( \exists v \in w \)'', and similarly for the universal quantifications. Then for every \( x \in \pre{\lambda \times \lambda}{2} \) we have \( (\lambda, {\in_x}) \restriction \beta \models \upsigma_0 \iff (\lambda,{\in_x}) \models \upsigma'_0[\beta/w] \), and the latter is a \(\lambda\)-Borel condition by Fact~\ref{fct:LopezEscobar}.
\end{remark}

We can now prove that, under some conditions, there exists a $\lambda$-coanalytic set without the \( \lPSP \). 

\begin{theorem}\label{thm:noPSPcoanalytic}
Assume that \( 2^{< \lambda} = \lambda \) and $(\lambda^+)^L=\lambda^+$. Then there is a $\lambda$-coanalytic set $A \subseteq \pre{\lambda}{2}$ without the \lPSP.
%
\end{theorem}

The idea behind the proof of Theorem~\ref{thm:noPSPcoanalytic}, an adaptation of \cite[Theorem 13.12]{Kanamori}, is the following.
For every ordinal \( \lambda \leq \alpha < \lambda^+ \), let%
\footnote{Here we use the fact that we can find a code for \( \alpha \) inside \( L \) because we assumed \( (\lambda^+)^L = \lambda^+ \).}
 \( x^{(\alpha)} \in \mathrm{WO}_\lambda \cap L\) be \( <_L \)-least
such that \( \mathrm{ot}(\preceq_{x^{(\alpha)}}) = \alpha \), and let 
\[ 
A = \{ x^{(\alpha)} \mid \lambda \leq \alpha < \lambda^+ \} . 
\]
The set \( A \) is designed to be a copy of \( \{ \alpha \in \On \mid \lambda \leq \alpha < \lambda^+ \} \) inside \( \pre{\lambda \times \lambda}{2}  \), in the sense that the map \( \alpha \mapsto x^{(\alpha)} \) is a (canonical) bijection between the two sets.
Thus \( |A| \nleq \lambda \) and, if \( A \) were \(\lambda\)-coanalytic, then by an easy argument involving the Boundedness Theorem~\ref{thm: boundednessforanalytic} the space \( \pre{\lambda}{2} \) could not be embedded into \( A \). Unfortunately, the \( <_L \)-minimality of \( x^{(\alpha)} \) in its definition cannot be rendered in a \( \lambda \)-coanalytic way. 
The best we can do is to observe that this can be checked by looking at a model of the form \( L_\delta \), for some limit \( \lambda \leq \delta < \lambda^+ \). Indeed, \( x^{(\alpha)} \) can be identified with a subset of \( \lambda \), hence there is some limit ordinal \( \lambda \leq \delta < \lambda^+ \) such that \( x^{(\alpha)} \in L_\delta \) (see e.g.\ the proof of~\cite[Theorem 13.20]{Jech2003}). By Proposition~\ref{prop:constructibility}, for every \( y \) we have 
\[ 
y <_L x^{(\alpha)} \iff y \in L_\delta \wedge (L_\delta,\in) \models \upvarphi_0[y/u,x^{(\alpha)}/v].
\]
Thus since all codes for \( \alpha \) are obviously ``isomorphic'' to each other, \( x^{(\alpha)} \) is the \( <_L \)-minimal code for \( \alpha \) if and only if there is \( z \in \pre{\lambda \times \lambda}{2} \) such that 
\begin{multline} \label{eq:A}
z \in \mathrm{BC}_\lambda \wedge (\lambda,{\in_z}) \models \upsigma_0 \wedge \exists \beta < \lambda \, (\pi_z(\beta) = x^{(\alpha)} \wedge \\
\forall \gamma < \lambda \,  \forall y \in \pre{\lambda \times \lambda}{2}  \, ((\lambda,{\in_z}) \models \upvarphi_0[\gamma/u,\beta/v] \wedge \pi_z(\gamma) = y \Rightarrow y \not\cong x^{(\alpha)}).
\end{multline}
Let us call any \( z \) as in~\eqref{eq:A} a witness for the \( <_L \)-minimality of \( x^{(\alpha)} \).
Using the results discussed in this section, one can easily see that the definition in~\eqref{eq:A} is \(\lambda\)-coanalytic, and thus we end up with a \( \lS^1_2 \)-condition because of the existential quantification on \( z \) preceding it. (Indeed, if we replace \( x^{(\alpha)} \) with \( x \) in~\eqref{eq:A} and add the requirement that \( x \in \mathrm{WO}_\lambda \), we get a \( \lS^1_2 \)-definition of \( A \) itself.)

To get rid of this extra quantification, one can make the choice of \( z \) canonical by first minimizing the ordinal \(\delta\) (that is, we only consider the smallest limit ordinal \( \lambda \leq \delta_1 < \lambda^+ \) for which \( x^{(\alpha)} \in L_{\delta_1} \)), and then taking the \( <_L \)-minimal \( z^{(\alpha)}_1 \in \mathrm{BC}_\lambda \) such that \( \mathrm{tr}(\lambda,{\in_{z^{(\alpha)}_1}}) = L_{\delta_1} \). 
But now we have only shifted our problem: how can we certify that \( z_1 \) is the \( <_L \)-least with the required properties? Repeating the above argument, we may look at witnesses \( z \in \mathrm{BC}_\lambda \) coding some large enough \( L_\delta \) for which \( z^{(\alpha)}_1 \in L_\delta \). Letting \( \lambda \leq \delta_2 < \lambda^+ \) be the smallest such \( \delta \), and \( z^{(\alpha)}_2 \) be the \( <_L \)-least \( z \in \mathrm{BC}_\lambda \) such that \( \mathrm{tr}(\lambda,{\in_z}) = L_{\delta_2} \), we get a canonical witness for the \( <_L \)-minimality of \( z^{(\alpha)}_1 \), and we have moved one step further. 

Iterating the process and setting \( z^{(\alpha)}_0 = x^{(\alpha)} \), we end up with a sequence \( (z^{(\alpha)}_n)_{n \in \omega} \in \pre{\omega}{(\pre{\lambda \times \lambda}{2})} \) such that \( z^{(\alpha)}_0 \) is the \( <_L \)-least code for \( \alpha \), and each \( z^{(\alpha)}_{n+1} \) is the ``canonical witness'' for the \( <_L \)-minimality of \( z^{(\alpha)}_n \).
Redefining \( A \) as 
\[
A = \{ (z^{(\alpha)}_n)_{n \in \omega} \in \pre{\omega}{(\pre{\lambda \times \lambda}{2})} \mid \lambda \leq \alpha < \lambda^+ \},
\]
we finally get a \( \lP^1_1 \)-set such that for each \( \lambda \leq \alpha < \lambda^+ \) there is exactly one element of \( A \) representing \( \alpha \) (in the sense that its first coordinate codes precisely \( \alpha \)), and thus \( A \) cannot have the \( \lPSP \) by the initial argument applied to this new \( A \).

\begin{proof}[Proof of Theorem~\ref{thm:noPSPcoanalytic}]
By Corollary~\ref{cor:transferPSP}, we can clearly work with the \(\lambda\)-Polish space \( Z = \pre{\omega}{(\pre{\lambda \times \lambda}{2})} \) instead of \( \pre{\lambda}{2} \), as they are \(\lambda\)-Borel isomorphic by Theorem~\ref{thm:Borelisomorphism}.

Let \( A \subseteq Z \) be the set of all \( (z_n)_{n \in \omega} \) such that for every \( n \in \omega \):
\begin{enumerate-(1)}
\item \label{thm:noPSPcoanalytic-1}
\( z_0 \in \mathrm{WO}_\lambda \);
\item \label{thm:noPSPcoanalytic-2}
\( z_{n+1} \in \mathrm{BC}_\lambda \wedge (\lambda,{\in_{z_{n+1}}}) \models \upsigma_0 \);
\item \label{thm:noPSPcoanalytic-3}
\( z_n \in \mathrm{tr}(\lambda,{\in_{z_{n+1}}}) \);
\item \label{thm:noPSPcoanalytic-4}
\( \forall \beta < \lambda \, [ (\lambda,{\in_{z_{n+1}}}) \restriction \beta \models \upsigma_0 \Rightarrow \neg \exists \gamma < \lambda \, (\gamma \in_{z_{n+1}} \beta \wedge \pi_{z_{n+1}}(\gamma)= z_n)] \);
\item \label{thm:noPSPcoanalytic-5}
for every \( y \in \pre{\lambda \times \lambda}{2} \), if
\[ 
\exists \beta, \gamma < \lambda \, (\pi_{z_{n+1}}(\beta) = y \wedge \pi_{z_{n+1}}(\gamma) = z_n \wedge (\lambda,{\in_{z_{n+1}}})\models \upvarphi_0[\beta/u,\gamma/v] ) ,
\]
then \( y \not\cong z_n \) .
\end{enumerate-(1)}

Condition~\ref{thm:noPSPcoanalytic-1} says that that \( z_0 \) codes some ordinal \( \omega \leq \alpha < \lambda^+ \), and is \(\lambda\)-coanalytic by Proposition~\ref{prop:NWO}.

Condition~\ref{thm:noPSPcoanalytic-2} says that each \( z_{n+1} \) has bounded collapse (so that \( (z_n)_{n \in \omega} \in \mathrm{WO}_\lambda \times \pre{\omega}{(\mathrm{BC}_\lambda)} \)) and codes \( L_{\delta_{n+1}} \) (necessarily for some \( \lambda \leq \delta_{n+1} < \lambda^+ \), since \( (\lambda,\in_{z_{n+1}}) \) is a structure of size \( \lambda \)),
and is \(\lambda\)-coanalytic by Lemma~\ref{lem:BC}\ref{lem:BC-1} and Fact~\ref{fct:LopezEscobar}.

Condition~\ref{thm:noPSPcoanalytic-3} says that \( z_n \in L_{\delta_{n+1}} \), and is \( \lambda \)-Borel relatively to \( \mathrm{WO}_\lambda \times \pre{\omega}{(\mathrm{BC}_\lambda)} \) by Lemma~\ref{lem:technical}\ref{lem:technical-3}.

Condition~\ref{thm:noPSPcoanalytic-4} says that \( \delta_{n+1} \) is the smallest ordinal \( \lambda \leq \delta < \lambda^+ \)  with \( z_n \in L_\delta \), and is \(\lambda\)-Borel relatively to  \( \mathrm{WO}_\lambda \times \pre{\omega}{(\mathrm{BC}_\lambda)} \) by Remark~\ref{rmk:restrictions} and Lemma~\ref{lem:technical}\ref{lem:technical-2}.

Finally, Condition~\ref{thm:noPSPcoanalytic-5} says that \( z_n \) is the \( <_L \)-least code for \( \alpha \) (if \( n = 0 \)) or for \( L_{\delta_n} \) (if \( n > 1 \)), and is \(\lambda\)-coanalytic by Fact~\ref{fct:LopezEscobar} and parts~\ref{lem:technical-1} and~\ref{lem:technical-2} of Lemma~\ref{lem:technical}.

This shows that \( A \) is \(\lambda\)-coanalytic. Moreover, for every ordinal \( \lambda \leq \alpha < \lambda^+ \) there is exactly one \( (z_n)_{n \in \omega} \) such that \( \mathrm{ot}(\preceq_{z_0}) = \alpha \). The existence part follows from the discussion before this proof. Uniqueness follows from the fact that if \( (z_n)_{n \in \omega} , (z'_n)_{n \in \omega} \in A \), then \( z_n = z'_n \) implies \( z_{n+1} = z'_{n+1} \) for every \( n \in \omega \) by definition of \( A \). Since if  \( \mathrm{ot}(\preceq_{z_0}) = \mathrm{ot}(\preceq_{z'_0}) = \alpha \) then \( z_0 = z'_0 \) by definition of \( A \) again, the result easily follows.

Finally, we show that \( A \) does not have the \lPSP. First of all, \( |A| \nleq \lambda \) because the map \( g \) sending each \( \lambda \leq \alpha < \lambda^+ \) to the unique \( (z_n)_{n \in \omega} \in A \) with \( \mathrm{ot}(\preceq_{z_0}) = \alpha \) is a well-defined bijection between \( \{ \alpha \in \On \mid \lambda \leq \alpha < \lambda^+ \} \) and \( A \). (No choice is needed here, as there is only one bijection with the described property.)
Suppose towards a contradiction that there is an embedding $f \colon \pre{\lambda}{2} \to A$, and let 
\[ 
B = \{ z \in \pre{\lambda \times \lambda}{2} \mid z = z_0 \text{ for some } (z_n)_{n \in \omega} \in \ran(f) \}. 
\]
Being the range of the composition of \( f \) with the projection from \( Z \) onto its first coordinate, \( B \in \lS^1_1(\pre{\lambda \times \lambda}{2}) \). Moreover, \( B \subseteq \mathrm{WO}_\lambda \)  by construction. Recall from (the proof of) Theorem~\ref{thm:Pi11isranked} that the map \( \mathrm{WO}_\lambda \to \lambda^+ \) sending each \( x \in \mathrm{WO}_\lambda \) to its order type \( \mathrm{ot}(\preceq_x) \) is a \( \lP^1_1 \)-rank, and therefore 
\[ 
C = \{ \mathrm{ot}(\preceq_z) \mid z \in B \} 
\] 
is bounded in \( \lambda^+ \) by Theorem~\ref{thm: boundednessforanalytic}, and therefore \( |C| \leq \lambda \). But by construction, \( C \) is the range of the injective map \( g^{-1} \circ f \), and hence we should have \( |C| \geq |\pre{\lambda}{2}| > \lambda \), a contradiction.
\end{proof}

The hypothesis \( (\lambda^+)^L = \lambda^+ \) in Theorem \ref{thm:noPSPcoanalytic} is not trivial, as it is an anti-large cardinal hypothesis, i.e., a hypothesis that is inconsistent with large enough cardinals. 
In fact, whether \( (\lambda^+)^L = \lambda^+ \) or not depends on the large cardinal object $0^\sharp$: if $0^\sharp$ exists, then by \cite[Theorem 9.17(c)]{Kanamori} any uncountable cardinal $\kappa$ is inaccessible in $L$, and therefore $(\kappa^+)^L\leq(2^\kappa)^L<\kappa^+$. On the other hand, if $0^\sharp$ does not exist, then by Jensen's Covering Theorem $(\lambda^+)^L=\lambda^+$ holds for all singular cardinals $\lambda$ (see~\cite[Corollary 18.32]{Jech2003}). 


%
Passing to the contrapositive, if \( \lPSP(\lP^1_1) \) holds for some cardinal \(\lambda\) satisfying \( 2^{< \lambda} = \lambda > \cf(\lambda) = \omega \), then $(\lambda^+)^L\neq\lambda^+$, and hence $0^\sharp$ exists. We can derive two remarks from this.

On the one hand, whether \( 0^\sharp \) exists or not is independent of the chosen cardinal \(\lambda\),
hence we can have the existence of $\lambda$-coanalytic sets without the $\lPSP$ for all strong limit cardinals, simultaneously.

\begin{corollary} \label{cor:counterexampleinV=L}
 Suppose $0^\sharp$ does not exist. Then for all strong limit cardinals \(\lambda > \omega \) of countable cofinality there $A \in \lP^1_1(\pre{\lambda}{2})$ without the $\lPSP$. 
 
\end{corollary}

The second remark is a negative result. Since \( \lPSP(\lP^1_1) \) implies a large cardinal assumption, it is not provable in $\ZFC$. This means that if we look for pointclasses $\boldsymbol{\Gamma}$ in the projective hierarchy such that \( \lPSP(\boldsymbol{\Gamma}) \), the best result that we can obtain in $\ZFC$ alone is \( \boldsymbol{\Gamma} = \lS^1_1 \), and to have more we need large cardinals. 

We further notice that it is easy to relativize the previous argument to $L[a]$ for any $a\in\pre{\lambda}{2}$, so that if \( \lPSP(\lP^1_1 )\), then $a^\sharp$ exists for all $a\in\pre{\lambda}{2}$. 

We end this section by noticing that, as in the classical case, in \( L \) there is a \( \lS^1_2 \)-well-ordering of \( \pre{\lambda}{2} \). This can be proved using the same methods developed for the proof of Theorem~\ref{thm:noPSPcoanalytic}.


\begin{proposition}[\( \mathsf{V=L} \)]\label{prop:definablewellorderinL}
For every limit ordinal \(\lambda\) with countable cofinality and every \(\lambda\)-Polish space \( X \), there is a \( \lS^1_2 \) well-ordering of \( X \).
\end{proposition}

\begin{proof}
By Theorem~\ref{thm:Borelisomorphism}, without loss of generality we may assume \( X = \pre{\lambda}{2} \). Then for every \( x,y \in \pre{\lambda}{2} \) we have that \( x <_L y \) if and only if
\begin{multline*}
\exists z \in \pre{\lambda \times \lambda}{2} \,
[z \in \mathrm{BC}_\lambda \wedge (\lambda,{\in_z}) \models \upsigma_0 \wedge \exists \alpha,\beta < \lambda \, (\pi_z(\alpha) = x \\ \wedge \pi_z(\beta) = y \wedge(\lambda,{\in_z}) \models \upvarphi_0[\alpha/u,\beta/v].
\end{multline*}
Using Fact~\ref{fct:LopezEscobar} and Lemma~\ref{lem:technical} (after having identified \( \pre{\lambda \times \lambda}{2} \) and \( \pre{\lambda}{2} \)), one easily sees that the above condition is \( \lS^1_2 \), hence we are done.
\end{proof}

\section{Games for the \( \lPSP \)} \label{sec:PSPgames}

There is a game-theoretic formulation of \( \lPSP \) similar%
\footnote{Indeed, when \( \lambda = \omega \) our game \( \omega \)-\( G^*(A) \) is game-theoretically equivalent to the game \( G^*(A) \) from~\cite[Section 21.A]{Kechris1995}: each of the players has a winning strategy in one of the games if and only if (s)he has a winning strategy in the other one.}
to the \( * \)-game $G^*(A)$ game in~\cite[Section 21.A]{Kechris1995}. Recall that we fixed a strictly increasing sequence \( \vec{\lambda} = (\lambda_k)_{k \in \omega} \)  of cardinals cofinal in \(\lambda\), and that without loss of generality we can assume that each \( \lambda_k \) is regular whenever \( 2^{< \lambda} = \lambda \).

\begin{defin} \label{def:game}
Let \( X \) be a \(\lambda\)-Polish space and fix a base \( \mathcal{B} = \{ U_\alpha \mid \alpha < \lambda \} \) for it with \( U_\alpha \neq \emptyset \) for all \( \alpha < \lambda \). 
Given \( A \subseteq X \), consider  the following infinite two-players zero-sum perfect information game \(\lambda\)-\( G^*(A) \)
\[
\begin{array}{c||c|c|c|c|c}
\pI & (\alpha^0_i)_{i < \lambda_0} & (\alpha^1_i)_{i < \lambda_1} & \dotsc & (\alpha^k_i)_{i < \lambda_k} & \dotsc \\
\hline
\pII & i_0 & i_1 & \dotsc & i_k & \dotsc
\end{array}
\]
The game is played with \( \pI\) and \( \pII \) moving alternatively, with \( \pI \) moving first. At each round \( k \in \omega \):
\begin{enumerate-(a)}
\item
\( \pI \) chooses a sequence \( (\alpha^k_i)_{i < \lambda_k} \) of ordinals smaller than \(  \lambda \) such that
\begin{itemizenew}
\item
\( (\alpha^k_i)_{i < \lambda_k} \) is  bounded in \(  \lambda \), i.e.\ there is \( \gamma < \lambda \) such that \( \alpha^k_i \leq \gamma \) for all \( i < \lambda_k \);
\item
\( U_{\alpha^k_i} \cap U_{\alpha^k_j} = \emptyset \) if \( i,j < \lambda_k \) are distinct;
\item
\( \mathrm{diam}(U^k_i) \leq 2^{-k} \) for all \( i < \lambda_k \), where the diameter is computed with respect to any compatible complete metric on \( X \) fixed in advance;
\item
\( \mathrm{cl}(U_{\alpha^k_i}) \subseteq U_{\alpha^{k-1}_{i_{k-1}}} \) for all \( i < \lambda_k \), where if \( k = 0  \) we set \( U_{\alpha^{-1}_{i_{-1}}} = X \).
\end{itemizenew}
\item
After \( \pI \)'s move, \( \pII \) replies by choosing some \( i_k < \lambda_k \).
\end{enumerate-(a)}
The game ends after \( \omega \)-many rounds, or when player \( \pI \) is not able to make his/her next move without violating the rules above. In the latter case \( \pII \) immediately wins. In the former case, \( \bigcap_{k \in \omega} U_{\alpha^k_{i_k}} = \{ x \} \) for some \( x \in X \), and \( \pI \) wins if and only if \( x \in A \).
\end{defin}

The interested reader can easily verify that  \(\lambda\)-\( G^*( A ) \) can be formalized as a Gale-Stewart game with rules (in the sense discussed at the end of Section~\ref{sec:noBoreldeterminacy}).

\begin{remark}
If \(\lambda\) is \(\omega\)-inaccessible and the space \( X \) is \(\lambda\)-perfect, then
no play of the game \(\lambda\)-\( G^*(A) \) will terminate at a finite stage because
 \( \pI \) will always be able to make the next move. To see this, use the fact that in such a case each nonempty open subset of \( X \) has density character \(\lambda\), and this guarantees that \( \pI \) can make one more legal move thanks to Lemma~\ref{lem:r-spaces} (see the proof of Theorem~\ref{thm:generalgameforPSP}\ref{thm:generalgameforPSP-1} for more details). 
\end{remark}

\begin{theorem} \label{thm:generalgameforPSP}
Let \( \lambda \) be such that \( 2^{< \lambda} = \lambda \), \( X \) be a \(\lambda\)-Polish space, and \( A \subseteq X \).
\begin{enumerate-(i)}
\item \label{thm:generalgameforPSP-1}
\( \pI \) has a winning strategy in \(\lambda\)-\( G^*(A) \) if and only if there is a continuous injection from \( C(\lambda) \) into \( A \).
\item \label{thm:generalgameforPSP-2}
\( \pII \) has a winning strategy in \(\lambda\)-\( G^*(A) \) if and only if \( A \) is well-orderable and \( |A| \leq \lambda \).
\end{enumerate-(i)}
Thus the game  \(\lambda\)-\( G^*(A) \) is determined if and only if \( \lPSP(A) \).
\end{theorem}

\begin{proof}
The additional part follows from
\( C(\lambda) \approx B(\lambda) \) and Corollary~\ref{cor:equivPSP}.

\ref{thm:generalgameforPSP-1}
Suppose that \( \pI \) has a winning strategy \(\sigma\) in \(\lambda\)-\( G^*(A) \). For \( x \in C(\lambda) \), consider the run of the game where \( \pI \) follows \(\sigma\) and \( \pII \) plays \( x(k) \) at each round \( k \in \omega \), and let \( f(x) \) be the unique element in the resulting set \( \bigcap_{k \in \omega} U_{\alpha^k_{i_k}} \). It is not hard to see that \( f \colon C(\lambda) \to A \) is continuous and injective.

Conversely, assume that there is a continuous injection from \( C(\lambda) \) into \( A \). By Corollary~\ref{cor:equivPSP}, there is a closed-in-\( X \) set \( C \subseteq A \) which is homeomorphic to \(  B(\lambda) \). We will describe a legal strategy \( \sigma \) for \( \pI \) such that \( \pI \) can always make the next move and with the additional property that \( U_{\alpha^k_i} \cap C \neq \emptyset \) for all \( k \in \omega \) and \( i < \lambda_k \). Such a strategy \( \sigma \) will then be winning for \( \pI \) because by construction the unique point in \( \bigcap_{k \in \omega} U_{\alpha^k_{i_k}} \) will be the limit of points \( x_k \in U_{\alpha^k_{i_k}} \cap C \), and thus will be in \( C \subseteq A \) because \(C \) is assumed to be closed. To define \( \sigma \), for each \( k \in \omega \) and \( \alpha < \lambda \) such that \( U_\alpha \cap C \neq \emptyset \) fix an \( r_{\alpha,k} \)-spaced set \( C_{\alpha,k} \subseteq U_\alpha \cap C \) with \( |C_{\alpha,k}| = \lambda_k \), for a suitable positive real \( r_{\alpha,k}\). (The existence of such a set is granted by Lemma~\ref{lem:r-spaces} applied with \( \nu = \lambda_k \) and the fact that \( U_\alpha \cap C \) is a metrizable space of density character \( \lambda \) because it is a relatively open subspace of a homeomorphic copy of \( B(\lambda) \).) Suppose that we are at turn \( k \in \omega \), so that \( \pII \) just selected the ordinal \( i_{k-1} < \lambda_{k-1} \), and that by inductive hypothesis  \( U_{\alpha^{k-1}_{i_{k-1}}} \cap C \neq \emptyset \). Let \( \{ x_i \mid i < \lambda_k \} \) be an enumeration without repetitions of \( C_{\alpha^{k-1}_{i_{k-1}},k} \) and for each \( i < \lambda_k \) let \( \beta_i < \lambda \) be such that \( x_i \in U_{\beta_i} \) (so that \( U_{\beta_i} \cap C \neq \emptyset \)), \( \mathrm{diam}(U_{\beta_i}) < \min \left\{ 2^{-k}, \frac{1}{3} r_{\alpha^{k-1}_{i_{k-1}},k} \right \} \), and \( \mathrm{cl}(U_{\beta_i}) \subseteq U_{\alpha^{k-1}_{i_{k-1}}} \). (This is possible because \( X \) is metrizable.) The sequence \( (\beta_i)_{i < \lambda_k} \) might be unbounded in \( \lambda \), however, since \( \lambda_k \) is regular the set \( B_n = \{ \beta_i \mid i < \lambda_k \wedge \beta_i < \lambda_n \} \) must be of size \( \lambda_k \) for at least one \( n \in \omega \). Letting \( (\alpha^k_i)_{i < \lambda_k} \) be an enumeration without repetitions of such a \( B_n \), we clearly get a legal move for \( \pI \) which moreover satisfies \( U_{\alpha^k_i} \cap C \neq \emptyset \) for all \( i < \lambda_k \), as required.

\ref{thm:generalgameforPSP-2}
Suppose that \( \pII \) has a winning strategy \( \tau \) in \(\lambda\)-\( G^*(A) \). Let \( P \) be the set of all legal positions in \( \lambda \)-\( G^*(A) \) in which \( \pII \) followed \( \tau \) and \( \pI \) has to move next. Formally, \( P \) is a collection of finite sequences of even length over the set
\[ 
\bigcup_{k \in \omega} \bigg(\lambda_k \cup \bigcup_{n \in \omega} \pre{\lambda_k}{(\lambda_n)}  \bigg).
 \] 
Notice that since we assumed \( 2^{< \lambda} = \lambda \), such a set is well-orderable and has size \( \lambda \), so that \( |P| \leq \lambda \) as well. Let \( p \in P \) be of length \( 2k \) for some \( k \in \omega \). A point \( x \in X \) is called \markdef{bad for \( p  \)} if%
\footnote{Recall that if \( k = 0 \) we set \( U_{\alpha^{-1}_{i_{-1}}} = X \).}
 \( x \in U_{\alpha^{k-1}_{i_{k-1}}} \) but for all \( (\alpha^k_i)_{i < \lambda_k} \), if it is a legal move for \( \pI \) in the next turn, then the reply \( i_k \) of \( \pII \) following \(\tau\) against such move of \( \pI \) is such that \( x \notin U_{\alpha^k_{i_k}} \). Let also 
 \[ 
 A_p = \{ x \in X \mid x \text{ is bad for }p \} .
 \]
(Notice that, in particular, \( A_p \subseteq U_{\alpha^{k-1}_{i_{k-1}}} \), and that  \( A_p = U_{\alpha^{k-1}_{i_{k-1}}}   \) if \( \pI \) cannot make any further legal move beyond \( p \).)

\begin{claim} \label{claim:generalgameforPSP}
For every \( p \in P \) of length \( 2k \) the set \( A_p \) is well-orderable and
\[ 
|A_p| \leq (\lambda_k)^\omega,
 \] 
and thus in particular \( |A_p| < \lambda \).
\end{claim}

\begin{proof}[Proof of the claim]
If not,  then \( \dens(A_p) > \lambda_k \) by Fact~\ref{fct:fromweighttosize} and the fact that weight and density character coincide on the metrizable space \( A_p \). By Lemma~\ref{lem:r-spaces} applied with \( \nu = \lambda_k \), there is some \( r \)-spaced set \( B = \{ x_i \mid i < \lambda_k \} \subseteq A_p \subseteq U_{\alpha^{k-1}_{i_{k-1}}} \) (for some \( r > 0 \)) such that \( x_i \neq x_j \) for all distinct \( i,j<\lambda_k \). Arguing as in the backward direction of part~\ref{thm:generalgameforPSP-1} above, from the points \( x_i \) we can then obtain a legal move \( (\alpha^k_i)_{i < \lambda_k} \) for \( \pI \) such that \( U_{\alpha^k_i} \cap B \neq \emptyset \) for all \( i < \lambda_k \). Let \( i_k \) be \( \pII \)'s reply to such move, and consider any point \( x \in U_{\alpha^k_{i_k}} \cap B  \): then \( x \) is not bad for \( p\), contradicting \( x \in B \subseteq A_p \).
\end{proof}

On the other hand, it easy to see that for every \( x \in A \) there is \( p \in P \) such that \( x \in A_p \). Otherwise, one could recursively construct%
\footnote{Notice that no choice beyond the one granted by \( 2^{< \lambda} = \lambda \) is needed here.}
a run \( p = \bigcup_{k \in \omega} p_k \) in \( \lambda \)-\( G^*(A) \) where each \( p_k \in P \) has length \( 2k \) and is such that \( x \) belongs to the last basic open set \( U_{\alpha^{k-1}_{i_{k-1}}} \) selected by \( \pII \) in the partial run \( p_k \). Indeed, since \( x \) is not bad for \( p_{k-1} \), player \( \pI \) has a legal move \( (\alpha^{k-1}_i)_{i < \lambda_{k-1}} \) such that if \( i_{k-1} \) is the reply by \( \pII \) following \( \tau \), then we still have \( x \in U_{\alpha^{k-1}_{i_{k-1}}} \). This contradicts the fact that \(\tau\) is winning for \( \pII \) because the run \( p = \bigcup_{k \in \omega} p_k \) does not stop at a finite stage, \(\bigcap_{k \in \omega} U_{\alpha^k_{i_k}} = \{ x \} \), but \( x \in A \). Thus \( A \subseteq \bigcup_{p \in P} A_p \), and hence \( |A| \leq \lambda \) by \( |P| \leq \lambda \) and Claim~\ref{claim:generalgameforPSP}.

Conversely, suppose that \( |A| \leq \lambda \), and enumerate it as \( (x_i)_{i < \lambda} \) (possibly with repetitions). We define a strategy for \( \pII \) as follows. In the first turn, (s)he plays any ordinal smaller than \( \lambda_0 \), e.g.\ \( i_0 = 0 \). Suppose now that at round \( k \geq 1 \) player \( \pI \) played \( (\alpha^k_i)_{i < \lambda_k} \). Since the open sets \( U_{\alpha^k_i} \) are pairwise disjoint and \( \lambda_k > \lambda_{k-1} \), there is \( i < \lambda_k \) such that \( x_j \notin U_{\alpha^k_i} \) for all \( j < \lambda_{k-1} \): let \( \pII \) play the least such \( i \) as his/her next move \( i_k \). If the corresponding run of the game \( \lambda \)-\( G^*(A) \) does not stop at a finite stage, by construction  \( \bigcap_{k \in \omega} U_{\alpha^k_{i_k}} \cap A = \emptyset \), hence \( \pII \) won such a run. It follows that the strategy we just described is winning for \( \pII \), as required.
\end{proof}

\begin{remark} \label{rmk:avoidPSPforanalytic}
In the proof of Theorem~\ref{thm:generalgameforPSP}\ref{thm:generalgameforPSP-1} we heavily used Corollary~\ref{cor:equivPSP}, which in turn relies on the fact that, under our assumptions on \(\lambda\), all \(\lambda\)-analytic sets have the \( \lPSP \) (Corollary~\ref{cor:analyticPSP}). This can be avoided by modifying the rules of \( \lambda \)-\( G^*(A) \) and asking that each \( \pI \)'s move \( (\alpha^k_i)_{i < \lambda_k} \) is such that there is \( r > 0 \) for which all pairs \( (U_{\alpha^k_i}, U_{\alpha^k_j}) \) have distance greater than \( r \) if \( i \neq j \) (this is obviously stronger than just requiring \( U_{\alpha^k_i} \cap U_{\alpha^k_j} = \emptyset \)). The very same proof would then show that, had we used this variant of \( \lambda \)-\( G^*(A) \), then%
\footnote{The sketched equivalence between each of the two games and the \( \lPSP \) for their payoff set proves that such games are indeed equivalent to each other in the game-theoretic sense.}
 \( \pI \) would have a winning strategy if and only if \( A \) contains a \emph{closed-in-\( X \) homeomorphic copy of \( C(\lambda) \)}; notably, however, the argument would now be direct and would not require passing through Corollary~\ref{cor:equivPSP}.
\end{remark}

The classic \( * \)-game \( G^*(A) \) is often used to show that if a certain class of sets is determined, then all sets in that class have the \( \mathrm{PSP} \). For example, all Borel sets have the \( \mathrm{PSP} \) because of Martin's Borel determinacy (Theorem~\ref{thm:Boreldeterminacy}), and the same is true for projective sets if we assume projective determinacy.
In contrast, in our generalized setting we lack determinacy results beyond closed and open sets (Section~\ref{sec:noBoreldeterminacy}), so Theorem~\ref{thm:generalgameforPSP} might seem useless. Nevertheless, we can reverse the approach and use that results to obtain that every time we prove a theorem of the form \( \lPSP(\boldsymbol{\Gamma}) \) for a certain boldface \(\lambda\)-pointclass \( \boldsymbol{\Gamma} \), then we automatically get the determinacy of an entire class of games associated to it, namely, all games of the form \( \lambda \)-\( G^*(A) \) for \( A \in \boldsymbol{\Gamma} \) (notice that the topological complexity of \( \lambda \)-\( G^*(A) \) is approximately the same of \( A \) itself); in view of Proposition~\ref{prop:noBoreldeterminacyheightomega}, such a consequence is highly nontrivial.

We next move to the unfolded version of the generalized \( * \)-game \( \lambda \)-\( G^*(A) \). 
\begin{defin} \label{def:unfoldedgame}
Let \( X \) and \( \mathcal{B} \) be as in Definition~\ref{def:game}. Given \( C \subseteq X \times B(\lambda) \), consider the variant \( \lambda \)-\( G^*_u(C) \)
\[
\begin{array}{c||c|c|c|c|c}
\pI & (\alpha^0_i)_{i < \lambda_0}, y(0) & (\alpha^1_i)_{i < \lambda_1}, y(1) & \dotsc & (\alpha^k_i)_{i < \lambda_k}, y(k) & \dotsc \\
\hline
\pII & i_0 & i_1 & \dotsc & i_k & \dotsc
\end{array}
\]
 of the generalized \( * \)-game from Definition~\ref{def:game}
in which at each turn \( \pI \) additionally plays \( y(k) < \lambda \), and \( \pI \) wins if and only if (s)he can move (legally) for \(\omega\)-many rounds and \( (x,y) \in C \), where \( x \) is the unique point in \( \bigcap_{k \in \omega} U_{\alpha^k_{i_k}} \) and \( y = (y(k))_{k \in \omega} \).
\end{defin}

\begin{theorem} \label{thm:generalgameforPSP2}
Let \( \lambda \) be such that \( 2^{< \lambda} = \lambda \), \( X \) be a \(\lambda\)-Polish space, and \( C \subseteq X \times B(\lambda) \).
\begin{enumerate-(i)}
\item \label{thm:generalgameforPSP2-1}
If \( \pI \) has a winning strategy in \(\lambda\)-\( G_u^*(C) \), then there is a continuous injection of \( C(\lambda) \) into \( \p(C) \).
\item \label{thm:generalgameforPSP2-2}
If \( \pII \) has a winning strategy in \(\lambda\)-\( G_u^*(A) \), then \( \p(C) \) is well-orderable and \( |\p(C)| \leq \lambda \).
\end{enumerate-(i)}
Thus if the game  \(\lambda\)-\( G_u^*(C) \) is determined, then \( \lPSP(\p(C)) \).
\end{theorem} 

\begin{proof}
\ref{thm:generalgameforPSP2-1}
It is easy to see that if \( \pI \) has a winning strategy in \( \lambda \)-\( G^*_u(C) \), then (s)he has a winning strategy in \( \lambda \)-\( G^*(\p(C)) \) as well, namely, the strategy obtained by forgetting the ordinal \( y(k) \) in each of them. Thus the conclusion follows from  Theorem~\ref{thm:generalgameforPSP}\ref{thm:generalgameforPSP-1}.

\ref{thm:generalgameforPSP2-2}
The proof is very similar to the proof of the akin statement in Theorem~\ref{thm:generalgameforPSP}\ref{thm:generalgameforPSP-2}, we just need to check that the presence of the ordinals \( y(k) \) does not significantly change the cardinality of the objects involved. Let again \( P \) be the set of all legal positions in \( \lambda \)-\( G^*_u(C) \) in which \( \pII \) followed his/her winning strategy \(\tau\) and \( \pI \) has to move next. Since there are only \( \lambda \)-many possible values for each \( y(k) \), arguing as in the proof of Theorem~\ref{thm:generalgameforPSP}\ref{thm:generalgameforPSP-2} it is easy to check that \( |P| \leq \lambda \) because we assumed \( 2^{< \lambda} = \lambda \). Next, given \( \gamma < \lambda \), \( p \in P \) of length \( 2k \), and \( x \in U_{\alpha^{k-1}_{i_{k-1}}} \) (where \( U_{\alpha^{k-1}_{i_{k-1}}} \) is the last basic open set selected by \( \pII \) along \( p \)), we say that \( x \) is \markdef{bad for \( (p,\gamma) \)} if \( x \in U_{\alpha^{k-1}_{i_{k-1}}} \) and the following holds: For every sequence \( (\alpha^k_i)_{i < \lambda_k} \), if the pair \( ((\alpha^k_i)_{i < \lambda_k}, \gamma) \) is a legal move for \( \pI \) and \( i_k < \lambda_k \) is \( \pII \)'s reply following \(\tau\), then \( x \notin U_{\alpha^k_{i_k}} \). Let
\[ 
A_{p,\gamma} = \{ x \in X \mid x \text{ is bad for } (p,\gamma) \} .
\] 
Notice that \( A_{p,\gamma} \subseteq U_{\alpha^{k-1}_{i_{k-1}}} \) by definition, and that if \( \pI \) cannot make any further legal move after \( p \) then \( A_{p,\gamma} = U_{\alpha^{k-1}_{i_{k-1}}} \). Arguing as in Claim~\ref{claim:generalgameforPSP}, one can then check that \( |A_{p,\gamma}| \leq (\lambda_k)^\omega < \lambda \). (Otherwise \( \mathrm{dens}(A_{p,\gamma}) > \lambda_k \), thus there would be an \( r \)-spaced set \( B \subseteq A_{p,\gamma} \) from which one could then extract a legal move \( ((\alpha^k_i)_{i < \lambda_k}, \gamma) \) for \( \pI \) such that \( U_{\alpha^k_i} \cap B \neq \emptyset \) for all \( i < \lambda_k \). Such a move would witness that the points in \( U_{\alpha^k_{i_k}} \cap B  \), where \( i_k \) is \( \pII \)'s reply according to \(\tau\), are not in \( A_{p,\gamma} \), a contradiction.) Finally, we claim that \( p(C) \subseteq \bigcup \{ A_{p,\gamma} \mid p \in P \wedge \gamma < \lambda \} \): this concludes the proof because then \( |\p(C)| \leq \lambda \) by the fact that \( |P| \leq \lambda \) and \( |A_{p,\gamma}| < \lambda \). Assume towards a contradiction that \( x \in \p(C) \setminus \bigcup \{ A_{p,\gamma} \mid p \in P \wedge \gamma < \lambda \} \). Let \( y \in B(\lambda) \) be such that \( (x,y) \in C \). Then using the inductive hypothesis \( x \notin A_{p_k, y(k)} \), one can recursively construct an increasing sequence of legal positions \( (p_k)_{k \in \omega} \) such that \( p_k \in P \) has length \( 2k \) and if \( k > 0 \)
\begin{itemizenew}
\item
the last move of \( \pI \) appearing in \( p_k \) is of the form \( ((\alpha^{k-1}_i)_{i < \lambda_{k-1}}, y(k-1)) \) (that is, \( y \) determines the second component of the pair), and
\item
\( x \in U_{\alpha^{k-1}_{i_{k-1}}} \), where \( i_{k-1} < \lambda_{k-1} \) is the last move of \( \pII \) in \( p_k \) (which is played according to \(\tau\) because \( p \in P \)).
\end{itemizenew}
Since by construction \( x \in \bigcap_{k \in \omega} U_{\alpha^k_{i_k}} \) and \( (x,y) \in C \), it then follows that \( \pI \) wins the run \( \bigcup_{k \in \omega} p_k \), contradicting the fact that \(\tau\) was a winning strategy for \( \pII \).
\end{proof}

\begin{remark}
Arguing as in Remark~\ref{rmk:avoidPSPforanalytic}, it is easy to define a variant of \( \lambda \)-\( G^*_u(C) \) in which the existence of a winning strategy for \( \pI \) directly implies that \( \p(C) \) contains a closed-in-\( X \) homeomorphic copy of \( C(\lambda) \) (without passing through Corollary~\ref{cor:equivPSP}). In this way we obtain an alternative purely game-theoretic proof of Corollary~\ref{cor:analyticPSP}. Indeed, if \( A \subseteq X \) is \(\lambda\)-analytic, then there is a closed set \( C \subseteq X \times B(\lambda) \) such that \( A = \p(C) \) by Proposition~\ref{prop:charanalytic}. This implies that the game \( \lambda \)-\( G^*_u(C) \), once coded as a Gale-Stewart game (with rules) on \( \lambda \), has a closed payoff set, and thus it is determined by Proposition~\ref{prop:closeddeterminacy}. By the analogue of Theorem~\ref{thm:generalgameforPSP2} for the modified game, we thus get that either \( A = \p(C) \) contains a closed-in-\( X \) homeomorphic copy of \( C(\lambda) \approx \pre{\lambda}{2} \) (if \( \pI \) has a winning strategy), or else \( |A| \leq \lambda \) (if \( \pII \) has a winning strategy).
\end{remark}

When \( \mathcal{B} \) is a tree-basis in the sense of Definition~\ref{def:tree-basis},
it is convenient to introduce further variants of 
\( \lambda \)-\( G^*(A) \) and \(\lambda\)-\( G^*_u(A) \) that better fit  the combinatorial structure of the index tree \( T \subseteq \pre{<\omega}{\lambda} \) of \( \mathcal{B} \). 

\begin{defin} \label{def:tildegame}
Let \( X \) be \(\lambda\)-Polish, and let \( \mathcal{B} = \{ U_s \mid s \in T \} \) be a tree-basis for \( X \) with index tree \( T \subseteq \pre{< \omega}{\lambda} \). Given \( A \subseteq X \) consider the following game \( \lambda \)-\( \widetilde{G}^*(A) \)
\[
\begin{array}{c||c|c|c|c|c}
\pI & (s^0_i)_{i < \lambda_0} & (s^1_i)_{i < \lambda_1} & \dotsc & (s^k_i)_{i < \lambda_k} & \dotsc \\
\hline
\pII & i_0 & i_1 & \dotsc & i_k & \dotsc
\end{array}
\]
where:
\begin{enumerate-(a)}
\item
\( \pI \) chooses a sequence \( (s^k_i)_{i < \lambda_k} \) of length \( \lambda_k \) of elements of \( T \) so that
\begin{itemizenew}
\item
there are \( 0 \neq j_k \in \omega \) and a cardinal \( \mu_k < \lambda \) such that \( s^k_i \in \pre{j_k}{\mu_k} \) for all \( i < \lambda_k \);
\item
\( U_{s^k_i} \cap U_{s^k_j} = \emptyset \) if \( i,j < \lambda_k \) are distinct;
\item
\( s^{k}_i \) is a proper extension of \( s^{k-1}_{i_{k-1}} \) for all \( i < \lambda_k \) (where \( s^{-1}_{i_{-1}} = \emptyset \)).
\end{itemizenew}
\item
After \( \pI \)'s move, \( \pII \) replies choosing some \( i_k < \lambda_k \).
\end{enumerate-(a)}
(Notice that the rules above impose that \( j_k \geq k+1 \), so that \( \mathrm{diam}(U_{s^k_i}) \leq 2^{-(k+1)} \) for all \( k \in \omega \) and \( i < \lambda_k \).)
The game ends if at some finite stage \( \pI \) cannot make any further legal move, in which case \( \pII \) wins, or after \(\omega\)-many rounds, in which case
\( \pI \) wins if and only if%
\footnote{Recall that \( f_{\mathcal{B}} \colon [T] \to X\) is the canonical continuous open surjection sending \( y \in [T] \) to the unique point in \( \bigcap_{k \in \omega} U_{y \restriction k} \).}
 $f_{\mathcal{B}}(\bigcup_{k\in\omega}s^k_{i_k})\in A$. 
\end{defin}

We are going to show that
 if \( 2^{< \lambda} = \lambda \), then \( \lambda \)-\( \widetilde{G}^*(A) \) is game-theoretically equivalent to \( \lambda \)-\( G^*(A) \). This will follow from the existence of a bijection 
\begin{equation} \label{eq:specialbijection} 
\phi \colon \pre{<\omega}{\lambda} \to \lambda 
\end{equation} 
such that
\begin{itemizenew}
\item
for all \( \alpha < \lambda \) there are \( j \in \omega \) and \( \mu < \lambda \) such that \( \phi^{-1}(\beta) \in \pre{j}{\mu} \) for all \( \beta < \alpha \) and, conversely,
\item
for every \( j \in \omega \) and \( \mu < \lambda \) the set \( \{ \phi(s) \mid s \in \pre{j}{\mu} \} \) is bounded in \(\lambda\).
\end{itemizenew}
(For example, \( \phi \) can be defined as the union over \( k \in \omega \) of bijections between \( \pre{\leq k+1}{\lambda_k} \setminus \pre{\leq k}{\lambda_{k-1}} \) and \( \lambda_k \setminus \lambda_{k-1} \), where  \( \lambda_{-1} = \emptyset \).)

\begin{theorem} \label{thm:equivalentgamesforPSP}
Assume that \( 2^{< \lambda} = \lambda \), and let \( X \), \( \mathcal{B} \), and \( T \subseteq \pre{<\omega}{\lambda} \) be as in Definition~\ref{def:tildegame}. Let \( A \subseteq X \). 
\begin{enumerate-(i)}
\item \label{thm:equivalentgamesforPSP-1}
\( \pI \) has a winning strategy in \(\lambda\)-\( \widetilde{G}^*(A) \) if and only if there is a continuous injection from $C(\lambda)$ into $A$. 
\item \label{thm:equivalentgamesforPSP-2}
\( \pII \) has a winning strategy in \(\lambda\)-\( \widetilde{G}^*(A) \) if and only if $|A|\leq\lambda$. 
\end{enumerate-(i)}
In particular, the game \(\lambda\)-\( \widetilde{G}^*(A) \) is game-theoretically equivalent to \( \lambda \)-\( G^*(A) \).
\end{theorem}

\begin{proof}
For the additional part, it is enough to combine the present result with Theorem~\ref{thm:generalgameforPSP}.

\ref{thm:equivalentgamesforPSP-1}
Suppose that \( \pI \) has a winning strategy in \(\lambda\)-\( \widetilde{G}^*(A) \). By identifying each of his/her legal moves \( (s^k_i)_{i < \lambda} \) with the sequence of ordinals \( (\phi(s^k_i))_{i < \lambda} \), where \( \phi \) is as in~\eqref{eq:specialbijection}, it is easy to check that such a strategy is also a winning strategy in \( \lambda \)-\( G^*(A) \) once we enumerate \( \mathcal{B} \) according to \( \phi \), hence there is a continuous injection from \( C(\lambda) \) into \( A \) by Theorem~\ref{thm:generalgameforPSP}\ref{thm:generalgameforPSP-1}. 

For the backward direction we argue as in the corresponding part of the proof of Theorem~\ref{thm:generalgameforPSP}\ref{thm:generalgameforPSP-1}. Suppose that there is a closed-in-\( X \) set \( C \subseteq A \) such that \( C \approx B(\lambda) \), whose existence is granted by Corollary~\ref{cor:equivPSP}. We can again recursively construct a legal strategy \( \sigma \) for \( \pI \) such that \( U_{s^k_i} \cap C \neq \emptyset \) for all \( k \in \omega \) and \( i < \lambda_k \), so that \(\sigma\) is automatically winning because \( C \) is closed. Indeed, once we know that \( U_{s^{k-1}_{i_{k-1}}} \cap C \neq \emptyset \), then by considering a suitable \( r \)-spaced subset of it we can find a \( j_k > \lh(s^{k-1}_{i_{k-1}}) \) and a sequence \( (t^k_i)_{i < \lambda_k} \) of elements of \( T \) such that the open sets \( U_{t^k_i} \) are pairwise disjoint, they still intersect \( C \), and moreover \( \lh(t^k_i) = j_k \) and \( t^k_i \supseteq s^{k-1}_{i_{k-1}} \) for every \( i < \lambda_k \). For each \( n \in \omega \), consider the set \( S_n \) of those \( t^k_i \) which belong to \( \pre{j_k}{\lambda_n} \). Since \( \lambda_k\) is regular, there is \( n \in \omega \) such that \( |S_n| = \lambda_k \). Let \( \bar{n} \) be smallest with such property, set \( \mu_k = \lambda_{\bar{n}} \), and let \( (s^k_i)_{i < \lambda_k} \) be an enumeration of \( S_{\bar{n}} \). Clearly, the move \( (s^k_i)_{i < \lambda_k} \) is then as required.

\ref{thm:equivalentgamesforPSP-2}
Assume that \( \pII \) has a winning strategy \( \tau \) in \(\lambda\)-\( \widetilde{G}^*(A) \). We basically follow the corresponding argument from Theorem~\ref{thm:generalgameforPSP}\ref{thm:generalgameforPSP-2}, although technically we need to be a bit more careful. Let \( P \) be the set of all legal positions \( p \) in \(\lambda\)-\( \widetilde{G}^*(A) \) in which \( \pII \) followed \( \tau \) and \( \pI \) has to move next, so that \( p \) has even length \( 2k \).  If \( k > 0 \), the last moves of \( \pI \) and \( \pII \) in \( p \) are of the form, respectively, \( (s^{k-1}_{i})_{i < \lambda_{k-1}} \) and \( i_{k-1} \), so that \( \pII \) has chosen \( s^{k-1}_{i_{k-1}} \). We extend this to the case \( k = 0 \) by setting \( s^{-1}_{i_{-1}} = \emptyset \). 
 Call \( \tilde{x} \in [T] \) \markdef{bad for \( p \)} if it extends \( s^{k-1}_{i_{k-1}} \) but for every legal move \( (s^k_i)_{i < \lambda_k} \) for \( \pI \) in the next turn, the reply \( i_k \) of \( \pII \) following \(\tau\) against \( \pI \)'s move is such that \( \tilde{x} \) does not extend \( s^k_{i_k} \). We let
\[ 
\widetilde{A}_{p} = \{ \tilde{x} \in [T] \mid \tilde{x} \text{ is bad for } p \}
 \] 
and
\[ 
A_{p} = \{ x \in X \mid \exists \tilde{x} \in \widetilde{A}_{p} \, (f_{\mathcal{B}}(\tilde{x}) = x) \} = f_\mathcal{B}(\widetilde{A}_p).
 \] 
 It is not hard to see that each \( \widetilde{A}_p\) is closed. Indeed, if \( \tilde{x} \in [T] \setminus \widetilde{A}_p \), it means that there is a legal move \( (s_i^k)_{i < \lambda_k} \) for \( \pI \) in the next turn such that \( \tilde{x} \supseteq s^k_{i_k} \), where \( i_k \) is \( \pII \)'s reply according to \(\tau\). It follows that the same move \( (s_i^k)_{i < \lambda_k} \) witnesses that \( \tilde{y} \notin \widetilde{A}_p \) for every \( \tilde{y} \in [T] \cap \Nbhd_{s^k_{i_k}} \), so that \( \Nbhd_{s^k_{i_k}} \) is an open neighborhood of \( \tilde{x} \) disjoint from \( \widetilde{A}_p \).
 Moreover, \( A_p \subseteq U_{s^{k-1}_{i_{k-1}}} \), and \( A_p = U_{s^{k-1}_{i_{k-1}}} \) if \( \pI \) cannot make any legal move beyond \( p \).


\begin{claim} \label{claim:gameforPSPclaim}
For every \( p \in P \) of length \( 2k \)
\[ 
|A_{p}| \leq (\lambda_k)^\omega,
 \] 
and thus \( |A_{p}| < \lambda \).
\end{claim}

\begin{proof}[Proof of the claim]
Assume towards a contradiction  that \( |A_{p}| > (\lambda_k)^\omega \), so that \( \dens(A_{p}) > \lambda_k \). Pick an \( r \)-spaced set \( B \subseteq A_{p} \) of size \( \lambda_k \) (for some real \( r > 0 \)). For each \( x \in B \subseteq A_{p} \), \( f^{-1}_{\mathcal{B}}(x) \) is closed because \( f_{\mathcal{B}} \) is continuous, hence so is \( f^{-1}_{\mathcal{B}}(x) \cap \widetilde{A}_p \): let \( \tilde{x} \in \widetilde{A}_{p} \) be the leftmost branch of the pruned \(\omega\)-tree \( T_x \) such that \( f^{-1}_{\mathcal{B}}(x) \cap \widetilde{A}_p = [T_x] \), so that \( f_{\mathcal{B}}(\tilde{x}) = x \) and  \( s^{k-1}_{i_{k-1}} \subseteq \tilde{x} \) (because \( \tilde{x} \in \widetilde{A}_p \)) for each such \( \tilde{x} \). Let \( j_k > j_{k-1} = \lh(s^{k-1}_{i_{k-1}}) \) be such that \( 2^{-{j_k}} < \frac{r}{2} \), so that \( |\{ \tilde{x} \restriction j_k \mid x \in B \} | = \lambda_k \) because \( B \) is \( r \)-spaced. 
Since \( \pre{j_k}{\lambda} = \bigcup_{n \in \omega} \pre{j_k}{\lambda_n} \) and \( \lambda_k \) is regular, there is (a minimal) \( \bar{n} \in \omega \) such that \( |S_{\bar{n}}| = \lambda_k \) where \( S_{\bar{n}} = \{  \tilde{x} \restriction j_k \mid  x \in B  \wedge {\tilde{x}_i \restriction j_k} \in \pre{j_k}{\lambda_{\bar{n}}} \} \). Let \( (s^k_i)_{i < \lambda_k} \) be an enumeration without repetitions of \( S_{\bar{n}} \), so that \(
(s^k_i)_{i < \lambda_k} \) becomes a legal move for \( \pI \) extending \( p \) in the game \(\lambda\)-\( \widetilde{G}^*(A) \) (as witnessed by \( j_k \) and \( \mu_k = \lambda_{\bar{n}} \)). Let also \( x_i \in B \subseteq A_p \) be such that \( \tilde{x}_i \restriction j_k = s^k_i \), for any \( i < \lambda_k \). If \( i_k \) is the reply of \( \pII \) to the above move of \( \pI \), then \( \tilde{x}_{i_k} \supseteq s^k_{i_k} \), contradicting \( \tilde{x}_{i_k} \in \widetilde{A}_{p} \).
\end{proof}

\begin{claim}
\( A \subseteq  \bigcup \{ A_{p} \mid p \in P   \} \).
\end{claim}

\begin{proof}[Proof of the claim]
Assume towards a contradiction that \( x \in A \setminus \bigcup \{ A_{p} \mid p \in P \} \), and let $\tilde{x}\in [T]$ be such that \( f_{\mathcal{B}}(\tilde{x}) = x \). We are going to construct an increasing sequence \( (p_k)_{k \in \omega} \) of positions in \( P \) of length \( 2k \) such that \(  s^{k}_{i_k} \subseteq \tilde{x} \) for all \( k \in \omega \): it follows that the run corresponding to the union of the positions \( p_k \) is won by \( \pI \), contradicting the fact that \( \pII \) is following the winning strategy \(\tau\).

The construction of the positions \( p_k \) is by recursion on \( k \), starting from the empty position \( p_0 = \emptyset \). For the inductive step, assume that \( p_k \) is as required, that is, that \( s^\ell_{i_\ell} \subseteq s^{k-1}_{i_{k-1}} \subseteq \tilde{x} \) for all \( \ell \leq k-1 \). 
Since \( x \notin A_{p_k} \), it follows that 
\( \tilde{x}' \notin \widetilde{A}_{p_k} \) for all \( \tilde{x}' \in [T] \) with \( f_{\mathcal{B}} (\tilde{x}') = x \): in particular, this holds for our \( \tilde{x} \). Thus by definition of \( \widetilde{A}_{p_k} \) there is \( (s^k_i)_{i < \lambda_k} \) which is a legal move for \( \pI \) such that \( \tilde{x} \) still extends \( s^k_{i_k} \), where \( i_k \) is the reply of \( \pII \) according to \( \tau \). Letting  \( p_{k+1} \) be the prolongation of \( p_k \) by such moves of \( \pI \) and \( \pII \), we get that \( p_{k+1} \in P \) is as required.
\end{proof}

Because of the size of \( P \) and each \( A_{p} \), we finally get \( |A| \leq \lambda \), as required.

Conversely, assume that \( |A| \leq \lambda \). By Theorem~\ref{thm:generalgameforPSP}\ref{thm:generalgameforPSP-2} there is a winning strategy \( \tau \) for \( \pII \) in \(\lambda\)-\( G^*(A) \). Using again the fact that every legal move for \( \pI \) in \(\lambda\)-\( \widetilde{G}^*(A) \) can be turned via \( \phi \) into a legal move for \( \pI \) in \(\lambda\)-\( G^*(A) \), we then get that \( \tau \) can be straightforwardly transformed into a winning strategy for \( \pII \) in \(\lambda\)-\( \widetilde{G}^*(A) \).
\end{proof}

The game \(\lambda\)-\( \widetilde{G}^*(A) \) admits an obvious unfolded version, where the payoff set \( C \) is a subset of \( X \times B(\lambda) \), and \( \pI \) additionally plays an ordinal \( y(k) < \lambda \) at each turn \( k \in \omega \). Assuming \( 2^{< \lambda} = \lambda \), one can mix the ideas from Theorems~\ref{thm:generalgameforPSP2} and~\ref{thm:equivalentgamesforPSP} and prove that if \( \pI \) has a winning strategy in such a game, then there is a continuous injection from \( C(\lambda) \) into \( \p(C) \), while if \( \pII \) has a winning strategy, then \( |\p(C) | \leq \lambda \); in particular, if the game is determined then \( \lPSP(\p(C) ) \).
However, for technical reasons (see Section~\ref{sec:representalePSP}, and in particular the proof of Theorem~\ref{thm:mainPSP}) we need to further generalize
 such a game by considering a version of it in which the payoff set \( \widetilde{C} \) is rather a subset of \( B(\lambda) \times B(\lambda) \), and the determinacy of the new game implies that \( f_\mathcal{B}(\p(\widetilde{C})) \) has the \( \lPSP \). (It is not hard to see that the simpler version of the unfolded game considered at the beginning of this paragraph corresponds then to the special case where \( \widetilde{C} = (f_\mathcal{B} \times \id)^{-1}(C) \), so that \( f_\mathcal{B}(\p(\widetilde{C})) =\p(C) \).)

\begin{defin} \label{def:unfoldedtildegame}
Let \( X \) be a \(\lambda\)-Polish space, and let \( \mathcal{B} = \{ U_s \mid s \in T \} \) be a tree-basis for \( X \) with index tree \( T \subseteq \pre{< \omega}{\kappa} \). Given \( \widetilde{C} \subseteq B(\lambda) \times B(\lambda) \), the game \( \lambda \)-\( \widetilde{G}^*_u(\widetilde{C}) \) 
\[
\begin{array}{c||c|c|c|c|c}
\pI & (s^0_i)_{i < \lambda_0}, y(0) & (s^1_i)_{i < \lambda_1}, y(1) & \dotsc & (s^k_i)_{i < \lambda_k}, y(k) & \dotsc \\
\hline
\pII & i_0 & i_1 & \dotsc & i_k & \dotsc
\end{array}
\]
is played as in Definition~\ref{def:tildegame} except that
at each turn \( \pI \) additionally plays \( y(k) < \lambda \). Player \( \pI \) wins if and only if (s)he can move (legally) for \(\omega\)-many rounds and \( (\tilde{x},y) \in \widetilde{C} \), where \( \tilde{x} = \bigcup_{k \in \omega} s^k_{i_k}  \) and \( y = (y(k))_{k \in \omega }  \).
\end{defin}

\begin{theorem} \label{thm:equivalentgamesforPSP2}
Assume that \( 2^{< \lambda} = \lambda \), and let \( X \), \( \mathcal{B} \), and \( T \subseteq \pre{<\omega}{\lambda} \) be as in Definition~\ref{def:unfoldedtildegame}. Let \( \widetilde{C} \subseteq B(\lambda) \times B(\lambda) \).
\begin{enumerate-(i)}
\item \label{thm:equivalentgamesforPSP2-1}
If \( \pI \) has a winning strategy in \(\lambda\)-\( \widetilde{G}_u^*(\widetilde{C}) \), then there is a continuous injection from \( C(\lambda) \) into \( f_{\mathcal{B}} (\p(\widetilde{C})) \).
\item \label{thm:equivalentgamesforPSP2-2}
If \( \pII \) has a winning strategy in \(\lambda\)-\( \widetilde{G}_u^*(\widetilde{C}) \), then \( f_{\mathcal{B}}(\p(\widetilde{C})) \) is well-orderable and \( |f_{\mathcal{B}}(\p(\widetilde{C}))| \leq \lambda \).
\end{enumerate-(i)}
Thus if the game  \(\lambda\)-\( \widetilde{G}_u^*(\widetilde{C}) \) is determined, then \( f_{\mathcal{B}}(\p(\widetilde{C})) \) has the \lPSP.
\end{theorem} 

\begin{proof}
\ref{thm:equivalentgamesforPSP2-1}
If \( \pI \) has a winning strategy in \( \lambda \)-\( \widetilde{G}^*_u(\widetilde{C}) \), then (s)he can also win \( \lambda \)-\( \widetilde{G}^*(f_\mathcal{B}(\p(\widetilde{C}))) \) by forgetting the additional moves \( y(k) \) played by \( \pI \) in \(\lambda\)-\( \widetilde{G}_u^*(\widetilde{C}) \), hence there is a continuous injection from \( C(\lambda) \) into \( f_\mathcal{B}(\p(\widetilde{C})) \) by Theorem~\ref{thm:equivalentgamesforPSP}.

\ref{thm:equivalentgamesforPSP2-2}
Let \( \tau \) be a winning strategy for \( \pII \) in \( \lambda \)-\( \widetilde{G}^*_u(\widetilde{C}) \). 
Let \( P \) the set of legal positions in \(\lambda\)-\( \widetilde{G}^*_u(\widetilde{C}) \) of even length where \( \pII \) followed \(\tau\). If \( p \) has length \( 2k \), the last moves of \( \pI \) and \( \pII \) in \( p \) are of the form \( ((s^{k-1}_i)_{i < \lambda_{k-1}}, y(k-1)) \) and \( i_{k-1} \), respectively. (If \( k = 0 \) we set \( s^{-1}_{i_{-1}} = \emptyset \) and \( i_{-1} = 0 \).) Given  \( \gamma < \lambda \), we call \( \tilde{x} \in [T] \) \markdef{bad for \( (p,\gamma) \)} if \( \tilde{x} \supseteq s^{k-1}_{i_{k-1}} \) and for every \( (s^k_i)_{i < \lambda_k} \), if the pair \( ((s^k_i)_{i < \lambda_k}, \gamma) \) is a legal move for \( \pI \) after \( p \), then \( \tilde{x} \not\supseteq s^k_{i_k} \) for \( i_k \) the reply by \( \pII \) according to \(\tau\). We then set
\[ 
\widetilde{A}_{p,\gamma} = \{  \tilde{x}\in [T] \mid \tilde{x} \text{ is bad for } (p, \gamma) \}
 \] 
and 
\[ 
A_{p,\gamma} = \{ x \in X \mid \exists \tilde{x} \in \widetilde{A}_{p,\gamma} \, (f_\mathcal{B}(\tilde{x}) = x) \} = f_\mathcal{B}(\widetilde{A}_{p,\gamma}).
 \] 
Notice in particular that \( A_{p,\gamma}  \subseteq U_{s^{k-1}_{i_{k-1}}} \), and that \( A_{p,\gamma}  = U_{s^{k-1}_{i_{k-1}}} \) if \( \pI \) cannot make any legal move after \( p \).

Arguing as in Claim~\ref{claim:gameforPSPclaim}, we get \( |A_{p,\gamma}| \leq (\lambda_k)^\omega < \lambda \). Moreover, \( \p(\widetilde{C}) \subseteq \bigcup \{ \widetilde{A}_{p,\gamma} \mid p \in P \wedge \gamma < \lambda \} \): if \( \tilde{x} \in \p(\widetilde{C}) \setminus \{ \widetilde{A}_{p,\gamma} \mid p \in P \wedge \gamma < \lambda \}  \),  picking any \(y \in B(\lambda) \) such that  \( (\tilde{x},y) \in \widetilde{C} \) we could then construct an increasing sequence of positions \( p_k \in P \) of length \( 2k \) so that each \( p_k \) is ``compatible'' with \( \widetilde{x} \) and \( y \), in the sense that those are the elements determined by the run \( \bigcup_{k \in \omega} p_k \); this contradicts the fact that \(\tau\) was winning for \( \pII \). 
It follows that 
\[ 
f_\mathcal{B}(\p(\widetilde{C})) \subseteq f_\mathcal{B}(\bigcup \{ \widetilde{A}_{p,\gamma} \mid p \in P \wedge \gamma < \lambda \}) = \bigcup \{ A_{p,\gamma} \mid p \in P \wedge \gamma < \lambda \},
\]
and since \( |P| \leq \lambda \) we conclude that \( |f_\mathcal{B}(\p(\widetilde{C}))| \leq \lambda \), as desired.
\end{proof}

Theorems~\ref{thm:equivalentgamesforPSP} and~\ref{thm:equivalentgamesforPSP2} become particularly significant when \( X \) is a closed subset of \( B(\lambda) \) (including the extreme case \( X = B(\lambda) \)) and the chosen tree-basis is the canonical one, that is, \( \{ \Nbhd_s(X) \mid s \in T \} \) for \( T = \{ s \in \pre{<\omega}{\lambda} \mid \Nbhd_s \cap X \neq \emptyset \} \). In such cases, \( f_\mathcal{B} \) is the identity map and the game \( \lambda\)-\( \widetilde{G}^*_u( \cdot) \) from Definition~\ref{def:unfoldedtildegame} is nothing else than the natural unfolding of the game \( \lambda \)-\( \widetilde{G}^*(\cdot) \) from Definition~\ref{def:tildegame}.

\section{Representable sets and \( \lPSP \)} \label{sec:representalePSP}

If we could have that all games of the form \( \lambda \)-\( G^*(A) \) are determined, we would conclude that all subsets of any $\lambda$-Polish space $X$ have the \lPSP. Unfortunately, as discussed in Section~\ref{sec:noBoreldeterminacy}, full determinacy is out of reach in the generalized context. But the existence of such a game-theoretic equivalent of \lPSP{} suggests a new development. Martin's proof that all $\boldsymbol{\Pi}^1_1$ sets of reals are determined if there exists a measurable cardinal (see \cite[Theorem 31.1]{Kanamori}) can be divided in two parts: in the first part, it is proven that if there exists a measurable cardinal, then all $\boldsymbol{\Pi}^1_1$ sets have a certain tree-structure; such sets were called afterwards ``homogeneously Souslin sets''. In the second part, it is proven that if the payoff of a game is a homogeneously Souslin set, then it is possible to define an unfolded version of the game, i.e.\ a closed  (and therefore determined, under \( \AC \)) game where:
\begin{itemizenew}
\item
one of the players has to play additional moves, which are designed to witness that if a run in the unfolded game goes on for $\omega$-many moves, then the play in the original game obtained by removing those additional moves is indeed winning for such player; 
\item
every winning strategy in the unfolded game for any of the two player yields a winning strategy in the original game for the same player.
\end{itemizenew}
The projection of a homogeneously Souslin set is called ``weakly homogeneously Souslin'' (see Definition \ref{defin:whS}). Such sets have been crucial in the proof of the consistency of \( \AD \) from the existence of large cardinals, and they enjoy the same regularity properties of determined set, even if they are not necessarily determined. The standard proof is the following: regularity properties like the \( \mathrm{PSP} \) have a game-theoretic equivalent, and the unfolded version of such game for weakly homogeneously Souslin sets is closed, therefore determined (under \( \AC \)). Then it is shown that 
the determinacy of the unfolded game implies the determinacy of a ``partially unfolded'' version of the game, and that the latter suffices to prove that the set enjoys the given regularity property. 
(See \cite[Chapter 32]{Kanamori} for more information.)

We are going to define a generalization of weakly homogeneously Souslin sets in $\lambda$-Polish spaces, called representable sets, and prove that representability implies the \lPSP. 
Our strategy for proving this will follow closely the one described in the previous paragraph in the context of weakly homogeneously Souslin sets. Indeed,
the game-theoretic equivalent \( \lambda \)-\( \widetilde{G}^*( \cdot) \) of the \( \lPSP \) and its ``partially unfolded'' version \( \lambda \)-\( \widetilde{G}^*_u( \cdot ) \) have been already introduced in Definitions~\ref{def:tildegame} and~\ref{def:unfoldedtildegame}, respectively; the ``fully'' unfolded version of \( \lambda \)-\( \widetilde{G}^*( \cdot) \), to be used with representable sets and denoted by \( G_{\pi,F} \), will instead appear in the proof of Theorem~\ref{thm:mainPSP}.

Our definition of representability is inspired by the notion of ``$\mathbb{U}(j)$-representable set'' under \( \mathsf{I0} \) coming from~\cite[Definition 110]{Woodin2011}, and we will see in Chapter~\ref{sec:applications} that our definition is indeed a direct generalization of both concepts.


\begin{defin} 
A family of ultrafilters \( \UU \) is \markdef{orderly} if there is a nonempty set \( K \), called \markdef{support of \( \UU \)}, such that every \( \U \in \UU \) concentrates on \( \pre{k}{K} \) for some \( k \in \omega \), that is \( \pre{k}{K} \in \U \); the natural number \( k \) is then called \markdef{level of \( \U \)}, and is denoted by \( \lev(\U) \).
An orderly family of ultrafilters \( \UU \) is \markdef{nonprincipal} if there is at least one \( \U \in \UU \) which is nonprincipal.
\end{defin}

When \( \UU \) is an orderly families of ultrafilters \( \UU \), we always denote by \( K \) its support. Notice that we can clearly assume that every \( \U \in \UU \) is actually an ultrafilter on \( \pre{\lev(\U)}{K} \). 

\begin{defin} \label{def:towers}
Let \( \UU \) be an orderly family of ultrafilters.
A \markdef{tower of ultrafilters} (\markdef{in \( \UU \)}) is a countable sequence \( (\U_k)_{k \in \omega } \) of ultrafilters from \( \UU \) such that for all \(  k \in \omega \) and \( \ell < k \)
\begin{enumerate-(a)}
\item
\( \lev(\U_k) = k \);
\item
\( \U_k \) projects onto \( \U_\ell \), that is, for every \( A \subseteq \pre{\ell}{K} \) we have
\[ 
A \in \U_\ell \iff  A^{\ell \to k } =  \{ s \in \pre{k}{K} \mid s \restriction \ell \in A \} \in \U_k.
 \] 
\end{enumerate-(a)}
A tower of ultrafilters \( (\U_k)_{k \in \omega} \) is \markdef{well-founded} if for every sequence of sets \( (A_k)_{k \in \omega} \) with \( A_k \in \U_k \) there is \( z \in \pre{\omega}{K} \) such that \( z \restriction k \in A_k \) for all \( k \in \omega \).
\end{defin}

\begin{remark}
The term ``well-founded'' comes from the analysis by Gaifman and Powell on such towers. If $\U$ and $\mathcal{V}$ are ultrafilters and $\U$ projects onto $\mathcal{V}$, then there exists an elementary embedding $j_{\mathcal{V},\U}:\Ult(V,\mathcal{V})\to\Ult(V,\U)$, where \( \Ult(V,\mathcal{V}) \) and \( \Ult(V,\mathcal{U}) \) are the ultrapowers of the universe of sets \( V \) by \( \mathcal{V} \) and \( \U \), respectively. Therefore a tower of ultrafilters \( (\U_k)_{k \in \omega} \) induces a directed system \( (\Ult(V,(\U_k)),j_{\U_k,\U_m})_{m\geq k } \), and the membership relation on the direct limit of such system is well-founded if the tower of ultrafilters is well-founded in the sense of Definition~\ref{def:towers} (or, in their terminology, countably complete; see~\cite{Gaifman1974} and~\cite{Powell1974}).
\end{remark}

\begin{defin} \label{def:repr}
Let \(\lambda\) be such that \( 2^{< \lambda} = \lambda \), \( \UU \) be an orderly family of ultrafilters, and 
\( \mathcal{B} = \{ U_s \mid s \in \pre{<\omega}{\lambda} \} \) be a tree-basis%
\footnote{Recall that by Proposition~\ref{prop:tree-basis} we can always assume without loss of generality that the index tree of the chosen tree-basis \( \mathcal{B} \) is the whole \( \pre{<\omega}{\lambda} \).}
 for a nonempty \(\lambda\)-Polish space \( X \). 
A \markdef{\( \UU \)-representation} for a set \( Z \subseteq X \) (with respect to \( \mathcal{B} \)) is a function 
\[ 
\pi \colon \bigcup_{k \in \omega} (\pre{k}{\lambda} \times \pre{k}{\lambda}) \to \UU 
 \] 
such that 
\begin{enumerate-(a)}
\item \label{def:repr-1}
for all $k\in\omega$ and all $s,t\in \pre{k}{\lambda}$, $\lev(\pi(s,t))=k$;
\item \label{def:repr-2}
for all \( \tilde{x}, y \in B(\lambda) \), the family \( T^\pi_{\tilde{x},y} =  (\pi(\tilde{x} \restriction k, y \restriction k))_{k \in \omega} \) is a tower of ultrafilters such that \( \pi({\tilde{x} \restriction k}, {y \restriction k}) \) is (at least) \( \lambda_k^+ \)-complete;
\item \label{def:repr-3}
for every \( \tilde{x} \in B(\lambda) \), \( f_{\mathcal{B}}(\tilde{x}) \in Z \) if and only if there is \( y \in B(\lambda) \) such that \( T^\pi_{\tilde{x},y} \) is well-founded. 
\end{enumerate-(a)}
A set \( Z \subseteq X \) will be called \markdef{\( \UU \)-representable} if there is a \( \UU \)-representation for \( Z \) (with respect to some tree-basis for \( X \)).
\end{defin}

The reference to \( \UU \) may be omitted when it is clear or irrelevant. 
Clearly, we can always take \( \UU = \ran(\pi) \), so that \( \pi \) is surjective: this will always be tacitly assumed when speaking of \( \UU \)-representable sets, unless otherwise stated.

\begin{remark} 
Condition~\ref{def:repr-3} in Definition~\ref{def:repr} is subtle, as it entails that for every \( x \in X \), the fact that \( x \in Z \) implies that \emph{for every \( \tilde{x} \in f^{-1}_{\mathcal{B}}(x) \)} there is \( y \in B(\lambda) \) such that \( T^\pi_{\tilde{x},y} \) is well-founded, while in order to prove that \( x \in Z \) it is enough to show that \emph{there is some \( \tilde{x} \in f^{-1}_{\mathcal{B}}(x) \)} such that \( T^\pi_{\tilde{x},y} \) is well-founded, for a suitable \( y \in B(\lambda) \). In other words,
\begin{align*}
x \in Z & \iff \forall \tilde{x} \in f^{-1}_{\mathcal{B}}(x) \, \exists y \in B(\lambda) \, (T^\pi_{\tilde{x},y} \text{ is well-founded}) \\
& \iff \exists \tilde{x} \in f^{-1}_{\mathcal{B}}(x) \, \exists y \in B(\lambda) \, (T^\pi_{\tilde{x},y} \text{ is well-founded}).
\end{align*}
In particular, \( Z \subseteq X \) is \( \UU \)-representable with respect to \( \mathcal{B} \) if and only if \( f^{-1}_{\mathcal{B}}(Z) \) is \( \UU \)-representable with respect to the canonical basis \( \{ \Nbhd_s \mid s \in \pre{<\omega}{\lambda} \} \) on \( B(\lambda) \).

Interestingly enough, all results except Proposition~\ref{prop:representpointclass} would continue to work if condition~\ref{def:repr-3} in Definition~\ref{def:repr} is weakened to
\begin{enumerate}[label={\upshape (c\( ' \))}, leftmargin=2pc]
\item
for every \( x \in X \), \( x \in Z \) if and only if there are \( \tilde{x} , y \in B(\lambda) \) such that \( f_\mathcal{B}(\tilde{x}) = x \) and  \( T^\pi_{\tilde{x},y} \) is well-founded. 
\end{enumerate}
\end{remark}

\begin{proposition} \label{prop:U-representabilityandmeasurablecardinals} 
Let \( \lambda \) be such that \( 2^{< \lambda} = \lambda \), and let \( \pi \) be a \( \UU \)-representation for a set \( Z \subseteq X \), for some \(\lambda\)-Polish space \( X \) and some orderly family of ultrafilters \( \UU = \ran(\pi) \). Then every ultrafilter in \( \UU \) is at least \( \lambda^+ \)-complete. 
If moreover \( \UU \) is nonprincipal, then there is a measurable cardinal \(\kappa > \lambda \) such that every ultrafilter in \( \UU \) is at least \( \kappa \)-complete.

In particular, if there is a \( \UU \)-representable set with \( \UU \) nonprincipal, then there is a measurable cardinal above \(\lambda\).
\end{proposition}

\begin{proof}
This essentially follows from the fact that it is not possible to project an ultrafilter to another ultrafilter with less completeness. Indeed, fix any infinite cardinal \( \nu \) and two natural numbers \( k > \ell \). Suppose that \( \U_\ell \)  is an ultrafilter on  \( \pre{\ell}{K}\) that is not \( \nu^+ \)-complete, and that \( \U_k \) is a \( \nu^+ \)-complete ultrafilter on \( \pre{k}{K} \) that projects onto \( \U_\ell \). Since $\U_\ell$ is not $\nu^+$-complete, there is a sequence $(A_\alpha)_{\alpha<\nu}$ of sets in $\U_\ell$ such that $A = \bigcap_{\alpha<\nu} A_\alpha\notin \U_\ell$. 
Setting \( B_\alpha = A_\alpha \setminus A \), we get a sequence \( (B_\alpha)_{\alpha < \nu} \) of sets in \( \U_\ell \) such that
$\bigcap_{\alpha<\nu} B_\alpha=\emptyset$. But then \( (B_\alpha^{\ell \to k})_{\alpha < \nu} \) is a sequence of sets in \(\U_k\) such that $\bigcap_{\alpha<\nu}B_\alpha^{\ell \to k}=\emptyset$, against the \( \nu^+ \)-completeness of $\U_k$.

Consider now any \( \tilde{x}, y \in B(\lambda) \) and the corresponding tower of ultrafilters \( T^\pi_{\tilde{x},y} \) from Definition~\ref{def:repr}\ref{def:repr-2}. The previous paragraph shows that each \( \U \in T^\pi_{\tilde{x},y} \) is at least \( \lambda_k^+ \)-complete, \emph{for every \( k \in \omega \)}. It follows that every ultrafilter in \( \UU = \ran(\pi) = \bigcup_{\tilde{x},y \in B(\lambda)} T^\pi_{\tilde{x},y} \) is \(\lambda\)-complete, and hence also \( \lambda^+ \)-complete because \( \lambda \) is singular.

Assume now that \( \UU \) is nonprincipal. Since the completeness number of a nonprincipal ultrafilter is always a measurable cardinal (see Section~\ref{subsec:ultrafilters}), the smallest measurable cardinal \( \kappa \) which is the completeness number of some nonprincipal \( \U \in \UU \) is as required.
\end{proof}

\begin{remark} \label{rmk:U-representabilityandmeasurablecardinals} 
It follows that in item~\ref{def:repr-2} of Definition~\ref{def:repr} we can equivalently require each of the ultrafilters \( \pi(\tilde{x} \restriction k, y \restriction k) \) to be (at least) \( \kappa \)-complete, for some fixed \( \kappa > \lambda \).
This reformulation of Definition~\ref{def:repr} makes sense also when \( \lambda = \omega \), and indeed we will show that in such case it becomes equivalent to the classical notion of \( \kappa \)-weakly homogeneous Souslin set
(see Section~\ref{sec:representabilityvsweaklyhomogeneouslysouslin}).
\end{remark}

\begin{remark} \label{rmk:principalrepresentations} 
Proposition~\ref{prop:U-representabilityandmeasurablecardinals} raises the question of which subsets \( Z \subseteq X \) of a \(\lambda\)-Polish space \( X \) are \( \UU \)-representable if \( \UU \) only contains principal ultrafilters. It is easy to see that this is the case for \( Z = X \). Conversely, we are now going to show that
if \( \pi \) is a \( \UU \)-representation for a set \( Z \subseteq X \) and every ultrafilter in \( \ran(\pi) \) is principal, then necessarily \( Z = X \). Indeed, fix any \( x \in X \). Let \( \tilde{x} \in B(\lambda) \) be such that \( f_{\mathcal{B}}(\tilde{x}) = x \), and pick an arbitrary \( y \in B(\lambda) \). For any \( k \in \omega \) there is  \( z_k \in \pre{k}{K} \) such that \( \{ z_k \} \in \pi(\tilde{x} \restriction k, y \restriction k) \) because \( \pi(\tilde{x} \restriction k, y \restriction k) \) is principal. 
But then for every \( \ell < k \) we have that \( \{ z_\ell \}^{\ell \to k} \cap \{ z_k\} \neq \emptyset \) because \( \pi(\tilde{x} \restriction k, y \restriction k) \) projects onto \( \pi(\tilde{x} \restriction \ell, y \restriction \ell) \), therefore \( z_\ell \subseteq z_k \). It follows that \( z = \bigcup_{k \geq 1} z_k \) witnesses that the tower \( T^\pi_{\tilde{x},y} \) is well-founded, and hence \( x \in Z \).
\end{remark}


One desirable property of weakly homogeneously Souslin subsets of the classical Baire space \( \pre{\omega}{\omega} \) is that, in particular, they are Souslin, i.e.\ they are the projection of a tree. For our definition of representable set, we will see that this is automatically true under \( \AC \) (this follows from Proposition~\ref{prop:TCautomatic} and the comment after it); however, it is not necessarily true in choiceless settings. For this reason, we define a stronger version of representability that does the job even if \( \AC \) is not assumed.

\begin{defin} \label{def:TC}
Let \( \lambda \), \( \UU \), \( X \), and \( \mathcal{B} \) be as in Definition~\ref{def:repr}. A \( \UU \)-representation \( \pi \) for a set \( Z \subseteq X \) has the \markdef{tower condition} if there is \( F \colon \ran(\pi) \to \bigcup \UU \) such that
\begin{enumerate-(a)}
\item \label{def:TC-1}
\( F(\U) \in \U \) for all \( \U \in \ran(\pi) \);
\item \label{def:TC-2}
for all \( \tilde{x},y \in B(\lambda) \), the tower \( T^\pi_{\tilde{x},y} \) is well-founded if and only if  there is \( z \in \pre{\omega}{K} \) (where \( K \) is the support of \( \UU \)) such that \( z \restriction k \in F(\pi(\tilde{x} \restriction k, y \restriction k)) \) for all \( k \in \omega \).
\end{enumerate-(a)}
If a set \( Z \subseteq X \) admits a \( \UU \)-representation with the tower condition, we will simply say that \( Z \) is \markdef{TC-representable}.
\end{defin}

\begin{remark} \label{rmk:TC}
The terms ``representable'' and ``tower condition'' are taken from~\cite{Woodin2011}. Woodin introduces these terms as a generalization of weakly homogeneously Souslin sets in $L(V_{\lambda+1})$, under the strong large cardinal assumption $\mathsf{I0}(\lambda)$. Such definitions are a particular case of our definition: we are going to discuss this in detail in Section~\ref{sec:representabilityvsU(j)-representability}.
\end{remark}

\begin{proposition}[\( \AC \)] \label{prop:TCautomatic}
Let \( \lambda \) be such that \( 2^{< \lambda} = \lambda \), and \( X \) be a \(\lambda\)-Polish space. If \( Z \subseteq X \) is \( \UU \) representable for some orderly family of ultrafilters \( \UU \), then it is also TC-representable.
\end{proposition}

\begin{proof}
Let \( \pi \) be a \( \UU \)-representation for \( Z \) (with respect to some tree-basis \( \mathcal{B} \) for \( X \)). If every \( \U \in \ran(\pi) \) is principal, it is enough to let \( F(\U) \) be the set generating \( \U \). Otherwise, by Proposition~\ref{prop:U-representabilityandmeasurablecardinals} there is a measurable cardinal \( \kappa > \lambda \) such that every \( \U \in \ran(\pi) \) is at least \( \kappa \)-complete. Since we are assuming \( \AC \), the cardinal \( \kappa \) is also strong limit, and thus \( \kappa > 2^{\lambda} \).

Let \( K \) be the support of \( \UU \).
Given any \( \tilde{x},y \in B(\lambda) \), consider the tower of ultrafilters \( T^\pi_{\tilde{x},y} \). If it is not well-founded, then there is a sequence \( (A^{\tilde{x},y}_k)_{k \in \omega} \) with \( A^{\tilde{x},y}_k \in \pi(\tilde{x} \restriction k, y \restriction k) \) witnessing this, i.e.\ such that there is no \( z \in \pre{\omega}{K} \) such that \( \forall k \in \omega \, ( z \restriction k \in A^{\tilde{x},y}_k )\). If instead \( T^\pi_{\tilde{x},y} \) is well-founded, let \( A^{\tilde{x},y}_k = \pre{k}{K} \) for all \( k \in \omega \). For every \( k \in \omega \) and \( s,t \in \pre{k}{\lambda} \), let
\[
F(\pi(s,t)) = \bigcap \{ A^{\tilde{x},y}_k \mid (\tilde{x},y) \in \Nbhd_s \times \Nbhd_t \}.
\]
We claim that the function \( F \colon \ran(\pi) \to \bigcup \UU \) we just defined witnesses that \( \pi \) has the tower condition. Condition~\ref{def:TC-1} of Definition~\ref{def:TC} is satisfied by choice of \( A^{\tilde{x},y}_k \) and the fact that  \( \pi(s,t) \) is a \( \kappa \)-complete ultrafilter for \( \kappa > 2^\lambda \). As for condition~\ref{def:TC-2} of the same definition, if \( T^{\pi}_{\tilde{x},y} \) is well-founded then, by definition, there is \( z \in \pre{\omega}{K} \) such that \( z \restriction k \in F(\pi(\tilde{x} \restriction k,y \restriction k)) \) for all \( k \in \omega \) simply because \( F(\pi(\tilde{x} \restriction k,y \restriction k)) \in \pi(\tilde{x} \restriction k,y \restriction k) \). 
If instead \( T^\pi_{\tilde{x},y} \) is not well-founded, then there cannot be \( z \in \pre{\omega}{K} \) as above by choice of \( A^{\tilde{x},y}_k \) and the fact that \( F(\pi(\tilde{x} \restriction k,y \restriction k)) \subseteq A^{\tilde{x},y}_k \).
\end{proof}

Thus the tower condition and TC-representability are more significant when we work in choiceless settings. Even in such situations, though,
every set \( Z \subseteq B(\lambda) \) which is TC-representable with respect to the canonical basis \( \mathcal{B} = \{ \Nbhd_s \mid s \in \pre{<\omega}{\lambda} \} \) for \( B(\lambda) \), so that \( f_{\mathcal{B}} \) is the identity map, can be written as \( Z = \p[T] \) for some pruned tree \( T \) over \( \lambda \times Y \), where \( Y = \lambda \times K \). Indeed, let
\[ 
T = \{ (s,(t,u)) \in \pre{k}{(\lambda \times Y)}  \mid k \in \omega \wedge u \in F(\pi(s,t)) \}. 
 \] 
Then \( x \in Z \) if and only if there is \( y \in B(\lambda) \) such that \( T^\pi_{x,y} \) is well-founded, if and only if there is \( z \in \pre{\omega}{K} \) such that \( (x,(y,z)) \in [T] \).

The main goal of this section is to prove the following general result.

\begin{theorem} \label{thm:mainPSP}
Assume that \( 2^{< \lambda} = \lambda \). Let \( X \) be a \(\lambda\)-Polish space, and let \( Z \subseteq X \) be a TC-representable set. Then either there is a continuous injection from \(  C(\lambda) \) into  \(Z \), or else \( Z \) is well-orderable and \( |Z| \leq \lambda \). In particular, \( Z \) has the \lPSP.
\end{theorem}

%

The proof of Theorem~\ref{thm:mainPSP} is reminiscent of the proof that all \( \kappa \)-weakly homogeneously Souslin sets have the \( \mathrm{PSP} \) (see~\cite[Theorem 32.7]{Kanamori}).

\begin{proof}
Fix a \( \UU \)-representation \( \pi \) for \( Z \subseteq X \)
with the tower condition, and let \( F \) be a witness of the latter. We denote by \( K \) the support of \( \UU \).
Consider the following infinite two-players zero-sum game%
\footnote{The game \( G_{\pi,F} \) is clearly reminiscent of the game \( \lambda \)-\( \widetilde{G}^*_u(\cdot) \), except that now \( \pI \) has to additionally play elements \( z(k) \in K \) from the second turn on.}
 with perfect information \( G_{\pi, F} \):
\[
\begin{array}{c||c|c|c|c|c}
\pI & (s^0_i)_{i < \lambda_0} , y(0) & z(0), (s^1_i)_{i < \lambda_1}, y(1) & \dotsc & z(k-1), (s^k_i)_{i < \lambda_k}, y(k) & \dotsc \\
\hline
\pII & i_0 & i_1 & \dotsc & i_k & \dotsc
\end{array}
\]
The game is played alternatively with \( \pI \) playing first. At each round \( k \in \omega \):
\begin{enumerate-(a)}
\item \label{towergame-1}
\( \pI \) chooses a sequence \( (s^k_i)_{i < \lambda_k} \) and an ordinal \( y(k) < \lambda \) as in the game \( \lambda \)-\( \widetilde{G}^*_u(\cdot) \) from Definition~\ref{def:unfoldedtildegame}, the only difference being that now the index tree \( T \) of the tree-basis \( \mathcal{B} \) of \( X \) is assumed to be the entire \( \pre{<\omega}{\lambda} \).
\item
If \( k > 0 \), then \( \pI \) additionally chooses some \( z(k-1) \in K \), ensuring that
\[ 
(z(j))_{j < k} \in F(\pi(s^{k-1}_{i_{k-1}} \restriction k, (y(j))_{j < k})).
 \] 
(Notice that the rules in~\ref{towergame-1} entail that \( \lh(s^{k-1}_{i_{k-1}}) \geq k \) for all \( k \geq 1 \), so the requirement makes sense.)
\item
As in \( \lambda \)-\( \widetilde{G}^*_u(\cdot) \), after \( \pI \)'s move player \( \pII \) replies by choosing some \( i_k < \lambda_k \).
\end{enumerate-(a)}
The game ends after \(\omega\)-many rounds or when player \( \pI \) is no longer able to  move according to the rules above; in the latter case \( \pII \) wins, otherwise (i.e.\ if \( \pI \) can survive for infinitely many rounds) \( \pI \) wins. 


The game $G_{\pi,F}$ is the ``fully unfolded'' version of the game \(\lambda\)-\( \widetilde{G}^*(Z) \). Indeed,  if \( \pI \) wins a run of \( G_{\pi, F} \), then the sequence \( z = (z(k))_{k \in \omega} \in \pre{\omega}{K} \) is a witness, through the tower condition witnessed by \( F \), of the well-foundedness of the tower \( T^\pi_{\tilde{x},y} \), where \( \tilde{x} = \bigcup_{k \in \omega} s^k_{i_k} \) and \( y = (y(k))_{k \in \omega} \), and therefore \( f_\mathcal{B}(\tilde{x})\in Z \). 

The game \( G_{\pi,F} \) is closed for \( \pI \), hence it is quasidetermined by Proposition~\ref{prop:closeddeterminacy}.
If \( \pI \) has a winning quasistrategy in $G_{\pi,F}$, then (s)he also has a winning quasistrategy in the game \(\lambda\)-\( \widetilde{G}^*(Z) \) --- simply ``forget'' the $z$ and $y$ part, and play only the sequences $s_i^k$. By \( 2^{< \lambda} = \lambda \), the set of all possible moves in \(\lambda\)-\( \widetilde{G}^*(Z) \) is well-orderable (in fact, it has size \( \lambda \)), and hence the winning quasistrategy of \( \pI \) in \(\lambda\)-\( \widetilde{G}^*(Z) \) can be refined to a winning strategy. Therefore $C(\lambda)$ embeds into $Z$ by  Theorem~\ref{thm:equivalentgamesforPSP}. 

Suppose now that \( \pII \) has a winning quasistrategy in \( G_{\pi,F} \). Player \( \pII \)'s set of moves in \( G_{\pi,F} \) is the cardinal \(\lambda\), hence (s)he also has a winning strategy in \( G_{\pi,F} \): we want to show that this implies 
 $|Z|\leq\lambda$. One could be tempted to do as for \( \pI \), i.e.\ to just ignore $y$ and $z$, but there is a crucial obstacle. Suppose that $\tau$ is a winning strategy for \( \pII \) in the game $G_{\pi,F}$, and suppose that there are two different legal moves $p$ and $q$ for \( \pI \) with the same $(s_i^k)_{i < \lambda_k}$ but different additional parts $y(k)$ and $z(k)$. Then it is possible that $\tau(p)\neq\tau(q)$, while the strategy for \( \pII \) in \(\lambda\)-\( \widetilde{G}^*(Z) \) must choose just one $i_k$, independently of the forgotten part of \( \pI \)'s moves. We now prove that  we can indeed make a single choice for all possible values of $z(k)$, so that such part of \( \pI \)'s moves can indeed be discarded: this leaves us with a winning strategy for \( \pII \) in a game of the form \( \lambda \)-\( \widetilde{G}^*_u(\widetilde{C}) \) for a certain set \( \widetilde{C} \subseteq B(\lambda) \times B(\lambda) \), which by Theorem~\ref{thm:equivalentgamesforPSP2}\ref{thm:equivalentgamesforPSP2-2} will be enough to ensure \( |Z| \leq \lambda \).

Let us consider the following variant of the game \( G_{\pi,F} \), which does not depend on \( F \) and is denoted by \( G_\pi \):
\[
\begin{array}{c||c|c|c|c|c}
\pI & (s^0_i)_{i < \lambda_0} , y(0) &  (s^1_i)_{i < \lambda_1}, y(1) & \dotsc &  (s^k_i)_{i < \lambda_k}, y(k) & \dotsc \\
\hline
\pII & i_0 & i_1 & \dotsc & i_k & \dotsc
\end{array}
\]
The rules are exactly those of \( G_{\pi,F} \), except that now \( \pI \) does not have to play \( z(k) \in K \) in any of the rounds. The main difference is in the winning condition: \( \pI \)  wins if (s)he can survive for \(\omega\)-many rounds and the tower \( T^\pi_{\tilde{x},y} \) is well-founded, where \( \tilde{x} = \bigcup_{k \in \omega} s^k_{i_k} \) and \( y = (y(k))_{k \in \omega} \) are obtained from the run as usual. This is a kind of ``intermediate unfolding'' of \(\lambda\)-\( \widetilde{G}^*(Z) \), but, in contrast to the case of \( G_{\pi,F} \), the complexity of the payoff set of the game \( G_\pi \) is now quite high and strictly depends on the complexity of \( \pi \) (and thus of the represented set \( Z \)), so it is not clear \emph{a priori} that the game \( G_\pi \) is (quasi)determined. However, if \( \pI \) has a winning (quasi)strategy in \( G_{\pi, F} \) then (s)he can turn it into a winning (quasi)strategy in \( G_\pi \) by forgetting the additional moves \( z(k) \). We are now going to show that a similar result holds for \( \pII \).

\begin{claim} \label{claim:Gpi}
If \( \pII \) has a winning strategy in \( G_{\pi,F } \), then (s)he also has a winning strategy in \( G_\pi \).
\end{claim}

\begin{proof}[Proof of the claim]
We show how to transform any winning strategy \( \tau \) for \( \pII \) in \( G_{\pi,F} \) into a winning strategy for \( \pII \) in \( G_{\pi} \), by simulating a given run of \( G_\pi \) in the game \( G_{\pi,F} \).

The first move in \( G_{\pi,F} \) and \( G_\pi \) are of the same form, so \( \pII \) can just use \( \tau \) to reply to \( \pI \)'s first move \( ((s^0_i)_{i < \lambda_0}, y(0)) \) and play some \( i_0  < \lambda_0 \). 

Let \( p_0 \) be the position reached after this first round, and let 
\( ((s^1_i)_{i < \lambda_1}, y(1) ) \) 
be \( \pI \)'s next move in \( G_\pi \). Simulate this move in the game \( G_{\pi,F} \), letting \( \pI \) randomly choose which \( a \in K \) to play as \( z(0) \) to complete his/her move. More precisely, for each \( a \in K \) let \( i_a < \lambda_1 \) be \( \pII \)'s reply in \( G_{\pi,F} \) according to \(\tau\) when \( \pI \) plays 
\( (a, (s^1_i)_{i < \lambda_1} , y(1)) \)
in the second turn. Set \( \U_1 = \pi(s^0_{i_0} \restriction 1,  (y(0)))  \) and  \( A^{(1)} = F(\U_1) \).
 For each \( i < \lambda_1 \), let
\[ 
A^{(1)}_i = \{ a \in A^{(1)} \mid i_a = i \}.
 \] 
Since the sets \( A^{(1)}_i \) form a weak partition of \( A^{(1)} = F(\U_1) \in \U_1 \) and \( \U_1 \) is at least \(\lambda_1^+\)-complete, there is \( i_1 < \lambda_1 \) such that \( A^{(1)}_{i_1} \in \U_1 \): let \( \pII \) play such \( i_1 \) in the second turn of \( G_\pi \), and denote by \( p_1 \) the position reached in this way. Notice that by construction \( A^{(1)}_{i_1} \subseteq F(\U_1) \), hence for each \( 	\langle a \rangle \in A^{(1)}_{i_1} \) the position
\[ 
\begin{array}{c||c|c}
\pI & (s^0_i)_{i < \lambda_0}, y(0) & a, (s^1_i)_{i < \lambda_1}, y(1)  \\
\hline
\pII & i_0 & i_1 
\end{array}
 \] 
is legal in \( G_{\pi,F} \) and such that \( \pII \) has followed \( \tau \) in each round.

Now let 
\( ((s^2_i)_{i < \lambda_1}, y(2)) \)
be \( \pI \)'s next move in \( G_\pi \) after \( p_1 \). Consider the ultrafilter \( \U_2 = \pi (s^1_{i_1} \restriction 2, ( y(0), y(1) ) ) \), and notice that 
\[ 
 (A^{(1)}_{i_1})^{1 \to 2} = \{ t \in K^2 \mid t \restriction 1 \in A^{(1)}_{i_1} \} \in \U_2
\]
because \( s^1_{i_1} \supseteq s^0_{i_0} \) by the rules of \( G_\pi \), and hence the ultrafilter \( \U_2 \) projects onto \( \U_1 \) by definition of representation. Set
\[ 
A^{(2)} = (A^{(1)}_{i_1})^{1 \to 2} \cap F(\U_2),
 \] 
so that \( A^{(2)} \in \U_2 \) as well. For each \( t \in A^{(2)} \) let \( i_t \) be \( \pII \)'s reply according to \(\tau\) in the partial run in \( G_{\pi, F} \) played as follows:
\[ 
\begin{array}{c||c|c|c}
\pI & (s^0_i)_{i < \lambda_0}, y(0) & t(0), (s^1_i)_{i < \lambda_1}, y(1) & t(1), (s^2_i)_{i < \lambda_2} , y(2) \\
\hline
\pII & i_0 & i_1 & i_t
\end{array}
 \] 
(Notice that in this way \( \pII \) has so far  played according to \( \tau \) by definition of \( A^{(2)} \).) For each \( i < \lambda_2 \) let
\[
A^{(2)}_i = \{ t \in A^{(2)} \mid i_t = i \}.
\]
Since the sets \( A^{(2)}_i \) for a weak partition of \( A^{(2)} \in \U_2 \) and \( \U_2 \) is at least \(\lambda^+_2\)-complete, there is \( i_2 < \lambda_2 \) such that \( A^{(2)}_{i_2} \in \U_2 \): let \( \pII \) play such \( i_2 \) in the next turn in \( G_\pi \), and denote by \( p_2 \) the position reached in this way. Notice that still \( A^{(2)}_{i_2} \subseteq F(\U_2) \), so that
for each \( 	t \in A^{(2)}_{i_2} \) the position
\[ 
\begin{array}{c||c|c|c}
\pI & (s^0_i)_{i < \lambda_0}, y(0) & t(0), (s^1_i)_{i < \lambda_1}, y(1) & t(1), (s^2_i)_{i < \lambda_2} , y(2) \\
\hline
\pII & i_0 & i_1 & i_2
\end{array}
 \] 
is legal in \( G_{\pi,F} \) and such that \( \pII \) has played according to \( \tau \) so far.

Continuing in this way, after round \( k > 0 \) we will have reached a position \( p_{k} \) in \( G_\pi \) of the form
\[ 
\begin{array}{c||c|c|c|c}
\pI & (s^0_i)_{i < \lambda_0}, y(0) &  (s^1_i)_{i < \lambda_1}, y(1) & \dotsc &  (s^k_i)_{i < \lambda_k} , y(k) \\
\hline
\pII & i_0 & i_1 & \dotsc & i_k
\end{array}
 \] 
and, simultaneously, constructed a set 
\[ 
A^{(k)}_{i_k} \in \U_k =  \pi (s^{k-1}_{i_{k-1}} \restriction k, ( y(j) )_{j < k})
\]
such that \( A^{(k)}_{i_k} \subseteq A^{(k)} = (A^{(k-1)}_{i_{k-1}})^{k-1 \to k} \cap F(\U_k) \), and for every \( t \in A^{(k)}_{i_k} \) the position 
\[ 
\begin{array}{c||c|c|c|c}
\pI & (s^0_i)_{i < \lambda_0}, y(0) &  t(0), (s^1_i)_{i < \lambda_1}, y(1) & \dotsc & t(k-1), (s^k_i)_{i < \lambda_k} , y(k) \\
\hline
\pII & i_0 & i_1 & \dotsc & i_k
\end{array}
 \] 
is legal in \( G_{\pi,F} \) and such that \( \pII \) has played according to \( \tau \) (in particular, the moves of \( \pII \)  according to \( \tau \) in such a partial run are actually independent of the specific choice of \( t \in A^{(k)}_{i_k} \) by construction).

Suppose that \( ((s^{k+1}_i)_{i < \lambda_{k+1}}, y(k+1)) \) is the next move of \( \pI \) after \( p_k \) in the game \( G_\pi \). Let  \( \U_{k+1} = \pi(s^k_{i_k} \restriction (k+1),  (y(j))_{j < k+1}) \), and notice that
\[ 
 (A^{(k)}_{i_k})^{k \to k+1} =  \{ t \in K^{k+1}  \mid t \restriction k \in A^{(k)}_{i_k} \} \in \U_{k+1}
 \]
by definition of representation. (Here we again use the fact that, by the rules of \( G_\pi \), the sequence \( s^k_{i_k} \) extends \( s^{k-1}_{i_{k-1}} \) and thus \( \U_{k+1} \) projects onto \( \U_k \)). Let 
\[ 
A^{(k+1)} =  (A^{(k)}_{i_k})^{k \to k+1} \cap F(\U_{k+1}),
 \] 
 so that \( A^{(k+1)} \in \U_{k+1} \).
 For each \( t \in A^{(k+1)} \), let \( i_t \) be \( \pII \)'s reply according to \(\tau\) in the partial run in \( G_{\pi, F} \) played as follows:
\[ 
\begin{array}{c||c|c|c|c}
\pI & (s^0_i)_{i < \lambda_0}, y(0) &   \dotsc & t(k-1), (s^k_i)_{i < \lambda_k} , y(k) & t(k), (s^{k+1}_i)_{i < \lambda_{k+1}} , y(k+1) \\
\hline
\pII & i_0 &  \dotsc & i_k  & i_t
\end{array}
 \] 
(Notice that in this way \( \pII \) has always played according to \( \tau \) by the properties of \( A^{(k)}_{i_k} \).) For each \( i < \lambda_{k+1} \), let
\[
A^{(k+1)}_i = \{ t \in A^{(k+1)} \mid i_t = i \}.
\]
Since the sets \( A^{(k+1)}_i \) form a weak partition of \( A^{(k+1)} \in \U_{k+1} \) and \( \U_{k+1} \) is at least \(\lambda^+_{k+1}\)-complete, there is \( i_{k+1} < \lambda_{k+1} \) such that \( A^{(k+1)}_{i_{k+1}} \in \U_{k+1} \): let \( \pII \) play such \( i_{k+1} \) in the next turn in \( G_\pi \), and denote by \( p_{k+1} \) the position reached in this way. It is clear that all the properties described in the previous paragraph are preserved after this new step.

The above procedure describes a strategy for \( \pII \) in \( G_\pi \): we claim that it is a winning one. Suppose not, i.e.\ that \( \pI \) has been able to go on for \(\omega\)-many rounds in such a way that \( T^\pi_{\tilde{x}, y}  = (\U_k)_{k \geq 1} \) is well-founded, where \( \tilde{x} = \bigcup_{k \in \omega} s^k_{i_k} \) and \( y = (y(k))_{k \in \omega} \).  Since for each \( k \geq 1 \) we have \( A^{(k)}_{i_k} \in \U_k \), it follows that there is \( z \in \pre{\omega}{K} \) such that \( z \restriction k \in A^{(k)}_{i_k} \) for all \( k \geq 1 \). But then
\[ 
\begin{array}{c||c|c|c|c|c}
\pI & (s^0_i)_{i < \lambda_0}, y(0) &  z(0), (s^1_i)_{i < \lambda_1} , y(1) &  \dotsc & z(k-1), (s^k_i)_{i < \lambda_k} , y(k) & \dotsc \\
\hline
\pII & i_0 &  i_1 & \dotsc & i_k  & \dotsc
\end{array}
 \] 
is an infinite legal run in \( G_{\pi,F} \) (because by \( A^{(k)}_{i_k} \subseteq F(\U_k) \) the condition on the additional moves \( z(k) \) is verified) in which \( \pII \) followed \(\tau\): this means that \( \pI \) has won such run, contradicting the fact that \(\tau\) was winning for \( \pII \).
\end{proof}

By Claim~\ref{claim:Gpi}, we now know that \( \pII \) has a winning strategy in \( G_\pi \).
To complete the proof, just observe that \( G_\pi \) is of the form \(\lambda\)-\( \widetilde{G}^*_u(\widetilde{C}) \), where
\[ 
\widetilde{C} = \{ (\tilde{x},y) \in B(\lambda) \times B(\lambda) \mid T^\pi_{\tilde{x},y} \text{ is well-founded} \}, 
 \] 
 and that \( f_\mathcal{B}(\p(\widetilde{C})) = Z \) by definition of \( \UU \)-representation. Thus \( |Z| \leq \lambda \) by Theorem~\ref{thm:equivalentgamesforPSP2}.
\end{proof}

The proof of Theorem~\ref{thm:mainPSP} shows that, indeed, the game $G_\pi = \lambda$-\( \widetilde{G}^*_u(\widetilde{C}) \) is determined, but gives no direct information on the games \(\lambda\)-\( G^*(Z) \) and  \(\lambda\)-\( \widetilde{G}^*(Z) \). But because of Theorems~\ref{thm:generalgameforPSP} and~\ref{thm:equivalentgamesforPSP}, we can \emph{a posteriori} conclude that since \( Z \) has the \lPSP, then also \(\lambda\)-\( G^*(Z) \) and  \(\lambda\)-\( \widetilde{G}^*(Z) \) are determined.

\begin{corollary} \label{cor:determinacyPSPgames}
Assume that \( 2^{< \lambda} = \lambda \), and let \( X \) be a \(\lambda\)-Polish space.
If $Z \subseteq X$ is TC-representable, then \(\lambda\)-\( G^*(Z) \) is determined, and the same is true for its variant \(\lambda\)-\( \widetilde{G}^*(Z) \).
\end{corollary}

\chapter{More on representability}
\label{sec:applications}

In this chapter we will show that our concept of $\UU$-representability is a generalization of both weakly homogeneously Souslin-ness (for subsets of $\pre{\omega}{\omega}$) and $\UU(j)$-representability (for subsets of $V_{\lambda+1}$, where $j$ witnesses that $\mathsf{I0}(\lambda)$ holds). 
As a consequence, we wil argue that under \( \mathsf{I0}(\lambda) \), Theorem~\ref{thm:mainPSP} yields that all definable subsets of any \(\lambda\)-Polish space, and in particular all its \( \lambda \)-projective sets, have the \( \lPSP \). 


\section{Representability and weakly homogeneously Souslin sets} \label{sec:representabilityvsweaklyhomogeneouslysouslin}

The goal of this section is to show that if we set \( \lambda = \omega \) in Definition~\ref{def:repr} (see Remark~\ref{rmk:U-representabilityandmeasurablecardinals}), then we obtain an equivalent reformulation of the following classical notion (see e.g.\ \cite[pp. 453]{Kanamori}).

\begin{defin} \label{defin:whS}
Let \( \kappa > \omega \).
A set $Z\subseteq\pre{\omega}{\omega}$ is \markdef{\( \kappa \)-weakly homogeneously Souslin} if and only if there exist a set \( Y \) and a tree $T$ on $\omega\times Y$ such that:
\begin{enumerate-(a)}
\item \label{defin:whS-0} 
$T(s)=\{u \in \pre{k}{Y} \mid (s,u)\in T\} \neq \emptyset$ for every \( s \in \pre{<\omega}{\omega} \);
\item \label{defin:whS-1}
for every \( k \in \omega \) and every $s,t\in\pre{k}{\omega}$ there is a $\kappa$-complete  ultrafilter $\U_{s,t}$ over \( T(s) \) such that $\U_{s,t}$ projects onto $\U_{s \restriction \ell, t \restriction \ell}$ for all $\ell <  k$; 
\item \label{defin:whS-2}
$x\in Z$ if and only if there is a $y\in\pre{\omega}{\omega}$ such that the tower of ultrafilters  $(\U_{x \restriction k,y \restriction k})_{k\in \omega}$ is well-founded. 
\end{enumerate-(a)}
We say that \( Z \subseteq \pre{\omega}{\omega} \) is \markdef{weakly homogeneously Souslin} if it is \( \kappa \)-weakly homogeneously Souslin for some \( \kappa > \omega \).
\end{defin}


\begin{remark} \label{rmk:whS}
As for Definition~\ref{def:repr}, if at least one ultrafilter \( \U_{s,t} \) in item~\ref{defin:whS-1} of Definition~\ref{defin:whS} is nonprincipal, then we can again assume without loss of generality that the cardinal \( \kappa \) appearing in Definition~\ref{defin:whS} is actually a measurable cardinal. In particular, the existence of a weakly homogeneous set with such a property implies that there is a measurable cardinal.
\end{remark}

\begin{proposition} \label{prop:representablevsweaklyhomogeneous}
A set \( Z \subseteq  \pre{\omega}{\omega} = B(\omega) \) is weakly homogeneously Souslin if and only if it is \( \UU \)-representable (in the sense of Definition~\ref{def:repr} applied with \( \mathcal{B} = \{ \Nbhd_s \mid s \in \pre{<\omega}{\omega} \} \) the canonical basis for \( \pre{\omega}{\omega} \) and \( \lambda = \omega \); see Remark~\ref{rmk:U-representabilityandmeasurablecardinals}).
\end{proposition}

\begin{proof}
Suppose first that $Z\subseteq\pre{\omega}{\omega}$ is \( \kappa \)-weakly homogeneously Souslin, for some \( \kappa > \omega \), as witnessed by the tree \( T \subseteq \pre{<\omega}{(\omega \times Y)} \).
For each \( k \in \omega \) and  \( s,t \in \pre{k}{\omega} \), let
\[
\pi(s,t) = \{ A \subseteq \pre{k}{Y} \mid H \subseteq A \text{ for some } H \in \U_{s,t} \},
\]
and let \( \UU = \ran(\pi) \). 
It is easy to see that $\UU$ is an orderly family of $\kappa$-complete ultrafilters with support $K=Y$, and that $\pi$ is a $\UU$-representation for $Z$ with respect to the canonical basis $\mathcal{B}$ of \( \pre{\omega}{\omega} \) (in which case \( f_{\mathcal{B}} \) is the identity map). 

On the other hand, let $\UU$ be an orderly family of $\kappa$-complete ultrafilters (for some $\kappa>\omega$) with support $K$, and let \( \pi \) be a \( \UU \)-representation for $Z$ (with respect to the canonical basis $\mathcal{B}$). Let \( Y = K \), and let $T = \pre{<\omega}{\omega \times Y}$ be the complete tree on $\omega \times Y$. Finally,   for every \( k \in \omega \) and \( s,t \in \pre{k}{Y} \), let $\U_{s,t}=\pi(s,t)$. It is easy to see that the tree $T$ and the ultrafilters $\U_{s,t}$ witness that $Z$ is $\kappa$-weakly homogeneously Souslin.
\end{proof}

As for \( \UU \)-representability, in the definition of weakly homogeneously Souslin set principal ultrafilters are allowed. However, arguing as in Remark~\ref{rmk:principalrepresentations} one can easily see that the existence of a proper weakly homogeneously Souslin subset \( Z \) of \( \pre{\omega}{\omega} \) implies that there is a measurable cardinal \( \kappa \). 
Indeed, if \( Z \neq \pre{\omega}{\omega} \) is weakly homogeneously Souslin, then some of the ultrafilters \( \U_{s,t} \) witnessing this need to be nonprincipal, and hence their completeness number is a measurable cardinal. Since under \( \AC \) any  measurable cardinal \( \kappa \) satisfies \( \kappa > 2^{\aleph_0} \), this entails that the tower condition from Definition~\ref{def:TC} is automatic when \( \lambda = \omega \) (see the proof of Proposition~\ref{prop:TCautomatic}), and this explains why such a condition does not explicitly appear in the literature on weakly homogeneously Souslin sets, where one usually works in \( \ZFC \).
It must also be noted that the equivalence among the different definitions of weakly homogeneously Souslin sets that can be found in the literature can break down when there is no measurable cardinal: for example, according to~\cite[Definition 33.32]{Jech2003} there are no weakly homogeneously Souslin sets if no measurable cardinal exists, while according to~\cite[Definition 1.3.4]{Larson2004}, in such a situation the only weakly homogeneously Souslin set is the empty set.

The existence of a measurable cardinal unlocks the possibility of having nontrivial weakly homogeneously Souslin subsets of \( \pre{\omega}{\omega} \). Indeed, it is known that if there exists a measurable cardinal, then all $\boldsymbol{\Sigma}^1_2$ sets are weakly homogeneously Souslin (see~\cite[Exercises 32.1 and 31.3]{Kanamori}). Moreover, if larger cardinal assumptions are made, then more and more subsets of \( \pre{\omega}{\omega} \) become weakly homogeneously Souslin. For example, Woodin showed that if there are $\omega$-many Woodin cardinals with a measurable cardinal above them, then all sets in $L(\RR)$ are weakly homogeneously Souslin (see~\cite[Theorem 32.13]{Kanamori}).

\section{Representability and \( \UU(j) \)-representable sets under \( \mathsf{I0}(\lambda) \)} \label{sec:representabilityvsU(j)-representability} 

In his seminal work~\cite{Woodin2011}, Woodin recognized that under \( \mathsf{I0}(\lambda) \) there are strong similarities between the theory of \( L(V_{\lambda+1}) \), and the theory of \( L(\RR) \) under \( \AD \).
For example, if $\mathsf{I0}(\lambda)$ holds, then $L(V_{\lambda+1}) \not\models \AC$ but \( L(V_{\lambda+1}) \models \DC_\lambda \), and this is the analogue of the fact that under \( \AD \) we have \( L(\RR) \not\models \AC \) but \( L(\RR) \models \DC_\omega \). Furthermore, under the mentioned assumptions we have that $\lambda^+$ is measurable in $L(V_{\lambda+1})$, just as $\omega_1$ is measurable in $L(\mathbb{R})$; a suitable version of Moschovakis' Coding Lemma holds in both \( L(V_{\lambda+1}) \) and \( L(\RR) \); and so on --- all of this is well described in~\cite{Woodin2011}. To further his case, Woodin isolates in~\cite[Definitions 103-111]{Woodin2011} an analogue of weakly homogeneously Souslin sets in the \( \mathsf{I0}(\lambda) \) context, namely, $\UU(j)$-representable sets. The definition of $\UU(j)$-representability is slightly more involved than our definition of $\UU$-representability, and needs some understanding of the model $L(V_{\lambda+1})$. 

Since under \( \mathsf{I0}(\lambda) \) the set $V_{\lambda+1}$ is not well-orderable in $L(V_{\lambda+1})$, 
in analogy with $L(\mathbb{R})$ we let 
\[
\Theta^{L(V_{\lambda+1})}=\sup\{\alpha \in \On \mid L(V_{\lambda+1}) \models \text{``there is a surjection } \pi \colon V_{\lambda+1}\to\alpha\text{''}\}. 
\]
This is a way to define the smallest cardinal ``larger'' than $|V_{\lambda+1}|$ in the case when the cardinality of $V_{\lambda+1}$, as computed in \( L(V_{\lambda+1}) \), is not a cardinal number (otherwise $\Theta$ is just $|V_{\lambda+1}|^+$). 

 


For the definition of $\mathbb{U}(j)$-representability, it is important to recall that if $j$ is an elementary embedding that witnesses $\mathsf{I0}(\lambda)$, then $j$ is iterable (see~\cite{Woodin2011} just before Definition 15 for the definition of an iterate, and Lemma 17 ibidem for the proof). We denote by $j^n$ the $n$-th iterate of \( j \).

The collection of ultrafilters $\mathbb{U}(j)$ is defined as a set of $L(V_{\lambda+1})$-ultrafilters that have the following three properties: they have size comparable to $V_{\lambda+1}$, they are fixed points for coboundedly many iterates of $j$, and they concentrate on elements that are fixed points for coboundedly many iterates of $j$ (see~\cite[Definition 103]{Woodin2011}).

\begin{defin} 
Given an elementary embedding \( j \) witnessing $\mathsf{I0}(\lambda)$, let $\mathbb{U}(j)$ be the collection of all $\U\in L(V_{\lambda+1})$ such that in $L(V_{\lambda+1})$ the following hold:
\begin{enumerate-(a)}
\item 
$\U$ is a $\lambda^+$-complete ultrafilter on some set \( K_\U \) satisfying \( K_\U \in L_{\xi}(V_{\lambda+1}) \), for some \( \xi <\Theta^{L(V_{\lambda+1})} \);
\item 
for some $\gamma<\Theta^{L(V_{\lambda+1})}$, \( \U \) is generated by $\U\cap L_\gamma(V_{\lambda+1})$;
\item 
$j^n(\U)=\U$ for all sufficiently large $n\in\omega$;
\item 
there is $A\in L(V_{\lambda+1})$ such that $A\subseteq K_\U$ and $\{a\in A \mid j^n(a)=a\}\in \U$.
 \end{enumerate-(a)}
\end{defin}

Woodin then defines a subfamily of $\mathbb{U}(j)$, denoted by $\mathbb{U}(j,\kappa, (a_i)_{i\in\omega})$, which is parametrized by a cardinal $\kappa<\Theta^{L(V_{\lambda+1})}$ and a sequence $(a_i)_{i\in\omega}$ of elements of $L_\kappa(V_{\lambda+1})$. Its exact definition goes beyond our purposes, and will thus be omitted: indeed, we mention the families \( \mathbb{U}(j,\kappa, (a_i)_{i\in\omega}) \) only to faithfully report~\cite[Definitions 110-111]{Woodin2011} in the definition below, but for the results of this section it is enough to keep in mind that \( \mathbb{U}(j,\kappa, (a_i)_{i\in\omega}) \subseteq \UU(j) \).

\begin{defin}
\label{def:representableforWoodin}
 Let $j$ be a witness for $\mathsf{I0}(\lambda)$, $\kappa<\Theta^{L(V_{\lambda+1})}$, and $(a_i)_{i\in\omega}$ be a sequence of elements of $L_\kappa(V_{\lambda+1})$.
A set $Z\subseteq V_{\lambda+1}$ belonging to $L(V_{\lambda+1})$ is \markdef{$\mathbb{U}(j,\kappa, (a_i)_{i\in\omega})$-representable} if there exists a strictly increasing sequence of infinite cardinals $(\lambda_k)_{k\in\omega}$ cofinal in $\lambda$ and a function 
\[
\pi \colon \bigcup\{V_{\lambda_k+1}\times V_{\lambda_k+1}\times\{k\} \mid k\in\omega\}\to \mathbb{U}(j,\kappa, (a_i)_{i\in\omega})
\]
such that the following conditions hold:
\begin{enumerate-(a)}
\item \label{def:representableforWoodin-1} 
for all $(a,b,k)\in\dom(\pi)$ there exists a set $A_{a,b,k}\subseteq \pre{k}{(L(V_{\lambda+1}))}$ such that $A_{a,b,k}\in\pi(a,b,k)$;
\item \label{def:representableforWoodin-2} 
for all $(a,b,k)\in\dom(\pi)$ and $\ell < k$ we have that $\pi(a,b,k)$ projects onto $\pi(a\cap V_{\lambda_\ell},b\cap V_{\lambda_\ell},\ell)$;
\item \label{def:representableforWoodin-3} 
for all $x\in V_{\lambda+1}$, $x\in Z$ if and only if there exists $y\in V_{\lambda+1}$ such that
\begin{itemizenew}
\item 
$(x\cap V_{\lambda_k},y\cap V_{\lambda_k},k)\in\dom(\pi)$ for all $k \in \omega$, and
\item 
the tower of ultrafilters $(\pi(x\cap V_{\lambda_k},y\cap V_{\lambda_k},k))_{k \in \omega}$ is well-founded.
\end{itemizenew}
\end{enumerate-(a)}
	
A set $Z\subseteq V_{\lambda+1}$ belonging to $L(V_{\lambda+1})$ is \markdef{$\mathbb{U}(j)$-representable} if and only if 
it is $\mathbb{U}(j,\kappa,(a_i)_{i\in\omega})$-representable for some \( \kappa \) and \( (a_i)_{i \in \omega} \) as above.
%
\end{defin}

We are going to show that our notion of $\UU$-representability is aligned with Woodin's Definition~\ref{def:representableforWoodin}, and since we will often work inside $L(V_{\lambda+1})$ under \( \mathsf{I0}(\lambda) \), we will also make sure to avoid any use of \( \AC \).
Notice that \( \mathsf{I0}(\lambda) \) implies that \( |V_\lambda| = \lambda \), and that any bijection witnessing this belongs to \( V_{\lambda+1} \); therefore $V_{\lambda+1}$ is a \(\lambda\)-Polish space homeomorphic to \( B(\lambda) \) in both \( V \) and \( L(V_{\lambda+1}) \) (see Theorem~\ref{thm:homeomorphictoCantor}\ref{thm:homeomorphictoCantor-3}), and it makes sense to consider $\UU$-representable subsets of it. 

First recall that for \( \U \) and \( \V \) ultrafilters on the sets \( X \) and \( Y \), respectively, their \markdef{tensor product} \( \U \otimes \V \) is the ultrafilter on \( X \times Y \) defined by 
\[
A \in \U \otimes \V \iff \{ x \in X \mid A_x \in \V \} \in \U,	
\]
where as usual \( A_x \) denotes the vertical section of \( A \) at \( x \). 
If \( X = \pre{k}{K} \) and \( Y = \pre{\ell}{K} \) for some set \( K \) and natural numbers \( k,\ell \in \omega \), we let \( \U {}^\smallfrown{} \V \) be the ultrafilter on \( \pre{k+\ell}{K} \) defined by stipulating that \( B \in \U {}^\smallfrown{} \V \) if and only if \( B \) is of the form
\[
B = \{ s {}^\smallfrown{} t \in \pre{k+\ell}{K} \mid (s,t) \in A \}
\]
for some \( A \in \U \otimes \V \).
It is easy to verify that for every \( \nu \in \Cn \), the ultrafilter \( \U {}^\smallfrown{}  \V \) is \( \nu \)-complete if so are \( \U \) and \( \V \), that \( \U {}^\smallfrown{} \V \) projects onto \( \U \), and that if \( \V \) projects onto \( \V' \) then \( \U {}^\smallfrown{} \V \) projects onto \( \U {}^\smallfrown{} \V' \).
In particular, if \( (\U_k)_{k \in \omega} \) and \( (\V_k)_{k \in \omega} \) are towers of ultrafilters in the same orderly family \( \U \), then for every \( n \in \omega \) the family \( (\mathcal{W}^{(n)}_k)_{k \in \omega} \) defined by
\[
\mathcal{W}^{(n)}_k = 
\begin{cases}
\U_k & \text{if } k \leq n \\
\U_n {}^\smallfrown{} \V_{k-(n+1)} & \text{if } k > n
\end{cases}
\] 
is a tower of ultrafilters in some orderly family \( \UU' \supseteq \UU \) with the same support \( K \).

\begin{proposition}
 \label{prop:representablethesame} 
 Let $j$ be a witness for $\mathsf{I0}(\lambda)$, and let $Z\subseteq V_{\lambda+1}$ be a set in \( L(V_{\lambda+1}) \). 
If \( Z \) is $\mathbb{U}(j)$-representable, then 
there is some orderly family of ultrafilters \( \UU \) such that
$Z$ is $\UU$-representable (with respect to a suitable tree-basis \( \mathcal{B} \) for \( V_{\lambda+1} \)).

\end{proposition} 

\begin{proof}
Let $(\lambda_k)_{k \in \omega}$ and $\pi'$ witness that \( Z \) is \( \UU(j) \)-representable. 
Recall from the proof of Theorem~\ref{thm:homeomorphictoCantor}\ref{thm:homeomorphictoCantor-3} the sets \( V^-_{\lambda_k+1} = \pow( V_{\lambda_{k}}\setminus V_{\lambda_{k-1}}) \), and for every \( k \in \omega \)
fix a surjection \( \sigma_k \colon \lambda \to V^-_{\lambda_k+1} \).
Recall also the basic open sets \( O_{(a,\alpha)} \) generating the Woodin's topology on \( V_{\lambda+1} \) (see Example~\ref{xmp:lambda-Polish}\ref{xmp:lambda-Polish-4}), which are clearly clopen sets.
Let \( U_\emptyset = V_{\lambda+1} \), and for every \( k > 0 \) and \( s \in \pre{k}{\lambda} \) let \( U_s = O_{(a(s),\alpha(s))} \), where \( \alpha(s) = \lambda_{k-1} \) and \( a(s) = \bigcup_{\ell < k} \sigma_\ell(s(\ell)) \).
Then \( \mathcal{B} = \{ U_s \mid s \in \pre{<\omega}{\lambda} \} \) is a  tree-basis for \( V_{\lambda+1} \), and clearly \( f_{\mathcal{B}}(\tilde{x})=\bigcup_{k \in \omega} \sigma_k(\tilde{x}(k)) \) for every \( \tilde{x} \in B(\lambda) \). 

For each \( (a,b,k) \in \dom(\pi') \), let \( K_{a,b,k} = \{ u(\ell) \mid u \in A_{a,b,k} \wedge \ell < k \} \), where \( A_{a,b,k} \) is as in condition~\ref{def:representableforWoodin-1} of Definition~\ref{def:representableforWoodin}. Let \( K = \bigcup_{(a,b,k) \in \dom(\pi')} K_{a,b,k} \), and fix any element \( p \in K \). Let \( \overline{\U} \) be the only proper ultrafilter on \( \pre{0}{K} = \{ \emptyset \} \),  \( \overline{\V} = \{ B \subseteq K \mid p \in B \} \) be the principal ultrafilter on \( K \) generated by \( p \), and for each \( (a,b,k) \in \dom(\pi') \), let \( \U_{a,b,k} = \{ B \subseteq \pre{k}{K} \mid B \cap A_{a,b,k} \in \pi'(a,b,k) \} \).  
Finally, let 
\[ 
\UU = \{ \overline{\U}, \overline{\V} \} \cup  \{ \overline{\V} {}^\smallfrown{}  \U_{a,b,k} \mid (a,b,k) \in \dom(\pi') \} . 
\]
By construction, \( \UU \) is an orderly family with support \( K \), and moreover 
\( \lev(\overline{\U}) = 0 \), \( \lev(\overline{\V}) = 1 \), and
\( \lev(\overline{\V} {}^\smallfrown{}  \U_{a,b,k}) = k+1 \) for every \( (a,b,k) \in \dom(\pi') \).
Now define
\[
 \pi \colon \bigcup_{k\in\omega}(\pre{k}\lambda\times\pre{k}\lambda)\to\UU
\]
by letting for every \( k \in \omega \) and \( s,t \in \pre{k}{K} \), \( \pi(s,t) = \overline{\U} \) if \( k = 0 \), \( \pi(s,t) = \overline{\V} \) if \( k = 1 \), and \( \pi(s,t) =\overline{\V} {}^\smallfrown{}  \U_{a(s),a(t),k-1} \) if \( k > 1 \).
%
We claim that \( \pi \) is a \( \UU \)-representation for \( Z \) with respect to the tree-basis \( \mathcal{B} \) defined above, i.e.\ that conditions~\ref{def:repr-1}--\ref{def:repr-3} from Definition~\ref{def:repr} are satisfied for our choice of \( \UU \), \( \mathcal{B} \), and \( \pi \).

Condition~\ref{def:repr-1} holds by construction. Notice also that all ultrafilters in \( \UU \) are \( \lambda^+ \)-complete, because so are \( \overline{\V} \) (being principal) and each \( \U_{a,b,k} \) (because \( \pi(a,b,k) \), being a member of \( \UU(j) \), is \( \lambda^+ \)-complete). In order to prove that \( T^\pi_{\tilde{x},y} = (\pi(\tilde{x} \restriction k, y \restriction k))_{k \in \omega} \) is a tower of ultrafilters for any \( \tilde{x},y \in B(\lambda) \), it is enough to show that for every \( 1 < \ell < k \), the ultrafilter \( \U_{a, b, k-1} \) projects onto \( \U_{a',b',\ell-1} \), where \( a = a(\tilde{x} \restriction k) \), \( b = y \restriction k \), \( a' = a(\tilde{x} \restriction \ell) \), and \( b' = y \restriction \ell \). This follows from the fact that by condition~\ref{def:representableforWoodin-2} in Definition~\ref{def:representableforWoodin}, \( \pi'(a,b,k-1) \) projects onto \( \pi'(a',b',\ell-1) \) because by construction \( a' = a \cap V_{\ell-1} \) and \( b' = b \cap V_{\ell-1} \). This shows that condition~\ref{def:repr-2} from Definition~\ref{def:repr} holds as well. Finally, we check that also condition~\ref{def:repr-3} from the same definition is satisfied. 
Let \( \tilde{x} \in B(\lambda) \), and let \( x = f_{\mathcal{B}}(\tilde{x}) \). Suppose first that \( x \in Z \). Let \( y \in V_{\lambda+1} \) be as in Definition~\ref{def:representableforWoodin}\ref{def:representableforWoodin-3}, and let \( \tilde{y} \in B(\lambda) \) be such that \( f_{\mathcal{B}}(\tilde{y}) = y \). By construction, for every \( k \in \omega \) we have that \( x \cap V_{\lambda_k} = a(\tilde{x} \restriction k+1)\) and \( y \cap V_{\lambda_k} = a(\tilde{y} \restriction k+1)\). Since by choice of \( y \) the tower of ultrafilters \( (\pi'(x \cap V_{\lambda_k}, y \cap V_{\lambda_k},k))_{k \in \omega} \) is well-founded, so is the tower \( (\U_{a(\tilde{x} \restriction k+1),a(\tilde{y} \restriction k+1),k})_{k \in \omega} \), which in turn implies that \( T^{\pi}_{\tilde{x},\tilde{y}} \) is well-founded. This proves the forward implication of Definition~\ref{def:repr}\ref{def:repr-3}. For the backward implication, we basically reverse the argument. If \(\tilde{y} \in B(\lambda) \) is such that \( T^\pi_{\tilde{x},\tilde{y}} \) is well-founded, then both \( (\U_{a(\tilde{x} \restriction k+1),a(\tilde{y} \restriction k+1),k})_{k \in \omega} \) and \( (\pi'(x \cap V_{\lambda_k}, y \cap V_{\lambda_k},k))_{k \in \omega} \) are well-founded, where again \( y = f_{\mathcal{B}}(\tilde{y}) \); therefore \( f_{\mathcal{B}}(\tilde{x}) =  x \in Z \) by Definition~\ref{def:representableforWoodin}\ref{def:representableforWoodin-3}.
%
%
%
%
\end{proof}

Arguing as in Proposition~\ref{prop:U-representabilityandmeasurablecardinals}, it is easy to see that if there is some \( \UU(j) \)-representable \( Z \subsetneq V_{\lambda+1} \), so that \( Z \) is also \( \UU \)-representable for a suitable (necessarily nonprincipal) orderly family of ultrafilters \( \UU \) by Proposition~\ref{prop:representablethesame}, then there is a measurable cardinal \( \kappa > \lambda \), and one can assume that all ultrafilters appearing in the range of both the \( \UU(j) \)-representation \( \pi' \) and the \( \UU \)-representation \( \pi \) for \( Z \) are at least \( \kappa \)-complete.
If \( \AC \) is assumed, then \( \pi \) would automatically have the tower condition (in the sense of our Definition~\ref{def:TC}, see~Proposition~\ref{prop:TCautomatic}), and hence every \( \UU(j) \)-representable set would be TC-representable. However, \( L(V_{\lambda+1}) \) does not satisfy \( \AC \) if \( \mathsf{I0}(\lambda) \) holds, and that is why Woodin introduced his own version of the tower condition in~\cite[Definition 121]{Woodin2011}.

\begin{defin} \label{def:WoodinTC}
 Let $j$ be a witness for $\mathsf{I0}(\lambda)$. Let $\mathbb{V}\subseteq\mathbb{U}(j)$ be such that $\mathbb{V} \in L(V_{\lambda+1})$ and $|\mathbb{V}|\leq\lambda$. Then the \markdef{tower condition holds for $\mathbb{V}$} if and only if there is a function $F \colon \mathbb{V} \to L(V_{\lambda+1})$ such that
\begin{enumerate-(a)}
\item \label{def:WoodinTC-1}
$F(\U)\in \U$ for each $\U\in A$;%
\footnote{This requirement actually is not in the original definition in~\cite{Woodin2011}, but it is in \cite[Definition 6.5]{Cramer2015}, where the crucial Theorem~\ref{thm:U(j)-representablesetshaveTC} is proved.}
\item \label{def:WoodinTC-2}
for each tower of ultrafilters%
\footnote{As one can expect, this means that for some set \( K \), each \( \U_k \) concentrates on \( \pre{k}{K} \), and for every \( k > \ell \) the ultrafilter \( \U_k \) projects onto \( \U_\ell \).}
$(\U_k)_{k\in\omega}$ in $\mathbb{V}$, the tower is well-founded in $L(V_{\lambda+1})$ if and only if there exists $z \in \pre{\omega}{L(V_{\lambda+1})}$ such that \( z \restriction k \in F(\U_k) \) for every \( k \in \omega \).
\end{enumerate-(a)}
\end{defin}

We are now going to show that, once again, Woodin's Definition~\ref{def:WoodinTC} aligns with our version of the tower condition from Definition~\ref{def:TC}. Notice that if \( \pi \) witnesses that a set is \( \UU(j) \)-representable, then \( |\ran(\pi)| \leq \lambda \), 
and hence it makes sense to ask if Woodin's tower condition holds for \( \mathbb{V} = \ran(\pi) \).
As for Proposition~\ref{prop:representablethesame}, also the next result is developed without assuming \( \AC \), and can thus be applied e.g.\ in \( L(V_{\lambda+1}) \).

\begin{proposition} 
 Let $j$ be a witness for $\mathsf{I0}(\lambda)$, and let \( Z \subseteq V_{\lambda+1} \) be a set in \( L(V_{\lambda+1} ) \).  If \( \pi' \) witnesses that $Z$ is $\mathbb{U}(j)$-representable and the tower condition (in the sense of Definition~\ref{def:WoodinTC}) holds for $\mathbb{V} = \ran(\pi')$, then $Z$ is TC-representable.
\end{proposition}

\begin{proof}
We use the same notation as in the proof of Proposition~\ref{prop:representablethesame}.
Let \( F' \colon \ran(\pi') \to L(V_{\lambda+1} ) \) witness the fact that the tower condition holds for \( \ran(\pi') \).  Without loss of generality, we can assume that \( F'(\pi'(a,b,k)) \in \U_{a,b,k} \): if not, just systematically replace \( F'(\pi'(a,b,k)) \) with \( F'(\pi'(a,b,k)) \cap A_{a,b,k} \).
Let \( \pi \) be the \( \UU \)-representation of \( Z \) obtained from \( \pi' \) as in the proof of Proposition~\ref{prop:representablethesame}. Set \( F(\pi(\emptyset,\emptyset)) = \{ \emptyset \} \), \( F(\pi(s,t)) = \{ p \} \) for every \( s,t \in \pre{1}{\lambda} \), and for every \( k > 1 \) and \( s,t \in \pre{k}{\lambda} \) let 
\[
F(\pi(s,t)) = p {}^\smallfrown{} F'(\pi'(a(s),a(t),k-1)), 
\]
where for \( B \subseteq \pre{k-1}{K} \) we let \( p {}^\smallfrown{} B = \{ p {}^\smallfrown{}  u \in \pre{k}{K} \mid u \in B \} \). Notice that in all cases, \( F(\pi(s,t)) \in \pi(s,t) \) by construction. 
Recalling that for every \( x,y \in V_{\lambda+1} \) and every \( \tilde{x},\tilde{y} \in B(\lambda) \) such that \( f_{\mathcal{B}}(\tilde{x}) = x \) and \( f_{\mathcal{B}}(\tilde{y}) = y \), the tower of ultrafilters \( T^\pi_{\tilde{x},\tilde{y}} \) is well-founded if and only if so is \( (\pi'(x \cap V_{\lambda_k}, y \cap V_{\lambda_k},k))_{k \in \omega} \), it is easy to see that \( F \) satisfies also condition~\ref{def:TC-2} of Definition~\ref{def:TC}. 
Indeed, suppose that \( z \in \pre{\omega}{K} \) is such that
\( z \restriction k \in F(\pi(\tilde{x} \restriction k, \tilde{y} \restriction k)) \)
for all \( k \in \omega \).
Then \( z(0) = p \) by choice of \( F \), and thus there is \( z' \in \pre{\omega}{K} \) such that \( z = p {}^\smallfrown{}  z' \). But then by construction
\( z' \restriction k \in F'(\pi'(x \cap V_{\lambda_k}, y \cap V_{\lambda_k},k)) \)
for all \( k \in \omega \).
It follows that \(  (\pi'(x \cap V_{\lambda_k}, y \cap V_{\lambda_k},k))_{k \in \omega} \) is well-founded, and hence so is \( T^\pi_{\tilde{x},\tilde{y}} \).
\end{proof}

The above discussion naturally raises the question of which \( \mathbb{V} \subseteq \UU(j) \) satisfy Woodin's tower condition. Immediately after~\cite[Definition 121]{Woodin2011}, it is observed that, assuming \( \DC_\lambda \), this is always true if \( |\mathbb{V}| < \lambda \), and from this Woodin derived a rather weak form of the \( \lambda \)-Perfect Set Property for all \( \UU(j) \)-representable sets (see Lemma~\ref{lem:WoodinsweakPSP}).
Later on, Scott Cramer conclusively answered the above question by proving that, indeed, all families \( \mathbb{V} \) as in Definition~\ref{def:WoodinTC} satisfy the tower condition~\cite[Theorem 6.8]{Cramer2015}.

\begin{theorem}[\( \DC_\lambda \)] \label{thm:U(j)-representablesetshaveTC} 
 Let $j$ be a witness for $\mathsf{I0}(\lambda)$, and work in \( L(V_{\lambda+1}) \).%
\footnote{This means that we only consider sets \( Z \in L(V_{\lambda+1}) \), and that both the \( \UU(j) \)-representability and the \( \UU \)-representability have to be verified in \( L(V_{\lambda+1}) \).} 
 Then Woodin's tower condition holds for every  $\mathbb{V} \subseteq\mathbb{U}(j)$ such that $|\mathbb{V}|\leq\lambda$.  
\end{theorem}

Since if \( \pi \) witnesses that a set \( X \subseteq V_{\lambda+1} \) is \( \UU(j) \)-representable then \( \mathbb{V} = \ran(\pi) \) has size at most \(\lambda\), 
combining Proposition~\ref{prop:representablethesame} and Theorem~\ref{thm:U(j)-representablesetshaveTC} we obtain:

\begin{corollary}[\( \DC_\lambda \)] 
\label{cor:fromTowertoTC}
Let $j$ be a witness for $\mathsf{I0}(\lambda)$, and work in \( L(V_{\lambda+1}) \). If $Z\subseteq V_{\lambda+1}$ is $\mathbb{U}(j)$-representable, then $Z$ is also TC-representable.
\end{corollary}

\section{Which sets are $\UU$-representable?} \label{sec:whichsetsarerepresentable} 

The base theory for both weakly homogeneously Souslin-ness and $\UU(j)$-repre\-sentability is $\ZFC$, and for this reason 
we will assume \( \AC \) throughout this section and the next one, unless otherwise specified. 

In classical descriptive set theory, regularity properties like the Perfect Set Property tend to exhibit the following pattern:

\begin{enumerate-(i)}
 \item \label{regularitypattern-1}
All analytic sets enjoy the given regularity property (see~\cite[Theorem 12.2]{Kanamori}). 
 \item \label{regularitypattern-2}
If $\mathsf{V=L}$, then there is a low-level projective set without the regularity property (see~\cite[Corollary 13.10 and Theorem 13.12]{Kanamori}).
For example, if the regularity property is the Perfect Set Property, then in the constructible universe \( L \) we have a coanalytic counterexample, while if we consider e.g.\ the Baire Property or Lebesgue Measurability, then we have counterexamples within the pointclass $\boldsymbol{\Delta}^1_2$.
\item \label{regularitypattern-3}
If we assume sufficiently strong large cardinal assumptions, then all definable sets have the regularity property at hand. For example, 
if there are $\omega$-many Woodin cardinals and a measurable above, then $L(\RR)\models \AD$, and therefore all the subsets of \( \RR \) in $L(\RR)$, projective sets included, enjoy the regularity property (see e.g.\ \cite[Theorem 8.24]{Neeman2010}; the proof that determined sets enjoy all the main regularity properties can be found e.g.\ in \cite[Chapter 21]{Kechris1995}).
\end{enumerate-(i)}

One of the most interesting proof of~\ref{regularitypattern-3} factors through the notion of weakly homogeneously Souslin set briefly discussed in Section~\ref{sec:representabilityvsweaklyhomogeneouslysouslin}. Indeed, weakly homogeneously Souslin sets themselves follow a pattern very similar to the one above:

\begin{enumerate-(I)}
\item \label{whSpattern-1}
If there exists a measurable cardinal, then all analytic sets are weakly homogeneously Souslin; in fact, all $\boldsymbol{\Sigma}^1_2$ sets are (see e.g.\ \cite[Exercises 32.1 and 32.3]{Kanamori}).
\item \label{whSpattern-2}
If $\mathsf{V=L}$, no nontrivial set is weakly homogeneously Souslin: this is because in \( L \) there are no measurable cardinals, and hence all ultrafilters appearing in the definition of weakly homogeneusly Souslin-ness are necessarily principal.
\item \label{whSpattern-3}
If there are $\omega$-many Woodin cardinals and a measurable above, all sets in $L(\RR)$ are weakly homogeneously Souslin (see e.g.\ \cite[Theorem 32.13]{Kanamori}).
\end{enumerate-(I)}

Since every weakly homogeneously Souslin set enjoy all the main regularity properties (\cite[Theorem 32.7]{Kanamori}), we can think of weakly homogeneous Souslin-ness as a sort of ``super'' regularity property.
Note that the large cardinal hypothesis for~\ref{whSpattern-3} is the same implying \( L(\RR) \models \AD \) in~\ref{regularitypattern-3} above.

The goal of this section is to observe that, remarkably, we have a very similar situation when moving to the generalized context, as long as we consider the \( \lPSP \) as our  regularity property, and \( \UU \)-representable and TC-representable sets as higher analogues of weakly homogeneously Souslin sets (see Proposition~\ref{prop:representablevsweaklyhomogeneous} and the discussion following it). Indeed, we already proved that all $\lambda$-analytic sets have the $\lPSP$ (Theorem~\ref{thm:analyticPSP}), and that there are $\lambda$-coanalytics sets without the $\lPSP$ if \( \mathsf{V=L} \) (Corollary~\ref{cor:counterexampleinV=L});
these results are the exact counterparts of~\ref{regularitypattern-1} and~\ref{regularitypattern-2} above.
As for~\ref{regularitypattern-3}, there has been a lot of work aiming at showing that we can get analogues of it in the generalized context if we assume \( \mathsf{I0}(\lambda) \) and we replace the inner model \( L(\RR) \) with its higher analogue \( L(V_{\lambda+1}) \).
The first result of this kind is due to Woodin and appears in~\cite[Lemma 122]{Woodin2011}.

\begin{lemma}[\( \AC \)] \label{lem:WoodinsweakPSP} 
Let $j$ be a witness for $\mathsf{I0}(\lambda)$, and let $Z \subseteq V_{\lambda+1}$ be $\mathbb{U}(j)$-representable in $L(V_{\lambda+1})$. If $|Z|>\lambda$, then there is a continuous injective function from $\pre{\omega}2$ into $Z$.
\end{lemma}

The proof of Lemma~\ref{lem:WoodinsweakPSP} exploits the representability structure and the fact that the tower condition in $L(V_{\lambda+1})$ holds for every $\mathbb{V} \subseteq \UU(j)$ with $|\mathbb{V}|<\lambda$. Notice that the property described in the conclusion of the theorem is much weaker than the \lPSP{} discussed in this paper, yet it was the first indication that \( \UU(j) \)-representable sets have interesting regularity properties. Another attempt was made by Shi in~\cite[Theorem 4.1]{Shi2015}.

\begin{theorem}[\( \AC \)] 
Assume that $\mathsf{I0}(\lambda)$ holds. Let $Z \subseteq V_{\lambda+1} \) be a set in \( L(V_{\lambda+1})$ that is definable in $V_{\lambda+1}$ (possibly with parameters in \( V_{\lambda+1} \)). If $|Z|>\lambda$, then there is a continuous injective function from $C(\lambda)$ into $Z$.
\end{theorem}

In view of Theorem~\ref{thm:homeomorphictoCantor} and Corollary~\ref{cor:equivPSP}, Shi's theorem essentially states that all $\lambda$-projective sets satisfy the \( \lPSP \). The proof uses generic absoluteness for \( \mathsf{I0} \), a deep result about the structure of $L(V_{\lambda+1})\cap V_{\lambda+2}$, and the generic extensions of its iterations. Shi comments that the same proof can be used for a larger class of subsets of $V_{\lambda+1}$: if $\gamma$ is closed enough and every $A\in L_\gamma(V_{\lambda+1})$ is $\mathbb{U}(j)$-representable, then all the subsets of \( V_{\lambda+1} \) belonging to $L_\gamma(V_{\lambda+1})$ have the \( \lPSP \). Finally, the conclusive theorem is obtained by Cramer in~\cite[Theorem 5.1]{Cramer2015}.

\begin{theorem}[\( \AC \)]  \label{thm:allPSPCramer} 
Assume that $\mathsf{I0}(\lambda)$ holds. Let $Z \subseteq V_{\lambda+1} $ be any set in \( L(V_{\lambda+1}) \). If $|Z|>\lambda$, then there is a continuous injective function from $B(\lambda)$ into $Z$.
\end{theorem}

Using again Theorem~\ref{thm:homeomorphictoCantor} and Corollary~\ref{cor:equivPSP}, this gives the ultimate result in this direction and an exact analogue of~\ref{regularitypattern-3} in the generalized context (under different large cardinal assumptions, of course).
Notably, Theorem \ref{thm:allPSPCramer} has no \( \UU(j) \)-representability assumption on the set \( Z \): its proof heavily uses the elementary embedding \( j \) that witnesses $\mathsf{I0}(\lambda)$, and its manipulations via inverse limits. 

\begin{remark} \label{rem:PSPinandout}
When dealing with $\mathsf{I0}(\lambda)$, there are two models of reference, namely, the universe of sets $V$, in which we assume \( \ZFC \), and the inner model $L(V_{\lambda+1})$, which instead only satisfies \( \ZF + \DC_\lambda \). Therefore, in principle we should be careful when considering the $\lPSP$, and always specify in which of the two models we intend it to be considered. 
However, if we restrict the attention to subsets of, say,%
\footnote{Similar observations hold if we replace \( B(\lambda) \) with one of the spaces $C(\lambda)$, $\pre{\lambda}{2}$, or $V_{\lambda+1}$.}
 the generalized Baire space $B(\lambda)$, the situation is rather indifferent: if \( A \subseteq B(\lambda) \) belongs to \( V_{\lambda+1} \), then \( A \) has the \( \lPSP \) in \( V_{\lambda+1} \) if and only if it has the \( \lPSP \) in \( V \).

To see this, the first thing to notice is that 
every function \( \phi \colon \lambda \to \lambda \) belongs to \( V_{\lambda+1} \), as its elements are pair of ordinals in \( \lambda \times \lambda \), which thus belong to \( V_\lambda \). It readily follows that
\( B(\lambda)^{L(V_{\lambda+1})} = B(\lambda)^V\). Also, since Theorem~\ref{thm:homeomorphictoCantor} is true in both \( V \) and \( L(V_{\lambda+1} ) \), it is enough to consider the \( \lPSP^* \) instead of the \( \lPSP \) (see the discussion after Definition~\ref{def:PSP}). Now it is easy to check that the witnesses of the \( \lPSP \) (in whatever model we are considering it) can be canonically coded as functions from \(\lambda\) into itself. Indeed, if \( f \colon B(\lambda) \to  B(\lambda) \) is an embedding of \( B(\lambda) \) into \( A \), then it is coded by the function \( \pre{<\omega}{\lambda} \to \pre{< \omega}{\lambda} \) sending each \( s \in \pre{<\omega}{\lambda} \) into the longest sequence \( t \in \pre{< \omega}{\lambda} \) such that \( \lh(t) \leq \lh(s) \) and \( f(\Nbhd_s) \subseteq \Nbhd_t \), which in turn can be converted into a function \( \phi \colon \lambda \to \lambda \) by identifying \( \pre{< \omega}{\lambda} \) with \( \lambda \) in the usual way.
If instead \( |A| \leq \lambda \), then \( A \) can be identified with a sequence \( (x_\beta)_{\beta < \lambda} \in \pre{\lambda}{(B(\lambda))} = \pre{\lambda}{(\pre{\omega}{\lambda})} \), which can then be canonically turned into a function \( \phi \colon \lambda \to \lambda \) using a suitable bijection between \( \lambda \times \omega \) and \( \lambda \). Our claim easily follows from the fact that, in any case, \( \phi \in L(V_{\lambda+1}) \).

Combining this observation with Corollary~\ref{cor:transferPSP}, one sees that when \( X \) is a \(\lambda\)-Polish space belonging to \( L(V_{\lambda+1}) \) and \( \boldsymbol{\Gamma} \) is a boldface \( \lambda \)-pointclass closed under \(\lambda\)-Borel preimages and such that \( \boldsymbol{\Gamma}(X) \subseteq L(V_{\lambda+1}) \), we can unambiguously say that all sets in \( \boldsymbol{\Gamma}(X) \) have the \( \lPSP \), without having to specify whether the truth of such assertion must be computed in $L(V_{\lambda+1})$ or in $V$.
\end{remark}

The tools developed in this paper allow us to see that, in analogy with what happens in the classical case, also Theorem~\ref{thm:allPSPCramer} has a proof that factors through representability. Indeed, we already showed in Theorem~\ref{thm:mainPSP} that all TC-repre\-sentable sets have the \lPSP, so it becomes crucial to understand which sets are representable, in one sense or another. Notice that the first part of the next result does not need any strong choice assumption, and thus it works also in inner models like \( L(V_{\lambda+1}) \). 

\begin{proposition}\label{prop:analyticrepresent} 
Let \(\lambda\) be such that \( 2^{< \lambda} = \lambda \), and assume that there is a measurable cardinal \( \kappa > \lambda \). Let $X$ be any $\lambda$-Polish space. Then there is an orderly family of ultrafilters \( \UU \) such that all \( \lambda \)-analytic subsets of \( X \) are \( \UU \)-representable (with respect to any tree-basis for \( X \)). If moreover we assume \( \AC \), then every \( Z \in \lS^1_1(X) \) is TC-representable.
\end{proposition}

\begin{proof} 
The additional part concerning TC-representability follows from the first one and Proposition~\ref{prop:TCautomatic}.

The following construction is somewhat implicit in Martin's proof that the existence of a measurable cardinal implies that every coanalytic set is determined (see~\cite[Theorem 31.1]{Kanamori}). Let $\U$ be a normal ultrafilter on $\kappa$. Recall that for every set \( H \) and every \( k \in \omega \), the symbol \( [H]^k \) denotes the collection of all subsets of \( H \) of size \( k \). When \( H \subseteq  \kappa \), we also let 
\begin{align*}
\pre{\uparrow k }{H} & = \{ s \in \pre{k}{H} \mid \forall i < j < k \, (s(i) < s(k)) \}, \quad \text{and} \\
\pre{\downarrow k }{H} & = \{ s \in \pre{k}{H} \mid \forall i < j < k \, (s(i) > s(k)) \}
\end{align*}
be the sets of all sequences in \( \pre{k}{H} \) which are strinctly increasing and strictly decreasing, respectively. Clearly, \( [H]^k \) can be naturally identified with both \( \pre{\uparrow k }{H} \) and \( \pre{\downarrow k}{H} \).
By Rowbottom's Theorem (see~\cite[Theorem 7.17]{Kanamori}), both
\begin{align*}
\U_k & =\{A \subseteq \pre{k}{\kappa} \mid \pre{\uparrow k}{H} \subseteq A \text{ for some }  H \in \U \}, \quad \text{and} \\
\V_k & =\{A \subseteq \pre{k}{\kappa} \mid \pre{\downarrow k}{H} \subseteq A \text{ for some }  H \in \U \} 
\end{align*}
are $\kappa$-complete ultrafilters on $\pre{k}{\kappa}$, and it is easy to see that if $k> \ell$ then $\U_k$ projects onto $\U_\ell$ and, similarly, \( \V_k \) projects onto \( \V_\ell \). This shows that both \( (\U_k)_{k \in \omega} \) and \( (\V_k)_{k \in \omega} \) are towers of \( \kappa \)-complete ultrafilters, the difference being that the former is well-founded (because \( \U \) is \( \kappa \)-complete and \( \kappa > \omega \)), while the latter is not (because there are no decreasing infinite sequences of ordinals).

Recall from the paragraph preceding Proposition~\ref{prop:representablethesame} the ``concatenation'' operation \( (\U,\V) \mapsto \U {}^\smallfrown{}  \V \) for ultrafilters based on the tensor product, and the ``mixed'' towers of \( \kappa \)-complete ultrafilters \( (\mathcal{W}^{(n)}_k)_{k \in \omega} \) obtained from the two towers \( (\U_k)_{k \in \omega} \) and \( (\V_k)_{k \in \omega} \), which in our cases are always non-well-founded because so is \( (\V_k)_{k \in \omega} \).
Let \( \UU = \{ \mathcal{W}^{(n)}_k \mid n \leq k \} \), 
and notice that the support of the orderly family \( \UU \) is \( K = \kappa \).
We claim that every \( Z \in \lS^1_1(X) \) is \( \UU \)-representable with respect to any tree-basis \( \mathcal{B} \) for \( X \).

Let \( Z' = f_{\mathcal{B}}^{-1}(Z) \). Then \( Z' \in \lS^1_1(B(\lambda)) \), hence there is an \( \omega \)-tree \( T \) on \( \lambda \times \lambda \) such that  \( Z' = \p[T] \).
For each \( k \in \omega \) and \( s,t \in \pre{k}{\lambda} \), let  \( \pi(s,t) = \mathcal{W}^{(n)}_k \), where \( n \leq k \) is largest such that \( (s \restriction n,t \restriction n) \in T \). We claim that \( \pi \) is a \( \UU \)-representation for \( Z \).

Consider any \( \tilde{x} \in B(\lambda) \). If \( f_{\mathcal{B}}(\tilde{x}) \in Z \), then \( \tilde{x} \in Z' \) and hence there is \( y \in B(\lambda) \) such that \( (\tilde{x},y) \in [T] \). Then \( T^\pi_{\tilde{x},y} = (\U_k)_{k \in \omega} \) is well-founded. 
If instead \( f_{\mathcal{B}}(\tilde{x}) \notin Z \), then \( \tilde{x} \notin Z' \), hence for every \( y \in B(\lambda) \) we must have \( (\tilde{x},y) \notin [T] \). Therefore \( T^\pi_{\tilde{x},y} = (\mathcal{W}^{(n)}_k)_{k \in \omega} \), where \( n \in \omega \) is largest such that \( ( \tilde{x} \restriction n, y \restriction n) \in T \), and so \( T^\pi_{\tilde{x},y} \) is not well-founded. 
\end{proof}


Next we show that, under the same hypotheses, representable sets give rise to a boldface \(\lambda\)-pointclass whose closure properties are only slightly weaker than those of \( \lS^1_1 \) (see Proposition~\ref{prop:closurepropertiesofanalytic}).
For comparison, we mention that in~\cite[Lemmas 114 and 115]{Woodin2011} it has been shown that some of these closure properties are enjoyed also by the collection of \( \UU(j) \)-representable subsets of \( V_{\lambda+1} \).

\begin{proposition} \label{prop:representpointclass}
Let \( \lambda \) be such that \( 2^{< \lambda} = \lambda \), and assume that there is a measurable cardinal \( \kappa > \lambda \). For any \(\lambda\)-Polish space \( X \), let \( \boldsymbol{\Gamma}_{\mathrm{repr}}(X) \) be the collection of all \( Z \subseteq X \) which are \( \UU \)-representable (with respect to \emph{some} tree-basis \( \mathcal{B} \) for \( X \)), where \( \UU \) varies over all orderly families of ultrafilters. Then: 
\begin{enumerate-(i)}
\item \label{prop:representpointclass-i}
\( \boldsymbol{\Gamma}_{\mathrm{repr}} \) is a boldface \( \lambda \)-poinclass containing \( \lS^1_1 \), in the sense that \( \lS^1_1(X) \subseteq \boldsymbol{\Gamma}_{\mathrm{repr}}(X) \) for every \(\lambda\)-Polish space \( X \).
\item \label{prop:representpointclass-ii}
\( \boldsymbol{\Gamma}_{\mathrm{repr}} \) is closed under projections, well-ordered unions of length at most \( \lambda \), and finite intersections.
\item \label{prop:representpointclass-iii}
\( \boldsymbol{\Gamma}_{\mathrm{repr}} \) is closed under \(\lambda\)-Borel images and preimages, in the following strong sense.
Let \( f \colon X \to Y \) be a \( \lambda \)-Borel function between the \( \lambda \)-Polish spaces \( X \) and \( Y \). Suppose that \( \mathcal{B}_X \) and \( \mathcal{B}_Y \) are tree-bases for \( X \) and \( Y \), respectively. 
\begin{itemizenew}
\item
If \( Z \subseteq X \) is \( \UU \)-representable with respect to \( \mathcal{B}_X \), then \( f(Z) \) is \( \mathbb{V} \)-respresentable with respect to \( \mathcal{B}_Y \), for some orderly family of ultrafilters \( \mathbb{V} \). 
\item
Conversely, if \( W \subseteq Y \) is \( \mathbb{V} \)-representable with respect to \( \mathcal{B}_Y \), then \( f^{-1}(W) \) is \( \UU \)-representable with respect to \( \mathcal{B}_Y \), for some orderly family of ultrafilters \( \UU \).
\end{itemizenew}
\end{enumerate-(i)}

If moreover we assume \( \AC \),%
\footnote{Notice that \( \AC \) is needed only to show the inclusion \( \lS^1_1(X) \subseteq \boldsymbol{\Gamma}_{\mathrm{repr}}(X) \) (which follows from Proposition~\ref{prop:analyticrepresent}), which in turn is used to show that \( \boldsymbol{\Gamma}_{\mathrm{TC}} \) is closed under \(\lambda\)-Borel images and preimages. However, our proof shows that even in a choiceless setting, \( \boldsymbol{\Gamma}_{\mathrm{TC}} \) is closed under projections, unions of size \( \lambda \), and finite intersections. Moreover, a different argument allows one to check that \( \boldsymbol{\Gamma}_{\mathrm{TC}} \) is at least closed under continuous preimages, i.e.\ it is a boldface \(\lambda\)-pointclass. } 
then the same holds for the pointclass \( \boldsymbol{\Gamma}_{\mathrm{TC}} \) obtained by letting \( \boldsymbol{\Gamma}_{\mathrm{TC}}(X) \) be the collection of all TC-representable sets \( Z \subseteq X \).
\end{proposition}

In particular, by Proposition~\ref{prop:representpointclass}\ref{prop:representpointclass-iii} the property of being representable is independent of the chosen tree-basis on the space.

\begin{proof}
The G\"odel pairing function \( \langle \cdot , \cdot \rangle \colon \lambda \times \lambda \to \lambda \) canonical induces bijections \( \pre{\omega}{\lambda} \times \pre{\omega}{\lambda} \to \pre{\omega}{\lambda} \) and \( \pre{k}{\lambda} \times \pre{k}{\lambda} \to \pre{k}{\lambda} \), for every \( k \in \omega \), which will still be denoted by \( \langle \cdot , \cdot \rangle \); their inverses will be denoted by \( ( \cdot )_0 \) and \( ( \cdot )_1 \), respectively, so that \( \langle (x)_0, (x)_1 \rangle = x \) for all elements \( x \) of interest.

We first prove that \( \boldsymbol{\Gamma}_{\mathrm{repr}} \) is closed under projections. Given two tree-bases \( \mathcal{B}_X = \{ U_s \mid s \in \pre{< \omega}{\lambda} \} \) and \( \mathcal{B}_Y = \{ V_s \mid s \in \pre{<\omega}{\lambda} \} \)  for the \(\lambda\)-Polish spaces \( X \) and \( Y \), respectively, we let \( \mathcal{B}_{X \times Y} = \{ W_s \mid s \in \pre{<\omega}{\lambda} \} \) be the tree-basis%
\footnote{To have that \( \mathcal{B}_{X \times Y} \) satisfies condition~\ref{def:tree-basis-4} of Definition~\ref{def:tree-basis}, the product \( X \times Y \) has to be equipped with the sup metric.}
 for their product \( X \times Y \) canonically induced by \( \mathcal{B}_X \) and \( \mathcal{B}_Y \), namely, \( W_s = U_{(s)_0} \times V_{(s)_1} \) for every \( s \in \pre{<\omega}{\lambda} \).
Notice that for every \( \tilde{x} \in B(\lambda) \), we have \( f_{\mathcal{B}_{X \times Y}}(\tilde{x}) = ( f_{\mathcal{B}_X((\tilde{x})_0)} , f_{\mathcal{B}_Y}((\tilde{x})_1) ) \).

\begin{claim} \label{claim:representpointclass1}
If \( C \subseteq X \times Y \) is \( \UU \)-representable with respect to \( \mathcal{B}_{X \times Y} \), then \( \p(C) \) is \( \UU \)-representable with respect to \( \mathcal{B}_X \) and, similarly, the projection \( \p_2(C) \) of \( C \) onto \( Y \) is \( \UU \)-representable with respect to \( \mathcal{B}_Y \).
\end{claim}

\begin{proof}[Proof of the claim]
Let \( \pi' \) be a \( \UU \)-representation for \( C \) with respect to \( \mathcal{B}_{X \times Y} \). Define \( \pi \colon \bigcup_{k \in \omega} (\pre{k}{\lambda} \times \pre{k}{\lambda}) \to \UU \) by setting
\[ 
\pi(s,t) = \pi'(\langle s , (t)_0 \rangle , (t)_1).
 \] 
Then \( \ran(\pi) = \ran(\pi') = \UU \), and
since for every \( \tilde{x},y \in B(\lambda) \) we have \( T^\pi_{\tilde{x},y} = T^{\pi'}_{\langle \tilde{x},(y)_0 \rangle, (y)_1} \) and \( T^{\pi'}_{\tilde{x},y} = T^\pi_{(\tilde{x})_0, \langle (\tilde{x})_1 , y \rangle} \), one can easily check that \( \pi \) is a \( \UU \)-representation for \( \p(C) \).
The case of \( \p_2(C) \) is similar, the only difference being that we have to set \( \pi(s,t) = \pi'(\langle (t)_0, s   \rangle , (t)_1) \).
\end{proof}

Conversely, we have the following result.

\begin{claim} \label{claim:representpointclass2}
If \( C \subseteq X \) is \( \UU \)-representable with respect to \( \mathcal{B}_X \), then \( C \times Y \) is \( \UU \)-representable with respect to \( \mathcal{B}_{X \times Y} \), where \( \mathcal{B}_Y \) is any tree-basis for \( Y \).
\end{claim}

\begin{proof}[Proof of the claim]
Let \( \pi' \) be a \( \UU \)-representation for \( C \) with respect to \( \mathcal{B}_X \). 
Let \( \pi \colon \bigcup_{k \in \omega} (\pre{k}{\lambda} \times \pre{k}{\lambda}) \to \UU \) be defined by
\[
\pi(s,t) = \pi'((s)_0,t).
\]
Then \( \ran(\pi) = \ran(\pi') = \UU \), and \( \pi \) is a \( \UU \)-representation for \( C \times Y \) with respect to \( \mathcal{B}_{X \times Y} \).
\end{proof}

Of course, by a similar argument we have that if \( D \subseteq Y \) is \( \UU \)-representable with respect to \( \mathcal{B}_Y \), then \( X \times D \) is \( \UU \)-representable with respect to \( \mathcal{B}_{X \times Y} \), where \( \mathcal{B}_X \) is any tree-basis for \( X \).

Next we show that each \( \boldsymbol{\Gamma}_{\mathrm{repr}} \) is closed under finite intersections.
Given two orderly family of ultrafilters \( \UU \) and \( \mathbb{V} \) with support \( K_0 \) and \( K_1 \), respectively, we can consider their tensor product
\[
\UU \otimes \mathbb{V} = \{ \U \otimes \V \mid {\U \in \UU} \wedge {\V \in \mathbb{V}} \wedge {\lev(\U) = \lev(\V)} \}.
\]
As written, \( \UU \otimes \mathbb{V} \) is not an orderly family of ultrafilters, as if \( \lev(\U) = \lev(\V) = k \) then \( \U \otimes \V \) is an ultrafilter on \( \pre{k}{K_0} \times \pre{k}{K_1} \). However, such an ultrafilter can be straightforwardly identified with an ultrafilter on \( \pre{k}{(K_0 \times K_1)} \), still denoted by \( \U \otimes \V \): under this identification, \( \UU \otimes \mathbb{V} \) becomes an orderly family of ultrafilters with support \( K_0 \times K_1 \), and \( \lev(\U \otimes \V) = \lev(\U) = \lev(\V) \).

In the proof of the following claim we will tacitly use the above identification, and freely switch from one presentation (and notation) to the other one when convenient.
We will also need the following general facts concerning tensor products.
First of all, for any cardinal \( \kappa \) we have that \( \U \otimes \V \) is  \( \kappa \)-complete if and only if so are \( \U \) and \( \V \).
Moreover, \( \U' \otimes \V' \) projects onto \( \U \otimes \V \) whenever \( \U' \) projects onto \( \U \) and \( \V' \) projects onto \( \V \). In particular, if \( (\U_k)_{k \in \omega} \) is a tower of ultrafilters in \( \UU \) and \( (\V_k)_{k \in \omega} \) is a tower of ultrafilters in \( \mathbb{V} \), then \( (\U_k \otimes \V_k)_{k \in \omega } \) is a tower of ultrafilters in \( \UU \otimes \mathbb{V} \).  Finally, the tower \( (\U_k \otimes \V_k)_{k \in \omega} \) is well-founded if and only if so are \( (\U_k)_{k \in \omega} \) and \( (\V_k)_{k \in \omega} \).
To see this, suppose that both \( (\U_k)_{k \in \omega} \) and \( (\V_k)_{k \in \omega} \) are well-founded, and
pick any sequence of sets \( (A_k)_{k \in \omega} \) with \( A_k \in \U_k \otimes \V_k \). For each \( k \in \omega \), let \( B_k^{(0)} =  \{ b \in \pre{k}{K_0} \mid (A_k)_b \in \V_k \} \), so that \( B^{(0)}_k \in \U_k \) by choice of \( A_k \). Since \( (\U_k)_{k \in \omega} \) is well-founded, there is \( z_0 \in \pre{\omega}{K_0} \) such that \( z_0 \restriction k \in B^{(0)}_k \) for every \( k \in \omega \). 
Let  \( B^{(1)}_k = (A_k)_{z_0 \restriction k} \), so that \( B^{(1)}_k \in \V_k \). Since \( (\V_k)_{k \in \omega} \) is well-founded, there is \( z_1 \in \pre{\omega}{K_1} \) such that \( z_1 \restriction k \in B^{(1)}_k \) for every \( k \in \omega \).
Let \( z = (z_0(k),z_1(k))_{k \in \omega} \in \pre{\omega}{(K_0 \times K_1)} \). 
Then \( z \restriction k \in A_k \) for every \( k \in \omega \), and since \( (A_k)_{k \in \omega} \) was arbitrary, this shows that \( (\U_k \otimes \V_k)_{k \in \omega} \) is well-founded too. Conversely, pick any \( (A^{(0)}_k)_{k \in \omega} \) and \( (A^{(1)}_k)_{k \in \omega} \) such that \( A^{(0)}_k \in \U_k \) and \( A^{(1)}_k \in \V_k \) for every \( k \in \omega \). Then \( A_k = A^{(0)}_k \times A^{(1)}_k \in \U_k \otimes \V_k \). 
If \( (\U_k \otimes \V_k)_{k \in \omega} \) is well-founded, there is \( z \in \pre{\omega}{(K_0 \times K_1)} \) such that \( z \restriction k \in A_k \) for every \( k \in \omega \). Let \( z_0 \in \pre{\omega}{K_0 } \) and \( z_1 \in \pre{\omega}{K_1} \) be such that \( z(k) = (z_0(k),z_1(k)) \) for every \( k \in \omega \). Then \( z_0 \restriction k \in A^{(0)}_k \) and \( z_1 \restriction k \in A^{(1)}_k \) by choice of \( A_k \). Since \( (A^{(0)}_k)_{k \in \omega} \) and \( (A^{(1)}_k)_{k \in \omega} \) were arbitrary, this shows that both  \( (\U_k)_{k \in \omega} \) and \( (\V_k)_{k \in \omega} \) are well-founded.

\begin{claim} \label{claim:representpointclass3}
If \( C \subseteq X \) is \( \UU \)-representable with respect to some tree-basis \( \mathcal{B} \) for \( X \) and \( D \subseteq X \) is \( \mathbb{V} \)-representable with respect to the same \( \mathcal{B} \), then \( C \cap D \) is \( \UU \otimes \mathbb{V} \)-representable with respect to \( \mathcal{B} \).
\end{claim}

\begin{proof}[Proof of the claim] 

Let \( K_0 \) and \( K_1 \) be the supports of \( \UU \) and \( \mathbb{V} \), respectively.
Given a \( \UU \)-representation \( \pi_C \) for \( C \) with \( \ran(\pi_C) = \UU \) and a \( \mathbb{V} \)-representation \( \pi_D \) for \( D \) with \( \ran(\pi_D) = \mathbb{V} \), let \( \pi \colon \bigcup_{k \in \omega} (\pre{k}{\lambda} \times \pre{k}{\lambda}) \to \UU \otimes \mathbb{V} \) be defined by
\begin{equation} \label{eq:representpointclass3}
\pi(s,t) = \pi_C(s,(t)_0) \otimes \pi_D(s,(t)_1).
\end{equation}
Then \( \ran(\pi) \subseteq \UU \otimes \mathbb{V} \), and we claim that \( \pi \) is a \( \UU \otimes \mathbb{V} \)-representation for \( C \cap D \).
Fix any \( \tilde{x},y \in B(\lambda) \). Since 
\begin{equation} \label{eq:tensorproductoftowers}
T^\pi_{\tilde{x},y} = (\pi_C(\tilde{x} \restriction k,(y)_0 \restriction k) \otimes \pi_D(\tilde{x} \restriction k,(y)_1 \restriction k))_{k \in \omega} , 
\end{equation}
we get that \( T^\pi_{\tilde{x},y} \) is a tower of ultrafilters because so are \( T^{\pi_C}_{\tilde{x},(y)_0} \) and \( T^{\pi_D}_{\tilde{x},(y)_1} \). 
Moreover, each \( \pi( \tilde{x} \restriction k, y \restriction k) \) is at least \( \lambda_k^+ \)-complete because so are \( \pi_C( \tilde{x} \restriction k, y \restriction k) \) and \( \pi_D( \tilde{x} \restriction k, y \restriction k) \).

Consider any \( \tilde{x} \in B(\lambda) \). If \( f_{\mathcal{B}}(\tilde{x}) \in C \cap D \), then there are \( y_0,y_1 \in B(\lambda) \) such that both \( T^{\pi_C}_{\tilde{x},y_0} \) and \( T^{\pi_D}_{\tilde{x},y_1} \) are well-founded.
It follows that there is also \( y \in B(\lambda) \) such that \( T^{\pi}_{\tilde{x}, y} \) is well-founded: by~\eqref{eq:tensorproductoftowers}, it is enough to set \( y = \langle y_0,y_1 \rangle \).
Conversely, suppose that \( T^\pi_{\tilde{x},y} \) is well-founded for some \( y \in B(\lambda) \).
Then \( T^{\pi_C}_{\tilde{x},(y)_0} \) and \( T^{\pi_D}_{\tilde{x},(y)_1} \) are well-founded (because of~\eqref{eq:tensorproductoftowers} again), and hence \( (y)_0 \) witnesses \( f_{\mathcal{B}}(\tilde{x}) \in C \), while \( (y)_1 \) witness \( f_{\mathcal{B}}(\tilde{x}) \in D \). Therefore \( f_{\mathcal{B}}(\tilde{x}) \in C \cap D \), and the proof is complete.
\end{proof}

We are now ready to show that \( \boldsymbol{\Gamma}_{\mathrm{repr}} \) is closed under \(\lambda\)-Borel images and preimages, in the sense of~\ref{prop:representpointclass-iii}. First suppose that \( Z \subseteq X \) is \( \UU \)-representable with respect to a given tree-basis \( \mathcal{B}_X \). Let \( \mathcal{B}_Y \) be a tree-basis for the \(\lambda\)-Polish space \( Y \), and let \( f \colon X \to Y \) be a \( \lambda \)-Borel function, so that its graph \( D = \mathrm{graph}(f) \) is \(\lambda\)-Borel in \( X \times Y \) by Proposition~\ref{prop:borelvsgraph}.
By Claim~\ref{claim:representpointclass2}, the set \( C = Z \times Y \) is \( \UU \)-representable with respect to \( \mathcal{B}_{X \times Y} \), while by Proposition~\ref{prop:analyticrepresent} there is an orderly family of ultrafilters \( \UU' \) such that \( D \) is \( \UU' \)-representable with respect to \( \mathcal{B}_{X,Y} \). 
Then by Claim~\ref{claim:representpointclass3} the set \( C \cap D \) is \( \mathbb{V} \)-representable with respect to \( \mathcal{B}_{X \times Y} \), where \( \mathbb{V} = \UU \otimes \UU' \). Then \( f(Z) = \p_2(C \cap D) \) is \( \mathbb{V} \)-representable with respect to \( \mathcal{B}_Y \) by Claim~\ref{claim:representpointclass1}. The case of preimages under \( f \) is similar. If \( W \subseteq Y \) is \( \mathbb{V} \) representable with respect to \( \mathcal{B}_Y \), then \( (X \times W) \cap D \) is \( \UU \)-representable with respect to \( \mathcal{B}_{X \times Y} \) for \( \UU = \mathbb{V} \otimes \UU' \), and hence \( f^{-1}(W) = \p ((X \times W) \cap D) \) is \( \UU \)-representable with respect to \( \mathcal{B}_X \).

Closure under \(\lambda\)-Borel preimages entails that \( \boldsymbol{\Gamma}_{\mathrm{repr}} \) is a boldface \( \lambda \)-poinclass. Together with the fact that \( \boldsymbol{\Gamma}_{\mathrm{repr}}(X) \supseteq \lS^1_1(X) \) by Proposition~\ref{prop:analyticrepresent}, this proves~\ref{prop:representpointclass-i}. 

As for~\ref{prop:representpointclass-ii}, we already proved closure under projections and finite intersections in Claim~\ref{claim:representpointclass1} and Claim~\ref{claim:representpointclass3}, respectively, so we only need to show that \( \boldsymbol{\Gamma}_{\mathrm{repr}}(X) \) is closed under well-ordered unions of size \( \lambda \). In what follows, all representations refer to a tree-basis \( \mathcal{B} \) for \( X \) fixed in advance.. 
Suppose that for every \( \alpha < \lambda \) we are given a  set \( Z_\alpha \subseteq X \) and a \( \UU_\alpha \)-representation \( \pi_\alpha \) for \( Z_\alpha \), where \( \UU_\alpha = \ran(\pi_\alpha ) \) is an orderly family of ultrafilters with support \( K_\alpha \). Let \( K = \bigcup_{\alpha < \lambda} K_\alpha \). 
For each \( \alpha < \lambda \), let \( \UU_\alpha^+ = \{ \U^+ \mid \U \in \UU_\alpha \} \), where if \( \lev(\U) = k \) we let \( \U^+  = \{ A \subseteq \pre{k}{K} \mid A \cap \pre{k}{K_\alpha} \in \U \} \). Then \( \UU = \bigcup_{\alpha < \lambda} \UU^+_\alpha \) is an orderly family of ultrafilters with support \( K \).
For each \( \alpha < \lambda \), the map \( \pi_\alpha^+ \colon \bigcup_{k \in \omega} (\pre{k}{\lambda} \times \pre{k}{\lambda}) \to \UU^+_\alpha \) defined by \( \pi_\alpha^+(s,t) = (\pi_\alpha(s,t))^+ \) is a \( \UU^+_\alpha \)-representation  for \( Z_\alpha \) with \( \ran(\pi^+_\alpha) = \UU^+_\alpha \). Define \( \pi \colon \bigcup_{k \in \omega} (\pre{k}{\lambda} \times \pre{k}{\lambda}) \to \UU \) by letting \( \pi(\emptyset,\emptyset) \) be the only proper ultrafilter on \( \pre{0}{K} = \{ \emptyset \} \), and
\begin{equation} \label{eq:representpointclass-union}
\pi(s,t) = \pi^+_{(t(0))_0}(s,(t)_1)
\end{equation}
for every \( s,t \in \pre{k}{\lambda} \) with \( k > 0 \).
Then \( \ran(\pi) = \UU \), and
we claim that \( \pi \) is a \( \UU \)-representation for \( Z = \bigcup_{\alpha < \lambda} Z_\alpha \). 

Fix any \( \tilde{x} \in B(\lambda) \). 
Suppose first that \( f_{\mathcal{B}}(\tilde{x}) \in Z_\alpha \) for some \( \alpha < \lambda \). Let \( y \in B(\lambda) \) be such that \( T^{\pi^+_\alpha}_{\tilde{x},y} \) is well-founded, and let \( y' \in B(\lambda) \) be such that \( (y')_1 = y \) and \( (y'(0))_0 = \alpha \). Then \( T^{\pi}_{\tilde{x},y'} = T^{\pi^+_\alpha}_{\tilde{x},y} \), and thus \( T^{\pi}_{\tilde{x},y'} \) is well-founded too.
Conversely, suppose that there is \( y \in B(\lambda) \) such that \( T^\pi_{\tilde{x},y} \) is well-founded, and let \( \alpha = (y(0))_0 \). Since \( T^{\pi^+_\alpha}_{\tilde{x},(y)_1} = T^\pi_{\tilde{x},y} \), the tower \( T^{\pi^+_\alpha}_{\tilde{x},(y)_1} \) is well-founded as well, so \( f_{\mathcal{B}}(\tilde{x}) \in Z_\alpha \subseteq Z \). 

This concludes the proof for what concerns \( \boldsymbol{\Gamma}_{\mathrm{repr}} \). The case of \( \boldsymbol{\Gamma}_{\mathrm{TC}} \) can be treated similarly, the only difference being that we have to additionally check that the tower condition is preserved along the various transformations of the orderly families of ultrafilters involved. This is trivial for Claim~\ref{claim:representpointclass1} and Claim~\ref{claim:representpointclass2}, as the family remain the same. 
Moving to Claim~\ref{claim:representpointclass3},  let \( F_0 \) and \( F_1 \) on \( \UU \) and \( \mathbb{V} \) witnessing that \( \pi_C \) and \( \pi_D \) have the tower condition, and let \( \pi \) be obtained from \( \pi_C \) and \( \pi_D \) as in~\eqref{eq:representpointclass3}. Define \( F \) on \( \UU \otimes \mathbb{V} \) setting \( F(\U \otimes \V ) = F_0(\U) \times F_1(\V) \): we claim that the restriction of \( F \) to \( \ran(\pi) \) witnesses that \( \pi \) has the tower condition. 
Suppose to have a tower \( T^\pi_{\tilde{x},y} \) such that there is \(  z \in \pre{\omega}{(K_0 \times K_1) } \) satisfying \( z \restriction k \in F(\pi(\tilde{x} \restriction k, y \restriction k)) \) for every \( k \in \omega \). Let \( z_0 \in \pre{\omega}{K_0} \) and \( z_1 \in \pre{\omega}{K_1} \) be such that \( z(k) = (z_0(k),z_1(k)) \) for each \( k \in \omega \). Then \( z_0 \restriction k \in F_0(\pi_C(\tilde{x} \restriction k, (y)_0 \restriction k)) \) and \( z_1 \restriction k \in F_1(\pi_C(\tilde{x} \restriction k, (y)_1 \restriction k)) \) for all \( k \in \omega \), hence \( T^{\pi_C}_{\tilde{x}, (y)_0} \) and \( T^{\pi_D}_{\tilde{x},(y)_1} \) are well-founded, and thus so is \( T^\pi_{\tilde{x},y} \) by~\eqref{eq:tensorproductoftowers}. 

The last case that requires attention is closure under unions of size \(\lambda\). Suppose that for each \( \alpha < \lambda \) there is a function \( F_\alpha \) on \( \UU_\alpha \) witnessing that \( \pi_\alpha \) has the tower condition. Since \( \U \subseteq \U^+ \) for every \( \U \in \UU_\alpha \), we can lift \( F_\alpha \) to a function \( F^+_\alpha \) on \( \UU^+_\alpha \) by setting \( F^+_\alpha(\U^+) = F_\alpha(\U) \). The map \( F^+_\alpha \) is well-defined because the operation \( \U \mapsto \U^+ \) is injective on \( \UU_\alpha \), and it witnesses that \( \pi^+_\alpha \) has the tower condition, too. 
Moreover, without loss of generality we can suppose that the supports \( K_\alpha \) were pairwise disjoint, so that the families \( \UU^+_\alpha \) are pairwise disjoint as well. This means that the map \( F = \bigcup_{\alpha < \lambda} F_\alpha \) on \( \UU \) is well-defined, and it is easy to check that it witnesses that the representation \( \pi \) from~\eqref{eq:representpointclass-union} has the tower condition. 
\end{proof}

What is missing in Proposition~\ref{prop:representpointclass} is the closure under complements of \( \boldsymbol{\Gamma}_{\mathrm{repr}}(X) \) and \( \boldsymbol{\Gamma}_{\mathrm{TC}}(X) \): this is unavoidable because such closure property might fail, even in presence of a measurable cardinal \( \kappa > \lambda \). 
In fact, one measurable cardinal is enough to prove that the $\lambda$-analytic sets have the $\lPSP$ (Proposition~\ref{prop:analyticrepresent}), but work in~\cite{BDM} built on the results of Section~\ref{sec:noPSPcoanalytic} shows that the consistency strength of the statement ``all the $\lambda$-coanalytic sets have the $\lPSP$'' is much higher. Therefore, in a model of $\ZFC$ plus the existence of a measurable cardinal above $\lambda$ plus the negation of the consistency of larger cardinals, all the $\lambda$-analytic sets have the $\lPSP$ and there is a $\lambda$-coanalytic set (i.e., a complement of a $\lambda$-analytic set) without the $\lPSP$, so that it cannot be \( \UU \)-representable or TC-representable by Proposition~\ref{prop:TCautomatic} and Theorem~\ref{thm:mainPSP}.  

Proposition~\ref{prop:analyticrepresent} is the exact counterpart in the generalized setting of item~\ref{whSpattern-1} in the discussion on which subsets of \( \pre{\omega}{\omega} \) are weakly homogeneously Souslin at the beginning of this section, and if we assume \( \AC \) it implies Theorem~\ref{thm:analyticPSP} (under stronger hypotheses) via Theorem~\ref{thm:mainPSP}. The analogue of item~\ref{whSpattern-2} holds as well, and for the same reason: by Proposition~\ref{prop:U-representabilityandmeasurablecardinals} and Remark~\ref{rmk:principalrepresentations}, if there is no measurable cardinal above \(\lambda\), and in particular if \( \mathsf{V=L} \), then the only \( \UU \)-representable subset of a \(\lambda\)-Polish space \( X \) is the whole space itself. 

As for the generalized version of item~\ref{whSpattern-3}, the natural move is to see what happens under \( \mathsf{I0} \). In Corollary~\ref{cor:fromTowertoTC} we observed that if we are concerned with subset of \( V_{\lambda+1} \) living in \( L(V_{\lambda+1}) \), then TC-representability is implied by \( \UU(j) \)-representability, so it becomes crucial to understand when a set as above is \( \UU(j) \)-representable.
The following result come from~\cite[Corollary 6.11]{Cramer2015}.

\begin{theorem}[\( \AC \)] \label{thm:cramerrepresentability1} 
Let $j$ be a witness for $\mathsf{I0}(\lambda)$, let $\kappa=\lambda^+$, and let $\eta=\sup\{(\kappa^+)^{L[A]} \mid A\subseteq\lambda\}$. Then every subset of \( V_{\lambda+1} \) belonging to $L_\eta(V_{\lambda+1})$ is $\mathbb{U}(j)$-representable in $L(V_{\lambda+1})$.
\end{theorem}

Since all \( \lambda \)-projective subsets of \( V_{\lambda+1} \) already appear in \( L_1(V_{\lambda+1}) \), combining the above result with Corollary~\ref{cor:fromTowertoTC}, Theorem~\ref{thm:mainPSP}, and Corollary~\ref{cor:transferPSP} (taking into account also Remark~\ref{rem:PSPinandout}), we obtain:

\begin{corollary}[\( \AC \)] \label{cor:PSPallPolishSpaces} 
Assume $\mathsf{I0}(\lambda)$. Then all $\lambda$-projective sets have the \( \lPSP \), that is, \( \lPSP(\lS^1_n) \) holds for every natural number \( n \geq 1 \). 
\end{corollary}

Note that the same result holds for any boldface $\lambda$-pointclass $\boldsymbol{\Gamma}$ closed under $\lambda$-Borel preimages such that $\boldsymbol{\Gamma}(V_{\lambda+1}) \subseteq L_\eta(V_{\lambda+1})$. 

Theorem~\ref{thm:cramerrepresentability1} has been extended by Cramer himself to encompass all sets in \( L(V_{\lambda+1}) \). The following statement appears in~\cite[Theorem 84]{Cramer2017}; unfortunately, although its proof has been widely circulated among experts, it has not been published yet, and so we cannot give a better reference.

\begin{theorem}[\( \AC \)] \label{thm:allrepresentable} 
 Let $j$ be a witness for $\mathsf{I0}(\lambda)$, and work in \( L(V_{\lambda+1}) \). Then every $Z \subseteq V_{\lambda+1}$ is $\mathbb{U}(j)$-representable.
\end{theorem} 

Together with Corollary~\ref{cor:fromTowertoTC}, this provides the desired analogue of~\ref{whSpattern-3}: If \( \mathsf{I0}(\lambda) \) holds, then all subsets of \( V_{\lambda+1} \) in \( L(V_{\lambda+1}) \) are TC-representable (in \( L(V_{\lambda+1}) \)). Combining this with our results linking TC-representability to the \lPSP{} (and taking into account Remark~\ref{rem:PSPinandout}), we obtain another proof of Cramer's Theorem~\ref{thm:allPSPCramer}, which can now be stated in full generality.

\begin{corollary}[\( \AC \)] \label{cor:allsetshavelpSP}  
Assume that $\mathsf{I0}(\lambda)$ holds. Let \( X \) be a \(\lambda\)-Polish space in \( L(V_{\lambda+1}) \).%
\footnote{We can e.g.\ let \( X \) be one of the spaces \( B(\lambda) \), \( C(\lambda) \), \( \pre{\lambda}{2} \), or \( V_{\lambda+1} \), among the others.}
 Then all \( Z \subseteq X \) belonging to \( L(V_{\lambda+1}) \) have the \( \lPSP \).
\end{corollary}

In particular, combining Corollary~\ref{cor:allsetshavelpSP} with Corollary~\ref{cor:transferPSP} we obtain:

\begin{corollary}[\( \AC \)] \label{cor:abitmorethanprojectivePSP}  
Assume that $\mathsf{I0}(\lambda)$ holds. If $\boldsymbol{\Gamma}$ is a boldface $\lambda$-pointclass closed under $\lambda$-Borel preimages and such that $\boldsymbol{\Gamma}(V_{\lambda+1}) \subseteq L(V_{\lambda+1})$, then $\lPSP(\boldsymbol{\Gamma})$.
\end{corollary}


The added value of this alternative proof is that it comes directly from the representability structure, and not from the embedding structure (as it was the case in Cramer's original proof). This actually reinforces the analogy between $\mathsf{I0}$ and (large cardinals implying) $\AD$, as it shows that not only there are parallels between single concepts, but also in the way they interact among each other. Moreover, it opens the possibility of finding new analogues structures, different than $\AD$ and $\mathsf{I0}(\lambda)$, by looking for $\lambda$-Polish spaces on which many sets are $\UU$-representable.

\section{Further results} \label{sec:furtherresults}






Recall that we are currently working in \( \ZFC \).
By Corollary~\ref{cor:allsetshavelpSP} and Remark~\ref{rem:PSPinandout}, under \( \mathsf{I0}(\lambda) \) we have that \( L(V_{\lambda+1}) \) satisfies
\begin{center}
``All subsets of an arbitrary \(\lambda\)-Polish space have the \( \lPSP \).''
\end{center}
One may ask whether $L(V_{\lambda+1})$ is the only model of set theory in which this happens. More generally, one can ask which (generalized) descriptive set theoretic consequences of $\mathsf{I0}(\lambda)$ can hold also in models where $\mathsf{I0}(\lambda)$ fails. It turns out that any property about $V_{\lambda+1}$ that can be expressed in a simple way can be reflected downwards, and therefore it is true also in a model in which $\mathsf{I0}$ does not hold. The key theorem is the following deep result by Woodin, which is part of~\cite[Theorem 137]{Woodin2011}.
\footnote{To recover Theorem~\ref{thm:DownwardReflectionI0} from \cite[Theorem 137]{Woodin2011}, set $A_0=\lambda_0$, so that $A=\lambda$: with this notation, Theorem~\ref{thm:DownwardReflectionI0} appears towards the end of the proof.}

\begin{theorem}[\( \AC \)] \label{thm:DownwardReflectionI0}
Let $j$ be a \emph{proper} elementary embedding witnessing \( \mathsf{I0}(\lambda) \), and let $\lambda_0 = \crt(j)<\lambda$. Then there exists a club $C$ in $\lambda_0$ such that for all $\lambda'\in C$ with $\cf(\lambda')=\omega$ and for all $\alpha<\lambda'$, 
\[
(L_\alpha(V_{\lambda'+1}),{\in},\lambda')\equiv (L_\alpha(V_{\lambda+1}),{\in},\lambda),
\]
where as usual \( \equiv \) denotes elementary equivalence.
\end{theorem}


%

\begin{corollary}[\( \AC \)] \label{cor:DownwardReflectionI0b}
Let $j$ be a \emph{proper} elementary embedding witnessing \( \mathsf{I0}(\lambda) \), and let $\lambda_0 = \crt(j) <\lambda$. Let $\alpha<\lambda_0$, and let $\upvarphi(u)$ be a first-order formula in the language of set theory such that $(L_\alpha(V_{\lambda+1}),{\in}) \models \upvarphi[\lambda/u]$. Then there is  $\lambda' < \lambda$ such that \( \cf(\lambda') = \omega \), $\alpha < \lambda' < \lambda_0$, \( 2^{< \lambda'} = \lambda' \), and $(L_\alpha(V_{\lambda'+1}),{\in}) \models \upvarphi[\lambda'/u]$. 
\end{corollary}

\begin{proof}
 The only part that is not immediately covered by Theorem~\ref{thm:DownwardReflectionI0} is \( 2^{< \lambda'} = \lambda' \), but this holds as it is expressible  in $(L_\alpha(V_{\lambda+1}),{\in},\lambda)$ by a first-order formula using $\lambda$ as a parameter.
\end{proof}

In order to apply Corollary~\ref{cor:DownwardReflectionI0b} in our context, we first need to understand how the property ``$Z$ has the $\lPSP$'' can be expressed in first-order logic, and for which ordinals $\alpha$ such property is true in $L_\alpha(V_{\lambda+1})$. (Exceptionally for the results in this section, the next proposition does not require \( \AC \).)


\begin{proposition} \label{prop:PSPformula}
Let \( \lambda \) be such that \( 2^{< \lambda} = \lambda \), 
$\alpha<\Theta^{L(V_{\lambda+1})}$ be an ordinal different from \( 0 \), and $Z \subseteq B(\lambda) \) be a set in $L_{\alpha}(V_{\lambda+1})$. Then there is a first-order formula $\uppsi(u,v)$ in the language of set theory such that $Z$ has the $\lPSP$ (as a subset of $B(\lambda)$) if and only if $(L_{\alpha}(V_{\lambda+1}),{\in}) \models \uppsi[\lambda/u,Z/v]$.
\end{proposition}

\begin{proofsketch}
By the argument in Remark~\ref{rem:PSPinandout}, if \( Z \in L_\alpha(V_{\lambda+1})\) has the \( \lPSP \), then the witnesses of this are definable from a corresponding function \( \phi \colon \lambda \to \lambda \), which necessarily belongs to \( V_{\lambda+1} \); therefore they belong to \( L_1(V_{\lambda+1}) \subseteq L_\alpha(V_{\lambda+1}) \), and witness that \( Z \) has the \( \lPSP \) also in \( L_\alpha(V_{\lambda+1}) \). We then get that \( Z \) has the \( \lPSP \) (in \( V \) or, equivalently, in \( L(V_{\lambda+1}) \)) if and only if \( L_\alpha(V_{\lambda+1}) \) satisfies the sentence
\begin{center}
\hspace{-1cm}``\( \exists \phi \in \pre{\lambda}{\lambda} \, [\)\( \phi \) codes either a continuous embedding from \( B(\lambda) \) into \( Z \), 

\hfill or a sequence \( (x_\beta)_{\beta < \lambda} \) enumerating \( Z ]\)''.\quad\quad\quad\
\end{center}
The latter can be easily formalized by a first-order formula in the language of set theory using \( \lambda \) and \( Z \) as parameters, hence we are done.
\end{proofsketch}


\begin{corollary}[\( \AC \)] \label{cor:DownwardReflectionI0b2}
Assume that $\mathsf{I0}(\lambda)$ holds. Then there is $\lambda'<\lambda$ such that \( \cf(\lambda') = \omega \), \( 2^{< \lambda'} = \lambda' \), and all the $\lambda'$-projective sets in any $\lambda'$-Polish space have the $\lambda'$-\( \mathrm{PSP} \). 
\end{corollary}

\begin{proof}
If $\mathsf{I0}(\lambda)$ holds, then there is a \emph{proper} elementary embedding \( j \) witnessing this (see the proof of~\cite[Lemma 5]{Woodin2011} or~\cite[Theorem 5.12 and Definition 5.13]{Dimonte2018}). This enables us to use, in what follows, Corollary~\ref{cor:DownwardReflectionI0b} with \( \alpha = 1 \), so that \( \alpha < \crt(j) \) is satisfied (independently of the choice of \( j \)).

Let $\uppsi(u,v)$ be as in Proposition~\ref{prop:PSPformula}, and let 
\( \upvarphi(u) \) be the formalization of 
\begin{center}
``\( \forall v \, [\,\)if every element of \( v \) is a function from \( \omega \) to \( u \), then \( \uppsi(u,v)] \)'',
\end{center}
Since every \( Z \subseteq B(\lambda) \) that belongs to \( L_1(V_{\lambda+1}) \) has the \( \lPSP \)  (Corollary~\ref{cor:allsetshavelpSP}), we have \( (L_1(V_{\lambda+1}),{\in}) \models \upvarphi[\lambda/u] \) by Proposition~\ref{prop:PSPformula}.
Then by Corollary~\ref{cor:DownwardReflectionI0b} there is
$1 < \lambda'<\lambda$ such that \( \cf(\lambda') = \omega \), \( 2^{< \lambda'} = \lambda' \), and $(L_{1}(V_{\lambda'+1}),{\in}) \models\upvarphi[\lambda'/u]$. Hence all the subsets of $B(\lambda')$ belonging to \( L_1(V_{\lambda'+1}) \), including the \( \lambda' \)-projective ones, have the $\lambda'$-\( \mathrm{PSP} \) by Proposition~\ref{prop:PSPformula} again. An application of Corollary~\ref{cor:transferPSP} allows us to transfer the $\lambda'$-\( \mathrm{PSP} \) from \( B(\lambda') \) to any \( \lambda' \)-Polish space, hence we are done. 
\end{proof}

\begin{corollary}[\( \AC \)] \label{cor:NoequivalenceI0PSP}
Assume that $\mathsf{I0}(\lambda)$ holds. Then there is a $\lambda'<\lambda$ such that $\mathsf{I0}(\lambda')$ does not hold, yet all $\lambda'$-projective subsets of any \( \lambda' \)-Polish space have the $\lambda'$-$\mathrm{PSP}$. 
\end{corollary}

\begin{proof}
Let $\bar{\lambda}$ be smallest such that $\mathsf{I0}(\bar{\lambda})$ holds, so that in particular $\bar{\lambda}\leq \lambda$. Let $\lambda' < \bar \lambda$ be as in Corollary~\ref{cor:DownwardReflectionI0b2} applied to \( \bar \lambda \). Then all $\lambda'$-projective subset of any \( \lambda' \)-Polish space have the $\lambda'$-$\mathrm{PSP}$, but $\mathsf{I0}(\lambda')$ does not hold by choice of \( \bar \lambda \). 
\end{proof}

In particular, for a given \( \lambda \) the assertion 
\begin{equation}\tag{\( \star \)}\label{eq:allPSP}
\text{``All $\lambda$-projective subsets of an arbitrary \( \lambda \)-Polish space have the $\lPSP$''}
\end{equation}
is not equivalent to $\mathsf{I0}(\lambda)$ in \( \ZFC \) alone.
Nevertheless, the construction leading to Corollary~\ref{cor:NoequivalenceI0PSP} assumes $\mathsf{I0}(\lambda)$, so equiconsistency is not ruled out. In fact, the analysis in~\cite{Laver2001} shows that even the equivalence with very large cardinal properties is not ruled out. 
Indeed, if \( j \) is an elementary embedding witnessing \( \mathsf{I0}(\lambda) \) and \( \lambda_0 = \crt(j) \), then for every \( \alpha < \lambda_0 \) there is an elementary embedding \( j_\alpha \colon L_\alpha(V_{\lambda+1})\to L_\alpha(V_{\lambda+1}) \)
(see e.g.\ \cite[Theorem 3(ii)]{Laver2001} for a much stronger proposition). The existence of such elementary embeddings \( j_\alpha \) may be seen as a weak (but still high in the consistency hierarchy) form of $\mathsf{I0}$. We want to show that such property is retained at the level of the cardinal \( \lambda' \) appearing in the proof of Corollary~\ref{cor:NoequivalenceI0PSP}: to this aim, 
let us go once again over such argument.

In order to simplify the notation, suppose that \( \lambda \) itself is smallest such that \( \mathsf{I0}(\lambda) \) holds, as witnessed by the elementary embedding \( j \). Let also \( \lambda_0 = \crt(j) \), and let \( j_\alpha \colon L_\alpha(V_{\lambda+1})\to L_\alpha(V_{\lambda+1}) \), for \( \alpha < \lambda_0 \), be the elementary embeddings isolated by Laver. By~\cite[Theorem 2(v)]{Laver2001}, each $j_\alpha$ is defined from $j \restriction V_\lambda$: since the latter belongs to $V_{\lambda+1}$, we have that $j_\alpha\in L_{\alpha+1}(V_{\lambda+1})$ and that 
\[
L_{\alpha+1}(V_{\lambda+1}) \models \text{``there is an elementary embedding } j_\alpha \colon L_\alpha(V_{\lambda+1})\to L_\alpha(V_{\lambda+1})\text{''}.
\]
Since $\On^{L_{\alpha+1}(V_{\lambda+1})}=\lambda+1+\alpha+1$, the ordinal $\alpha$ is definable in $L_{\alpha+1}(V_{\lambda+1})$ from $\lambda$. Also, $V_{\lambda+1}$ is easily definable inside $L_\alpha(V_{\lambda+1})$.
Moreover, we can build $L_\alpha(V_{\lambda+1})$ as computed inside $L_{\alpha+1}(V_{\lambda+1})$, and by absoluteness the resulting set will be $L_\alpha(V_{\lambda+1})$ itself. Therefore $L_\alpha(V_{\lambda+1})$ is definable in $L_{\alpha+1}(V_{\lambda+1})$ from $\lambda$, and thus the first-order sentence formalizing 
\begin{center}
``there is an elementary embedding $j_\alpha \colon L_\alpha(V_{\lambda+1})\to L_\alpha(V_{\lambda+1})$''
\end{center}
uses only $\lambda$ as a parameter.
Recall that the $\lambda'<\lambda$ in Corollary~\ref{cor:NoequivalenceI0PSP} comes from an application of Theorem~\ref{thm:DownwardReflectionI0}. Therefore $\lambda'<\lambda_0$, and $(L_\alpha(V_{\lambda'+1}),{\in},\lambda')\equiv (L_\alpha(V_{\lambda+1}),{\in},\lambda)$ for all $\alpha<\lambda'$. It follows that for all $\alpha<\lambda'<\lambda_0$ 
\[
L_{\alpha+1}(V_{\lambda'+1}) \models \text{``there is an elementary embedding } j_\alpha \colon L_\alpha(V_{\lambda'+1})\to L_\alpha(V_{\lambda'+1})\text{''}.
\]
So while $\mathsf{I0}(\lambda')$ does not hold, still the fragment of $\mathsf{I0}(\lambda')$ corresponding to the existence of the elementary embeddings \( j_\alpha \)	still holds. 
This means that, as claimed, the proof of Corollary~\ref{cor:NoequivalenceI0PSP} cannot rule out that the statement~\eqref{eq:allPSP} is equivalent to weak forms of $\mathsf{I0}$.

\backmatter
\bibliographystyle{amsalpha}
\bibliography{GeneralizedDSTunderI0}

@book{Moschovakis2009,
    AUTHOR = {Moschovakis, Yiannis N.},
     TITLE = {Descriptive set theory},
    SERIES = {Mathematical Surveys and Monographs},
    VOLUME = {155},
   EDITION = {Second},
 PUBLISHER = {American Mathematical Society, Providence, RI},
      YEAR = {2009},
     PAGES = {xiv+502},
      ISBN = {978-0-8218-4813-5},
   MRCLASS = {03-02 (03E15 54H05)},
  MRNUMBER = {2526093},
       DOI = {10.1090/surv/155},
       URL = {https://doi-org.bibliopass.unito.it/10.1090/surv/155},
}

@article {Roy1968,
    AUTHOR = {Roy, Prabir},
     TITLE = {Nonequality of dimensions for metric spaces},
   JOURNAL = {Trans. Amer. Math. Soc.},
  FJOURNAL = {Transactions of the American Mathematical Society},
    VOLUME = {134},
      YEAR = {1968},
     PAGES = {117--132},
      ISSN = {0002-9947},
   MRCLASS = {54.70},
  MRNUMBER = {227960},
MRREVIEWER = {E. Hemmingsen},
       DOI = {10.2307/1994832},
       URL = {https://doi-org.bibliopass.unito.it/10.2307/1994832},
}

@article {HytKul2015,
    AUTHOR = {Hyttinen, Tapani and Kulikov, Vadim},
     TITLE = {On {$\Sigma^1_1$}-complete equivalence relations on the
              generalized {B}aire space},
   JOURNAL = {MLQ Math. Log. Q.},
  FJOURNAL = {MLQ. Mathematical Logic Quarterly},
    VOLUME = {61},
      YEAR = {2015},
    NUMBER = {1-2},
     PAGES = {66--81},
      ISSN = {0942-5616},
   MRCLASS = {03E15 (03C45 03C55 03E10 03E55)},
  MRNUMBER = {3319714},
MRREVIEWER = {Jan Kraszewski},
       DOI = {10.1002/malq.201200063},
       URL = {https://doi.org/10.1002/malq.201200063},
}

@article {FKK2016,
    AUTHOR = {Friedman, Sy David and Khomskii, Yurii and Kulikov, Vadim},
     TITLE = {Regularity properties on the generalized reals},
   JOURNAL = {Ann. Pure Appl. Logic},
  FJOURNAL = {Annals of Pure and Applied Logic},
    VOLUME = {167},
      YEAR = {2016},
    NUMBER = {4},
     PAGES = {408--430},
      ISSN = {0168-0072},
   MRCLASS = {03E15 (03E05 03E30 03E40)},
  MRNUMBER = {3459046},
MRREVIEWER = {Arnold W. Miller},
       DOI = {10.1016/j.apal.2016.01.001},
       URL = {https://doi.org/10.1016/j.apal.2016.01.001},
}

@incollection {HytKul2018,
    AUTHOR = {Hyttinen, Tapani and Kulikov, Vadim},
     TITLE = {{${\rm Borel}^\ast$} sets in the generalized {B}aire space and
              infinitary languages},
 BOOKTITLE = {Jaakko {H}intikka on knowledge and game-theoretical semantics},
    SERIES = {Outst. Contrib. Log.},
    VOLUME = {12},
     PAGES = {395--412},
 PUBLISHER = {Springer, Cham},
      YEAR = {2018},
   MRCLASS = {03E47},
  MRNUMBER = {3753133},
}

@article{Cramer2015,
    AUTHOR = {Cramer, Scott S.},
     TITLE = {Inverse limit reflection and the structure of
              {$L(V_{\lambda+1})$}},
   JOURNAL = {J. Math. Log.},
  FJOURNAL = {Journal of Mathematical Logic},
    VOLUME = {15},
      YEAR = {2015},
    NUMBER = {1},
     PAGES = {1550001, 38},
      ISSN = {0219-0613},
   MRCLASS = {03E55},
  MRNUMBER = {3367775},
MRREVIEWER = {Vincenzo Dimonte},
       DOI = {10.1142/S0219061315500014},
       URL = {https://doi.org/10.1142/S0219061315500014},
}

@incollection {Cramer2017,
    AUTHOR = {Cramer, Scott},
     TITLE = {Implications of very large cardinals},
 BOOKTITLE = {Foundations of mathematics},
    SERIES = {Contemp. Math.},
    VOLUME = {690},
     PAGES = {225--257},
 PUBLISHER = {Amer. Math. Soc., Providence, RI},
      YEAR = {2017},
   MRCLASS = {03E55 (03E60)},
  MRNUMBER = {3656314},
MRREVIEWER = {Mohammad Golshani},
       DOI = {10.1090/conm/690},
       URL = {https://doi.org/10.1090/conm/690},
}

@article{Dimonte2018,
    AUTHOR = {Dimonte, Vincenzo},
     TITLE = {I0 and rank-into-rank axioms},
   JOURNAL = {Boll. Unione Mat. Ital.},
  FJOURNAL = {Bollettino dell'Unione Matematica Italiana},
    VOLUME = {11},
      YEAR = {2018},
    NUMBER = {3},
     PAGES = {315--361},
      ISSN = {1972-6724},
   MRCLASS = {03E55 (03E05 03E35 03E45)},
  MRNUMBER = {3855762},
MRREVIEWER = {Jing Zhang},
       DOI = {10.1007/s40574-017-0136-y},
       URL = {https://doi.org/10.1007/s40574-017-0136-y},
}

@article {LS,
    AUTHOR = {L{\"u}cke, Philipp and Schlicht, Philipp},
     TITLE = {Continuous images of closed sets in generalized {B}aire
              spaces},
   JOURNAL = {Israel J. Math.},
  FJOURNAL = {Israel Journal of Mathematics},
    VOLUME = {209},
      YEAR = {2015},
    NUMBER = {1},
     PAGES = {421--461},
      ISSN = {0021-2172},
   MRCLASS = {03E15 (03E35)},
  MRNUMBER = {3430247},
}

@book {Gao2009,
    AUTHOR = {Gao, Su},
     TITLE = {Invariant descriptive set theory},
    SERIES = {Pure and Applied Mathematics (Boca Raton)},
    VOLUME = {293},
 PUBLISHER = {CRC Press, Boca Raton, FL},
      YEAR = {2009},
     PAGES = {xiv+383},
      ISBN = {978-1-58488-793-5},
   MRCLASS = {03-02 (03E15 22E05 28A05 28D05 37A20 54A35)},
  MRNUMBER = {2455198},
MRREVIEWER = {\'{E}tienne Matheron},
}

@article {DzaVaa,
    AUTHOR = {D{\v{z}}amonja, Mirna and V\"a\"an\"anen, Jouko},
     TITLE = {Chain models, trees of singular cardinality and dynamic
              {EF}-games},
   JOURNAL = {J. Math. Log.},
  FJOURNAL = {Journal of Mathematical Logic},
    VOLUME = {11},
      YEAR = {2011},
    NUMBER = {1},
     PAGES = {61--85},
}

@article {MekVaa1993,
    AUTHOR = {Mekler, Alan and V\"a\"an\"anen, Jouko},
     TITLE = {Trees and {$\Pi^1_1$}-subsets of {$^{\omega_1}\omega_1$}},
   JOURNAL = {J. Symbolic Logic},
  FJOURNAL = {The Journal of Symbolic Logic},
    VOLUME = {58},
      YEAR = {1993},
    NUMBER = {3},
     PAGES = {1052--1070},
      ISSN = {0022-4812,1943-5886},
   MRCLASS = {03E15 (03E35 54H05)},
  MRNUMBER = {1242054},
MRREVIEWER = {Eva\ Coplakova},
       DOI = {10.2307/2275112},
       URL = {https://doi.org/10.2307/2275112},
}

@article {Vaught1974,
    AUTHOR = {Vaught, Robert},
     TITLE = {Invariant sets in topology and logic},
   JOURNAL = {Fund. Math.},
  FJOURNAL = {Polska Akademia Nauk. Fundamenta Mathematicae},
    VOLUME = {82},
      YEAR = {1974/75},
     PAGES = {269--294},
      ISSN = {0016-2736,1730-6329},
   MRCLASS = {04A15 (02B25)},
  MRNUMBER = {363912},
MRREVIEWER = {Andreas\ Blass},
       DOI = {10.4064/fm-82-3-269-294},
       URL = {https://doi.org/10.4064/fm-82-3-269-294},
}

@article{MottoRos2017h,
    AUTHOR = {Motto Ros, Luca},
     TITLE = {Can we classify complete metric spaces up to isometry?},
   JOURNAL = {Boll. Unione Mat. Ital.},
  FJOURNAL = {Bollettino dell'Unione Matematica Italiana},
    VOLUME = {10},
      YEAR = {2017},
    NUMBER = {3},
     PAGES = {369--410},
      ISSN = {1972-6724},
   MRCLASS = {03E15 (54E50)},
  MRNUMBER = {3691805},
MRREVIEWER = {Benjamin Vejnar},
       DOI = {10.1007/s40574-017-0125-1},
       URL = {https://doi.org/10.1007/s40574-017-0125-1},
}

@article{Friedman:2011nx,
	AUTHOR = {Friedman, Sy-David and Hyttinen, Tapani and Kulikov, Vadim},
     TITLE = {Generalized descriptive set theory and classification theory},
   JOURNAL = {Mem. Amer. Math. Soc.},
  FJOURNAL = {Memoirs of the American Mathematical Society},
    VOLUME = {230},
      YEAR = {2014},
    NUMBER = {1081},
     PAGES = {vi+80},
      ISSN = {0065-9266},
      ISBN = {978-0-8218-9475-0},
   MRCLASS = {03C45 (03C55 03C75 03E15 03E35 03E47 54H05)},
  MRNUMBER = {3235820},
MRREVIEWER = {Philipp L{\"u}cke},}

@article {Moreno2017,
    AUTHOR = {Hyttinen, Tapani and Kulikov, Vadim and Moreno, Miguel},
     TITLE = {A generalized {B}orel-reducibility counterpart of {S}helah's
              main gap theorem},
   JOURNAL = {Arch. Math. Logic},
  FJOURNAL = {Archive for Mathematical Logic},
    VOLUME = {56},
      YEAR = {2017},
    NUMBER = {3-4},
     PAGES = {175--185},
      ISSN = {0933-5846},
   MRCLASS = {03E15 (03C45)},
  MRNUMBER = {3633791},
MRREVIEWER = {Sebastien Vasey},
       DOI = {10.1007/s00153-017-0521-3},
       URL = {https://doi.org/10.1007/s00153-017-0521-3},
}

@article {Mildenberger2019,
    AUTHOR = {Calderoni, Filippo and Mildenberger, Heike and Motto Ros,
              Luca},
     TITLE = {Uncountable structures are not classifiable up to
              bi-embeddability},
   JOURNAL = {J. Math. Log.},
  FJOURNAL = {Journal of Mathematical Logic},
    VOLUME = {20},
      YEAR = {2020},
    NUMBER = {1},
     PAGES = {2050001, 49},
      ISSN = {0219-0613},
   MRCLASS = {03E15 (03C75)},
  MRNUMBER = {4094553},
MRREVIEWER = {Arnold W. Miller},
       DOI = {10.1142/S0219061320500014},
       URL = {https://doi.org/10.1142/S0219061320500014},
}

@article {Mangraviti2018,
    AUTHOR = {Mangraviti, Francesco and Motto Ros, Luca},
     TITLE = {A descriptive main gap theorem},
   JOURNAL = {J. Math. Log.},
  FJOURNAL = {Journal of Mathematical Logic},
    VOLUME = {21},
      YEAR = {2021},
    NUMBER = {1},
     PAGES = {Paper No. 2050025, 40},
      ISSN = {0219-0613,1793-6691},
   MRCLASS = {03E15 (03C45)},
  MRNUMBER = {4194559},
MRREVIEWER = {Miguel\ Moreno},
       DOI = {10.1142/S0219061320500257},
       URL = {https://doi.org/10.1142/S0219061320500257},
}

@article {AM,
    AUTHOR = {Andretta, Alessandro and Motto Ros, Luca},
     TITLE = {Souslin quasi-orders and bi-embeddability of uncountable
              structures},
   JOURNAL = {Mem. Amer. Math. Soc.},
  FJOURNAL = {Memoirs of the American Mathematical Society},
    VOLUME = {277},
      YEAR = {2022},
    NUMBER = {1365},
     PAGES = {vii+189},
      ISSN = {0065-9266,1947-6221},
      ISBN = {978-1-4704-5273-5; 978-1-4704-7097-5},
   MRCLASS = {03E15 (03E10 03E45 03E47 03E60)},
  MRNUMBER = {4409720},
MRREVIEWER = {Miroslav\ Repick\'y},
       DOI = {10.1090/memo/1365},
       URL = {https://doi.org/10.1090/memo/1365},
}

@article {AgoMotSch,
    AUTHOR = {Agostini, Claudio and Motto Ros, Luca and Schlicht, Philipp},
     TITLE = {Generalized {P}olish spaces at regular uncountable cardinals},
   JOURNAL = {J. Lond. Math. Soc. (2)},
  FJOURNAL = {Journal of the London Mathematical Society. Second Series},
    VOLUME = {108},
      YEAR = {2023},
    NUMBER = {5},
     PAGES = {1886--1929},
      ISSN = {0024-6107,1469-7750},
   MRCLASS = {03E15 (54H05)},
  MRNUMBER = {4668519},
MRREVIEWER = {Jaros\l aw\ Swaczyna},
       DOI = {10.1112/jlms.12797},
       URL = {https://doi.org/10.1112/jlms.12797},
}

@unpublished{MorenoUnp,
Author={Moreno, M.},
Note = {Preprint (arXiv:2308.07510), 2024.},
Title = {Shelah's {M}ain {G}ap and the generalized {B}orel-reducibility}}

@unpublished{AgoMot,
Author={Agostini, C. and Motto Ros, L.},
Note = {Preprint.},
Title = {Generalized {P}olish spaces at all uncountable cardinals}}

@article{Motto-Ros:2011qc,
	Author = {Motto Ros, L.},
	Journal = {Annals of Pure and Applied Logic},
	Number = {12},
	Pages = {1454--1492},
	Title = {The descriptive set-theoretical complexity of the embeddability relation on models of large size},
	Volume = {164},
	Year = {2013}}

@incollection {Katetov1986,
    AUTHOR = {Kat\v{e}tov, M.},
     TITLE = {On universal metric spaces},
 BOOKTITLE = {General topology and its relations to modern analysis and
              algebra, {VI} ({P}rague, 1986)},
    SERIES = {Res. Exp. Math.},
    VOLUME = {16},
     PAGES = {323--330},
 PUBLISHER = {Heldermann, Berlin},
      YEAR = {1988},
   MRCLASS = {54E35 (54A25 54E40)},
  MRNUMBER = {952617},
MRREVIEWER = {Stoyan Nedev},
}

@book {Engelking1978,
    AUTHOR = {Engelking, Ryszard},
     TITLE = {Dimension theory},
      NOTE = {Translated from the Polish and revised by the author,
              North-Holland Mathematical Library, 19},
 PUBLISHER = {North-Holland Publishing Co., Amsterdam-Oxford-New York;
              PWN---Polish Scientific Publishers, Warsaw},
      YEAR = {1978},
     PAGES = {x+314 pp. (loose errata)},
      ISBN = {0-444-85176-3},
   MRCLASS = {54FXX},
  MRNUMBER = {0482697},
MRREVIEWER = {Teodor Przymusinski},
}

@article {Stone1962,
    AUTHOR = {Stone, A. H.},
     TITLE = {Non-separable {B}orel sets},
   JOURNAL = {Rozprawy Mat.},
  FJOURNAL = {Rozprawy Matematyczne},
    VOLUME = {28},
      YEAR = {1962},
     PAGES = {41},
      ISSN = {0860-2581},
   MRCLASS = {04.40},
  MRNUMBER = {152457},
MRREVIEWER = {Djuro Kurepa},
}

@article {Stone1972,
    AUTHOR = {Stone, A. H.},
     TITLE = {Non-separable {B}orel sets. {II}},
   JOURNAL = {General Topology and Appl.},
  FJOURNAL = {General Topology and its Applications},
    VOLUME = {2},
      YEAR = {1972},
     PAGES = {249--270},
      ISSN = {0016-660X},
   MRCLASS = {54H05 (04A15 28A05)},
  MRNUMBER = {312481},
MRREVIEWER = {Djuro Kurepa},
}

@article {Rudin1969,
    AUTHOR = {Rudin, Mary Ellen},
     TITLE = {A new proof that metric spaces are paracompact},
   JOURNAL = {Proc. Amer. Math. Soc.},
  FJOURNAL = {Proceedings of the American Mathematical Society},
    VOLUME = {20},
      YEAR = {1969},
     PAGES = {603},
      ISSN = {0002-9939},
   MRCLASS = {54.50},
  MRNUMBER = {236876},
MRREVIEWER = {C. E. Aull},
       DOI = {10.2307/2035708},
       URL = {https://doi-org.bibliopass.unito.it/10.2307/2035708},
}

@article {Hansell1971,
    AUTHOR = {Hansell, R. W.},
     TITLE = {Borel measurable mappings for nonseparable metric spaces},
   JOURNAL = {Trans. Amer. Math. Soc.},
  FJOURNAL = {Transactions of the American Mathematical Society},
    VOLUME = {161},
      YEAR = {1971},
     PAGES = {145--169},
      ISSN = {0002-9947},
   MRCLASS = {28.20 (54.00)},
  MRNUMBER = {288228},
MRREVIEWER = {C. J. Himmelberg},
       DOI = {10.2307/1995934},
       URL = {https://doi-org.bibliopass.unito.it/10.2307/1995934},
}

@article {Han72,
    AUTHOR = {Hansell, R. W.},
     TITLE = {On the representation of nonseparable analytic sets},
   JOURNAL = {Proc. Amer. Math. Soc.},
  FJOURNAL = {Proceedings of the American Mathematical Society},
    VOLUME = {39},
      YEAR = {1973},
     PAGES = {402--408},
      ISSN = {0002-9939,1088-6826},
   MRCLASS = {54H05},
  MRNUMBER = {380752},
MRREVIEWER = {Djuro\ Kurepa},
       DOI = {10.2307/2039654},
       URL = {https://doi.org/10.2307/2039654},
}

@article {Han73,
    AUTHOR = {Hansell, R. W.},
     TITLE = {On the nonseparable theory of {B}orel and {S}ouslin sets},
   JOURNAL = {Bull. Amer. Math. Soc.},
  FJOURNAL = {Bulletin of the American Mathematical Society},
    VOLUME = {78},
      YEAR = {1972},
     PAGES = {236--241},
      ISSN = {0002-9904},
   MRCLASS = {04A15 (28A05)},
  MRNUMBER = {294138},
MRREVIEWER = {Mary\ Powderly},
       DOI = {10.1090/S0002-9904-1972-12936-7},
       URL = {https://doi.org/10.1090/S0002-9904-1972-12936-7},
}

@article {Hansell1974,
    AUTHOR = {Hansell, R. W.},
     TITLE = {On {B}orel mappings and {B}aire functions},
   JOURNAL = {Trans. Amer. Math. Soc.},
  FJOURNAL = {Transactions of the American Mathematical Society},
    VOLUME = {194},
      YEAR = {1974},
     PAGES = {195--211},
      ISSN = {0002-9947},
   MRCLASS = {54H05 (54C50)},
  MRNUMBER = {362270},
MRREVIEWER = {A. H. Stone},
       DOI = {10.2307/1996801},
       URL = {https://doi-org.bibliopass.unito.it/10.2307/1996801},
}

@article {Hansell1983,
    AUTHOR = {Hansell, Roger W.},
     TITLE = {Hereditarily additive families in descriptive set theory and
              {B}orel measurable multimaps},
   JOURNAL = {Trans. Amer. Math. Soc.},
  FJOURNAL = {Transactions of the American Mathematical Society},
    VOLUME = {278},
      YEAR = {1983},
    NUMBER = {2},
     PAGES = {725--749},
      ISSN = {0002-9947},
   MRCLASS = {54H05 (04A15 28A05)},
  MRNUMBER = {701521},
MRREVIEWER = {John E. Jayne},
       DOI = {10.2307/1999181},
       URL = {https://doi-org.bibliopass.unito.it/10.2307/1999181},
}

@article {Sierpinski1945,
    AUTHOR = {Sierpi\'{n}ski, Waclaw},
     TITLE = {Sur les espaces m\'{e}triques universels},
   JOURNAL = {Fund. Math.},
  FJOURNAL = {Polska Akademia Nauk. Fundamenta Mathematicae},
    VOLUME = {33},
      YEAR = {1945},
     PAGES = {123--136},
      ISSN = {0016-2736},
   MRCLASS = {27.2X},
  MRNUMBER = {15452},
       DOI = {10.4064/fm-33-1-123-136},
       URL = {https://doi.org/10.4064/fm-33-1-123-136},
}

@BOOK{Kechris1995,
  author       = {Kechris, Alexander S.},
  title        = {Classical descriptive set theory},
  series       = {Graduate Texts in Mathematics},
  volume       = {156},
  publisher    = {Springer-Verlag, New York},
  year         = {1995},
  pages        = {xviii+402},
  isbn         = {0-387-94374-9},
  mrclass      = {03E15 (03-01 03-02 04A15 28A05 54H05 90D44)},
  mrnumber     = {1321597 (96e:03057)},
  mrreviewer   = {Jakub Jasi{\'n}ski},
  doi          = {10.1007/978-1-4612-4190-4},
  url          = {http://dx.doi.org/10.1007/978-1-4612-4190-4},
}

@article{Shi2015,
    AUTHOR = {Shi, Xianghui},
     TITLE = {Axiom {$I_0$} and higher degree theory},
   JOURNAL = {J. Symb. Log.},
  FJOURNAL = {The Journal of Symbolic Logic},
    VOLUME = {80},
      YEAR = {2015},
    NUMBER = {3},
     PAGES = {970--1021},
      ISSN = {0022-4812},
   MRCLASS = {03E55 (03D30 03E45 03E57 03E60)},
  MRNUMBER = {3395357},
MRREVIEWER = {Vincenzo Dimonte},
       DOI = {10.1017/jsl.2015.15},
       URL = {https://doi.org/10.1017/jsl.2015.15},
}

@article {MMS2016,
    AUTHOR = {L\"{u}cke, Philipp and Motto Ros, Luca and Schlicht, Philipp},
     TITLE = {The {H}urewicz dichotomy for generalized {B}aire spaces},
   JOURNAL = {Israel J. Math.},
  FJOURNAL = {Israel Journal of Mathematics},
    VOLUME = {216},
      YEAR = {2016},
    NUMBER = {2},
     PAGES = {973--1022},
      ISSN = {0021-2172},
   MRCLASS = {03E15 (03E35)},
  MRNUMBER = {3557473},
MRREVIEWER = {Giorgio Laguzzi},
       DOI = {10.1007/s11856-016-1435-1},
       URL = {https://doi.org/10.1007/s11856-016-1435-1},
}

@book {Jech2003,
    AUTHOR = {Jech, Thomas},
     TITLE = {Set theory},
    SERIES = {Springer Monographs in Mathematics},
      NOTE = {The third millennium edition, revised and expanded},
 PUBLISHER = {Springer-Verlag, Berlin},
      YEAR = {2003},
     PAGES = {xiv+769},
      ISBN = {3-540-44085-2},
   MRCLASS = {03Exx (03-01 03-02)},
  MRNUMBER = {1940513},
MRREVIEWER = {Eva Coplakova},
}

@article {Woodin2011,
    AUTHOR = {Woodin, W. Hugh},
     TITLE = {Suitable extender models {II}: beyond {$\omega$}-huge},
   JOURNAL = {J. Math. Log.},
  FJOURNAL = {Journal of Mathematical Logic},
    VOLUME = {11},
      YEAR = {2011},
    NUMBER = {2},
     PAGES = {115--436},
      ISSN = {0219-0613},
   MRCLASS = {03E35 (03E25 03E45 03E55 03E60)},
  MRNUMBER = {2914848},
MRREVIEWER = {A. Kanamori},
       DOI = {10.1142/S021906131100102X},
       URL = {https://doi.org/10.1142/S021906131100102X},
}

@inproceedings {Gaifman1974,
    AUTHOR = {Gaifman, Haim},
     TITLE = {Elementary embeddings of models of set-theory and certain
              subtheories},
 BOOKTITLE = {Axiomatic set theory ({P}roc. {S}ympos. {P}ure {M}ath., {V}ol.
              {XIII}, {P}art {II}, {U}niv. {C}alifornia, {L}os {A}ngeles,
              {C}alif., 1967)},
     PAGES = {33--101},
      YEAR = {1974},
   MRCLASS = {02K15 (02H13)},
  MRNUMBER = {0376347},
MRREVIEWER = {L. Bukovsky},
}

@article {Powell1974,
    AUTHOR = {Powell, William C.},
     TITLE = {Variations of {K}eisler's theorem for complete embeddings},
   JOURNAL = {Fund. Math.},
  FJOURNAL = {Polska Akademia Nauk. Fundamenta Mathematicae},
    VOLUME = {81},
      YEAR = {1973/74},
    NUMBER = {2},
     PAGES = {121--132},
      ISSN = {0016-2736},
   MRCLASS = {02K15},
  MRNUMBER = {342394},
MRREVIEWER = {M. Boffa},
       DOI = {10.4064/fm-81-2-121-132},
       URL = {https://doi.org/10.4064/fm-81-2-121-132},
}

@book {Kanamori,
    AUTHOR = {Kanamori, Akihiro},
     TITLE = {The higher infinite},
    SERIES = {Springer Monographs in Mathematics},
   EDITION = {Second},
      NOTE = {Large cardinals in set theory from their beginnings},
 PUBLISHER = {Springer-Verlag, Berlin},
      YEAR = {2003},
     PAGES = {xxii+536},
      ISBN = {3-540-00384-3},
   MRCLASS = {03E55 (01A55 01A60 03-02 03-03)},
  MRNUMBER = {1994835},
}

@article {Laver1997,
    AUTHOR = {Laver, Richard},
     TITLE = {Implications between strong large cardinal axioms},
   JOURNAL = {Ann. Pure Appl. Logic},
  FJOURNAL = {Annals of Pure and Applied Logic},
    VOLUME = {90},
      YEAR = {1997},
    NUMBER = {1-3},
     PAGES = {79--90},
      ISSN = {0168-0072},
   MRCLASS = {03E55 (03E35)},
  MRNUMBER = {1489305},
MRREVIEWER = {Douglas R. Burke},
       DOI = {10.1016/S0168-0072(97)00031-6},
       URL = {https://doi.org/10.1016/S0168-0072(97)00031-6},
}

@incollection {Neeman2010,
    AUTHOR = {Neeman, Itay},
     TITLE = {Determinacy in {$L(\Bbb R)$}},
 BOOKTITLE = {Handbook of set theory. {V}ols. 1, 2, 3},
     PAGES = {1877--1950},
 PUBLISHER = {Springer, Dordrecht},
      YEAR = {2010},
   MRCLASS = {03E60 (03E15 03E45 03E55)},
  MRNUMBER = {2768701},
MRREVIEWER = {D. A. Martin},
       DOI = {10.1007/978-1-4020-5764-9\_23},
       URL = {https://doi.org/10.1007/978-1-4020-5764-9_23},
}

@article {MaStee1988,
    AUTHOR = {Martin, Donald A. and Steel, John R.},
     TITLE = {Projective determinacy},
   JOURNAL = {Proc. Nat. Acad. Sci. U.S.A.},
  FJOURNAL = {Proceedings of the National Academy of Sciences of the United
              States of America},
    VOLUME = {85},
      YEAR = {1988},
    NUMBER = {18},
     PAGES = {6582--6586},
      ISSN = {0027-8424},
   MRCLASS = {03E15 (03E55 03E60)},
  MRNUMBER = {959109},
MRREVIEWER = {Thomas J. Jech},
       DOI = {10.1073/pnas.85.18.6582},
       URL = {https://doi.org/10.1073/pnas.85.18.6582},
}

@article {Laver2001,
    AUTHOR = {Laver, Richard},
     TITLE = {Reflection of elementary embedding axioms on the
              {$L[V_{\lambda+1}]$} hierarchy},
   JOURNAL = {Ann. Pure Appl. Logic},
  FJOURNAL = {Annals of Pure and Applied Logic},
    VOLUME = {107},
      YEAR = {2001},
    NUMBER = {1-3},
     PAGES = {227--238},
      ISSN = {0168-0072},
   MRCLASS = {03E55 (03E45)},
  MRNUMBER = {1807846},
MRREVIEWER = {Ralf-Dieter Schindler},
       DOI = {10.1016/S0168-0072(00)00035-X},
       URL = {https://doi.org/10.1016/S0168-0072(00)00035-X},
}

@book {Jech1973,
    AUTHOR = {Jech, Thomas J.},
     TITLE = {The axiom of choice},
    SERIES = {Studies in Logic and the Foundations of Mathematics},
    VOLUME = {Vol. 75},
 PUBLISHER = {North-Holland Publishing Co., Amsterdam-London; American
              Elsevier Publishing Co., Inc., New York},
      YEAR = {1973},
     PAGES = {xi+202},
   MRCLASS = {04A25 (02K05)},
  MRNUMBER = {396271},
}

@book {Larson2004,
    AUTHOR = {Larson, Paul B.},
     TITLE = {The stationary tower},
    SERIES = {University Lecture Series},
    VOLUME = {32},
      NOTE = {Notes on a course by W. Hugh Woodin},
 PUBLISHER = {American Mathematical Society, Providence, RI},
      YEAR = {2004},
     PAGES = {x+132},
      ISBN = {0-8218-3604-8},
   MRCLASS = {03-02 (03E15 03E35 03E40 03E55 03E60)},
  MRNUMBER = {2069032},
MRREVIEWER = {Miroslav\ Repick\'y},
       DOI = {10.1090/ulect/032},
       URL = {https://doi.org/10.1090/ulect/032},
}

@unpublished{ACMRP,
    author ={Agostini, Caudio and Chapman, Nick  and Motto Ros, Luca and Pitton, Beatrice },
    title = {Generalized {B}orel sets},
    note = {In preparation}
}

@unpublished{BDM,
    author ={Barrera, Fernando and Dimonte, Vincenzo and M\"{u}ller, Sandra},
    title = {The $\lambda$-{PSP} at $\lambda$-coanalytic sets},
		Note = {Preprint (arXiv:2504.15675). Submitted}
}
\printindex

\end{document}